\documentclass[11pt]{article}
\usepackage{graphicx} % Required for inserting images
\pdfoutput=1
\usepackage[margin=1.0in]{geometry}
\usepackage{xcolor}

\usepackage{amsmath}
\usepackage{amsthm}
\usepackage{mathrsfs}

\newtheorem{theorem}{Theorem}[section] % numbering by section (optional)

\newtheorem{example}{Example}

\newtheorem{remark}{Remark}[section]
\newtheorem{lemma}{Lemma}[section]

\usepackage{algorithm}
\usepackage{algpseudocode}

\DeclareMathOperator\supp{supp}

\usepackage{mathtools} % optional but nice
\numberwithin{equation}{section} % makes numbers like (2.1), (2.2), ...

\usepackage{indentfirst}
\usepackage{titlesec}
\titleformat{\section}
  {\normalfont\Large\bfseries}
  {\thesection.}
  {0.5em}
  {}
\titleformat{\subsection}
  {\normalfont\large\bfseries}
  {\thesubsection.}
  {0.5em}
  {}

\graphicspath{{./}}
\usepackage{float}
\usepackage{capt-of}
\usepackage{diagbox}
\usepackage{subcaption}

\usepackage{amssymb,latexsym}

\usepackage{setspace}

\usepackage{amsmath}

\usepackage{mathrsfs,amsfonts}
\usepackage{bbm}

\usepackage{amsthm,amsxtra}

\usepackage[final]{showkeys}

\usepackage{stmaryrd}

\usepackage{algorithm}
\usepackage{algorithmicx}
\usepackage{algpseudocode}
\usepackage{etaremune}
\usepackage{enumitem} 

\usepackage[hidelinks]{hyperref}

\usepackage{graphicx} % Required for inserting images
\usepackage{cite}  % Optional for better citation formatting
\newcommand{\meas}{u^\infty_{\text{meas}}}

\newcommand{\Jacobian}{\mathcal{J}}

\newcommand{\floor}[1]{\lfloor {#1} \rfloor}

\title{Iterative Truncated SVD  for Inverse Scattering Problems} %Based on Information Content of the Forward Operator} %with Application to Inverse Scattering Problems}
\author{Zhinan Han \and Ying Liang \and Hongkai Zhao}
\date{}  % Leave empty to remove the date

\begin{document}

\maketitle

\begin{abstract}
We propose an iterative truncated SVD (iTSVD) algorithm based on a nonlinear least squares formulation for Newton's method for inverse scattering problems. The key idea is to estimate the effective rank of the corresponding linearized operator (or the Jacobian matrix after discretization), which defines the dimension of the signal space of the associated forward operator. Once the effective rank is estimated, truncated SVD (TSVD) reconstruction in the signal space can be computed efficiently based on fast random SVD (RSVD) -- only requiring matrix-vector multiplication which can be implemented in a matrix-free formulation by solving multiple forward and adjoint problems of the same type. Our analysis and numerical experiments show that iTSVD can markedly improve computational efficiency by reducing the number of iterations and enhance stability by physics/information based inversion. Moreover, the effective rank estimate is naturally model- and data-based. Extensive experiments, including far- or near-field, limited aperture, and random measurement data without or with noise, are presented.   
\end{abstract}

\section{Introduction}
 Due to ubiquitous applications of inverse problems in science, engineering, technology, and everyday life, they have attracted a lot of research from both mathematical and computational perspectives in the past few decades. Various formulations, computational approaches, and numerical algorithms have been proposed, developed, and studied to deal with the key challenges due to ill-posedness, nonlinearity, and the large scale of the underlying inverse problem to achieve a desired balance of stability vs. accuracy and overall computational efficiency. In real practice, data-driven and problem-specific approaches are crucial in achieving these goals. 

Most popular and generic approaches are based on optimization by designing a cost functional (in continuous formulation) or cost function (in discrete version) without or with constraints that include data fidelity, forward models, and possible regularization and/or prior information. Then various computational methods based on different optimization techniques and fast algorithms are applied to solve the optimization problem. A general inverse problem often requires solving a large-scale nonconvex optimization problem after discretization. Hence, first-order gradient-based methods, such as adjoint methods, are the popular choices instead of second-order methods, for which forming the Hessian matrix and computing its (pseudo) inverse poses serious computational challenges. However, intrinsic ill-posedness of the underlying inverse problem, overparameterization due to oversampling in numerical discretization,  and/or insufficient measurement data will commonly lead to ill-conditioning of the corresponding optimization problems, which will result in slow convergence and hence inefficiency for gradient-based methods. In this work, we propose an effective Newton's method in least-squares formulation, iterative truncated singular value decomposition (iTSVD), based on an effective rank estimate and randomized SVD algorithms to solve an inverse problem.

The effectiveness of iTSVD is based on the following observations.
First, the linearized operator is often a compact operator, e.g., an integral operator the kernel of which is related to the Green's function of the underlying forward problem. Often, the probing sources and measurements are exterior to the domain of interest, which means no singularity for the integral kernel. On the one hand, this is the cause of the ill-posedness of the inverse problem. On the other hand, it may render an approximate low-rank structure that one can take advantage of when solving the linearized problem. In particular, randomized SVD provides an ideal computational tool that takes advantage of the underlying low-rank structure to compute a TSVD solution of the linearized problem with computational efficiency, accuracy, and stability.
Compared to gradient descent-based updates along a single compromised direction, the TSVD solution provides a properly regularized simultaneous update in all significant directions, which span the signal space of the linearized operator, in one step.  
As an important computational note, RSVD only involves matrix-vector multiplications. This may lead to a matrix-free implementation for the RSVD to compute the TSVD solution to the linearized problem, which is crucial for the computational feasibility of using Newton's method in practice.  
%In our case, as we will show in Section~\ref{sec: lstsq}, the application of the linearized operator can be transformed to solving a fixed linearized problem with the current state of parameters/coefficients but different forcing terms, i.e., a set of linear systems with the same matrix but different right-hand sides in computation.
  
%The key point of the method is to have a good characterization and estimate of the effective rank of the linearized operator based on the forward problem and measurement data. 
Hence, the key idea of iTSVD is to exploit an approximate low-rank structure determined together by the forward model and measurement data and obtain an effective rank estimate for the linearized operator. It provides both the basis for fast RSVD algorithms and the appropriate regularization threshold for TSVD solution. As a result, iTSVD accelerates the convergence and reduces the computational cost significantly while improving the stability due to the use of Newton's method, fast RSVD algorithms, and a properly regularized inversion. iTSVD is also naturally data-driven and model-based through an effective rank estimate. In this work, we focus on wave-related inverse problems, e.g., inverse medium scattering.

The inverse medium scattering problem concerns the recovery of spatially varying material properties, such as the refractive index or dielectric permittivity, from measured wave field data resulting from the interaction of incident waves with a heterogeneous medium. Although the forward scattering problem is linear and well-posed, where the wave field is computed for a known medium, the inverse problem is inherently nonlinear and ill-posed. This class of problems plays a critical role in a wide range of applications, including medical imaging \cite{inverse_imaging,medical_imaging}, geophysical exploration \cite{geophysics}, and nondestructive testing \cite{nondestructive_testing}. Despite recent advances in regularization theory, optimization-based methods, and machine learning, fundamental difficulties remain, rooted in the following physical, mathematical, and computational features of the problem:
\begin{enumerate}
    \item While interactions between the incident waves and inhomogeneity can encode significant medium information in the measurements, they may lead to strong nonlinearity in the inverse problem, the strength of which is proportional to the product of the squared wavenumber/frequency and the medium contrast. This reality is the ultimate physical limit to practical stability and resolution for wave-based inverse problems. A good initial guess and/or prior information combined with effective iterative reconstruction can alleviate the difficulty.     
    \item  Accurate simulation of wave fields, which requires a mesh resolving the wavelength and a numerical algorithm that captures wave propagation and scattering correctly and effectively,  poses computational challenges that grow dramatically with the ratio between the domain size and the wavelength, and spatial dimensions. 
    \item Ill-posedness of the inverse problem leads to ill-conditioning and instability for inversion algorithms, which can be further exacerbated by sparse measurements, e.g., due to limited aperture, and noisy data. Desirable approaches should be adaptive to the forward model, measurement configuration, and data noise.
\end{enumerate}
Addressing these challenges is essential for developing efficient, robust, and accurate inversion techniques for complex media. 

 Most often, inverse problems are cast into an optimization problem. 
 % For wave-based inverse problems, the wavelength, which determines the scale of resolution, is usually a fraction of the size of the domain or the object of interest for a desirable resolution.
 For wave-based inverse problems, achieving a desirable resolution typically requires the wavelengths to be small relative to the size of the domain or object of interest. Hence, for the nonlinear inverse problem, a sufficiently fine mesh is required to resolve the wavelength for faithful numerical simulations, leading to a large-scale non-convex optimization problem after discretization. For such problems, gradient-based methods, e.g., the adjoint method, are usually the favorable approach. However, the optimization problem is often ill-conditioned due to the ill-posedness of the underlying inverse problem, oversampled numerical discretization -- requiring a mesh size that is a fraction of the wavelength, and potentially incomplete data. The ill-conditioning of the optimization problem results in slow convergence for first-order gradient-based methods and, hence, requires many iterations in practice.  
 Moreover, each iteration involves solving many forward problems, e.g., computing the residual (measurement mismatch) and backpropagation, which is typically the most computationally expensive part. Hence, an effective iterative method is crucial for computational efficiency.  

 Here we propose to use  Newton's method in the least-squares formulation. The main computational challenge comes from solving the linearized problem. In this work, we show that solving the linearized problem by TSVD is computationally effective and stable by exploiting the spectral property induced by the physical and mathematical properties of the underlying wave scattering operator. The key ingredients are:
 \begin{itemize}
\item 
Estimate of the effective rank of the linearized operator (in continuous space) or the Jacobian matrix (in discrete space). For inverse scattering problems, the linearized operator behaves like a low-pass filter. There is a physical wavelength-induced spectral gap or threshold that naturally defines the signal space. For the far-field scattering pattern, the dimension of the signal space can be related to the effective dimension of the space of “essentially” space and band-limited functions, i.e., the Shannon number~\cite{Shannon_number_estimate}.

\item
Fast SVD algorithm based on the estimate of effective rank. It leads to efficient and stable computation of the TSVD solution of the linearized problem in the signal space. In particular, RSVD only requires the application of the matrix to a set of random vectors whose cardinality is comparable to the effective rank. More importantly, these matrix-vector multiplications can be implemented in a matrix-free formulation which only requires to solve the same forward problem with different forcing terms. 
\end{itemize}

In summary, the TSVD solution based on the effective rank estimate and the corresponding signal space -- determined by the underlying forward model and the measurement data -- provides the optimal (in the least squares sense) linearized reconstruction one can compute stably at the current iteration in one step. In other words, it reduces the need for frequent and incremental updates of the medium -- each update requires solving many forward problems.

We also show an interesting connection to information or function approximation perspectives when far-field scattering data is properly transformed to the Fourier domain. The transform leads to a weighted least squares formulation that has no frequency bias. The effective rank is also related to finding and representing signals/functions that are optimally concentrated in both the spatial and Fourier domains. The effective rank can be estimated by the Landau-Pollak-Slepian "Shannon number" introduced in\cite{Simons2015, Landau_Pollak}. We provide numerical approximations of the "Shannon number" in more general setups and for discrete data. 

%Our numerical results show that iTSVD not only significantly reduces the number of iterations but also enhance the stability markedly compared to gradient based methods. Extensive tests on far-field and near-field data, limited aperture data, sparse random measurements, noisy data.

Another important factor to achieve the overall stability of Newton's method is a good initial guess. For the wave scattering problem, one can take the frequency continuation approach: start from more stable and faster reconstruction at a coarse scale with lower frequency data and use it as the initial guess for reconstruction at a finer scale with higher frequency data as proposed in \cite{Bao_example}. We show the effectiveness of iTSVD by using larger frequency leaps compared to gradient-based methods. 

We use extensive numerical experiments in different setups, including far-field data, near-field data, limited aperture data, sparse random measurements, and noisy data, to demonstrate that iTSVD not only significantly reduces the number of iterations but also enhances the stability markedly compared to gradient methods because the TSVD solution is adaptive to the model and the data. 
 
 Our work is most related to the recent article\cite{zhou2024exploringlowrankstructureinverse}, in which fast linear reconstruction around a homogeneous background is obtained using explicit eigenfunctions and eigenvalues for the far-field scattering operator for a homogeneous medium with full aperture data, computed based on the disk spheroidal wave functions and connection to a Sturm-Liouville operator.  In contrast, iTSVD is an iterative method for nonlinear reconstruction that utilizes effective rank estimates and efficient RSVD algorithms to compute the singular values and singular vectors of the linearized operator corresponding to a heterogeneous medium in general setups. 
 In earlier work \cite{Greengard}, recursive linearization and frequency continuation were used to solve two-dimensional inverse scattering problems. Each linearized problem is solved through a normal matrix system corresponding to a least-squares formulation projected to a fixed bandwidth of Fourier modes. In other words, the solution to each linearized problem is low-pass filtered using Fourier modes. In contrast, iTSVD solves each linearized system through truncated SVD, i.e., low-pass filtered using the true leading singular modes of the linearized operator. Moreover, the matrix is never explicitly formed for each linearized problem in iTSVD.
 % However, the ill-posedness and non-convexity of the optimization problem requires the successive frequency marching and many iterations at each frequency to obtain certain accuracy.
%In this paper, we explore a consistent low rank structure of the far-field operator so that the number of iterations per frequency is massively reduced, and only data at high frequency is needed. (\textcolor{red}{This is a very dangerous statement, need to try higher frequency with some contrast containing low to high frequency})

% 

Here is the organization of this paper. We set up the mathematical formulation for the inverse scattering problem in Section ~\ref{sec:formulation}.   Then, we establish spectral analysis of the forward scattering operator, its connection to the band-limited Fourier transform, and an effective rank estimate in Section ~\ref{sec:spectral}. The iTSVD algorithm is provided in Section ~\ref{sec: algorithm}, and its convergence analysis is in Section ~\ref{sec: analysis}. We present implementation details, complexity analysis, and extensive numerical experiments in Section ~\ref{sec: numerical}. Finally, Section ~\ref{sec: conclusion} is on conclusion remarks and future work.

\section{Forward and  Inverse Scattering Problems}
\label{sec:formulation}
In this section, we introduce the forward medium scattering problems and the least-squares formulations of Newton's method for the inverse problems using far-field and near-field measurements.
\subsection{Forward Scattering Problem}
\label{sec:forward}
Let the domain of interest $\Omega \subset \mathbb{R}^2$ be bounded with Lipschitz boundary $\partial \Omega$. 
The real-valued contrast function $q(x) \in L^\infty(\mathbb{R}^d)$ satisfies $q > -1$ in 
$\Omega$ and $q = 0$ outside $\Omega$. Let $k > 0$ be the wavenumber and denote the incident wave by $u^{in}(x)$, which satisfies the Helmholtz equation in the homogeneous medium. Let $u^t(x)=u^s(x)+u^{in}(x)$ be the total field, where $u^s(x)$ is the scattered field. The forward scattering problem is to find the total field  $u^t(x) \in H^1_{\mathrm{loc}}(\mathbb{R}^2)$ given $u^{in}(x)$ and the medium coefficient $q(x)$ by solving the Helmholtz equation
\begin{equation}\label{eq:total-field}
    \Delta u^t(x) + k^2\bigl(1+q(x)\bigr)\,u^t(x) = 0 
    \quad x\in \Omega,
\end{equation}
 with the scattered field $u^s$ satisfying the Sommerfeld radiation condition
\begin{equation}\label{eq:sommerfeld}
    \lim_{r:=|x|\to\infty} r^{\frac{1}{2}}
    \left(\frac{\partial u^s(x)}{\partial r} - ik\,u^s(x)\right) = 0,
\end{equation}
uniformly with respect to all directions $x/|x|$. In particular, the scattered field satisfies
\begin{equation}\label{eq:scattered-pde}
    \Delta u^s(x) + k^2(1+q(x))\,u^s = -k^2 q(x) u^{in}(x) \quad x\in \Omega,
\end{equation}
or equivalently, the Lippmann-Schwinger integral equation\cite{kirsch1996introduction}, 
\begin{equation}\label{eq:LS-plane}
\begin{array}{lc}
%& u^s(x) - k^2\!\int_{\Omega} %G_k(x,y)\,q(y)\,u^s(y)\,dy = k^2\!\int_{\Omega} %G_k(x,y)\,q(y)\,u^{in}(y)\,dy, \\
%\mbox{or} &
u^s(x) = k^2\!\int_{\Omega} G_k(x,y)\,q(y)\,u^t(y)\,dy,
\end{array}
\end{equation}
where $G_k(x,y)$ is the fundamental solution to the free-space Helmholtz equation. In the weak scattering regime, the linearized Born approximation is obtained by replacing the total field $u^t$ in the integral by the incident field $u^{in}$,
\begin{equation}\label{eq:Born}
u^s(x) = k^2\!\int_{\Omega} G_k(x,y)\,q(y)\,u^{in}(y)\,dy.   
\end{equation}

To formulate the forward maps used in the inverse problem, we consider two measurement configurations: far-field measurements generated by a plane wave, and near-field measurements generated by a point source. For each configuration, we specify the incident field and the corresponding observation of the scattered field.

\textbf{Far-Field Case.} Let the incident wave be a plane wave with propagation direction 
$\hat{\theta} \in \mathbb{S} := \{x \in \mathbb{R}^2 : |x| = 1\}$,
\begin{equation}\label{eq:incident}
    u^{in}(x;\hat{\theta}) := e^{ikx\cdot\hat{\theta}}, \quad x \in \mathbb{R}^d.
\end{equation}
Since $|x-y| = r - \hat{x} \cdot y + O(\frac{1}{r})$ as $r=|x| \rightarrow \infty$, where $\hat{x} := \frac{x}{|x|}$, $G_k(x,y) = \frac{e^{i\frac{\pi}{4}}}{\sqrt{8k\pi}}\frac{e^{i k r}}{\sqrt{r}} e^{-ik\hat{x}\cdot y} + o(r^{-\frac{1}{2}})$,
% \[\Rightarrow G_k(x,y) \sim \frac{e^{i\frac{\pi}{4}}}{\sqrt{8k\pi}} \frac{e^{ikr}}{\sqrt{r}} e^{-i k \hat{x} \cdot y}\]
from the integral equation~\eqref{eq:LS-plane}, the scattered field has an asymptotic behavior
\[
u^s(x;\hat{\theta}) = \frac{e^{ikr}}{\sqrt{r}} u_o^{\infty} (\hat{x};\hat{\theta}) + O(r^{-\frac{3}{2}}),
\] 
and the integral relation becomes
\[
u_o^{\infty} (\hat{x};\hat{\theta}) - k^2\int_{\Omega} \frac{e^{i\frac{\pi}{4}}}{\sqrt{8k\pi}}\, e^{-i k \hat{x} \cdot y}\,q(y)\,u^s(y;\hat{\theta})\,dy = k^2\int_{\Omega} \frac{e^{i\frac{\pi}{4}}}{\sqrt{8k\pi}}\,e^{-i k \hat{x} \cdot y}\,q(y)\,e^{i k y\cdot \hat{\theta}}\,dy.
\] 
For convenience,
we introduce the rescaled far-field pattern  $u^\infty(\hat{x};\hat{\theta}) := \frac{1}{k^2} \sqrt{8k\pi}e^{-i\frac{\pi}{4}}u_o^\infty(\hat{x};\hat{\theta})$, 
then we can define the far-field operator $\mathcal{F}$ that maps $q$ to the scaled far-field measurement $u^\infty$:
\begin{equation}\label{eq:farfield-compact}
u^\infty(\hat{x};\hat{\theta})= \mathcal{F}q(\hat{x};\hat{\theta}) := \int_{\Omega} e^{-ik\hat{x}\cdot y}\,q(y)\,u^{t}(y;\hat{\theta})\,dy .
\end{equation}

\textbf{Near-Field Case.}
The incident wave corresponding to a point source located at $z\in \mathbb{R}^d$ ($d=2,3$) outside $\Omega$ is given by
\begin{equation}\label{eq:point-source}
    u^{in}(y;z)=G_k(y,z), \quad y\in \mathbb{R}^d.
\end{equation}
The measurement is taken at $x\in \mathbb{R}^d$ outside $\Omega$. Hence, the near-field operator $\mathcal{F}$ that maps $q$ to near-field measurement corresponding to source-observation pair $(x;z)$ is
\begin{equation}\label{eq:near-field}
    u^s(x;z) = \mathcal{F}q(x; z) := k^2\!\int_{\Omega} G_k(x,y)\,q(y)\,u^t(y;z)\,dy,
\end{equation}
where $u^t(y;z)$ is the total field corresponding to $u^{in}(y;z)=G_k(y,z)$.

\begin{remark}
In general, the incident wave can be a distribution or superposition of point sources or plane waves. 
\end{remark}

\subsection{Newton's method in least squares formulation for the inverse problem} \label{sec: lstsq}
The inverse scattering problem is to recover $q$ from the source-to-measurement operator, $\mathcal{F}$, e.g., the far-field operator~\eqref{eq:farfield-compact}, near-field operator~\eqref{eq:near-field}, or some combinations. Although the forward scattering problem is linear (in terms of $u$), the inverse problem is nonlinear (in terms of $q$) since $u$ depends on $q$ (globally). Moreover, the forward problem is well-posed while the inverse problem is ill-posed. Due to model error, measurement noise, or other uncertainties, the measurement data cannot be exactly matched in practice. The most general approach is to formulate the inverse problem as an optimization problem that minimizes the data mismatch in some norm. One can also view each measurement as an equation for the unknown and hence the inverse problem is to solve a collection of nonlinear equations concurrently in some optimal way. Below, we propose Newton's method in a least-squares formulation for inverse scattering problems. 

\textbf{Far-Field Case.}
With the far-field measurement in \eqref{eq:farfield-compact}, the inverse scattering problem can be formulated as: given an initial guess of the medium, $q_0(x)$, we look for an update, $\delta q(x)$, that satisfies the non-linear relation in:
\begin{equation}\label{eq:min-far-field}
\min_{\delta q}\Bigl\|\meas
- \mathcal{F}(q_0 + \delta q)\Bigr\|^2_{L^2}.
\end{equation}
% \begin{equation}\label{eq:min-far-field}
% \min_{\delta q}\Bigl\|\meas
% - \mathcal{F}(q_0 + \delta q)\Bigr\|_{L^2(\mathbb{S} \times \mathbb{S},\mathbb{C})}.
% \end{equation}
Here we prefer to be vague about the $L^2$ norm in the measurement space, the choice of which could result in undesirable or desirable bias, e.g., regularization, in practice as we will discuss in Section \ref{sec:spectral}.
% For example, a straightforward choice with equal weight for all measurement data in discrete form is
% \begin{equation}
%     \min_{\delta q}\frac{1}{N^2}\sum_{i=1}^{N}\sum_{j=1}^N \Bigl|u_{meas}^{\infty}(\hat{x}_i,\hat{\theta}_j)
% - \mathcal{F}(q_0 + \delta q)(\hat{x}_i,\hat{\theta}_j)\Bigr|^2.
% \end{equation}
% For full aperture data and evenly distributed directions, this discrete least squares formulation converges to 
% $L^2(\mathbb{S} \times \mathbb{S},\mathbb{C})$ in 2D as $N\rightarrow \infty$. As we will see in the next section, this $L^2$ norm will introduce some frequency bias, which can be corrected with appropriate weighting. \ylfn{Do we need to mention frequency bias here? We can't properly define or explain it yet.}

Define $F: q \mapsto u_q^t$ as the map from $q$ to the total field that satisfies \eqref{eq:total-field} and \eqref{eq:sommerfeld}.
It then follows \eqref{eq:farfield-compact} that \eqref{eq:min-far-field} gives
% \[\min_{\delta q}\Bigl\|u^{\infty}(\hat{x};\hat{\theta})
% - \int_{\Omega} e^{-i k\hat{x}\cdot y}\cdot\,u_{q_0+\delta q}^{t}(y;\hat{\theta})\cdot\,(q_0+\delta q)\,dy\Bigr\|_{L^2(\mathbb{S} \times \mathbb{S},\mathbb{C})}\]
% \begin{equation*}\label{eq:min3}
% \min_{\delta q}\Bigl\|\int_{\Omega} e^{-ik\hat{x}\cdot y}\Big(F(q)\cdot\,q
% - F(q_0+\delta q) \cdot\,(q_0+\delta q)\Big) (y;\hat{\theta})\,dy\Bigr\|_{L^2(\mathbb{S} \times \mathbb{S},\mathbb{C})}.
% \end{equation*}

% \begin{equation}\label{eq:min2}
% = \min_{\delta q}\Bigl\|u^{\infty}(\hat{x};\hat{\theta})
% - \int_{\Omega} e^{-i k\hat{x}\cdot y}\cdot\,F(q_0+\delta q)(y;\hat{\theta})\cdot\,(q_0+\delta q)\,dy\Bigr\|_{L^2(\mathbb{S} \times \mathbb{S},\mathbb{C})}.
% \end{equation}
% Plug in \eqref{eq:farfield-compact}, we get
\begin{equation}\label{eq:min3}
\min_{\delta q}\Bigl\|\meas - \int_{\Omega} e^{-ik\hat{x}\cdot y}\Big(F(q_0+\delta q) \cdot\,(q_0+\delta q)\Big) (y;\hat{\theta})\,dy\Bigr\|^2_{L^2}.
\end{equation}
% \ylfn{notations in the eq above need double check ($q$ does not depend on $\hat{\theta}$}
Using the 1st-order functional Taylor expansion,
\begin{align*}
F(q_0+\delta q) \cdot (q_0 + \delta q)
% = &F(q)\cdot q - F(q_0)\cdot (q_0 + \delta q) - F'(q_0)[\delta q] \cdot (q_0+\delta q) + O(\|\delta q\|^2)\\
= F(q_0)\cdot q_0 + (F(q_0) \cdot \delta q + F'(q_0)[\delta q] \cdot q_0) + O(\|\delta q\|^2),
\end{align*}
where $F'(q_0)$ denotes the Fr\'{e}chet derivative of $F$ at $q_0$. Define the linear operator
% Define the data mismatch $\delta\tilde{u}^{\infty}(\hat{x};\hat{\theta}) := \meas(\hat{x};\hat{\theta}) - \mathfrak{F}q_0(\hat{x};\hat{\theta})$
% \ylfn{Please add a line to explain how this Taylor expansion is related to Eq. 3.2. Shall eq 3.2 LHS start from 3.1? Also we need to define $F'$ delete line 2}
% \begin{multline}\label{eq:linearized-eq}
% u^{\infty}(\hat{x};\hat{\theta})
% -\int_{\Omega} e^{-ik\hat{x}\cdot y}\cdot\,F(q_0)(y;\hat{\theta})\cdot\,q_0(y)\,dy \\
% = \int_{\Omega} e^{-ik\hat{x}\cdot y}\,\big(F(q_0)(y;\hat{\theta}) \cdot\delta q(y)
% + F'(q_0)[\delta q] \cdot q_0(y)\big)\,dy + O(\|\delta q \|^2).
% \end{multline}
% \begin{equation}\label{eq:linearized-eq}
% \delta\tilde{u}^{\infty}(\hat{x};\hat{\theta}) 
% = \int_{\Omega} e^{-ik\hat{x}\cdot y}\,\big(F(q_0)(y;\hat{\theta}) \cdot\delta q(y)
% + F'(q_0)[\delta q] \cdot q_0(y)\big)\,dy + O(\|\delta q \|^2).
% \end{equation}
% the linearized operator around $q_0$ is
\begin{equation}\label{eq:linearized-operator}
\Jacobian_{q_0}\delta q
:= \int_{\Omega} e^{-ik\hat{x}\cdot y}\,\big(F(q_0)(y;\hat{\theta}) \cdot\delta q(y)
+ F'(q_0)[\delta q] \cdot q_0(y)\big)\,dy ,
\end{equation}
which maps a medium contrast to the far-field data. Then the linearized problem in each step of Newton's method for \eqref{eq:min3} can be formulated as a linear least squares problem:
\begin{equation}\label{eq:min-angle}
\min_{\delta q}\Bigl\| \delta\tilde{u}^{\infty}(\hat{x};\hat{\theta}) - \Jacobian_{q_0}\delta q\Bigr\|^2_{L^2} ,
\end{equation}
where $\delta\tilde{u}^{\infty}(\hat{x};\hat{\theta}) := \meas(\hat{x};\hat{\theta}) - \mathcal{F}q_0(\hat{x};\hat{\theta})$ is the residual and $\mathcal{F}$ is the far-field operator~\eqref{eq:farfield-compact}.
% In practice, $\Jacobian_{q_0}$ can be formulated by solving a corresponding adjoint problem without computing the Fr\'{e}chet derivative $F'(q_0)$.

To obtain a representation of $\mathcal{J}_{q_0}$ in \eqref{eq:min-angle} that avoids explicitly computing the Fréchet derivative \(F'(q_0)\), we introduce the linearized field and use an adjoint formulation. Let $\delta u = F'(q_0)[\delta q]$, then $\delta u$ satisfies
\begin{equation}\label{eq:lin-PDE}
\Delta \delta u + k^2(1+q_0)\delta u = -k^2 \delta q u^t_{q_0}
\end{equation}
with the Sommerfeld radiation condition \eqref{eq:sommerfeld}.
For the second term in \eqref{eq:linearized-operator},
\begin{equation}\label{eq:adj-equivalence}
   \int_{\Omega} e^{-ik\hat{x}\cdot y}\cdot F'(q_0)[\delta q]  \cdot q_0(y)\,dy = \int_{\Omega} e^{-ik\hat{x}\cdot y}\,\delta u(y;\hat{\theta})  \, q_0(y)\,dy =
\int_\Omega 
\lambda_{q_0}^s(y;\hat{x})
\,u_{q_0}^t(y;\hat{\theta})
\,\delta q(y)\,dy, 
\end{equation}
% \[ \int_{\Omega} e^{-ik\hat{x}\cdot y}\cdot F'(q_0)[\delta q]  \cdot q_0(y)\,dy = \int_{\Omega} e^{-ik\hat{x}\cdot y}\,\delta u(y;\hat{\theta})  \, q_0(y)\,dy =
% \int_\Omega 
% \lambda_{q_0}^s(y;\hat{x})
% \,u_{q_0}^t(y;\hat{\theta})
% \,\delta q(y)\,dy, 
% \]
where $\lambda_{q_0}^s(y;\hat{x})$ satisfies the adjoint problem
% Let \(\psi(y;\hat{x})\) solve the adjoint equation
% \[
% A_0^T\psi=b,
% \]
% and define $\lambda(y;\hat{x})=k^2\mathcal G^T\psi(y;\hat{x})$ which satisfies
\begin{equation}\label{eq:adj-PDE}
\Delta_y \lambda_{q_0}^s(y;\hat{x}) + k^2(1+q_0)\lambda_{q_0}^s(y;\hat{x}) = -k^2q_0(y)e^{-ik\hat{x}\cdot y}
\end{equation}
and the Sommerfeld radiation
condition \eqref{eq:sommerfeld}. Hence, define the adjoint total field $\lambda_{q_0}^t(y;\hat{x}) := e^{-ik\hat{x} \cdot y} + \lambda_{q_0}^s(y;\hat{x}) $ and we can write $\Jacobian_{q_0}$ as
\begin{equation}\label{eq:operator-adj}
\Jacobian_{q_0}\delta q = \int_\Omega 
\lambda^t_{q_0}(y;\hat{x})
\,u_{q_0}^t(y;\hat{\theta})
\,\delta q(y)\,dy.
\end{equation}
This formulation allows one to apply $\Jacobian_{q_0}$ to a function by solving the forward and adjoint problem without computing the Fr\'{e}chet derivative $F'(q_0)$ explicitly. We leave the derivation in the appendix and refer to \cite{Bao_example} for detailed analysis. More importantly, it provides a way to apply the Jacobian matrix to a vector without explicitly forming the matrix in numerical computation, which is crucial in applying RSVD to the Jacobian matrix. 

When $q_0 = 0$, $\Jacobian_{q_0}$ is reduced to
\begin{equation}\label{eq:linearized-operator_0}
\Jacobian_{0}\delta q := \int_{\Omega} e^{-i k\hat{x}\cdot y}\, e^{iky\cdot\hat{\theta}} \,\delta q(y)\,dy ,
\end{equation} 
which is known as the linear Born approximation, and the corresponding Born far-field data
\begin{equation}\label{eq:Born data}
u_b^{\infty}(\hat{x};\hat{\theta};k) := \int_{\Omega} e^{-i k\hat{x}\cdot y}\, e^{iky\cdot\hat{\theta}} \, q(y)\,dy
\end{equation}
leads to a linear inverse scattering problem.
% In the next section, we will link it to the restricted Fourier transform and show its spectral property.

\textbf{Near-Field Case.}
Following the same derivation, the linearized operator for near-field data is exactly of the same form as~\eqref{eq:operator-adj} with $k^2$ scaling except that $u^t_{q_0}(y;z)$ and $\lambda^t_{q_0}(y;x)$ are the total fields (with respect to medium $q_0$) corresponding to incident waves from point sources located at $z$ and $x$, i.e.,  $G_k(y,z)$ and $G_k(y,x)$, respectively.
% \textcolor{cyan}{Better present the linear inverse problems corresponding to 2.13, then 2.16 and 2.17 explicitly.}

\section{Spectral Behavior}\label{sec:spectral}
Truncated SVD (TSVD) is a typical choice to solve the linear least squares problem~\eqref{eq:min-angle}. Hence, it is important to understand the spectral behavior of the linear operator $\Jacobian_{q_0}$. In this section, we study the spectral behaviors of the linearized operators in each Newton iteration, and illustrate the low-rank structure across different measurement configurations.

\subsection{\texorpdfstring{$p$}{p}-transform and reweighted least-squares formulation for far-field case}
\label{sec:far-field data}

%Innovated by the spectrum of the restricted Fourier transform which connects to $\Jacobian_0$, in the section we study the spectral behavior of $\Jacobian_{q_0}$. 
% \ylfn{How to justify why we need to do the p-transform and why balance in $p$ space is important? Starting from \eqref{eq:linearized-operator_0},  this $p$ is in the position $e^{i 2kp\cdot x}$ thus:  1. it is convenient to introduce this change of variable and 2. it is naturally a frequency variable ( for $\delta q$)  (when $q_0 = 0$), which is why we say frequency bias in the space of $q$ later. But moving to \eqref{eq:operator-adj}, it is not clear that $p$ is a frequency variable (is it?), thus it is something we propose and then verify with numerical results.}
Before studying the spectrum of the forward operator \eqref{eq:linearized-operator} for the general medium scattering problem, we first study the problem in the weak scattered regime, which motivates the TSVD algorithm. Recall that in the weak scattered regime, the least-squares problem corresponding to the Born far-field data \eqref{eq:Born data} is based on the linear operator $\Jacobian_{0}$ in \eqref{eq:linearized-operator_0}. We introduce a map $P$ to represent the change of variables
\begin{equation}\label{eq:p transform}
    P:\mathbb{S} \times \mathbb{S} \mapsto B,\quad p=P(\hat{x},\hat{\theta}) := \frac{\hat{\theta} - \hat{x}}{2},
\end{equation}
% $$
% P:\mathbb{S} \times \mathbb{S} \mapsto B(0,1) ,\quad p=P(\hat{x},\hat{\theta}) := \frac{\hat{\theta} - \hat{x}}{2},$$
%an observation on the right-hand side of \eqref{eq:linearized-operator_0} links the Born model to the restricted Fourier transform by a change of variable $p = \frac{\hat{\theta}-\hat{x}}{2}$ which maps from $\mathbb{S} \times \mathbb{S}$ to $B(0,1) =: B$.
where $B:=B(0,1)\subset \mathbb{R}^2$ is the unit ball centered at the origin, and we refer to the mapping as the $p$-transform.
Then, the linear operator $\Jacobian_{0}$  becomes the restricted Fourier transform of $q$ (since $\supp q \subset \Omega$):
\begin{equation}\label{eq:J0}
\Jacobian_{0} q (\hat{x};\hat{\theta}) = \int_{\Omega} e^{-i k\hat{x}\cdot y}\, e^{iky\cdot\hat{\theta}} \,q(y)\,dy = \widehat{q}(2kp):=\hat{\Jacobian_0}q(p).
\end{equation}
% After applying the $p$-transform
%\[(\hat{x},\hat{\theta}) \mapsto \frac{\hat{\theta} - \hat{x}}{2}.\] 

Next we consider the least-squares formulation under the $p$-transform. We first present our study in 2D. The situation is similar in 3D and the difference is explained in Remark~\ref{remark:3D}.
% \textcolor{blue}{Previous line:In order to have bijection between $B$ and $\mathbb{S} \times \mathbb{S}$, we consider the pair where $\det(\hat{\theta},\hat{x}) = \hat{\theta}_1\hat{x}_2 - \hat{\theta}_2\hat{x}_1 \ge 0$ and $\hat{x} \neq \hat{\theta}$.}

If we restrict the domain of the $p$-transform to $D = \{ (\hat{x}, \hat{\theta})\in  \mathbb{S} \times \mathbb{S}:\det(\hat{\theta},\hat{x}) = \hat{\theta}_1\hat{x}_2 - \hat{\theta}_2\hat{x}_1 \ge 0,\hat{x} \neq \hat{\theta}\}$, then the $p$-transform and its inverse exist, i.e., for any $p \in B\setminus \{0\}$, there exists a unique pair $(\hat{x}(p),\hat{\theta}(p))=P^{-1}(p) \in D$ such that
%\begin{equation}\label{eq:p transform}
$p = \frac{\hat{\theta}(p) - \hat{x}(p)}{2}$.
%\end{equation}
% Under the $p$-transform, the Born far-field data for $q(x)$
% \begin{equation}\label{eq:J0}
% \Jacobian_{0} q = \int_{\Omega} e^{-i k\hat{x}\cdot y}\, e^{iky\cdot\hat{\theta}} \,q(y)\,dy =\widehat{q}(2kp):=\hat{\Jacobian_0}q.
% \end{equation}
Moreover, the product measure \(ds_{\hat{x}} \, ds_{\hat{\theta}}\) on $D$ pushes forward to the $p$-measure $\frac{2}{|p|\sqrt{1-|p|^2}}dp$ on \(B\). More precisely 
% \[
% (p)_\#(ds_x\,ds_\theta)
% =
% \frac{2}{|p|\sqrt{1-|p|^2}}\,dp,
% \qquad 0<|p|<1.
% \]
for any integrable function $f$ on $D$,
% \[
% \int_{\mathbb{S}}\int_{\mathbb{S}} f(\frac{\hat{\theta}-\hat{x}}{2})ds_{\hat{x}}ds_{\hat{\theta}} = \int_{B} f(p) \frac{2}{|p|\sqrt{1-|p|^2}}dp.
% \]
\[
\int_D f(\hat{x}, \hat{\theta})ds_{\hat{x}}ds_{\hat{\theta}} = \int_{B} f(\hat{x}(p), \hat{\theta}(p)) \frac{2}{|p|\sqrt{1-|p|^2}}dp.
\]
\begin{remark} \label{rmk:reciprocity}
    Note that for $\forall p \in B\setminus \{0\}$, there exists 2 pairs $\Big(\hat{x}(p),\hat{\theta}(p)\Big)$ and $\Big(-\hat{\theta}(p),-\hat{x}(p)\Big)$ that corresponds to the same $p$ by $p$-transform. However, there is no inconsistency issue in using all measurement data due to the reciprocity principle \cite{kirsch1996introduction},
\begin{equation*}\label{eq:reciprocity}
    u^\infty(\hat{x},\hat{\theta}) = u^\infty(-\hat{\theta},-\hat{x})\quad for \,\, all \,\,\hat{x}, \hat{\theta} \in \mathbb{S}. 
\end{equation*}
\end{remark}
 To optimize \eqref{eq:min-far-field}, a straightforward choice with equal weight for all measurement data in discrete form is
\begin{equation}\label{eq: discrete unweighted}
    \min_{\delta q}\frac{1}{M_1M_2}\sum_{i=1}^{M_2}\sum_{j=1}^{M_1} \Bigl|u_{meas}^{\infty}(\hat{x}_i,\hat{\theta}_j)
- \mathcal{F}(q_0 + \delta q)(\hat{x}_i,\hat{\theta}_j)\Bigr|^2.
\end{equation}
For evenly distributed directions, this discrete least squares formulation converges to 
$L^2(\mathbb{S} \times \mathbb{S},\mathbb{C})$ in 2D as $M_1,M_2\rightarrow \infty$. By the argument in Remark ~\ref{rmk:reciprocity}, the measurement space $\mathbb{S} \times \mathbb{S}$ is equal to two copies of $D$. Let $\hat{u}^\infty(p): = u^\infty(P^{-1}(p))$ for $p\in B$, then  
\[
\frac{1}{2}\Bigl\|u^{\infty}(\hat{x};\hat{\theta}) - \Jacobian_{0} \delta q\Bigr\|^2_{L^2(\mathbb{S}\times\mathbb{S},\mathbb{C})}=\Bigl\|u^{\infty}(\hat{x};\hat{\theta}) - \Jacobian_{0} \delta q\Bigr\|^2_{L^2(D,\mathbb{C})} = \int_{B} |\widehat{\delta q}(2kp) - \hat{u}^\infty(p)|^2 \frac{2}{|p|\sqrt{1-|p|^2}} dp, 
\] 
% By the argument in Remark ~\ref{rmk:reciprocity}, $
% \Bigl\|u^{\infty}(\hat{x};\hat{\theta}) - \Jacobian_{0} \delta q\Bigr\|^2_{L^2(\mathbb{S} \times \mathbb{S},\mathbb{C})} = 2 \Bigl\|u^{\infty}(\hat{x};\hat{\theta}) - \Jacobian_{0} \delta q\Bigr\|^2_{L^2(D,\mathbb{C})}$, thus the least squares formulation is equivalent to minimizing RHS of .... Then frequency bias... then the discretized form \eqref{eq:weighted discrete} ...
% Previous lines: Let $\hat{u}^\infty(p)$ denote the $p$-transform of the far-field data $u^\infty(\hat{x};\hat{\theta})$ from $\mathbb{S} \times \mathbb{S}$ to $B$, and denote $\hat{\mathcal{F}}$ the far-field operator after the $p$-transform. In particular, in the case of $q_0 = 0$, the discrete least-squares formulation with equal weight for all measurement data for the linearized inverse scattering problem converges to the continuous limit
% \[
% \Bigl\|u^{\infty}(\hat{x};\hat{\theta}) - \Jacobian_{0} \delta q\Bigr\|^2_{L^2(\mathbb{S} \times \mathbb{S},\mathbb{C})}=\int_{B} |\widehat{\delta q}(2kp) - \hat{u}^\infty(p)|^2 \frac{2}{|p|\sqrt{1-|p|^2}} dp, 
% \] 
which exhibits a bias towards the low-frequency regime (near $p=0$) and the high-frequency regime near $|p|=1$ ($|2kp|=2k$). %See Figure ~\ref{fig:density_p} for a density plot of $p$ from $\hat{x}$ and $\hat{\theta}$ evenly distributed on $\mathbb{S}$.
% The Jacobian $\frac{2}{|p|\sqrt{1-|p|^2}}$ exhibits a bias towards the low-frequency regime ($p = 0$) and the high-frequency regime ($p=1$) if we consider \eqref{eq:min-angle}.
% However, this $L^2$ norm will introduce some frequency bias, which can be corrected with appropriate weighting.
% \textcolor{cyan}{This motivates us to propose a reweighted least-squares formulation that, after the $p$-transform, assigns uniform weight to the frequency components of $\delta q$ within the accessible frequency disk. In the discrete form, the least-squares formulation becomes:} 
Here we propose to remove the frequency bias by using $L^2(B, \mathbb{C})$ for the data mismatch, which is equivalent to a weighted least-squares formulation in discrete form
\begin{equation}\label{eq:weighted discrete}
    \min_{\delta q}\frac{1}{M_1M_2}\sum_{i=1}^{M_2}\sum_{j=1}^{M_1} \frac{|p_{ij}|\sqrt{1-|p_{ij}|^2}}{2}\Bigl|u_{meas}^{\infty}(\hat{x}_i,\hat{\theta}_j)
- \mathcal{F}(q_0 + \delta q)(\hat{x}_i,\hat{\theta}_j)\Bigr|^2, \quad\text{where} ~~p_{ij}=\frac{\hat{\theta}_j-\hat{x}_i}{2}.
\end{equation}
% \begin{figure}[H]
% \centering
% \begin{tabular}{c@{\vspace{2pt}}c@{\vspace{2pt}}c}
% \includegraphics[width=0.4\linewidth]{density_p1.png}&
% \includegraphics[width=0.4\linewidth]{density_p2.png}\\
% (a)&(b)
% \end{tabular}
% \caption{(a) $p$ distribution from  $\hat{x}$, $\hat{\theta}$ evenly distributed in $\mathbb{S}$; (b) area-uniform $p$ distribution.}
% \label{fig:density_p}
% \end{figure}

% \begin{figure}[H]
% \centering
% \includegraphics[width=0.3\linewidth]{density_p1.png}
% \caption{$p$ distribution from  $\hat{x}$, $\hat{\theta}$ evenly distributed on $\mathbb{S}$.}
% \label{fig:density_p}
% \end{figure}
% Thus, it reduces the redundancy of the data while retaining the information.

With this observation on frequency bias, we adopt the unbiased least squares formulation in $L^2(B, \mathbb{C})$ for~\eqref{eq:min-far-field}, and hence \eqref{eq:min-angle} is transformed into
\begin{equation}\label{eq:min_p}
\min_{\delta q}\Bigl\| \delta\hat{u}^\infty(p) - \hat{\Jacobian}_{q_0}\delta q(p)\Bigr\|^2_{L^2(B,\mathbb{C})}, 
\end{equation}
where 
\begin{equation}\label{eq:p_integral}
\hat{\Jacobian}_{q_0}\delta q(p) := \int_\Omega 
\lambda^t_{q_0}(y;\hat{x})
\,u_{q_0}^t(y;\hat{\theta})
\,\delta q(y)\,dy, \quad (\hat{x},\hat{\theta}) = \Big(\hat{x}(p),\hat{\theta}(p)\Big),
\end{equation}
and $\delta\hat{u}^\infty(p) = \delta\tilde{u}^{\infty}(P^{-1}(p))$ in \eqref{eq:min-angle}.
% Since $\mu(\{0\}) = 0$ \textcolor{cyan}{Do we also need this argument for the homogeneous part?} where $\mu$ denotes the Lebesgue measure, $L^2(B)$ is isomorphic to $L^2(B\setminus \{0 \})$,
% $\hat{\mathcal{J}}_{q_0}$ is compact from $L^2(\Omega)$ to
% $L^2(B)$. 
% By compactness, there exists a singular value and corresponding singular mode system $\{\hat{\sigma}_i,\hat{u}_i,\hat{v}_i\}_{i=1}^\infty$ such that
% \begin{equation}\label{eq:SVD}
% \hat{\mathcal{J}}_{q_0}\hat{v}_i= \hat{\sigma}_i \hat{u}_i,
% \qquad 
% \{\hat{u}_i\}\subset L^2(B),\ \{\hat{v}_i\}\subset L^2(\Omega),
% \end{equation}
% where $\{\hat{u}_i\}$ and $\{\hat{v}_i\}$ form orthonormal bases of $L^2(B)$
% and $L^2(\Omega)$ respectively, 

%\subsection{Far-Field Case} \label{sec:far-field data}

First, we study the spectral property for $\hat{\Jacobian}_0$ defined in \eqref{eq:J0}. For full aperture data and $\Omega = B$, the analysis in \cite{zhou2024exploringlowrankstructureinverse} connects the eigen-system of $\hat{\Jacobian}_{0}$ with prolate spheroidal wave functions (PSWFs) and the restricted Fourier transform
\[\mathcal{F}_{B}^{2k}\psi_{m,n,l}(p;2k) = \int_{B} e^{i2kp \cdot y}\psi_{m,n,l}(y;2k)dy = \alpha_{m,n}(2k)\psi_{m,n,l}(p;2k),\]
% From the 1st-order Born approximation and the asymptotic behavior of the scattered wave field (c.f. \cite{book}), we have
% \[u_b^s(x;\hat{\theta};k) = \frac{e^{i\frac{\pi}{4}}}{\sqrt{8k\pi}} \frac{e^{ikr}}{\sqrt{r}} \left(u_b^{\infty}(\hat{x};\hat{\theta};k) + O\left(r^{-1}\right)\right) \quad \text{as} \enspace r := |x| \rightarrow \infty \]
% \[u_b^{\infty}(\hat{x};\hat{\theta};k) = k^2 \int_{\Omega} e^{-ik\hat{x}\cdot y} q(y) e^{i k y \cdot \hat{\theta}} dy = k^2 \int_{\Omega} e^{ik (\hat{\theta}-\hat{x})\cdot y} q(y) dy.\]
% A low rank structure of mapping from $q$ to Born far-field data is discovered in \cite{zhou2024exploringlowrankstructureinverse} using prolate spheroidal wave functions and restricted Fourier transform, i.e.
% \[\mathcal{F}_{B(0,1)}^c \psi_{m,n,l}(x;c) = \int_{B(0,1)} e^{icx \cdot y}\psi_{m,n,l}(y;c)dy = \alpha_{m,n}(c)\psi_{m,n,l}(x;c), \]
where $\psi_{m,n,l}(x;2k)$ are the \textit{disk prolate spheroidal wave functions}, and $\alpha_{m,n}$ are the corresponding eigenvalues. The efficiency of the linearized inversion algorithm in \cite{zhou2024exploringlowrankstructureinverse} relies on two features: the explicit form of the eigensystem of $\hat{\Jacobian}_{0}$ by its connection to the Sturm-Liouville operator,
%which allows fast and accurate computation of $\psi_{m,n,l}$ and $\alpha_{m,n}$. In particular, 
and the effective low-rank structure of $\mathcal{F}_{B}^{2k}$, reflected by the 'flat-then-cliff' phenomenon of its spectrum (see Figure 1 in \cite{zhou2024exploringlowrankstructureinverse}  and Figure ~\ref{fig:spectrum for both}(a) in this paper). 

%The idea of the paper is to explore the low rank structure iteratively so that a better approximation can be obtained.  However, a more general and efficient approach of finding the eigenbasis for an the inhomogeneous medium is needed.  For explicit characterization of the estimate given in section 7 and alignment of the spectrum of the forward operator with $\alpha_{m,n}$, we set $\Omega := B(0,1)$ without any harm, see \cite{Meng_2023} for analysis of generalization of $\Omega$.

% \textcolor{cyan}{This effective low-rank structure has been studied through the theory of spatial–spectral concentration, which provides estimates of the effective rank and motivates our choice of a spectral cutoff.}
From the $L^2$ perspective, the eigenvalue of $\hat{\Jacobian}_{0}^*\hat{\Jacobian}_{0}$,
\begin{align*}
(\hat{\Jacobian}_{0}^*\hat{\Jacobian}_{0})f(p) = \int_{B}\int_{\Omega}e^{i 2k x \cdot z}\,f(z)\,dz\,e^{-i 2k x \cdot p}\,dx
 = \int_{\Omega} \int_{B} e^{i2kx \cdot (z-p)}\, dx\, f(z)\,dz,
\end{align*}
quantifies how much a space-limited function can be concentrated on a given frequency domain, which is often referred to as the \textit{spatio-spectral concentration problem}. Specifically, for the case of $\Omega = B$, the 'numerical dimension' of the space of bandlimited functions is finite and depends on the decay of the singular values, as shown in \cite{Landau1975OnSE}; also see \cite{SHKOLNISKY2007235}, and the dimension is $k^2 + o(k^2)$ asymptotically. To obtain a more accurate eigenvalue distribution, we refer to \cite{hughes2026wavepacketseigenvalueestimates}, which considers a more general case where the spatial domain $\Omega$ is an ellipse centered at the origin in $\mathbb{R}^2$, and the frequency domain is well-shaped in $\mathbb{R}^2$ following Definition 2.1 in \cite{hughes2026wavepacketseigenvalueestimates} (indeed, $B$ is well-shaped).
By a change of variable $x' = 2kx$, the scaled spatial domain is $\Omega' = 2k\Omega$ and the frequency domain is $B$, then we can define the 
% Set $c = 1$, $\Omega = [0,1]^d$, $B = B(0,2k)$, and define the
\textit{spatio-spectral limiting operator} $\tilde{T}_{\Omega',B}=\frac{k^2}{\pi^2}\hat{\Jacobian}_{0}^*\hat{\Jacobian}_{0}$. Theorem 1.1 in \cite{hughes2026wavepacketseigenvalueestimates} demonstrates the clustering behavior of eigenvalues of $\tilde{T}_{\Omega',B}$ near 0 and 1. The number of eigenvalues close to 1 is $\frac{1}{4\pi^2}\mu(\Omega')\mu(B)$, aligned with the Shannon number estimate given in \cite{Simons2015}, which can be generalized to any $B$ and $\Omega$. The number of intermediate eigenvalues between 0 and 1 is at most $O(k(\log k)^{1+2s})$ for any $s>1$.
% To obtain a more accurate eigenvalue distribution, we refer to \cite{Eigenvalue_distribution}, which considers a more general case where the spatial domain $\Omega$ is a hypercube in $\mathbb{R}^2$, and the frequency domain is a convex body in $\mathbb{R}^2$ (Indeed, $B$ is convex in our case).
% By a change of variable $x' = 2kx$ and setting $\Omega = [0,1]^2$, the scaled spatial domain is $\Omega' = [0,2k]^2$ and the frequency domain is $B$, then we can define the 
% % Set $c = 1$, $\Omega = [0,1]^d$, $B = B(0,2k)$, and define the
% \textit{spatio-spectral limiting operator} $\tilde{T}_{\Omega',B}=\frac{k^2}{\pi^2}\hat{\Jacobian}_{0}^*\hat{\Jacobian}_{0}$. Theorem 1.1 in \cite{Eigenvalue_distribution} demonstrates the clustering behavior of eigenvalues of $\tilde{T}_{\Omega',B}$ near 0 and 1. The number of eigenvalues close to 1 is $\frac{1}{4\pi^2}\mu(\Omega')\mu(B) = \frac{1}{\pi}k^2$, and the number of intermediate eigenvalues between 0 and 1 is at most $O(k(\log k)^{\frac{5}{2}})$.
As a result, singular values of $\hat{\Jacobian}_{0}$ would have the same clustering phenomenon near $0$ and $\frac{\pi}{k}$. 
% \ylfn{I assume that they are not clustering near $1$ due to some change of variable. Shall we show Theorem 1.1 in [8] as a lemma and write the change explicitly?}
% It is worth noting that the estimate of the number of eigenvalues close to 1? given in \cite{Eigenvalue_distribution} aligns with the Shannon number estimate given in \cite{Simons2015}, which can be generalized to any $B$ and $\Omega$. 
% Without loss of generality, we set $\Omega = B$ for the rest of the paper. \textcolor{red}{delete, case $\Omega =B$}
% \begin{figure}%[HTB]
%     \centering
%     \includegraphics[width=0.8\linewidth]{spectrum_shannon.png}
%     \caption{Spectrum of $\hat{\Jacobian}_0$ with different $k$; the dash line indicates the Shannon number.}
%     \label{fig:svd-spectrum}
% \end{figure}

Motivated by the clustering phenomenon, we propose to choose the spectral cutoff by computing the Landau-Pollak-Slepian "Shannon number" introduced in\cite{Simons2015, Landau_Pollak}, which is used to define the dimension of a signal subspace spanned by eigenmodes with eigenvalues close to $1$,
\begin{equation}\label{eq:Shannon number}
S(\Omega;\mathcal{P}) = (\frac{k}{\pi})^2 \mu(\mathcal{P}) \cdot \mu(\Omega),
% (2k)^2\frac{\mu(\mathcal{P}) \cdot \mu(\Omega)}{(2\pi)^2},
\end{equation}
where $\mathcal{P} := \{p: p = \frac{\hat{\theta} - \hat{x}}{2} \}$, and $\mu(A)$ means the measure of set $A$. For example, for full aperture data, $\mathcal{P}=B(0,1)$, $\Omega = B(0,1)$, and  $S = k^2$.
% \ylfn{$\mu$ seems to be used to denote measure in $\mathbb{R}^n$?}
In our case, the TSVD solution corresponding to the Shannon number cutoff provides the best (in terms of least squares) update in the signal subspace at each Newton step, which significantly reduces the number of steps with stability control.

Next, we compare $\hat{\Jacobian}_{q_0}$ in \eqref{eq:p_integral} and $\hat{\Jacobian}_{0}$. In particular, we compare the singular values of these two operators around $S$ defined in \eqref{eq:Shannon number}, since such a comparison would inform whether $S$ remains useful for SVD truncation of $\hat{\Jacobian}_{q_0}$ for the algorithm proposed in Section ~\ref{sec: algorithm}.
In fact, numerical evidence shows that both the regularity and the amplitude of $q_0$ factor into the effectiveness of the Shannon number for describing the spectrum transition. A smoother $q_0$ or $q_0$ with smaller amplitude makes them more comparable between $\sigma_{S}(\hat\Jacobian_{q_0})$ and $\sigma_{S}(\hat\Jacobian_0)$, where smaller amplitude belongs to the Born regime, where we can apply Weyl's inequality to quantify the gap between the singular values close to $\frac{\pi}{k}$ and those close to $0$. Outside the Born regime, we observe that even if a few leading singular values of $\hat{\Jacobian}_{q_0}$ could be amplified, the spectrum remains comparable around $S$, which will be studied further in our future work. From the computational viewpoint, one only needs a good estimate of the effective rank $N$ such that $\sigma_N(\hat\Jacobian_{q_0})$ starts to lie in the tail part (close to 0) of the spectrum of the compact operator $\hat\Jacobian_{q_0}$, then RSVD (see Section ~\ref{sec: algorithm}) can provide a good approximation with $N+r$ ($r$ is a smaller number) sampling of random vectors along which the matrix is projected. After RSVD, one can apply a simple threshold strategy to control the stability of Algorithm ~\ref{alg:iterative-low-rank} if needed.     
\subsection{Limited Aperture Case}\label{sec: limited aperture analysis}

We next examine how limited-aperture geometry affects the information available for reconstruction. The incident and observation directions, denoted by $\Gamma^i$ and $\Gamma^o$ respectively, are restricted to certain directions, i.e., $\Gamma^i\times \Gamma^o \subset \mathbb{S}\times \mathbb{S}$, which leads to a more ill-posed problem compared to full aperture data for inverse scattering problems. In the Born approximation~\eqref{eq:Born data}, limited aperture data correspond to the Fourier transform of $q$ on a subset of the Fourier ball $B(0,2k)$, denoted by $\mathcal{P}(\Gamma^o\times\Gamma^i)$, after the $p$-transform \eqref{eq:p transform}.
Although the far-field pattern is analytic in $(\hat{x},\hat{\theta})\in \mathbb{S}\times \mathbb{S}$, which means the full-aperture data are uniquely determined by analytic continuation \cite{GriesmaierHyvonenSeiskari2013}, this continuation is severely unstable. Moreover, not only the size of the aperture matters, but also the geometric configuration of $\Gamma^i$ and $\Gamma^o$, which determines $\mathcal{P}(\Gamma^o\times\Gamma^i)$ as shown in Figure \ref{fig:limited_angle}, will lead to different reconstruction results in terms of both stability and resolution in different directions and at different scales.

% the problem gets more ill-posed. Since the far-field pattern is analytic in $\hat{x}$ and $\hat{\theta}$,
%  the full-aperture data are uniquely determined by analytic continuation \cite{GriesmaierHyvonenSeiskari2013}, but this continuation is severely unstable.
%  %The far-field pattern is analytic in $\hat{x}$ and $\hat{\theta}$ and hence by analytic continuation \textcolor{cyan}{Cite a work on uniqueness}, full aperture data can be uniquely determined. However, analytic continuation is exponentially unstable.
% % There exists a H\"older-logarithmic stability estimate for the inverse of restricted Fourier transform\cite{Holder_estimate}.
% In the Born approximation case, the Born data from \eqref{eq:Born data} is equivalent to the Fourier transform of $q$ at frequency $k(\hat{\theta} - \hat{x})$. Hence, limited aperture data determines $\widehat{q}$ on a subset of the Fourier ball of radius $2k$ after the $p$-transform \eqref{eq:p transform}. Moreover, the subset is typically anisotropic in directions in the Fourier domain as shown in Figure \ref{fig:limited_angle}, which implies different directional stability and/or resolution in the spatial domain. 

We analyze how the size and relative orientation of the incident and observation apertures affect the reconstruction results. 
% First, in the continuous setting, the dimension of the signal space, or the effective rank, can be estimated by the Shannon number \eqref{eq:Shannon number}.
The Shannon number \eqref{eq:Shannon number} provides an estimate of the effective rank associated with a given aperture configuration.
To be specific, consider limited aperture data corresponding to directions $\{u^{\infty}(\hat{x};\hat{\theta}): (\hat{x},\hat{\theta}) \in \Gamma^o\times \Gamma^i \}$, where $  \Gamma^o=\Gamma^i=\mathbb{S}_L:= \{\hat{x} \in \mathbb{S}: \arg \hat{x} \in [-L/2, L/2] \}$ for  $0<L<2\pi$, which corresponds to transmission measurement as shown in Figure \ref{fig:limited_angle}(a). 
After the $p$-transform, one can determine the region in the Fourier domain covered by the limited aperture data and compute its area $\mu(\mathcal{P}(\mathbb{S}_L\times \mathbb{S}_L)) = \frac{1}{2}(L - \sin L)$. More generally, let's consider different measurement setups, i.e., different $\Gamma^o$ and $\Gamma^i$ configurations, and hence, $\mathcal{P}(\Gamma^o\times\Gamma^i)$ corresponds to different regions inside a Fourier ball after the $p$-transform. Based on these Fourier regions, one can compute the effective rank and infer different imaging results. For example, $\Gamma^o = \Gamma^i$ represents the transmission measurement, while $\Gamma^o = -\Gamma^i$ corresponds to reflection measurement. Under the Born approximation, the coverage in the Fourier domain dictates the frequency modes of $q$ available from measurement by ~\eqref{eq:J0}. Figure ~\ref{fig:limited_angle} gives the coverage under different configurations of $\Gamma^o$ and $\Gamma^i$ with a fixed aperture size $L = \frac{3}{4}\pi$. Let $c_{\hat{\theta}}$ and $c_{\hat{x}}$ denote the center points of $\Gamma^i$ and $\Gamma^o$, respectively, and set $\omega = |c_{\hat{\theta}} -c_{\hat{x}}|$. For $L < \frac{\pi}{2}$, the maximal area of coverage happens at $\omega = \frac{\pi}{2}$. For $L \ge \pi$, it happens when $\omega \leq L-\pi$. In general, a larger area of coverage means a larger effective rank and more significant modes in the signal space. The practical implication is that only those frequency modes of $q$ inside $\mathcal{P}(\Gamma^o\times\Gamma^i)$ can be recovered well from the measurement data. Conversely, one can use both the area of coverage and information on $q$'s frequency modes as a prior to choose an appropriate limited aperture setup when possible. Moreover, the geometry of $\mathcal{P}(\Gamma^o\times\Gamma^i)$ implies reconstruction resolution and stability in different frequency modes. For example, Figure ~\ref{fig:limited_angle}(a) shows that the transmission measurement covers Fourier modes close to the transversal (with respect to the central direction of the aperture) directions with much more extent than those close to the range directions, which means much better imaging resolution and stability along the transversal direction than along the range direction. 
%The covered frequency region corresponding to transmission measurement is symmetric with respect to the origin due to the symmetry in incident directions and observation directions and results in double coverage and redundancy since $q$ is real.
The skewness towards low frequency implies better stability but worse resolution. In comparison, the region of $p$-transformed reflection measurement, as demonstrated in Figure ~\ref{fig:limited_angle}(e), not only has a smaller coverage in the frequency domain but also is highly concentrated on the high-frequency modes along the illumination directions. It implies a smaller effective rank and signal space dimension and, more importantly, more instability in the reconstruction, since in the higher-frequency regime, the nonlinearity is stronger (reflected by the remainder estimate in Section ~\ref{sec: analysis}), making the reconstruction more sensitive to the initial guess and perturbations, e.g., noise. The configuration shown in Figure~\ref{fig:limited_angle}(c), where $\omega = \frac{\pi}{2}$, renders the largest area of coverage in the Fourier domain among all configurations. From Figure~\ref{fig:limited_angle}, one can see that, as the configuration transitions from transmission to reflection, i.e., as $\omega$ increases, the covered frequency region becomes more and more skewed towards higher frequency. 

In summary, 1) $\mathcal{P}(\Gamma^o\times\Gamma^i)$ corresponds to those Fourier modes that can be recovered well from the measurement data, 2) lower frequency modes can be recovered more stably than higher frequency modes in general, 3) a reconstruction result depends on to what extent the significant Fourier modes of the medium contrast $q(x)$ are
contained in $\mathcal{P}(\Gamma^o\times\Gamma^i)$ corresponding to a specific measurement setup.   
Our numerical experiments in Section~\ref{sec: numerical on limited aperture} corroborate our analysis.

%An iterative method usually reconstructs the medium from low-frequency to high-frequency components, since higher frequency strengthens the non-linearity (reflected by the remainder estimate in Section ~\ref{sec: analysis}), making reconstruction more sensitive to the initial guess. A lower frequency generally means a more stable reconstruction.

\begin{figure}[htbp]
\centering
\begin{tabular}{c@{\vspace{2pt}}c@{\vspace{2pt}}c@{\vspace{2pt}}c@{\vspace{2pt}}c}
\includegraphics[width=0.2\linewidth]{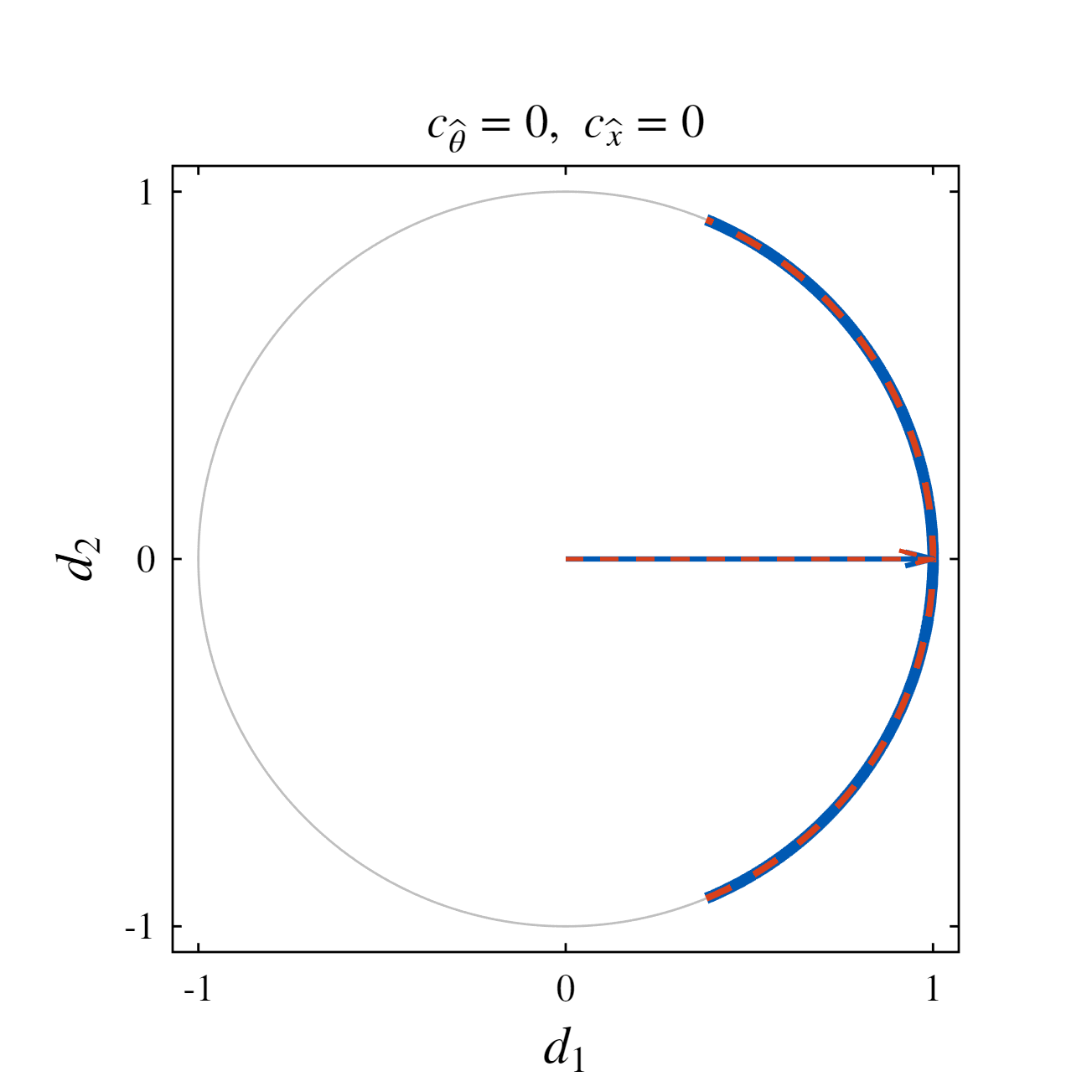}&
\includegraphics[width=0.2\linewidth]{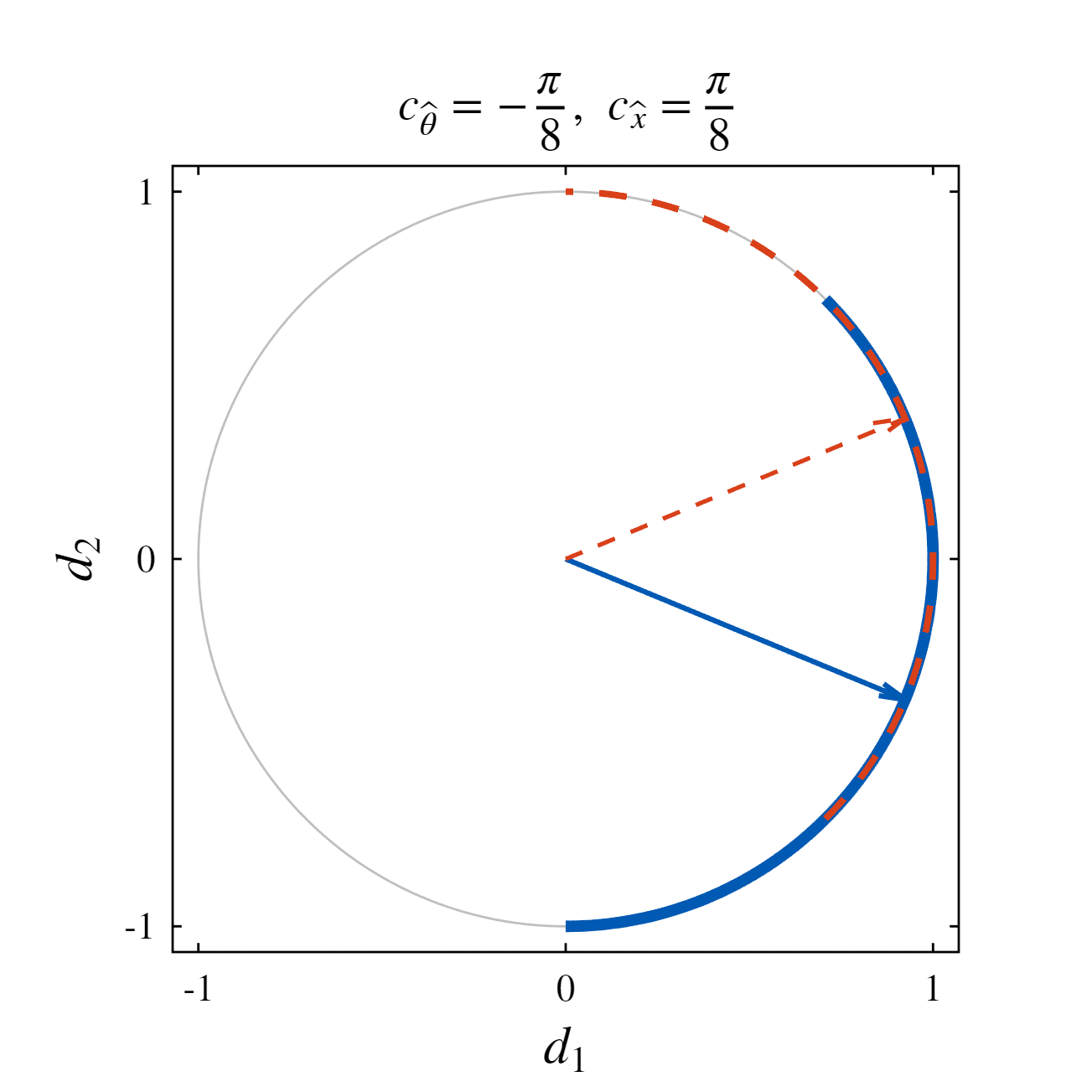}&
\includegraphics[width=0.2\linewidth]{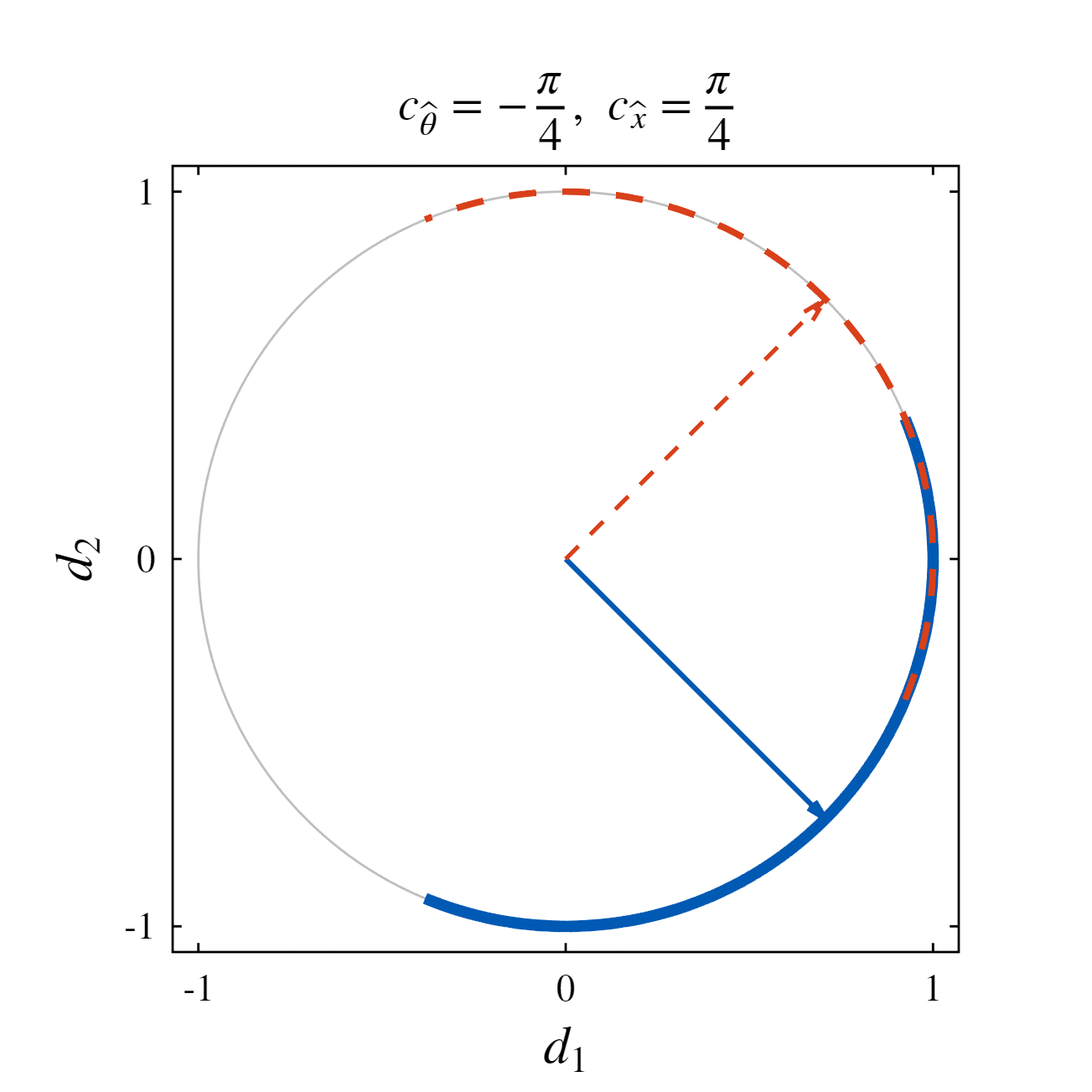}&
\includegraphics[width=0.2\linewidth]{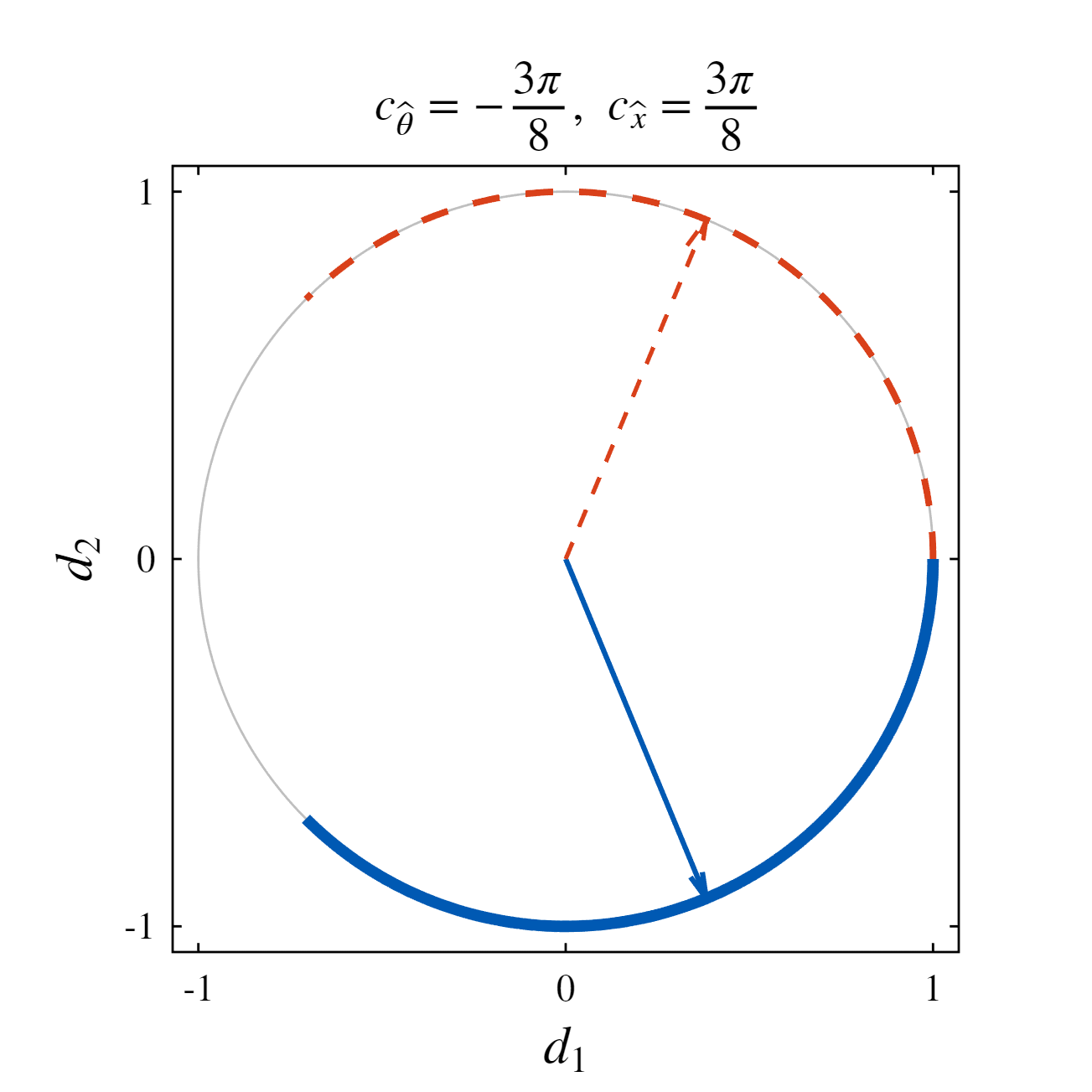}&
\includegraphics[width=0.2\linewidth]{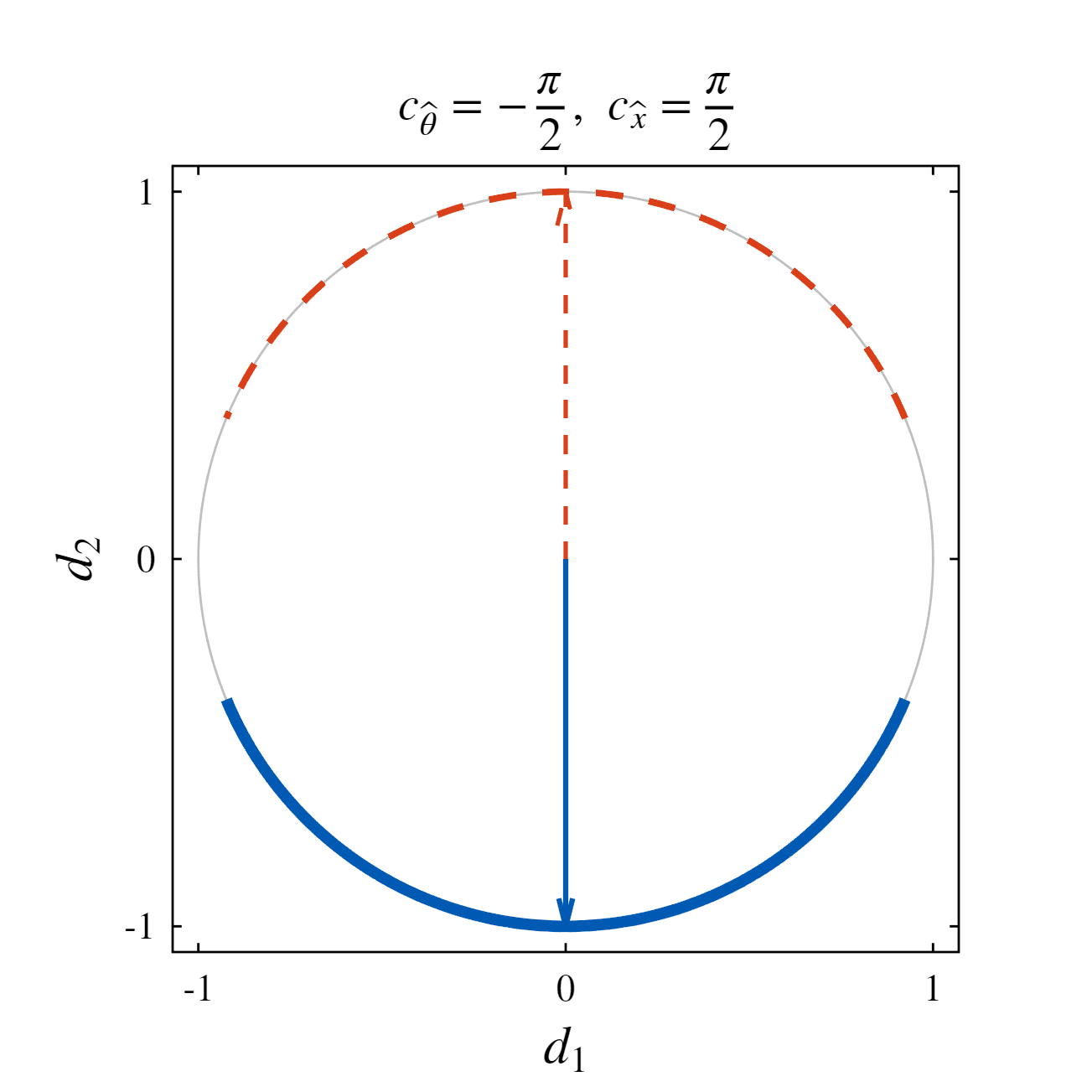}\\
\includegraphics[width=0.2\linewidth]{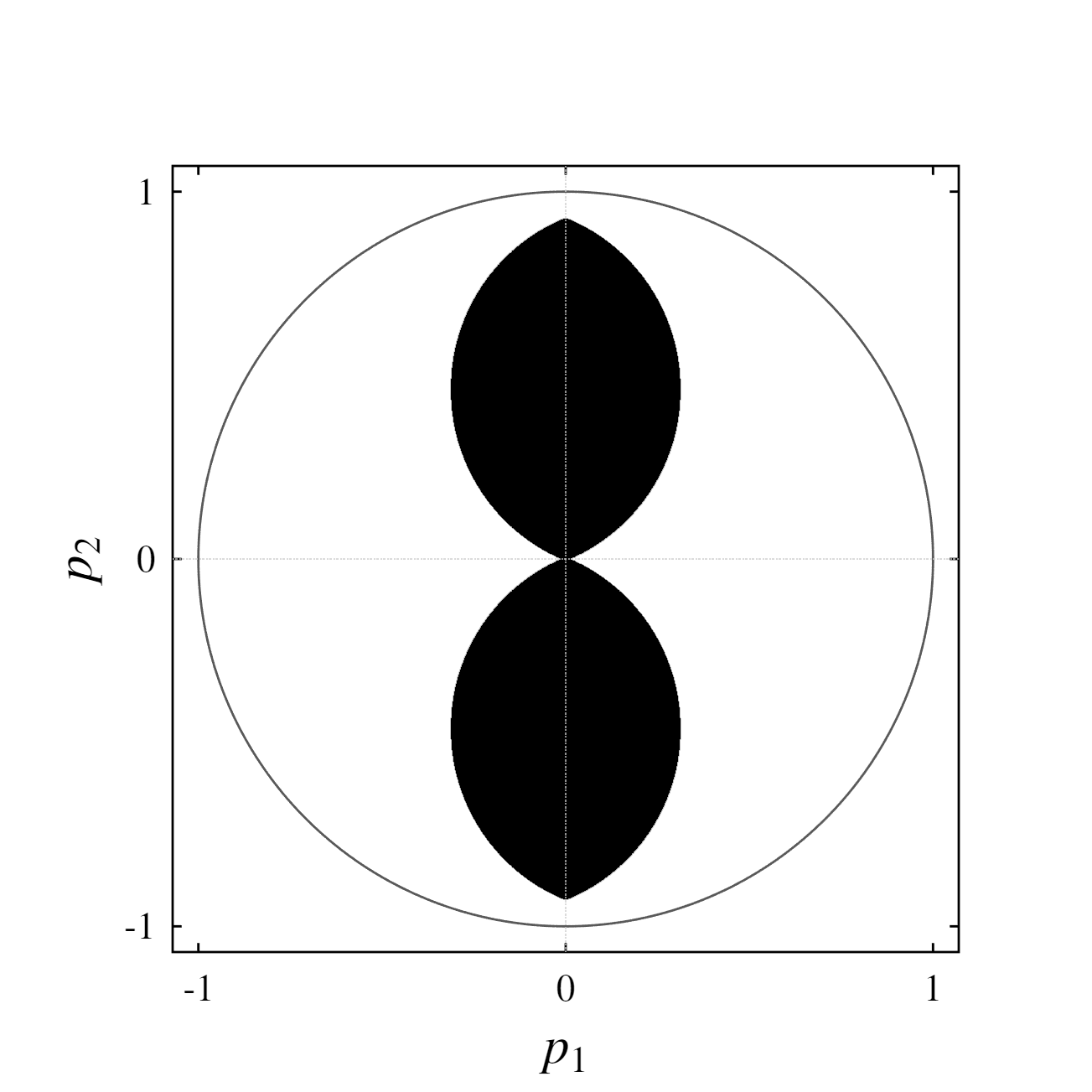}&
\includegraphics[width=0.2\linewidth]{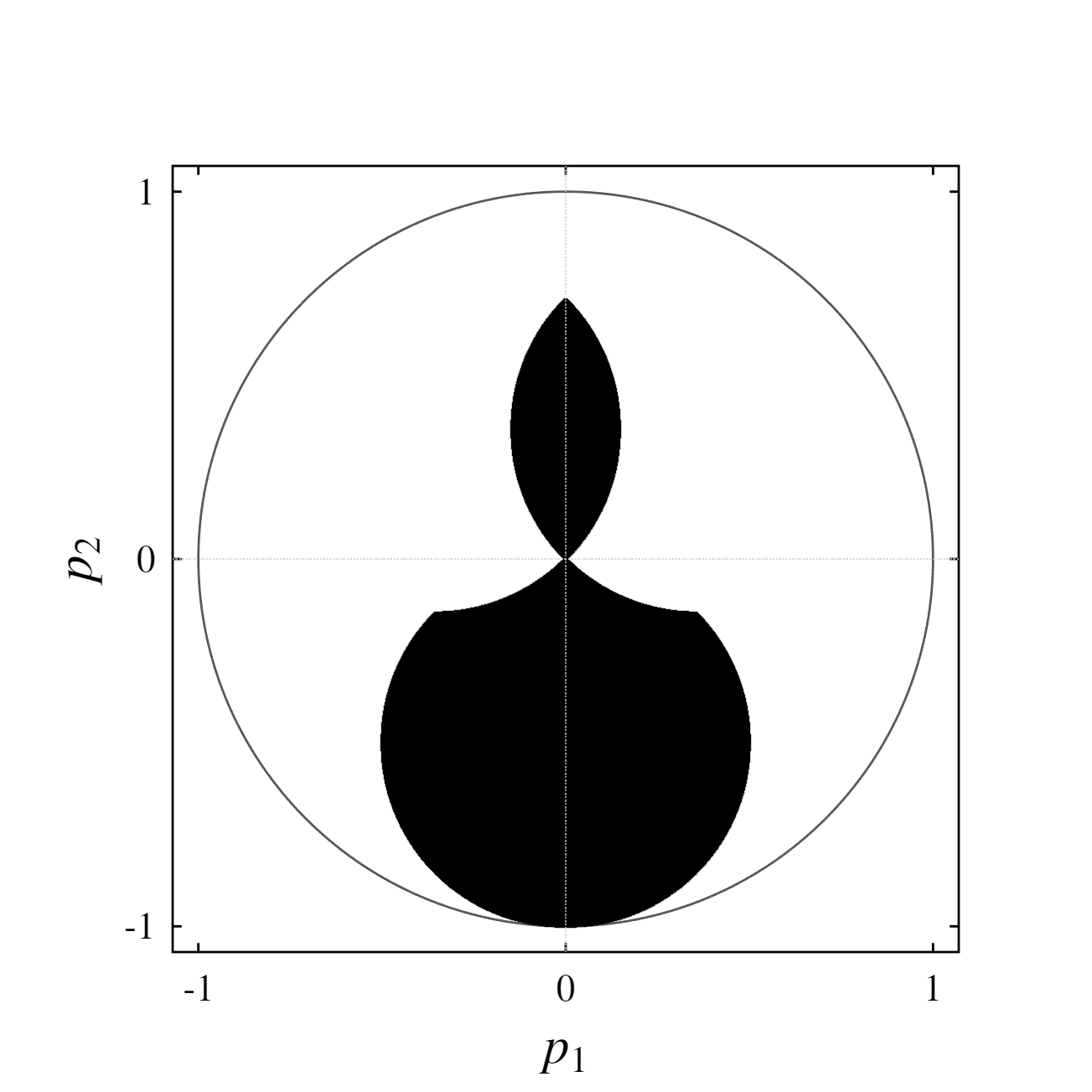}&
\includegraphics[width=0.2\linewidth]{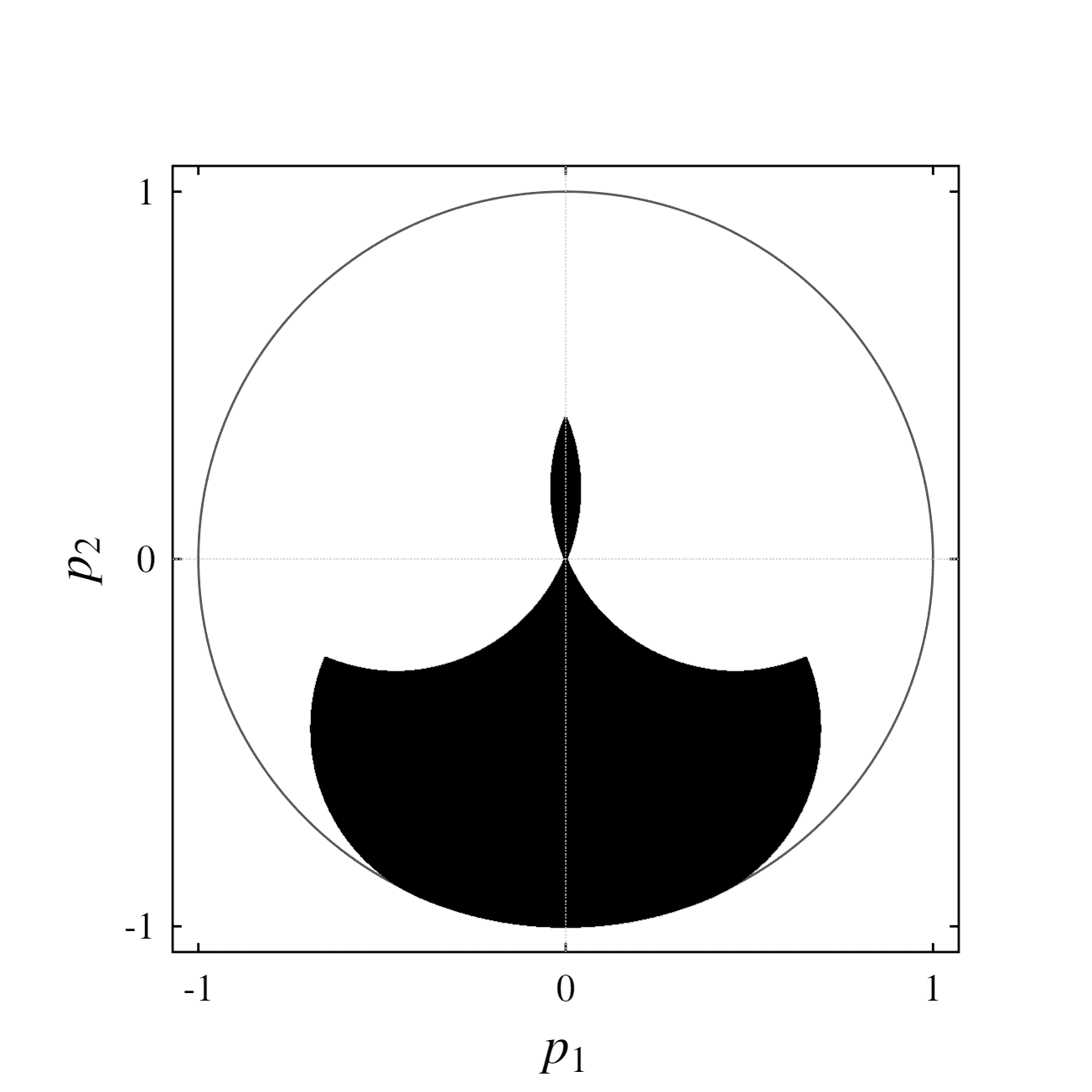}&
\includegraphics[width=0.2\linewidth]{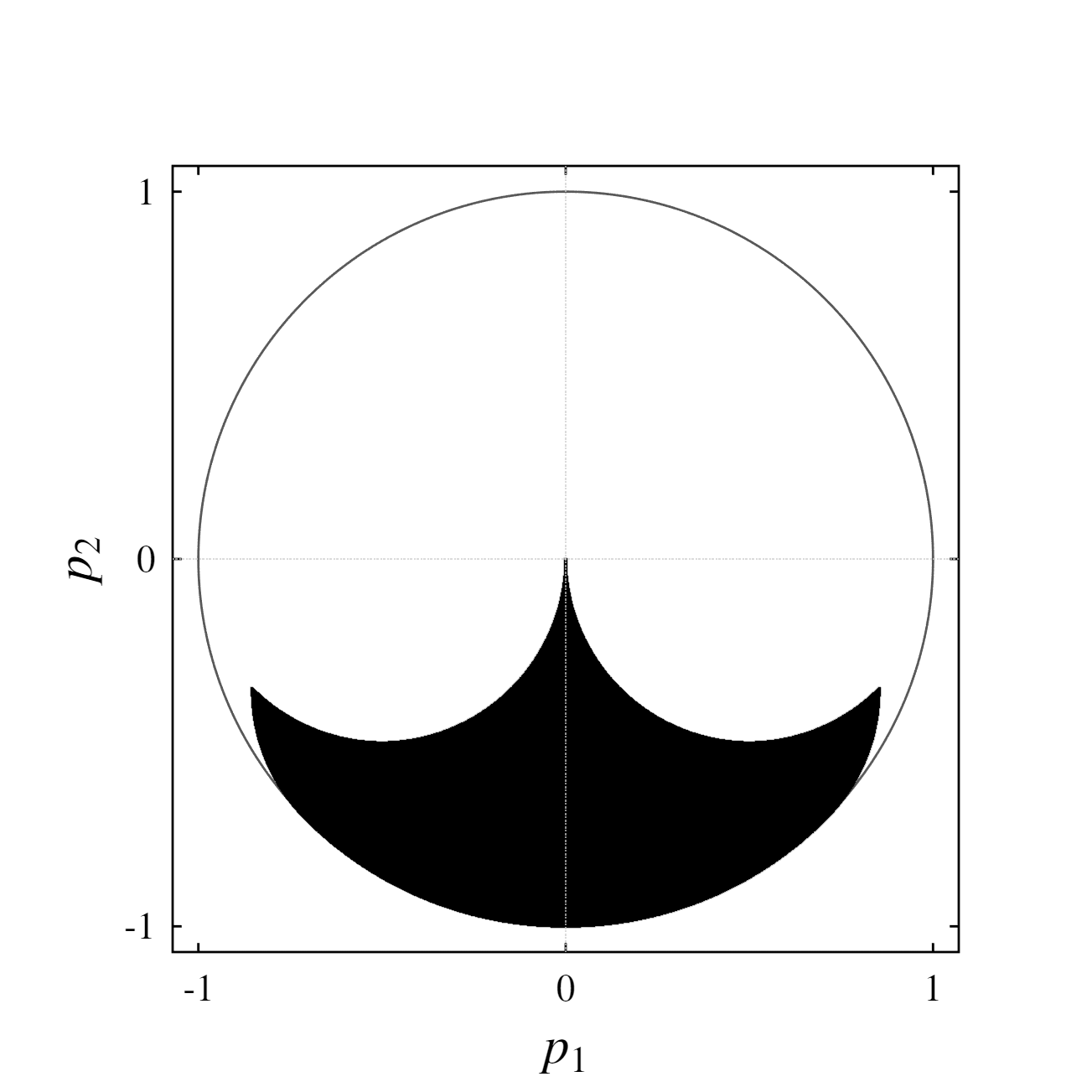}&
\includegraphics[width=0.2\linewidth]{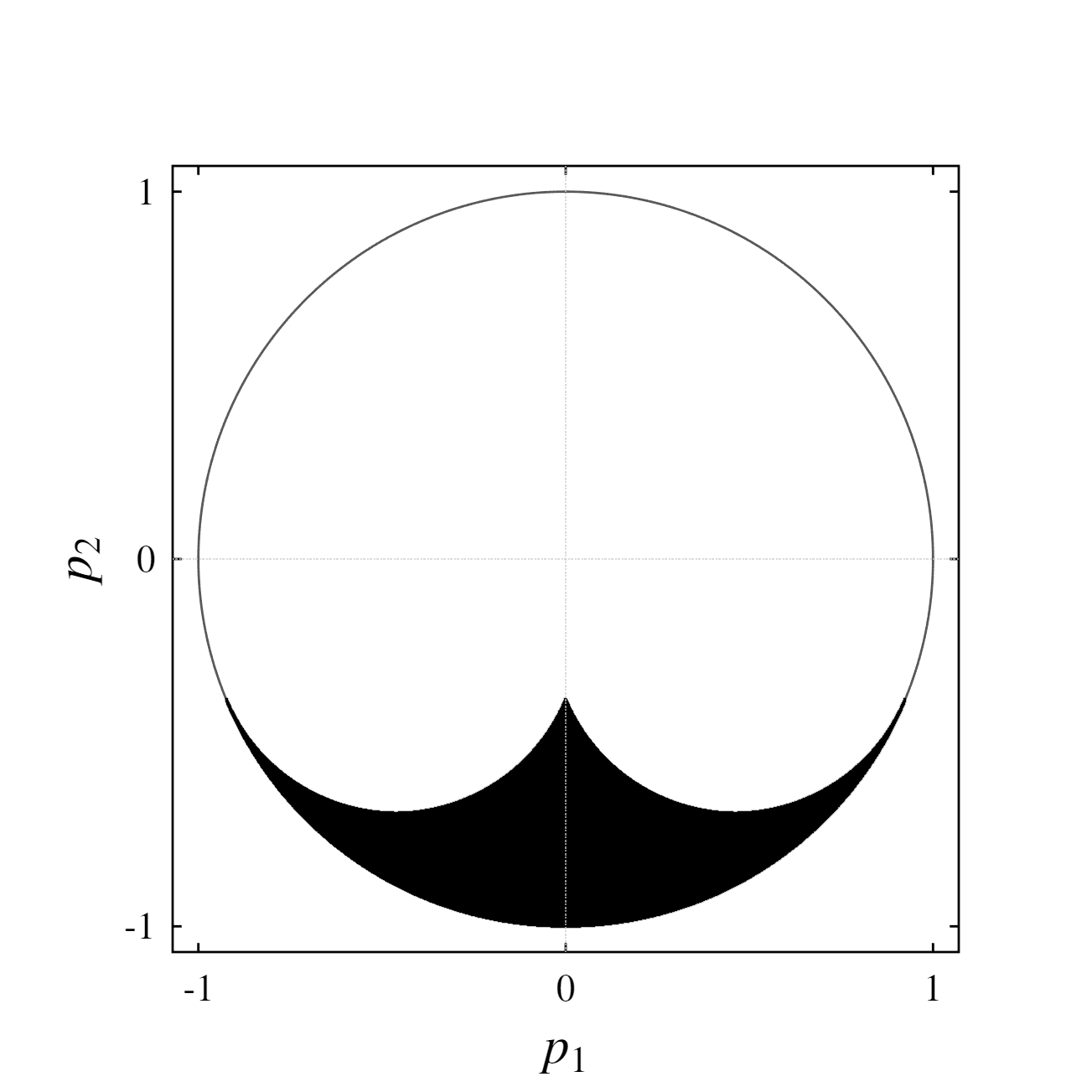}
\\
(a) $\omega = 0$ &(b) $\omega = \frac{\pi}{4}$ &(c) $\omega = \frac{\pi}{2}$ &(d) $\omega = \frac{3}{4}\pi$ &(e) $\omega = \pi$
\end{tabular}
\caption{Top: incident
and observation windows; the blue and orange curves show the incident directions $\hat{\theta}$ and observation directions $\hat{x}$, respectively. Bottom: corresponding $p$-space (Fourier domain) coverage in black.}
\label{fig:limited_angle}
\end{figure}

% \begin{figure}[H]
% \centering
% \begin{minipage}{0.75\linewidth}
%     \noindent
%     \includegraphics[width=0.905\linewidth]
%         {limited aperture/separation of angles/windows.png}

%     \par\vspace{4pt}

%     \noindent\makebox[\linewidth][l]{%
        
%         \includegraphics[width=0.905\linewidth]
%             {limited aperture/separation of angles/p_bands.png}%
%     }
% \end{minipage}
% \caption{Top: incident
% and observation windows; blue curve gives $\{\hat{\theta}\}$, orange curve gives $\{\hat{x}\}$. Bottom: corresponding $p$-space (Fourier domain) coverage in black.}
% \label{fig:limited_angle}
% \end{figure}

% In \cite{Meng_2023}, it assumes that $\{\hat{\theta} - \hat{x} \}$ contains a small neighborhood of the origin so that generalized PSWFs exist. However, this assumption is not necessary for our algorithm.
% From a more general perspective, a density argument of the data can be made, i.e., the data is only dense in $\mathbb{S}_L$, which leads to a general framework that connects the continuous case for analysis and the discrete case for numerical computation.

\begin{remark}\label{remark: small disk}
In practice, one has discrete measurements, i.e., discrete $\hat{\theta}_i, \hat{x}_j$ and corresponding $p_{i,j}$. An important question is how to compute an approximate measure for a set of discrete measurement points in $p$-space. For a discrete set that has dense enough sampling points to approximate a continuous limit (for full or limited aperture), $\mu(\mathcal{P})$ in \eqref{eq:Shannon number} can be computed. In the general situation, we propose to approximate $\mu(\mathcal{P})$ by the area of the union of discs of radius $\frac{\rho}{k}$ centered at each discrete point $p\in B(0,1)$, where $\rho$ is some constant which we will specify later in Section \ref{sec: numerical}. 
\end{remark}

\begin{remark} \label{remark:3D}
The only difference for the $p$-transform in 3D is that $\mathbb{S}^2 \times \mathbb{S}^2$ is one dimension higher than $B_3(0,1) \subset \mathbb{R}^3$, hence each point $p\ne 0 \in B_3(0,1)$ corresponds to two circles on the unit sphere whose lengths are $2\pi\sqrt{1-|p|^2}$. 
 Under the $p$-transform, the coarea formula~\cite{negro2021sampledistributiontheoryusing} gives, for integrable $g(\hat{x},\hat{\theta})$,
\begin{equation}\label{eq:p transform in 3D}
    \int_{S^2\times S^2}g(\hat{x},\hat{\theta})
                  d\sigma_{\hat{x}}d\sigma_{\hat{\theta}}=\int_{B_3(0,1)}\frac4{|p|}
       \int_0^{2\pi}g\bigl(\hat{x}(p,\phi),\hat{\theta}(p,\phi)\bigr)
                    d\phi dp,
\end{equation}
where $\phi$ is a parametrization along the fiber which corresponds to the same $p$. 
Hence, the unweighted discrete least squares~\eqref{eq: discrete unweighted} can lead to a low-frequency bias and can be corrected by introducing the Jacobian weights as in the 2D case. All the study and analysis for 2D far-field data can be extended to 3D. For example, the corresponding Shannon number is
\begin{equation}\label{eq:Shannon number 3D}
    S(\Omega;\mathcal{P}) = (\frac{k}{\pi})^3\mu(\Omega)\cdot\mu(\mathcal{P}),
\end{equation}
where $\mathcal{P}$ follows the same definition in 2D case.
% On the other hand, the study of PSWFs has been extended to 3D in \cite{zhou2026theoreticalcomputationalframeworkdimensional}.
\end{remark}

\subsection{Near-Field Case}
In the near-field case, assume we have a source-observation location pair: $(x,z)$, which are outside the domain of interest $\Omega$. The measurement, $u^s(x;z)$, corresponding to the scattered wave field at $x$ due to a point source at $z$ is given by the Lippmann-Schwinger integral equation~\eqref{eq:near-field}. In the following, we will provide asymptotic estimates for the effective rank using the Born approximation,
\begin{equation}\label{near-field-Born}
    u^s(x;z) = k^2\int_{\Omega} G_k(x,y)G_k(y,z)q(y)dy, 
\end{equation}
based on the separability of the Green's function of the Helmholtz equation in the high-frequency limit analyzed in Engquist and Zhao~\cite{Zhao-Engquist}. 

\textbf{Physical Setup.} Assume the domain of interest $\Omega\in \mathbb{R}^d, d=2, 3$ is compact. All near-field point sources and observation points are located on smooth compact manifolds, source manifold $\Gamma^s$ and observation manifold $\Gamma^o$, respectively, that are outside of $\Omega$. 

The kernel of the compact near-field operator defined by \eqref{near-field-Born} from the domain of interest $\Omega$ to the source-observation domain, $L^2(\Omega) \rightarrow L^2(\Gamma^2)$, where $\Gamma^2:=\Gamma^s\times\Gamma^o$, is 
\[
\mathcal{K}_k(x,z;y):= G_k(x,y)G_k(y,z), \quad y\in \Omega, ~~x\in \Gamma^o, ~~z\in  \Gamma^s.
\]

\textbf{Approximate Separability.} The approximate separability of the kernel ${\cal K}_k(x,z;y)$ is defined as the following: given $\epsilon >0$, there is a smallest integer $N^\epsilon_k \in {\cal N}$ such that there are $f_n(y), g_n(x,z), n=1, 2, \ldots, N^\epsilon_k,$ 
\[
\|{\cal K}_k(x,z;y)-\sum_{n=1}^{N^\epsilon_k}f_n(y)g_n(x,z)\|_{L^2(\Omega\times \Gamma^2)}\le \epsilon \|{\cal K}_k(x,z;y)\|_{L^2(\Omega\times \Gamma^2)}, \quad y\in \Omega, ~~x\in \Gamma^o, ~~z\in \Gamma^s.
\]
In our physical setup, ${\cal K}_k$ is a Hilbert-Schmidt (integral) operator from $L_2(\Omega)$ to $L_2(\Gamma^2)$ and let $\lambda_1\ge \lambda_2 \ge \cdots \ge \lambda_m \ge \cdots \ge 0$ be the singular values of  the operator,
 %singular value decomposition (SVD) of $G(\bx,\by)$,  
 then $N_k^{\epsilon}$ is the smallest integer such that $\sum_{n=N_k^{\epsilon}+1}^{\infty}\lambda^2_n\le \epsilon^2 \sum_{n=1}^{\infty}\lambda^2_n$. If one views ${\cal K}_k(x,z;y)$ as a family of functions in $L^2(\Omega)$ parameterized by $(x,z)\in \Gamma^2$ (or vice versa), this is related to the Kolmogorov $n$-width
\footnote{Kolmogorov $n$-width of a set $S$ in a normed space $W$ is its worst-case distance to the best $n$ dimensional linear subspace $L_n$:
\[
d_n(S, W) :=\inf_{L_n}\sup_{f\in S}\inf_{g\in L_n} \|f-g\|_W. 
\]
} ~\cite{Kolmogorov}.
An estimate of $N_k^{\epsilon}$ in terms of $k$ and $\epsilon$ is of great interest in both theory and practice.
For example, the separability estimate implies the rank estimate in the discrete setting when the mesh size for $\Omega$ and measurement setup resolve the wavelength $\frac{2\pi}{k}$. In the following, we will provide the separability estimate of the kernel ${\cal K}_k(x,z;y)$ in the high-frequency limit. 

The following relations are used to simplify notations: $x \gtrsim y$ means that there is a constant $\infty > c>0$ such that $x\ge cy$,  $x \lesssim y$ means that there is a constant $\infty>C>0$ such that $x\le Cy$
%and $x \sim y$ means there are two constants $0<c \le C <\infty$ such that $cy\le x \le Cy$. 
For our asymptotic results, all constants are independent of the wave number $k$ as $k\rightarrow \infty$. Generic constants $c$ and $C$ will typically depend on the geometry of $\Omega, \Gamma^s, \Gamma^o$ and their relative locations. Dependence of these constants on other factors will be specified in the statement.

\begin{theorem}\label{thm:near-field-2d}
Assume $\Gamma^s, \Gamma^o$ are smooth curves outside the compact $\Omega \subset \mathbb{R}^2$ that do not contain line segments, then there are constants $C, C_1, C_2 >0$ independent of $k$ and $\epsilon$ such that 
\[
C_1[1-(C_2\epsilon)^2]^2   k^{2-\delta}\le  N^\epsilon_k \le C k^{2+\delta}, \quad \mbox{as} ~k\rightarrow \infty.
\]
\end{theorem}
\begin{proof}

We first estimate the decorrelation rate between $G_k(\cdot,y_1)$ and $G_k(\cdot,y_2)$ as functions along $\Gamma^o$
following the argument in Theorem 2.1 and Remark 2.4 in ~\cite{Zhao-Engquist}.  Given two points $y_1, y_2\in \Omega$, the ray, which is a straight line in a homogeneous medium, passing through the two points is a line, denoted by $l_{y_1}^{y_2}$.

Case 1: If $\Gamma^o$ is a closed curve without straight pieces, $l_{y_1}^{y_2}$ intersects $\Gamma^o$ outside the interval between $y_1$ and $y_2$ at a finite number (most likely two) of points. Since $\Gamma^o$ is a closed smooth curve of dimension 1, these intersections are points of stationary phase for the oscillatory function $G_k(x,y_1)\bar{G}_k(x,y_2)$ of $x$ along $\Gamma^o$ and contribute to the leading terms for the integral of order $\mathcal{O}(k^{-\frac{1}{2}})$ according to the stationary phase theorem \cite{evans, Zhao-Engquist}, i.e., 
\[
|\left\langle G_k(\cdot, y_1),G_k(\cdot,y_2)\right\rangle_{\Gamma^o}|=|\int_{\Gamma^o}G_k(x,y_1)\bar{G}_k(x,y_2) d\Gamma^o |\lesssim k^{-\frac{1}{2}}.
\]

Case 2: If $\Gamma^o$ is not closed, there are two situations: 
\\
a) The ray $l_{y_1}^{y_2}$ still intersects $\Gamma^o$ at a finite number of points. Then these stationary points still contribute to the leading terms and yield the above estimate.
\\
b) The ray $l_{y_1}^{y_2}$ does not intersect $\Gamma^o$. Using integration by parts, the leading terms of the contribution come from the boundary, which are the endpoints of $\Gamma^o$, with order $\mathcal{O}(k^{-1})$.

In summary, the following estimate holds 
\begin{equation}\label{eq:decorrelation-2d-1}
|\left\langle G_k(\cdot, y_1),G_k(\cdot,y_2)\right\rangle_{\Gamma^o}|=|\int_{\Gamma^o}G_k(x,y_1)\bar{G}_k(x,y_2) d\Gamma^o |\lesssim k^{-\beta_1}, \quad \beta_1 \ge \frac{1}{2}.
\end{equation}

For the same reason, we have 
\begin{equation}\label{eq:decorrelation-2d-2}
|\left\langle G_k(y_1, \cdot),G_k(y_2,\cdot)\right\rangle_{\Gamma^s}|=|\int_{\Gamma^s}G_k(y_1,z)\bar{G}_k(y_2, z) d\Gamma^s |\lesssim k^{-\beta_2},
\quad \beta_2 \ge \frac{1}{2}.
\end{equation}
Combining the above two estimates, we have 
\begin{equation*}
\begin{array}{l}
|\left\langle {\cal K}_k(\cdot, y_1),{\cal K}_k(\cdot,y_2)\right\rangle_{\Gamma^2}|=|\iint_{\Gamma^2}{\cal K}_k(x,z;y_1)\bar{\cal K}_k(x,z;y_2) d\Gamma^2 |
\\
=|\left\langle G_k(\cdot, y_1),G_k(\cdot,y_2)\right\rangle_{\Gamma^o}||\left\langle G_k(y_1, \cdot),G_k(y_2,\cdot)\right\rangle_{\Gamma^s}|\lesssim k^{-\beta_1-\beta_2}\lesssim k^{-1}.
\end{array}
\end{equation*}
For $\Omega \subset \mathbb{R}^d, d=2$, the decorrelation power $\beta_1+\beta_2\ge \frac{d}{2}=1$. The result comes from a direct application of Theorem 3.1 and Theorem 3.2 in \cite{Zhao-Engquist}.
\end{proof}

\begin{theorem}\label{thm:near-field-3d}
Assume $\Gamma^s, \Gamma^o$ are smooth manifolds outside the compact $\Omega \subset \mathbb{R}^3$ that do not contain line segments. If one of $\Gamma^s, \Gamma^o$ is a surface and the other is a surface or a curve,  then there are constants $C, C_1, C_2 >0$ independent of $k$ and $\epsilon$ such that 
\[
C_1[1-(C_2\epsilon)^2]^2   k^{3-\delta}\le  N^\epsilon_k \le C k^{3+\delta},
\quad \mbox{as} ~k\rightarrow \infty\]
\end{theorem}
\begin{proof}
 Without loss of generality, take $\Gamma^o$ to be a smooth surface without line segments. Again, we follow the argument in Theorem 2.1 and Remark 2.4 in ~\cite{Zhao-Engquist}. Given two points $y_1, y_2\in \Omega$, the ray, which is a straight line in a homogeneous medium, passing through the two points is a line, denoted by $l_{y_1}^{y_2}$.

 Case 1: If $\Gamma^o$ is a closed surface without line segments, $l_{y_1}^{y_2}$ intersects $\Gamma^o$ outside the interval between $y_1$ and $y_2$ at a finite number (most likely two) of points.  Since $\Gamma^o$ is a closed smooth surface of dimension 2, these intersections are points of stationary phase for the oscillatory function $G_k(x,y_1)\bar{G}_k(x,y_2)$ of $x$ along $\Gamma^o$ and contribute to the leading terms for the integral of order $\mathcal{O}(k^{-1})$ according to the stationary phase theorem \cite{evans, Zhao-Engquist}, 
\[
|\left\langle G_k(\cdot, y_1),G_k(\cdot,y_2)\right\rangle_{\Gamma^o}|=|\int_{\Gamma^o}G_k(x,y_1)\bar{G}_k(x,y_2) d\Gamma^o |\lesssim k^{-1}.
\]

 Case 2: If $\Gamma^o$ is a surface that contains no line segment and has boundaries, there are two situations: 
\\
a) The ray $l_{y_1}^{y_2}$ still intersects $\Gamma^o$ at a finite number of points. Then these stationary points still contribute to the leading terms and yield the same estimate as above.
\\
b) The ray $l_{y_1}^{y_2}$ does not intersect $\Gamma^o$. Using integration by parts, the leading terms of the contribution come from the boundaries, which are smooth curves of dimension 1, of order $\mathcal{O}(k^{-\frac{3}{2}})$.

In summary, the following estimate holds 
\begin{equation}\label{eq:decorrelation-3d-1}
|\left\langle G_k(\cdot, y_1),G_k(\cdot,y_2)\right\rangle_{\Gamma^o}|=|\int_{\Gamma^o}G_k(x,y_1)\bar{G}_k(x,y_2) d\Gamma^o |\lesssim k^{-\beta_1}, \quad \beta_1 \ge 1.
\end{equation}

Now consider $\Gamma^s$.
\begin{enumerate}
    \item 
 If $\Gamma^s$ is a smooth surface and does not contain line segments, the same estimate as for $\Gamma^o$ holds. 
 \item 
 If $\Gamma^s$ is a smooth curve without line segments, we get the same estimate as in Case 2 in Theorem \ref{thm:near-field-2d}. 
\end{enumerate}
In either case, we have 
\begin{equation}\label{eq:decorrelation-3d-2}
|\left\langle G_k(y_1, \cdot),G_k(y_2,\cdot)\right\rangle_{\Gamma^s}|=|\int_{\Gamma^s}G_k(y_1,z)\bar{G}_k(y_2, z) d\Gamma^s |\lesssim k^{-\beta_2},
\quad \beta_2 \ge \frac{1}{2}.
\end{equation}

Combining estimates $\Gamma^o$ and $\Gamma^s$, we have 
\begin{equation*}
\begin{array}{l}
|\left\langle {\cal K}_k(\cdot, y_1),{\cal K}_k(\cdot,y_2)\right\rangle_{\Gamma^2}|=|\iint_{\Gamma^2}{\cal K}_k(x,z;y_1)\bar{\cal K}_k(x,z;y_2) d\Gamma^2 |
\\
=|\left\langle G_k(\cdot, y_1),G_k(\cdot,y_2)\right\rangle_{\Gamma^o}||\left\langle G_k(y_1, \cdot),G_k(y_2,\cdot)\right\rangle_{\Gamma^s}|\lesssim k^{-\beta_1-\beta_2}\lesssim k^{-\frac{3}{2}}.
\end{array}
\end{equation*}
For $\Omega \subset \mathbb{R}^d, d=3$, the decorrelation power $\beta_1+\beta_2\ge \frac{d}{2}=\frac{3}{2}$. The result comes from a direct application of Theorem 3.1 and Theorem 3.2 in \cite{Zhao-Engquist}.

\end{proof}

\begin{remark}
The above separability result can be extended to a heterogeneous medium inside $\Omega$ as long as there are no trapped rays inside the domain of interest $\Omega$. Hence, each ray passing through two points $y_1,y_2\in \Omega$ will leave $\Omega$ and become a straight line. The geometric scenarios are similar to those discussed above as shown in Section 2.2 in \cite{Zhao-Engquist}.
\end{remark}
\begin{remark}
The above analysis can also be applied to far-field data. For example, the kernel for the far-field operator for 2D in homogeneous medium, $\Jacobian_0$ in \eqref{eq:J0}, is $e^{i k(\hat\theta-\hat{x})\cdot y}$.  The correlation between the kernels associated with two points $y_1,y_2\in \Omega$ is proportional to \(J_0^2(k|y_1-y_2|)\), and therefore decays $~\simeq (k|y_1-y_2|)^{-1}$ as $k\rightarrow \infty$. Hence, one can derive the same separability results. 
\end{remark}

The above estimates show that with dimension match between the physical space and the measurement space, i.e., $\Gamma^o\times \Gamma^s$, 
the effective rank is $O(k^2)$ generically (bar those special conditions stated in the Theorems), if the mesh size and measurement resolution resolve the wavelength. However, the constant depends on the concrete settings of the physical domain, the medium, and the measurement configuration, which is difficult to estimate. In practice, after normalizing the physical domain in a unit ball, one may use RSVD algorithms to numerically estimate the effective rank in a homogeneous medium first and use it as a reference for subsequent iterations (see Section~\ref{sec: numerical} for more detailed discussion). 

In Figure ~\ref{fig:spectrum for both} we plot the spectrum of the near-field operator \eqref{near-field-Born} corresponding to the following setup: the domain of interest $\Omega$ is a unit disk $B$ and both the observation and the sources are collocated on the boundary of a disk of radius 2 co-centered with $\Omega$.
% We also introduce the $p$-transform for source-observation pairs corresponding to their far-field analogues.
% \begin{figure}[H]
%     \centering
%     \includegraphics[width=0.8\linewidth]{near field/SVD/nearfield_spectrum_q0.png}
%     \caption{Spectrum of \eqref{near-field-Born} with different $k$; the dash line indicates the Shannon number.}
%     \label{fig:SVD near field}
% \end{figure}
\begin{figure}[H]
\centering
\begin{tabular}{@{}c@{\hspace{6pt}}c@{}}
\includegraphics[width=0.48\linewidth]{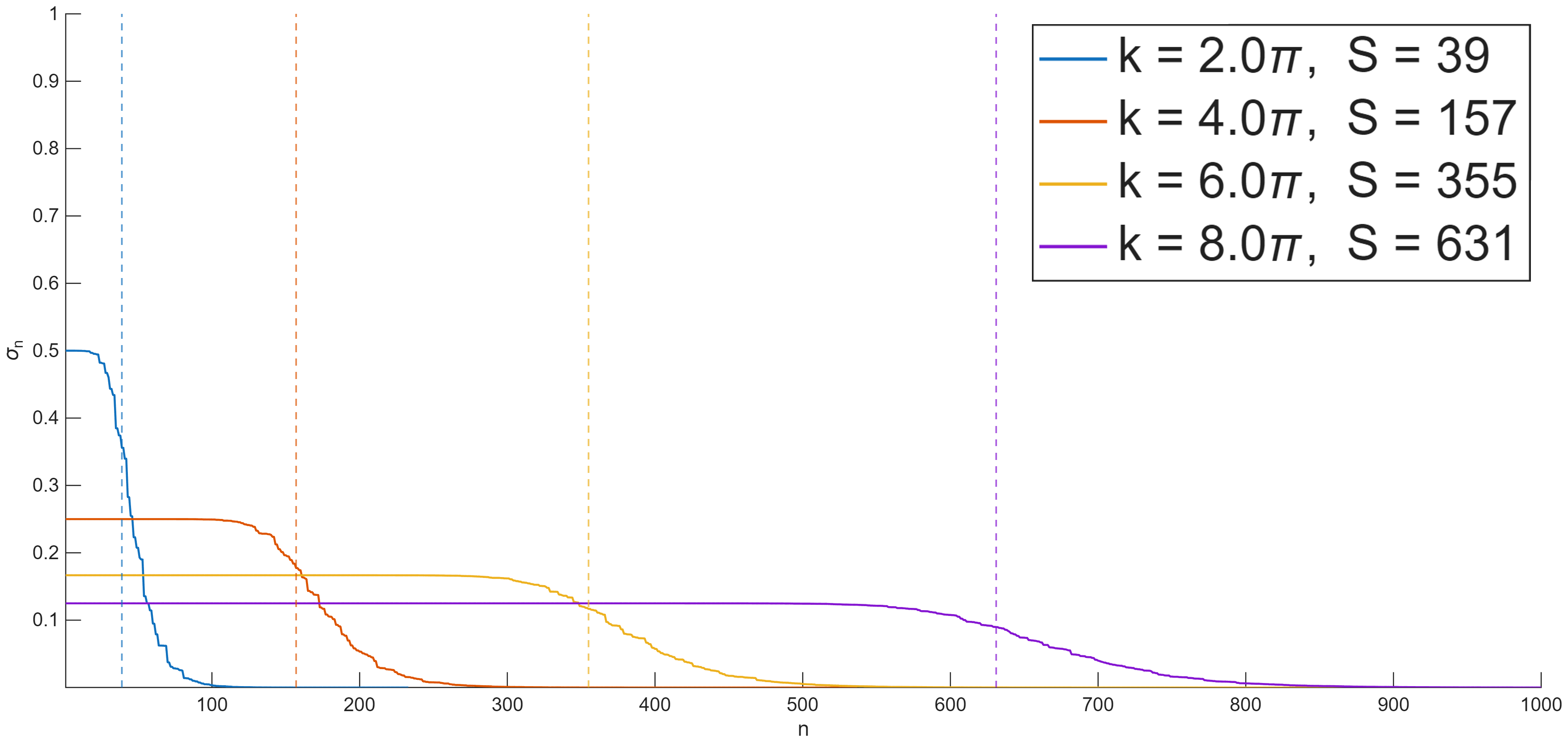}
&
\includegraphics[width=0.48\linewidth]{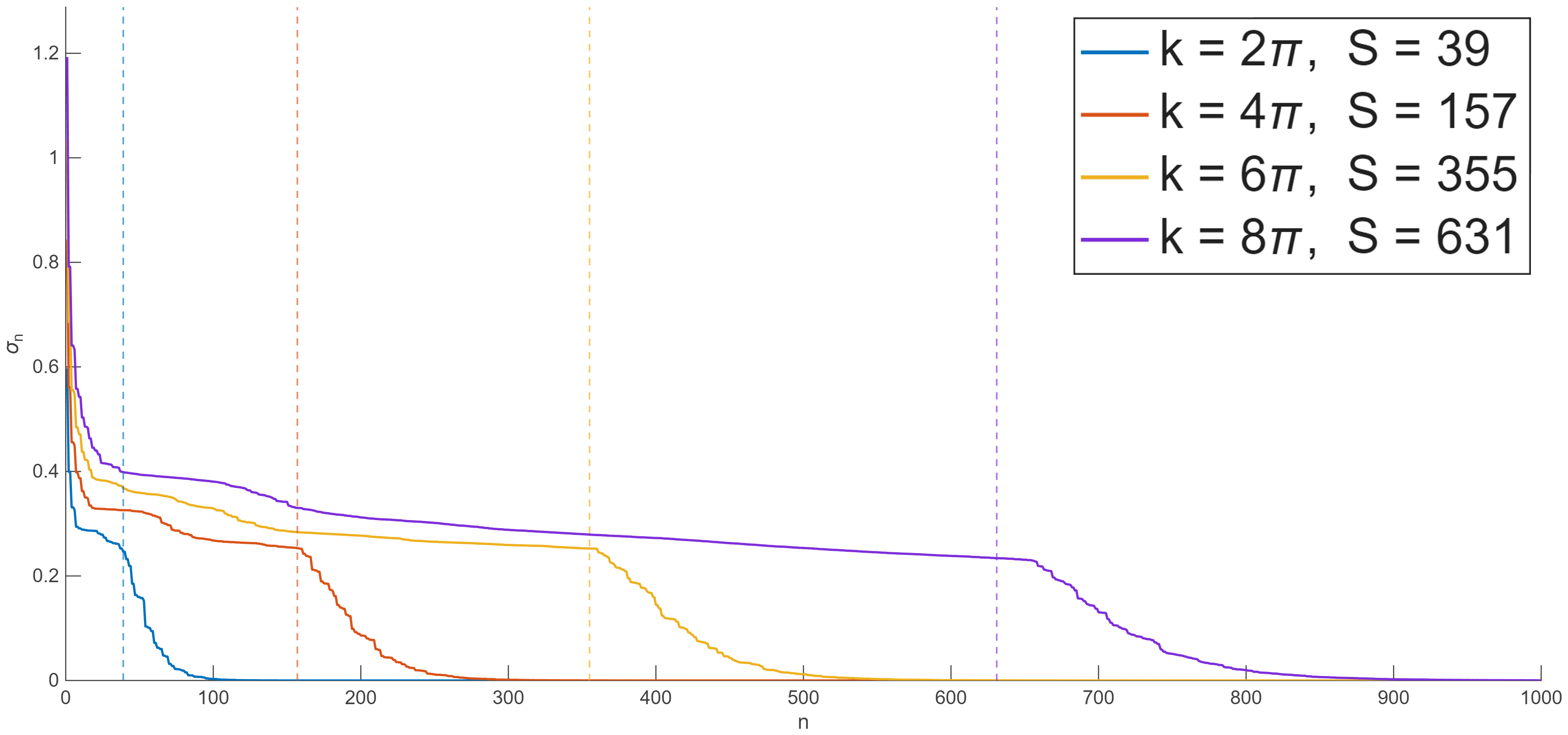}
\\[2pt]
(a) far-field operator $\hat{\Jacobian}_0$ \eqref{eq:J0} with $\Omega = B$ & (b) near-field operator \eqref{near-field-Born}
\end{tabular}
\caption{Spectrum of the linearized operator with different $k$. The dashed line indicates the Shannon number.}
\label{fig:spectrum for both}
\end{figure}

\section{iTSVD Algorithm}\label{sec: algorithm}
iTSVD is an iterative Newton's method that solves the linearized problem at each step using TSVD based on effective rank estimation and RSVD for efficient computation of leading singular values and singular vectors. $\hat{\mathcal{J}}_{q_0}$ is compact from $L^2(\Omega)$ to
$L^2(B)$. 
By compactness, there exists a singular value and corresponding singular mode system $\{\hat{\sigma}_i,\hat{u}_i,\hat{v}_i\}_{i=1}^\infty$ such that
\begin{equation}\label{eq:SVD}
\hat{\mathcal{J}}_{q_0}\hat{v}_i= \hat{\sigma}_i \hat{u}_i,
\qquad 
\{\hat{u}_i\}\subset L^2(B),\ \{\hat{v}_i\}\subset L^2(\Omega),
\end{equation}
where $\{\hat{u}_i\}$ and $\{\hat{v}_i\}$ form orthonormal bases of $L^2(B)$
and $L^2(\Omega)$ respectively. Given the current state $q_0$ and the corresponding singular  decomposition of the linearized operator $\hat{\Jacobian}_{q_0}$ by \eqref{eq:SVD} and the optimization \eqref{eq:min_p}, the TSVD solution with effective rank $N$ from measurement mismatch $\delta\hat{u}^\infty=\hat{u}^\infty-\mathcal{F}q_0$ yields the new update
\begin{equation}\label{eq:update}
a_i :=\langle \delta\hat{u}^\infty, \hat{u}_i\rangle_{L^2(B)}, ~ i=1, \ldots,N \quad  \delta q= \sum_{i=1}^N \frac{1}{\hat{\sigma}_i} a_i\, \hat{v}_i(y),\quad  \hat{q}(y)= q_0(y)+\delta q.
\end{equation}

The choice of effective rank $N$ controls the trade-off between approximation accuracy and numerical stability for the linearized problem. While Tikhonov regularization is a classical technique for stabilizing the inversion of modes with small singular values, it does so by introducing a regularization bias. Here, we instead use truncated singular value decomposition (TSVD) for the inversion.
% \textcolor{cyan}{While Tikhonov regularization stabilizes the inversion by smoothly attenuating singular modes, TSVD leaves retained modes unattenuated and discards those below a prescribed threshold.} %Consequently by \eqref{eq:Shannon number}, we take $N = S$ for analysis. 
% With a physical argument introduced by \cite{Landau_Pollak} and the numerical stability concern, we perform truncated SVD for the inversion. 
An effective numerical technique for approximating the leading singular values and modes in a low-rank structure is RSVD. Hence, we exploit Algorithm 4.2 in \cite{Random_SVD} for the implementation of RSVD to acquire a fixed-precision approximation: given a matrix $A \in \mathbb{R}^{m\times n}$, a tolerance $\epsilon$, and an integer $r$, produce an orthonormal matrix $Q$ such that
\begin{equation}\label{eq:fixed-precision}
\bigl\|(I-QQ^*)A\bigr\|_{2} \le \epsilon ,
\end{equation}
with probability no smaller than $ 1 - \min\{m,n\}10^{-r}$. The key idea of RSVD to obtain a good approximation of $A$ is to have enough samples such that the singular values (in decreasing order) after the number of samplings is small enough w.r.t. $\epsilon$, which is justified by an estimate of the fixed-rank approximation; see Theorem 10.8 in \cite{Random_SVD}. The consequence is twofold for step 7 in Algorithm ~\ref{alg:iterative-low-rank}. First, the number of columns of $Q$ is close to the effective rank of $A$ dictated by $\epsilon$. Second, based on the effective rank estimate $S$ of $\hat{\Jacobian}$ and the fast transition property of its spectrum, we propose to use $2S + r$ samples in practice since $r$ dominates the number of samples when $k$ is small, while $S$ dominates when $k$ is large; the extra $S$ simply acts as a safeguard since the width of the transition region of the spectrum is $o(S)$. This allows us to perform RSVD in one bulk of samplings, which is computationally efficient compared to gradual rank-revealing with one-by-one sampling. Furthermore, RSVD only requires matrix-vector multiplication, which is crucial in our case since forming the large Jacobian matrix for the linearized operator $\hat\Jacobian_{q_0}$ for a modest-size inverse scattering problem can be memory prohibitive. Moreover, the application of $\hat\Jacobian_{q_0}$ on a set of random perturbations $\delta q$ is given by \eqref{eq:p_integral}, which only requires solving 
a set of forward and adjoint problems in the same medium, the cardinality of which equals the number of sources and the number of observations, respectively, i.e., $u^t_{q_0}, \lambda^t_{q_0}$. Also, in the discrete formulation, the $L^2$ norm corresponds to the weighted square sum in \eqref{eq:weighted discrete}.

Since $\delta\hat{u}^\infty \in L^2(B,\mathbb{C})$, $\hat{\Jacobian}_{q_0}: L^2(\Omega;\mathbb{R}) \to L^2(B,\mathbb{C})$, and $\delta q \in L^2(\Omega;\mathbb{R})$, we split the real and imaginary part of $\hat{\Jacobian}_{q_0}$ in the implementation to obtain a viable basis for $\delta q$, i.e.,
\begin{equation}\label{eq:real-aug-ls_continuous}
\min_{\delta q}
\left\|
\begin{bmatrix}
\Re(\delta\hat{u}^\infty) \\
\Im(\delta\hat{u}^\infty)
\end{bmatrix}
-
\begin{bmatrix}
\Re(\hat{\Jacobian}_{q_0})\\
\Im(\hat{\Jacobian}_{q_0})
\end{bmatrix}
\,\delta q
\right\|_{L^2(B,\mathbb{R}^2)}=: \|\delta\hat{u}^\infty_R - \hat{\Jacobian}_{q_0,R}\,\delta q \|_{L^2(B,\mathbb{R}^2)}.
\end{equation}
Equivalence between the norms on $L^2(B,\mathbb{C})$ and $L^2(B,\mathbb{R}^2)$ is asserted in \cite{BLECHER2026111226} so that we can carry all the estimates by considering complex data and $\hat{\Jacobian}_{q_0}$.

\begin{algorithm}[H]
\caption{Iterative Truncated SVD}
\label{alg:iterative-low-rank}
\begin{algorithmic}[1]

\Require Far-field data
         $\left\{
         u^{\infty}(\hat{x}_i;\hat{\theta}_j;k)
         \right\}_{\substack{i=1,\ldots,M_2\\ j=1,\ldots,M_1}}$,
         initial guess $q^{(0)}$,
         spectral tolerance $\epsilon>0$,
         stopping criterion $\eta>0$,
         maximum iterations $T\geq 1$

\Ensure Reconstructed contrast $\hat{q}$

% \State Build spatial grid $\{y_\ell\}$ on $B(0,1)\supset\Omega$

\For{$t=1,2,\ldots,T$}

  \State Set the current background $q_0=q^{(t-1)}$

  \State \textit{Forward solve:}
         Solve the Lippmann--Schwinger equation
         \eqref{eq:LS-plane} for
         $\{u_{q_0}^{t}(y;\hat{\theta}_j)\}_{j=1}^{M_1}$
         via block GMRES

  \State \textit{Adjoint solve:}
         For each $\hat{x}_i$, $i=1,\ldots,M_2$,
         solve \eqref{eq:adj-PDE} for
         $\lambda_{q_0}^{s}(y;\hat{x}_i)$ and set
         $
           \lambda_{q_0}^{t}(y;\hat{x}_i)
           =
           e^{-ik\hat{x}_i\cdot y}
           +
           \lambda_{q_0}^{s}(y;\hat{x}_i).
         $

  \State \textit{Assemble Jacobian:}
         Form the discrete operator via
         \eqref{eq:operator-adj}, with columns indexed by
         $y$ and rows by
         $(\hat{x}_i,\hat{\theta}_j)$, and apply the $p$-transform

  \State \textit{Data mismatch:}
         Compute
         $
           \delta\tilde{u}^{\infty}
           =
           u^{\infty}_{\mathrm{meas}}
           -
           \mathcal{F}(q_0),
         $
         apply the $p$-transform and form the real-augmented system
         \eqref{eq:real-aug-ls_continuous}.
         In the following steps, denote its Jacobian and
         residual by $\hat{J}_R$ and $\delta\hat{u}_R^{\infty}$,
         respectively

  \State \textit{Randomized SVD:}
         Apply Algorithm~4.2 in \cite{Random_SVD}
         to $\hat{J}_R$ with tolerance $\epsilon$ to obtain
         $\hat{J}_R\approx U\Sigma V^*$.
         Compute $\hat{S}$ in
         Section ~\ref{sec: implementation}
         and retain $\hat{S}$ leading modes

  \State \textit{Update:}
         Set
    $a_i=\langle\delta\hat{u}_R^{\infty},U_i\rangle$,
         $i=1,\ldots,\hat{S}$, and
         \[
           q^{(t)}
           =
           q^{(t-1)}
           +
           \sum_{i=1}^{\hat{S}}
           \frac{a_i}{\sigma_i}\,V_i.
         \]
  \If{$\left\|q^{(t)}-q^{(t-1)}\right\|_{L^2(B(0,1))}<\eta$}
    \State \Return $\hat{q}=q^{(t)}$
    % \Comment{Small-update stopping criterion}
  \EndIf

\EndFor

\State \Return $\hat{q}=q^{(T)}$

\end{algorithmic}
\end{algorithm}

\section{Error Analysis}
\label{sec: analysis}
In this section, we derive the error reduction at each Newton step for the iTSVD algorithm for a fixed frequency, then analyze the sufficient condition under which incorporating multi-frequency data reduces the reconstruction error.
Decompose the measurement mismatch as $\hat{\mathcal{F}}(q)-\hat{\mathcal{F}}(q_0)= \hat{\Jacobian}_{q_0}(q-q_0)+\hat{R}(q-q_0)$ after $p$-transform, the linearization part plus a nonlinear remainder term $\hat{R}(q-q_0)$. For each \(i\ge 1\), we have  the coefficient $a_i$ into
\begin{equation}\label{a_i}
    a_i:= \left\langle \mathcal{F}(q)-\mathcal{F}(q_0), \hat{u}_i \right\rangle
=
\left\langle \hat{\Jacobian}_{q_0}(q-q_0),\hat{u}_i\right\rangle_{L^2(B)}
+
\left\langle \hat{R}(q-q_0),\hat{u}_i\right\rangle_{L^2(B)}.
\end{equation}

For fixed $q_0$, define the forward operator and the resolvent
\[
 L_{q_0}=\Delta+k^2(1+q_0),\qquad
 \mathcal R_{q_0}(k)=L_{q_0,\mathrm{out}}^{-1}.
\]
Choose $\chi\in C_c^\infty(\mathbb R^2)$ such that
$0\leq\chi\leq1$ and $\chi\equiv 1$ on $\Omega$. 
\\
% \textcolor{cyan}{Label the assumption and cite it for all lemmas or theorems requiring it}
{\bf Assumption}: We consider the non-trapping and non-resonance case for the medium $q_0(x)$ and define the following bounds (see \cite{Resolvent,field_O1}), with emphasis on the dependence on $k$:
\begin{align}
 \beta_k
 &=k^2\|\chi\mathcal R_{q_0}(k)\chi\|_{L^2(\mathbb R^2)\to L^2(\mathbb R^2)},
 \label{eq:rem-beta}\\
 U_k
 &=\sup_{\hat\theta\in \mathbb{S}}
   \|u_{q_0}^t(\cdot;\hat\theta)\|_{2}.
 \label{eq:rem-U}
\end{align}

\begin{lemma}[Remainder estimate]\label{lem:remainder}
Let $q = q_0 + \delta q$, where $q_0$ satisfies the above assumption, and $\delta q \in H_0^2(\Omega)$ satisfies
\begin{equation}
 \beta_k\|\delta q\|_{\infty}\le \tau <1.
 \label{eq:rem-smallness}
\end{equation}
There exists a positive constant \(C(k) = O(k)\), where the constant in $O$ depends on $q_0$, $\Omega$, and $\tau$ such that
\begin{equation}\label{eq:remainder}
\|\hat{R}(\delta q)\|_{L^2(B)}
\leq
C(k)\|\delta q\|_\infty^2.
\end{equation}
% For fixed $q_0$, $\chi$, and $\Omega$, one may choose $C(k)=O(k)$ provided $U_k=O(1)$ and $1-\beta_k\|\delta q\|_\infty$ is bounded below by a positive constant independent of $k$.
\end{lemma}

\begin{proof}
Write $\delta u=u^t_q-u^t_{q_0}$, then $\delta u$ is outgoing and satisfies
\begin{equation}
 L_{q_0}\delta u=-k^2\delta q(u_{q_0}^t+\delta u).
 \label{eq:rem-difference-pde}
\end{equation}
Since the source (in particular $\delta q$) in \eqref{eq:rem-difference-pde} is supported in
$\Omega$,
\[
 \chi \delta u=-k^2\chi\mathcal R_{q_0}(k)\chi\,[\delta q(u^t_{q_0}+\delta u)].
\]
$\chi=1$ on $\Omega$ gives
\[
 \|\delta u(\cdot;\hat\theta)\|_{2}
 \leq\beta_k\|\delta q\|_{\infty}\Bigl(
   \|u^t_{q_0}(\cdot;\hat\theta)\|_{2}
   +\|\delta u(\cdot;\hat\theta)\|_{2}\Bigr).
\]
Taking the supremum over $\hat{\theta}$ gives
\begin{equation}
 \sup_{\hat\theta\in \mathbb{S}}
 \|\delta u(\cdot;\hat\theta)\|_{2}
 \leq\frac{\beta_k U_k}{1-\beta_k\|\delta q\|_{\infty}}\,\|\delta q\|_{\infty}.
 \label{eq:rem-field-bound}
\end{equation}
The remainder can be written using \eqref{eq:adj-equivalence} and \eqref{eq:p_integral} as
\begin{align}
 \hat R_{q_0}(\delta q)
 & \quad =\hat{\mathcal F}(q_0+\delta q)
  -\hat{\mathcal F}(q_0)-\hat J_{q_0}\delta q
 \notag\\
 &\quad=\int_\Omega e^{-ik\hat x(p)\cdot y}
              \bigl(\delta q(y)u^t_q(y;\hat\theta(p))
                    +q_0(y)\delta u(y;\hat\theta(p))\bigr)\,dy -\hat J_{q_0}\delta q
 \notag\\
 &\quad=\int_\Omega\lambda_{q_0}^t(y;\hat x(p))
                u^t_q(y;\hat\theta(p))\delta q(y)\,dy - \hat J_{q_0}\delta q
 \notag\\
 &\quad=\int_\Omega\lambda_{q_0}^t(y;\hat x(p))
                  \delta u(y;\hat\theta(p))\delta q(y)\,dy.               
 \label{eq:rem-mixed-identity}
\end{align}
Since 
$\|\lambda_{q_0}^t(\cdot;\hat x)\|_{L^2(\Omega)}\leq U_k$,
Cauchy--Schwarz and \eqref{eq:rem-field-bound} imply
\[
 |\hat R_{q_0}(\delta q)|
 \leq U_k\|\delta q\|_{\infty}
       \|\delta u(\cdot;\hat\theta(p))\|_{2}
 \leq\frac{U_k^2\beta_k}{1-\beta_k\|\delta q\|_{\infty}}\,\|\delta q\|_{\infty}^2.
\]
Finally, taking the $L^2(B)$ norm gives
\begin{equation}
 \|\hat R_{q_0}(\delta q)\|_{L^2(B)}
 \leq
 \frac{\sqrt{\pi}\,U_k^2\beta_k}
      {1-\beta_k\|\delta q\|_{\infty}}
 \|\delta q\|_{\infty}^2.
 \label{eq:rem-full-bound}
\end{equation}

In particular for the non-trapping case, by the weighted semiclassical resolvent estimate stated in Theorem 2
% \((0.2)\)
in
Robert--Tamura \cite{Resolvent}
% (quoted from Yafaev ?)
,
there exists a constant \(C_{q_0,\chi}>0\), 
\begin{equation}\label{eq:resolvent estimate}
    \|\chi \mathcal{R}_{q_0}(k+i0)\chi\|_{L^2(\mathbb R^2)\to L^2(\mathbb R^2)}
\le C_{q_0,\chi}\,k^{-1},
\end{equation}
% The coefficient in \eqref{eq:rem-full-bound} gives the stated remainder estimate.
which yields $\beta_k\le C_{q_0,\chi}k$. With ~\eqref{eq:rem-smallness}, we take
\[
C(k) = \frac{\sqrt{\pi}\,U_k^2C_{q_0,\chi}k}
      {1-\tau}.
\]
Combining the argument in
\cite[Lemma~9.4]{field_O1} where the non-trapping resolvent estimate implies that
$U_k=O(1)$ uniformly in $k$ in ~\eqref{eq:rem-U}, we get $C(k) = O(k)$.
\end{proof}

The following result provides a sufficient condition for an error reduction in one step of iTSVD.
\begin{theorem}\label{thm:contraction}
Let $q, q_0 \in H_0^2(\Omega)$, and assume
\begin{align}&\label{eq:regularity}\sum_{i=1}^N |\left\langle q-q_0,\hat{v}_i\right\rangle|^2 \ge a \|q-q_0 \|_2^2,
\end{align} 
for some $a>0$, then there exists $\epsilon > 0$ depending on $a,\,k,\,q_0,\,N$ such that if
$\|q-q_0\|_{H^2(\Omega)}<\epsilon$, the update
\eqref{eq:update} satisfies
\begin{equation}\label{eq:contraction}
\|q-\hat{q}\|_{2}
< \|q-q_0\|_{2} .
\end{equation}
\end{theorem}

\begin{proof}
% First for $d\leq 3$, $H^2(\Omega) \hookrightarrow L^\infty(\Omega) \hookrightarrow L^2(\Omega)$
We compute 
\begin{align*}
\|q-\hat{q}\|_{2}^2 - \|q-q_0\|_{2}^2
% & =  \int_{\Omega} |q(y) - q_0(y) - \sum_{i=1}^N \frac{a_i}{\sigma_i} \hat{v}_i(y)|^2 dy - \int_{\Omega} |q(y) - q_0(y)|^2 dy\\
& = \int_{\Omega} |\sum_{i=1}^N \frac{a_i}{\sigma_i} \hat{v}_i(y)|^2 - 2(q-q_0)\sum_{i=1}^N \frac{a_i}{\sigma_i} \hat{v}_i(y) dy \\
& = \sum_{i=1}^N |\frac{a_i}{\sigma_i}|^2 - \sum_{i=1}^N 2 \frac{a_i}{\sigma_i} \int_{\Omega} (q-q_0)(y)\, \hat{v}_i(y) dy\\
& \leq \sum_{i=1}^N |\frac{a_i}{\sigma_i} - \int_{\Omega} (q-q_0)(y)\, \hat{v}_i(y) dy |^2 -  \sum_{i=1}^N |\int_{\Omega} (q-q_0)(y)\, \hat{v}_i(y) dy |^2\\
& \leq \sum_{i=1}^N \frac{1}{\sigma_i^2} | \int_{B} \hat{R}(q-q_0)(p)\, \hat{u}_i(p) dp|^2 - \sum_{i=1}^N |\int_{\Omega} (q-q_0)(y)\, \hat{v}_i(y) dy |^2\\
& \leq \frac{1}{\sigma_N^2} \|\hat{R}(q-q_0) \|^2_2 - a \|q-q_0\|^2_{2} \\
& \leq \frac{C^2(k)}{\sigma^2_N} \|q-q_0 \|^4_\infty - a \|q-q_0\|^2_{2},
\end{align*}
where we used the definition of $a_i$ in \eqref{a_i} and Lemma~\ref{lem:remainder}. Since $\Omega \subset \mathbb{R}^2$ is a bounded Lipchitz domain, by the Gagliardo-Nirenberg interpolation\cite{evans}, there exists a constant $C_1$ depending on $\Omega$ such that
\begin{align}
    \frac{C(k)^2}{\sigma^2_N} \|q-q_0 \|^4_\infty
    - a \|q-q_0\|^2_{2}
    &\leq
    C_1 \frac{C^2(k)}{\sigma^2_N}
    \|q-q_0 \|^2_2
    \|q-q_0 \|^2_{H^2(\Omega)}
    - a \|q-q_0\|^2_{2}
    \notag \\
    & =
    \left(
    C_1 \frac{C^2(k)}{\sigma^2_N}
    \|q-q_0 \|^2_{H^2(\Omega)}
    - a
    \right)
    \|q-q_0 \|^2_2 .
    \label{eq:contraction-condition}
\end{align}
Hence, $\|q - q_0\|_{H^2(\Omega)}$ sufficiently small yields \eqref{eq:contraction}.
\end{proof}
Condition \eqref{eq:regularity} states that the error has a nonzero component in the subspace spanned by the leading $N$ singular modes. This excludes the degenerate case in which the error lies entirely outside the recoverable subspace and guarantees that the truncated update can produce a nontrivial decrease in the \(L^2\) norm of the error. In particular, an explicit sufficient condition for \eqref{eq:contraction} is obtained from \eqref{eq:contraction-condition}
\begin{align}
 &\label{eq:sufficient}C_1 \frac{C^2(k)}{\sigma^2_N} \|q-q_0 \|^2_{H^2(\Omega)}  < a,
\end{align}
which shows how the regularity
$\|q - q_0\|_{H^2(\Omega)}$, the threshold of singular values, and quadratic remainder constant $C(k)$ jointly determine the improvement condition of the update. In particular, $\sigma_N^{-2}$ reflects the noise amplification factor in the retained modes.
%In particular, $\sum\limits_{i=1}^N \sigma_i^{-2}$ reflects noise amplification in the retained modes.

\paragraph{Frequency Continuation}

To extend Algorithm \ref{alg:iterative-low-rank} to the multi-frequency case, we adopt the recursive linearization approach; that is, we use the reconstruction at each frequency as the initial guess at the next, higher frequency. Choosing the frequency increment requires balancing efficiency and convergence: larger increments reduce the number of frequency steps, but may place the current reconstruction outside the convergence region of Newton’s method at the next frequency. To guide the choice of the next frequency, $k:= k_2$, we derive a sufficient condition for an error reduction. The estimate of $k_2$ can be conversely derived from the convergence analysis, i.e., given an error estimate between $q$ and $q_0$, we find the largest $k_2$ that guarantees a better approximation. 
%Although the theoretical estimate from Newton's method is informative for choosing the frequency increment, it consists of $k$ to a high power which renders the estimate impractical; see \cite{Bao_Faouzi} for detailed analysis.

Following the assumptions in Theorem ~\ref{thm:contraction} and $C(k)$ defined in \eqref{eq:remainder} of Lemma~\ref{lem:remainder}, assume the current estimate is $q_0$, 
an admissible $k_2$ needs to satisfy\eqref{eq:sufficient}. By taking $N = S$ and considering the spectral behavior, $\sigma_N = O(\frac{1}{k_2})$, 
we get
    \[
    k_2^2 C^2(k_2) \|q-q_0 \|^2_{H^2(\Omega)} \lesssim 1 \Rightarrow k_2 \, C(k_2)\,\|q-q_0 \|_{H^2(\Omega)} \lesssim 1
    \]
Since $C(k_2) = O(k_2)$, we have
    \[
k_2 \lesssim \|q-q_0 \|^{-\frac{1}{2}}_{H^2(\Omega)}.
    \]

\paragraph{One Step Error Estimate}
To establish a recursive connection between the error tolerance and the wavenumber, it remains to estimate $\epsilon$ for a fixed $k$. Because our algorithm does not rely on the Born approximation, the relevant quantity is not a Born-linearization error, but rather the one-step update error associated with linearization around the current background $q_0$ at the wavenumber $k$.
% Together with the estimate derived above, which gives an admissible $k$ for a given $\epsilon$, this yields a closed loop between $\epsilon$ and $k$.?
% To have recursive connection between the, we also need to evaluate $\epsilon$ at given $k$. Since our algorithm does not rely on the Born approximation, an error estimate for a one-step update given $q_0$ and $k$ is due.
Define the one step update $\hat{q}$ as in \eqref{eq:update} and assume $q_0,q \in H^2(\Omega)$, then
\[
\|\hat{q}-q\|_{2} 
\le \frac{1}{\sigma_N}\,\|\hat{R}(q-q_0)\|_{L^2(B)} + \sqrt{1-a}\|q-q_0 \|_{2}.
\]
Such an estimate comes from Lemma 4.1 in \cite{zhou2024exploringlowrankstructureinverse}, where we replace the first term by $\hat{R}(q-q_0)$, and $\beta(\alpha) = \sigma_N$; the second term in Lemma 4.1 comes from the truncation error of the finite expansion of $\delta q$ by $\{v_i \}$, which is replaced by the assumption \eqref{eq:regularity}.

\section{Numerical Experiments}
\label{sec: numerical}
 In this section, we perform extensive numerical experiments to demonstrate the effectiveness of the proposed iTSVD algorithm in terms of efficiency and robustness on various setups for inverse scattering problems. In particular, our tests include far-field data, limited aperture data, random measurement data, near-field data, and frequency continuation. Specifically, we illustrate stability by adding noise to the data, and the efficiency by counting the number of Newton steps.  

Throughout this section, we set $\Omega = B$, and the total field inside the domain of interest $\Omega$, in which the unknown contrast $q(x)$ is supported, is computed by solving the Lippmann-Schwinger integral equation. We adopt \cite{ANDERSON2024113091} for fast and accurate evaluation of the volume potential $\mathcal{G}$.
Far-field data are then obtained by the integral  \eqref{eq:farfield-compact}. In all our tests, we assume no prior knowledge of $q$ by setting the initial guess $q_0 = 0$. To quantify the accuracy of the reconstruction $\tilde{q}$, we use the relative $L^2$ error with respect to the ground truth $q$
\begin{equation}\label{eq:e2}
e_2 = \frac{\|\tilde{q} - q\|_{L^2(\Omega)}}{\|q\|_{L^2(\Omega)}}.
\end{equation}

\paragraph{Effective rank estimate}
\label{para: rank estimate} If a discrete setting can resolve the wavelength, one can apply the continuous setup estimate, e.g., using the Shannon number~\eqref{eq:Shannon number}, as for the full aperture cases and limited aperture cases (see Section ~\ref{sec: limited aperture analysis}) in our test examples. 

For general discrete cases, we develop numerical approximations for the effective rank, e.g., for random sparse far-field data in Section ~\ref{sec: random data} with a quantitative analysis connecting the discrete case to the continuous case. For general random sparse near-field data, see Section ~\ref{sec: near field data}, we estimate the effective rank using RSVD for the linearized operator in homogeneous medium, $\hat{\mathcal{J}}_0$,  where no PDE needs to be solved since the Green's function is available in free space. Then the effective rank is used for subsequent Newton steps. This strategy can be used in an arbitrary setup. If needed, the effective rank estimate can be updated after every few Newton steps.

As a remark, iTSVD is quite tolerant with respect to the effective rank estimate due to its iterative nature and the threshold of small singular values when computing the TSVD solution to the linearized problem. If the estimate is conservative, i.e., larger than the optimal effective rank, more computational cost may be incurred in the Newton step. However, stability is controlled by thresholding. In contrast, if the estimate is aggressive, i.e., smaller than the optimal effective rank, more Newton steps may be needed for convergence. Overall, iTSVD is both efficient and robust. See more implementation details in Appendix~\ref{sec: implementation}.

\paragraph{Computation cost}
\label{para: compute cost}

For each iteration in Algorithm ~\ref{alg:iterative-low-rank}, the main computational cost consists of two parts. The first part is solving the forward and adjoint problems using the Lippmann-Schwinger integral equation. The volume integral operator is evaluated via polynomial density interpolation introduced in \cite{ANDERSON2024113091} that provides high-order numerical accuracy by resolving the singularity of the Helmholtz kernel $G_k$, which is then compressed by the Fast Multipole Method to accelerate the computation. Since these forward and adjoint scattering problems share the same PDE with only different forcing, i.e., sharing the same matrix numerically (see Section ~\ref{sec: implementation}), we use block GMRES to accelerate the linear solve. In particular, a preconditioner is available in \cite{anderson2026fasthighordersolverslippmannschwinger} to reduce the number of iterations for such an iterative method. Here, we count the total number of PDEs solved for the first part. For iTSVD, $(M_1+M_2)$ PDEs are solved, where \(M_1, M_2\) are the number of
incident directions/sources and observation directions/measurements, respectively. 

The second part comes from RSVD, which depends on the number of matrix-vector multiplications needed (see \cite{Random_SVD} for a detailed cost analysis). In our setting, the number is related to our effective rank estimate. For example, for full aperture far-field data, we use Shannon number $S=k^2$ in ~\eqref{eq:Shannon number}. To include a buffer zone for the spectral transition window of size $o(k^2)$, we sample $\floor{2S + 10}$ random vectors as discussed in Section~\ref{sec:spectral}. However, the matrix-free matrix-vector multiplications based on~\eqref{eq:adj-equivalence} only involve direct integral evaluations, which are much cheaper than solving PDEs in the first part. It is worth noting that the gradient-based method, such as the adjoint method, requires to solve $2M_1$ forward/adjoint problems in each iteration. Hence, the major factor for the overall computational cost is the number of iterations needed for convergence to a satisfactory result.
In our tests, we count the number of iterations for convergence with a stopping criterion $\eta = 10^{-2}$ on the $L^2$ norm of the update to demonstrate the efficiency of iTSVD.

\paragraph{Data Noise} \label{sec: noisy data}
Let $\mathbf{D} \in \mathbb{C}^{M_2 \times M_1}$ denote the noiseless measurement data.
% where \(M\) is the number of
% observation directions and incidence directions. 
Let $\iota $ denote the noise level and set
\[
\sigma_\iota
=
\iota \frac{\|\mathbf{D}\|_F}{\sqrt{M_1M_2}},
\]
where \(M_1, M_2\) are the number of
sources and observations, respectively. The noisy measurement data is generated as
\[
\mathbf{D}^\iota
=
\mathbf{D}+\mathbf{E},
\]
where the noise matrix \(\mathbf{E} \in \mathbb{C}^{M_2 \times M_1}\) is defined entry-wise by
\[
\mathbf{E}_{mn}
=
\frac{\sigma_\iota}{\sqrt{2}}
\left(
\xi_{mn}
+i\zeta_{mn}
\right),
\qquad
1\le m\le M_2,\quad 1\le n\le M_1,
\]
with
\[
\xi_{mn},\zeta_{mn}\overset{\mathrm{i.i.d.}}{\sim}\mathcal{N}(0,1).
\]
We compare the reconstruction between noiseless data and noisy data with noise level $\iota  = 0.2$, and take $M_1 = M_2 =: M$ with $M = 50$ in our tests.

\subsection{Full Aperture} \label{sec: numerical on full aperture}
\begin{example}\label{ex:multi_gaussian}
% Far-field data with 5 Gaussian bumps.
We consider a contrast consisting of five truncated Gaussian bumps with different amplitudes, centers, and widths:
    \[
q(x_1,x_2) = \sum_{i=1}^5 A_i \exp\big[-\big(\frac{(x_1 - c_{i,x})^2}{\sigma_{i,x}^2} + \frac{(x_2 - c_{i,y})^2}{\sigma_{i,y}^2}\big)      \big] \mathbf{1}_{\{|x_1 - c_{i,x}| \le 5\sigma_{i,x},|x_2 - c_{i,y}| \le 5\sigma_{i,y}\}} 
\]
with
\[
\begin{array}{c|c|c|c}
i & A_i & (c_{i,x},c_{i,y}) & (\sigma_{i,x},\sigma_{i,y}) \\
\hline
1 & 0.6 & (0,0) & (0.15,0.10) \\
2 & 1.0 & (-0.3,0.25) & (0.18,0.18) \\
3 & 0.6 & (0.3,0.25) & (0.10,0.15) \\
4 & 0.8 & (-0.4,-0.4) & (0.12,0.12) \\
5 & 1.0 & (0.5,-0.5) & (0.10,0.10) . 
\end{array}
\]
\end{example}
% \paragraph{Example 1} Let
% \[
% q(x_1,x_2) = \sum_{i=1}^5 A_i \exp\big[-\big(\frac{(x_1 - c_{i,x})^2}{\sigma_{i,x}^2} + \frac{(x_2 - c_{i,y})^2}{\sigma_{i,y}^2}\big)      \big] \mathbf{1}_{\{|x_1 - c_{i,x}| \le 5\sigma_{i,x},|x_2 - c_{i,y}| \le 5\sigma_{i,y}\}} 
% \]
% with
% \[
% \begin{array}{c|c|c|c}
% i & A_i & (c_{i,x},c_{i,y}) & (\sigma_{i,x},\sigma_{i,y}) \\
% \hline
% 1 & 0.6 & (0,0) & (0.15,0.10) \\
% 2 & 1.0 & (-0.3,0.25) & (0.18,0.18) \\
% 3 & 0.6 & (0.3,0.25) & (0.10,0.15) \\
% 4 & 0.8 & (-0.4,-0.4) & (0.12,0.12) \\
% 5 & 1.0 & (0.5,-0.5) & (0.10,0.10) . 
% \end{array}
% \]
Figure ~\ref{fig:multi-Gaussian} shows the reconstruction for Example ~\ref{ex:multi_gaussian} using a single frequency $k = 4\pi$ measurement without noise and with $20\%$ noise. Figure ~\ref{fig:Convergence of Multi-Gaussian} presents the relative error $e_2$ ~\eqref{eq:e2} w.r.t. Newton steps. iTSVD converges in 5 iterations with 500 PDEs solved for both noiseless and noisy data, while the gradient descent method takes 300 iterations with 30000 PDEs solved to reach a comparable result. The reduction in error with iterations demonstrates the importance of consecutive updates rather than a single linearized update around $q=0$.

% \begin{figure}[htbp]
% \centering
%     % \begin{subfigure}{0.2\linewidth}
%     % \captionsetup{labelformat=empty}  % Remove the alphabetic label
%     %     \includegraphics[width=\linewidth]{paper experiments/far-field/noisy data/Gaussian/fig_q0.png}
%     %     % \caption{initial guess $q_0$}
%     % \end{subfigure}
%     \begin{subfigure}{0.3\linewidth}
%     \captionsetup{labelformat=empty}  % Remove the alphabetic label
%         \includegraphics[width=\linewidth]{noisy data/Gaussian/fig_q_exact.png}
%         % \caption{exact $q$}
%     \end{subfigure}
%     \begin{subfigure}{0.3\linewidth}
%     \captionsetup{labelformat=empty}  % Remove the alphabetic label
%         \includegraphics[width=\linewidth]{noisy data/Gaussian/fig_q_recon_noiseless.png}
%         % \caption{noiseless}
%     \end{subfigure}
%     \begin{subfigure}{0.3\linewidth}
%     \captionsetup{labelformat=empty}  % Remove the alphabetic label
%         \includegraphics[width=\linewidth]{noisy data/Gaussian/fig_q_recon_noisy.png}
%         % \caption{with 20\% noise}
%     \end{subfigure}
%     \caption{From left to right: true $q$; reconstruction $\tilde{q}$ from noiseless data; reconstruction $\tilde{q}$ from data with $\delta=20\%$.}
%     \label{fig:multi-Gaussian}
% \end{figure}

\begin{figure}[htbp]
\centering
\begin{tabular}{c@{\vspace{2pt}}c@{\vspace{2pt}}c}
\includegraphics[width=0.25\linewidth]{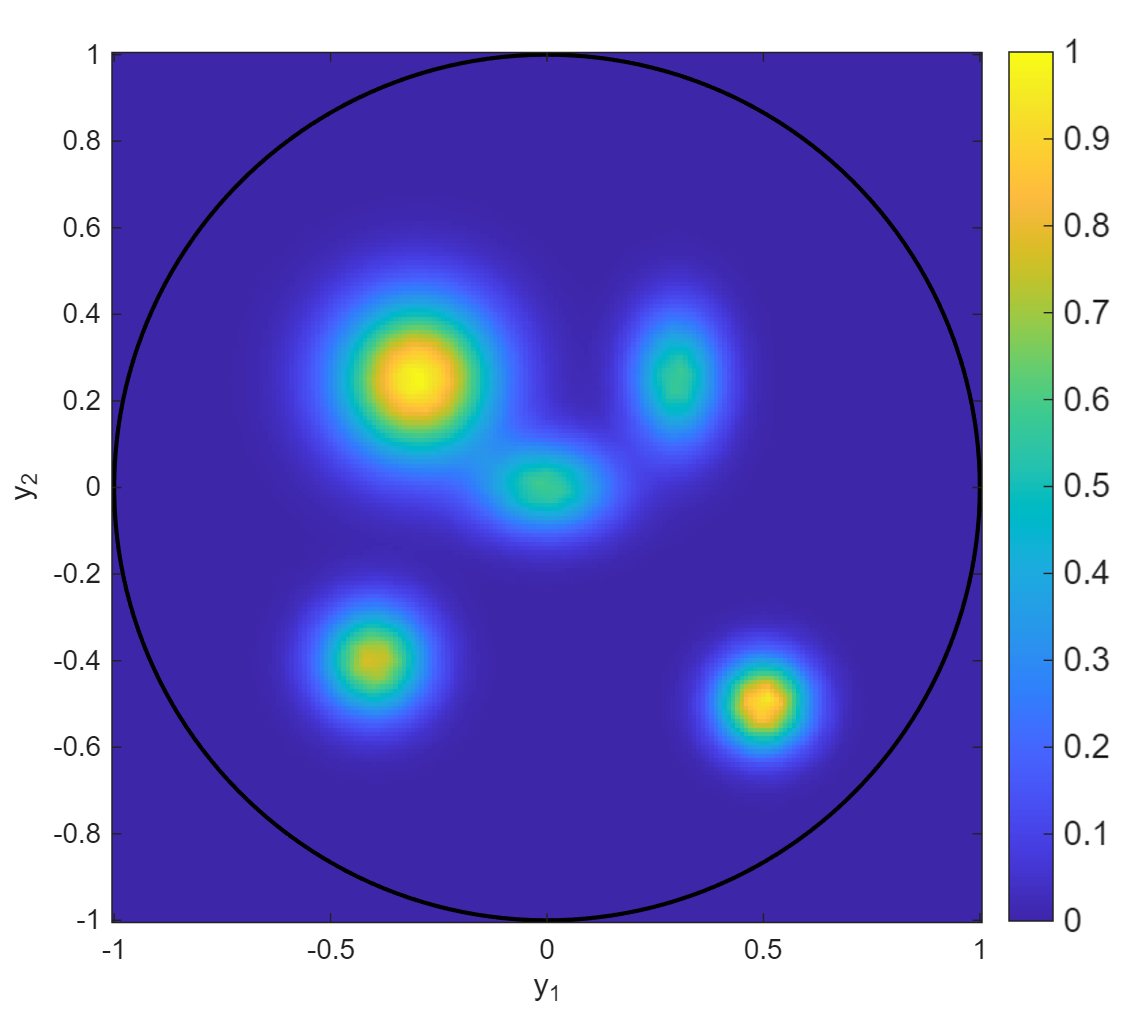}&
\includegraphics[width=0.25\linewidth]{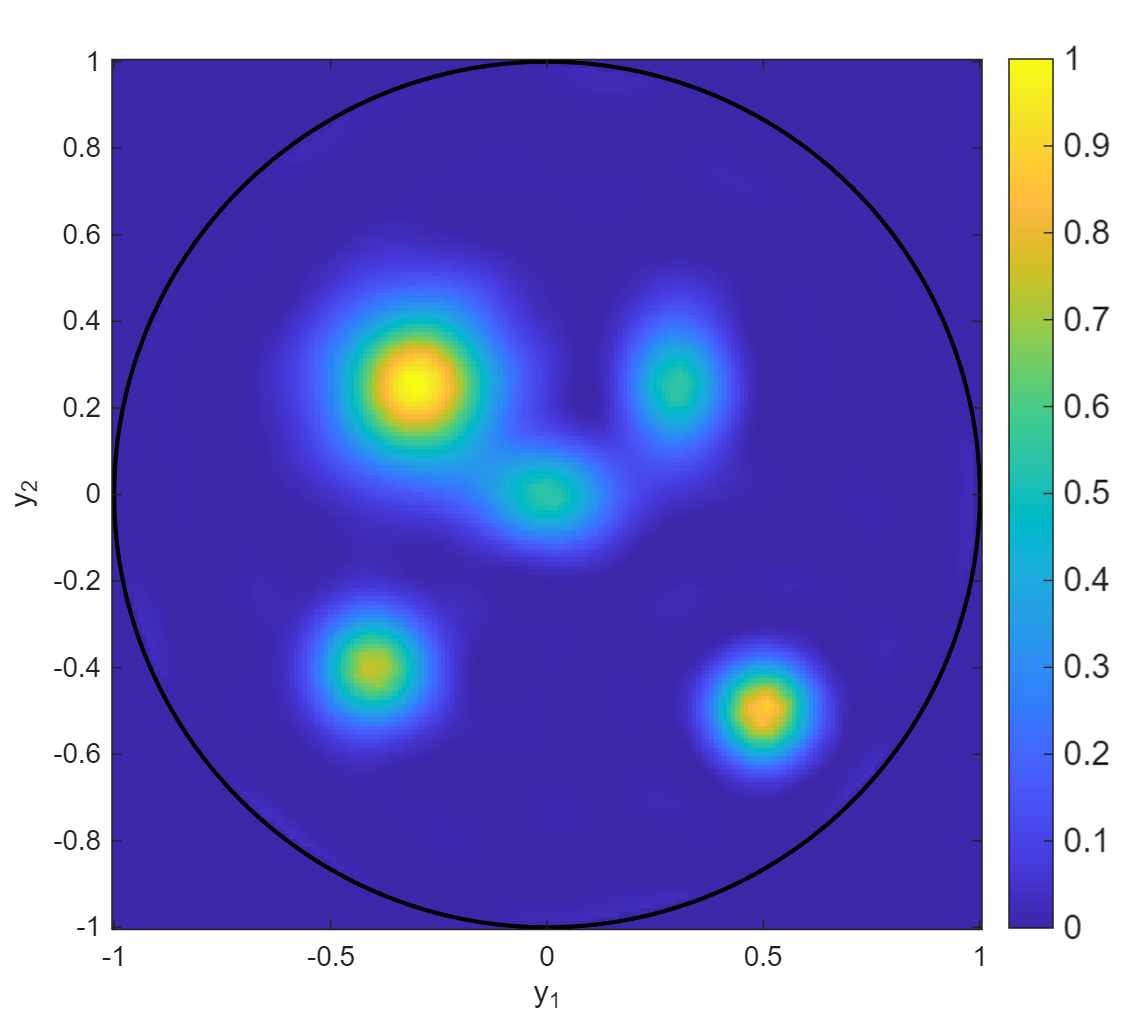}&
\includegraphics[width=0.25\linewidth]{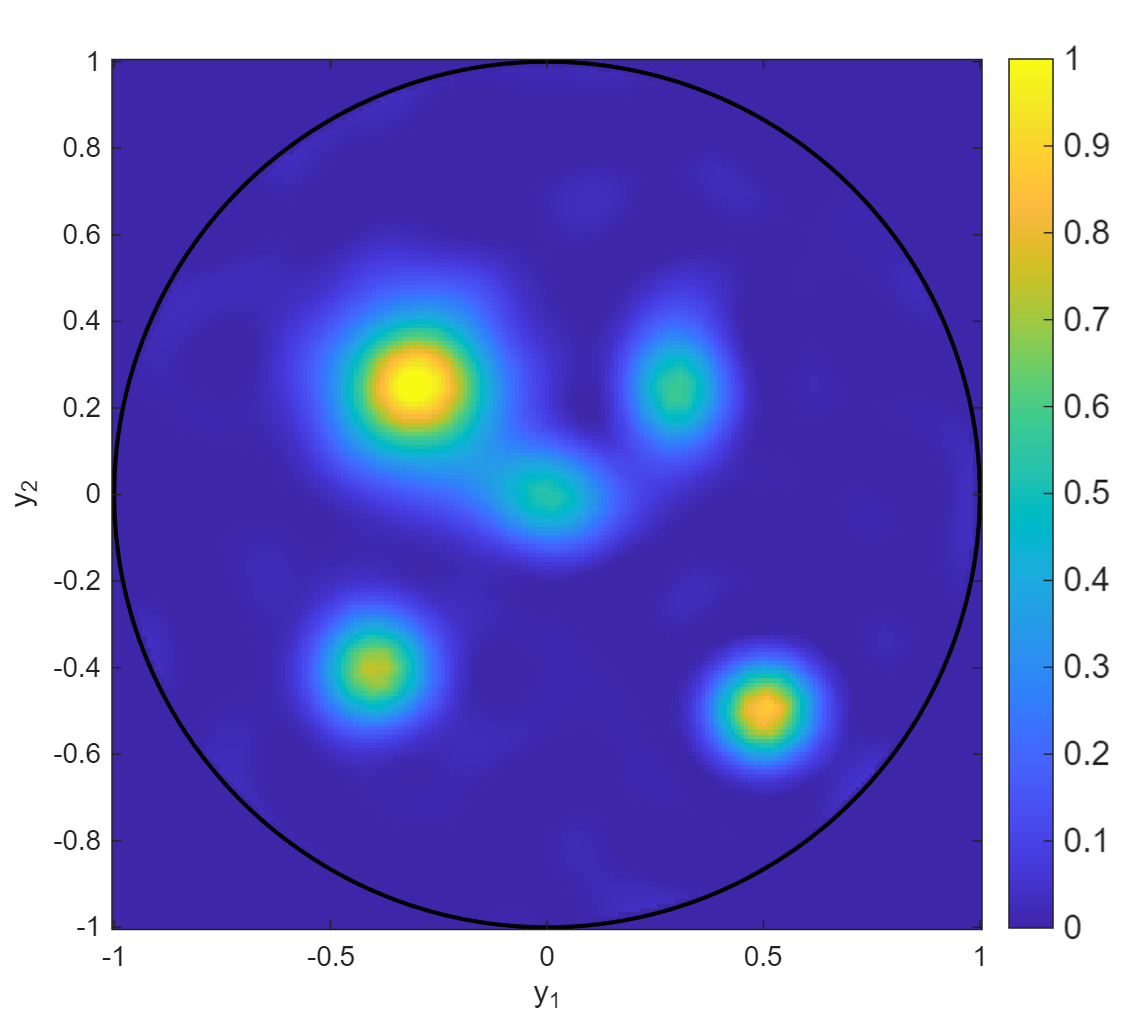}
\\
(a) true $q$ &(b) clean data  &(c) data with $20\%$ noise
\end{tabular}
\caption{Reconstruction results for Example ~\ref{ex:multi_gaussian}.}
\label{fig:multi-Gaussian}
\end{figure}

\begin{figure}[htbp]
\centering
\begin{tabular}{c@{\vspace{2pt}}c@{\vspace{2pt}}c}
\includegraphics[width=0.35\linewidth]{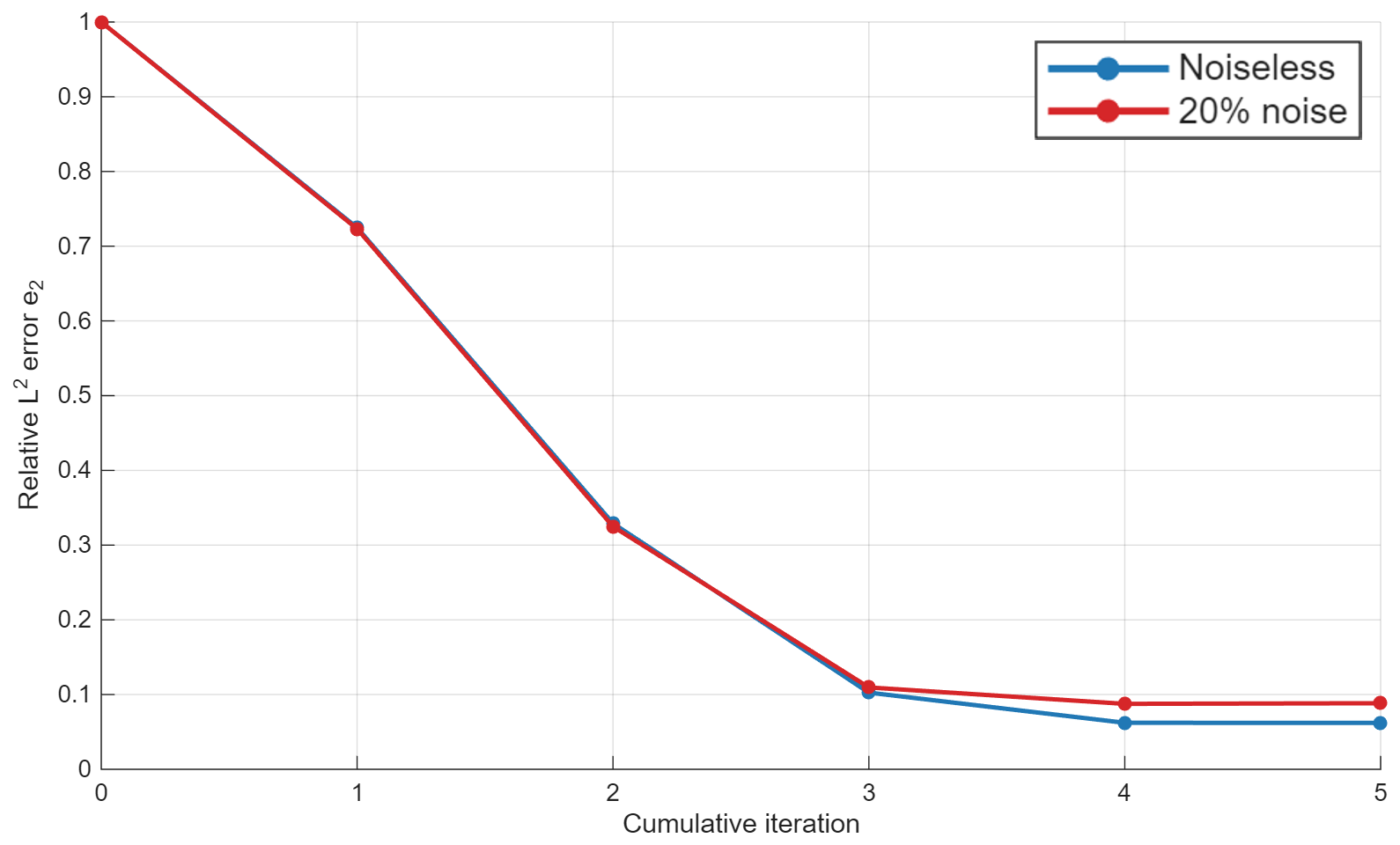}&
\includegraphics[width=0.35\linewidth]{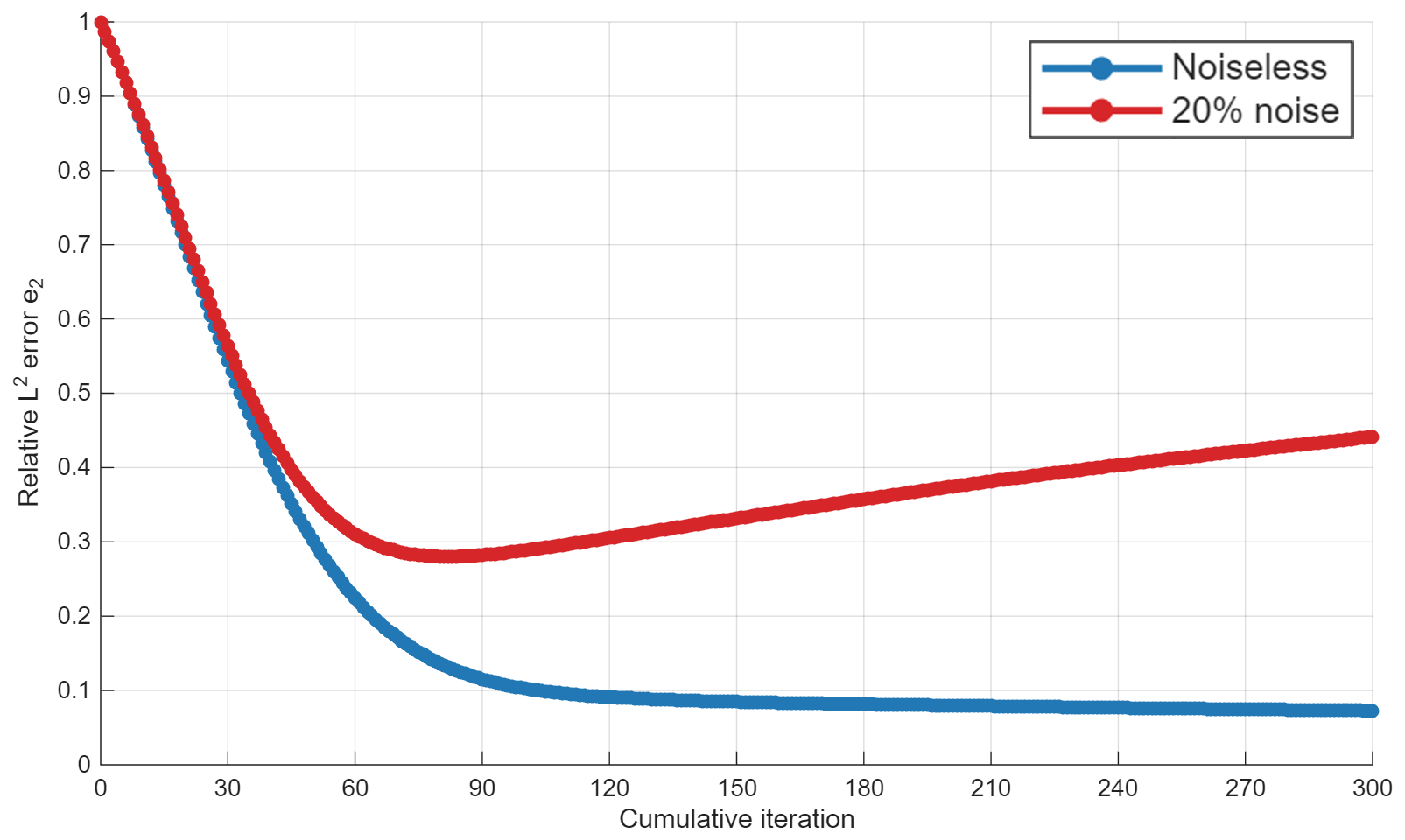}
\\
(a) iTSVD &(b) gradient descent
\end{tabular}
\caption{Comparison of convergence for (a) iTSVD and (b) gradient descent for Example ~\ref{ex:multi_gaussian}.}
\label{fig:Convergence of Multi-Gaussian}
\end{figure}

We also compare the reconstruction results between using the uniform-weighted least-squares formulation~\eqref{eq: discrete unweighted} and the frequency-balanced least-squares formulation~\eqref{eq:weighted discrete} based on the $p$-transform in Figure ~\ref{fig:numerical on p-transform}. The result using~\eqref{eq:weighted discrete} is significantly better, which demonstrates that the choice of data-misfit norm
% and its consistent discretization
can substantially affect the reconstruction accuracy.

\begin{figure}[htbp]
\centering
\begin{tabular}{c@{\vspace{2pt}}c@{\vspace{2pt}}c}
\includegraphics[width=0.2\linewidth]{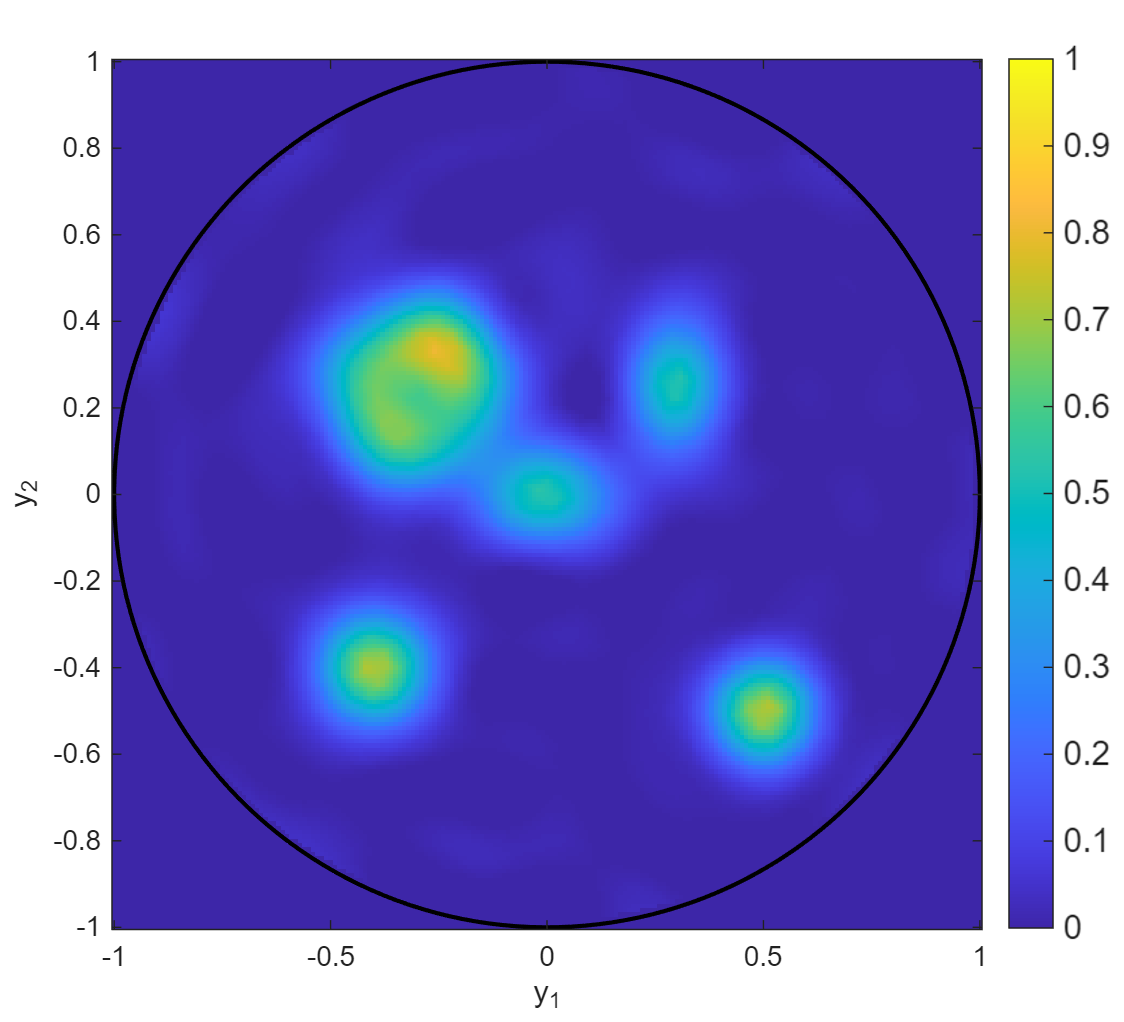}&
\includegraphics[width=0.2\linewidth]{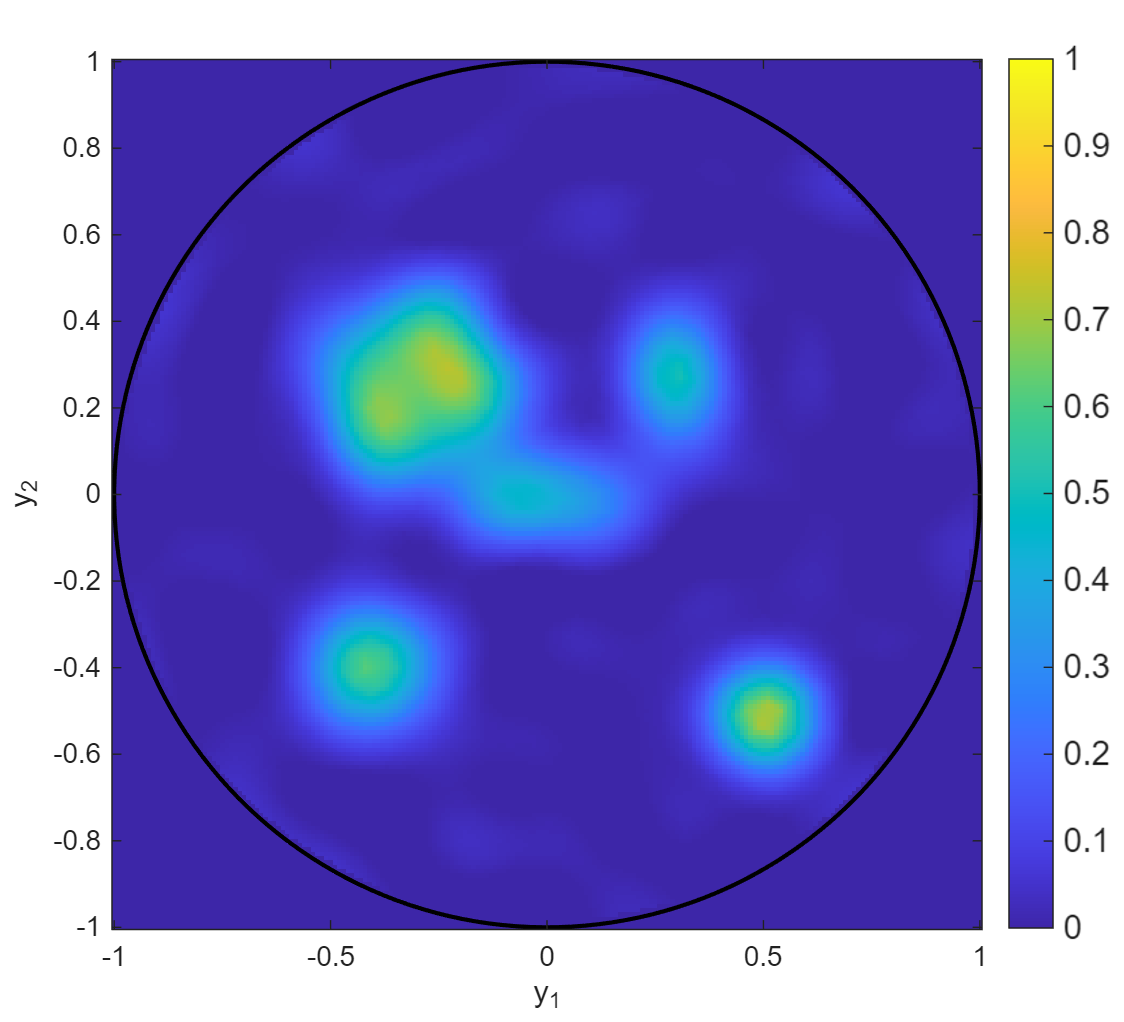}&
\includegraphics[width=0.3\linewidth]{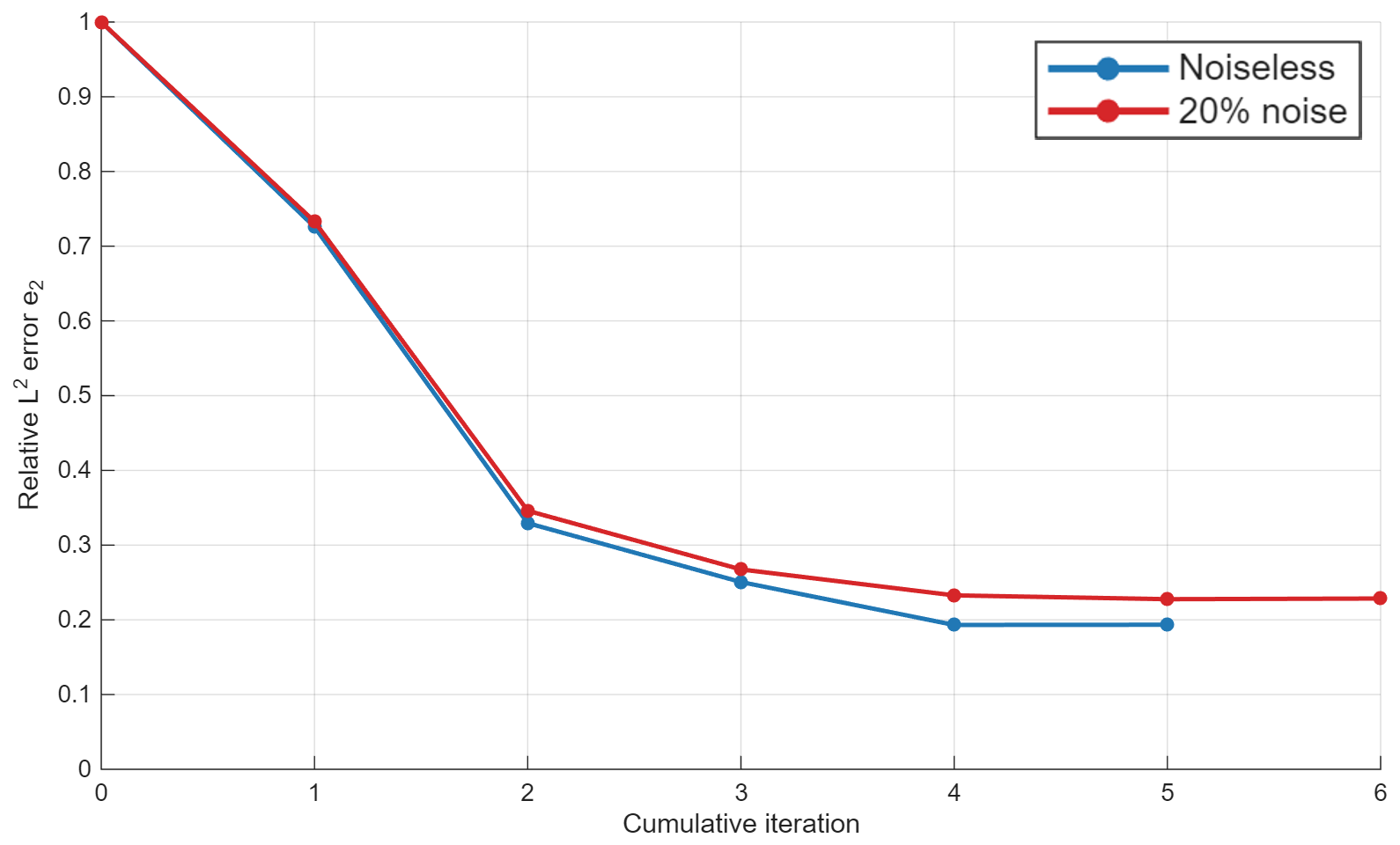}\\
\includegraphics[width=0.2\linewidth]{figures/noisy_data/Gaussian/fig_q_recon_noiseless.png}&
\includegraphics[width=0.2\linewidth]{figures/noisy_data/Gaussian/fig_q_recon_noisy.png}&
\includegraphics[width=0.3\linewidth]{figures/noisy_data/Gaussian/fig_convergence.png}
\\
(a) clean data &(b) data with $20\%$ noise &(c) convergence
\end{tabular}
\caption{Reconstruction results for Example~\ref{ex:multi_gaussian}. First row: results using uniformly weighted least squares~\eqref{eq: discrete unweighted}. Second row: results using frequency-balanced least squares~\eqref{eq:weighted discrete}.}
\label{fig:numerical on p-transform}
\end{figure}

% \begin{figure}[H]
%     \centering
%     \includegraphics[width=0.5\linewidth]{noisy data/Gaussian/fig_convergence.png}
%     \caption{Convergence of $e_2$.}
%     \label{fig:Convergence of Multi-Gaussian}
% \end{figure}

\begin{example}\label{ex: greengard}
We next consider a sign-changing contrast adapted from Example 2 in \cite{Greengard}, defined by
    \[
\begin{aligned}
q(x_1,x_2)
=
&
\frac{3}{20}
\left(1-\frac{x_1^2}{\sigma^2}\right)
\exp\left\{
-\left[
\frac{x_1^2}{\sigma^2}
+
\left(\frac{x_2}{\sigma}+1\right)^2
\right]
\right\}
-
\frac{1}{60}
\exp\left\{
-\left[
\frac{x_2^2}{\sigma^2}
+
\left(\frac{x_1}{\sigma}+1\right)^2
\right]
\right\}
\\
&
-
\sigma
\left[
\frac{2}{5}x_1
-
\left(\frac{x_1}{\sigma}\right)^3
-
\left(\frac{x_2}{\sigma}\right)^5
\right]
\exp\left(
-\frac{x_1^2+x_2^2}{\sigma^2}
\right)
, \qquad \sigma = \dfrac{1}{2\sqrt{\pi}}.
\end{aligned}
\] 
\end{example}
% \paragraph{Example 2} Let\[
% \begin{aligned}
% q(x_1,x_2)
% =
% 3\Bigg[
% &
% \frac{3}{20}
% \left(1-\frac{x_1^2}{\sigma^2}\right)
% \exp\left\{
% -\left[
% \frac{x_1^2}{\sigma^2}
% +
% \left(\frac{x_2}{\sigma}+1\right)^2
% \right]
% \right\}
% -
% \frac{1}{60}
% \exp\left\{
% -\left[
% \frac{x_2^2}{\sigma^2}
% +
% \left(\frac{x_1}{\sigma}+1\right)^2
% \right]
% \right\}
% \\
% &
% -
% \sigma
% \left[
% \frac{2}{5}x_1
% -
% \left(\frac{x_1}{\sigma}\right)^3
% -
% \left(\frac{x_2}{\sigma}\right)^5
% \right]
% \exp\left(
% -\frac{x_1^2+x_2^2}{\sigma^2}
% \right)
% \Bigg],
% \end{aligned}
% \]
We adopt this example from \cite{Greengard} with a different $\sigma$ since our spatial domain is $B(0,1)$ instead of $[-\frac{\pi}{2},\frac{\pi}{2}]^2$ as in the original experiment. In \cite{Greengard}, the recursive linearization method is used from $k = 1$ to $k = 9$ with frequency step $\delta k = 0.25$, and high fidelity is achieved with $L^2$ error smaller than $10^{-5}$ when $k = 9$.  Our algorithm converges in 3 steps using a single frequency $k=4\pi$. Figure ~\ref{fig:Greengard} shows the reconstruction with $20\%$ noise and without noise. Also worth noting is that iTSVD does not require a low-frequency start as the recursive linearization in  \cite{Greengard}, and the single frequency $k = 4\pi$ used in iTSVD is even higher than the highest frequency $k = 9$ in ~\cite{Greengard}.

\begin{figure}[H]
\centering
\begin{tabular}{c@{\vspace{2pt}}c@{\vspace{2pt}}cc}
\hspace{-0.3cm}\includegraphics[width=0.25\linewidth]{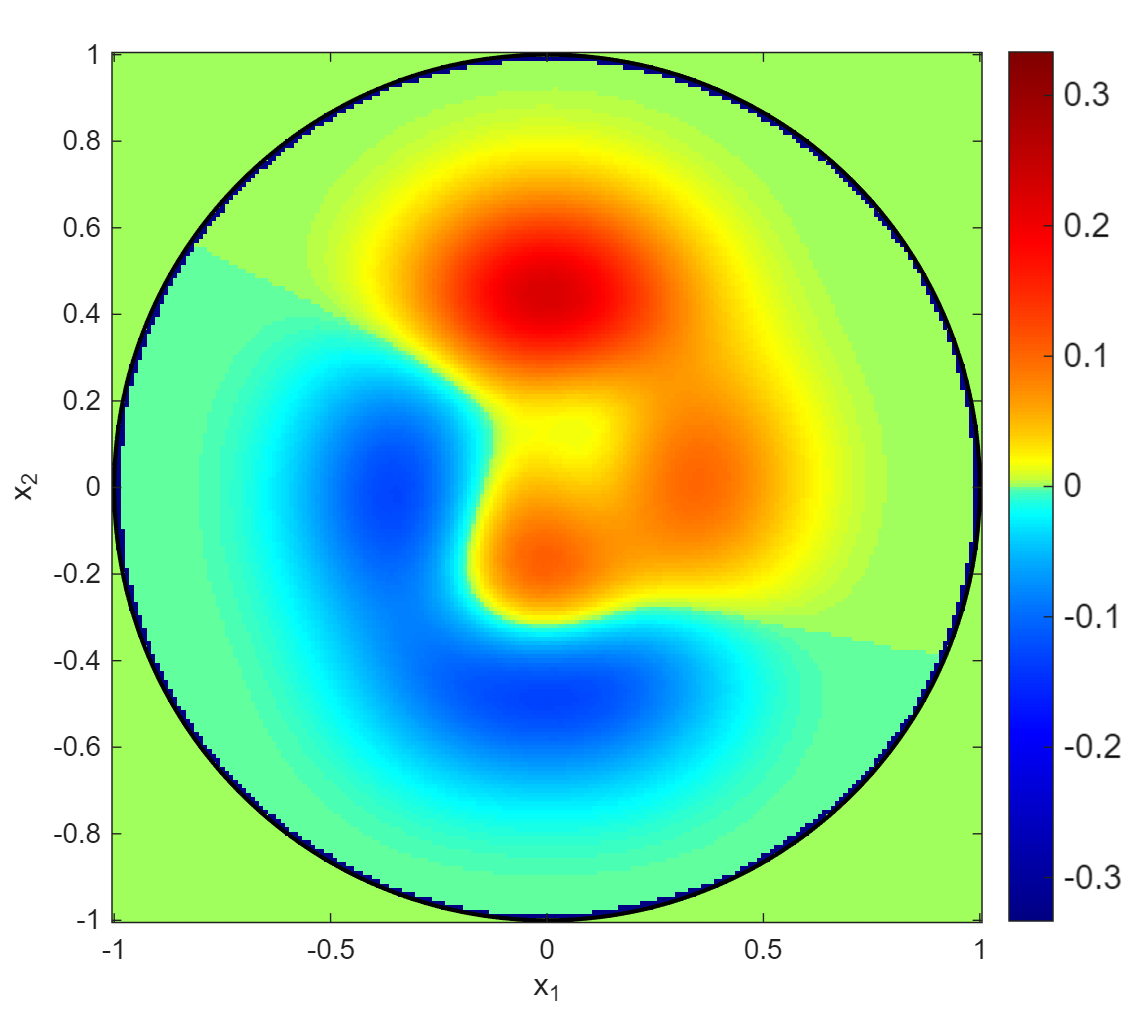}&
\includegraphics[width=0.25\linewidth]{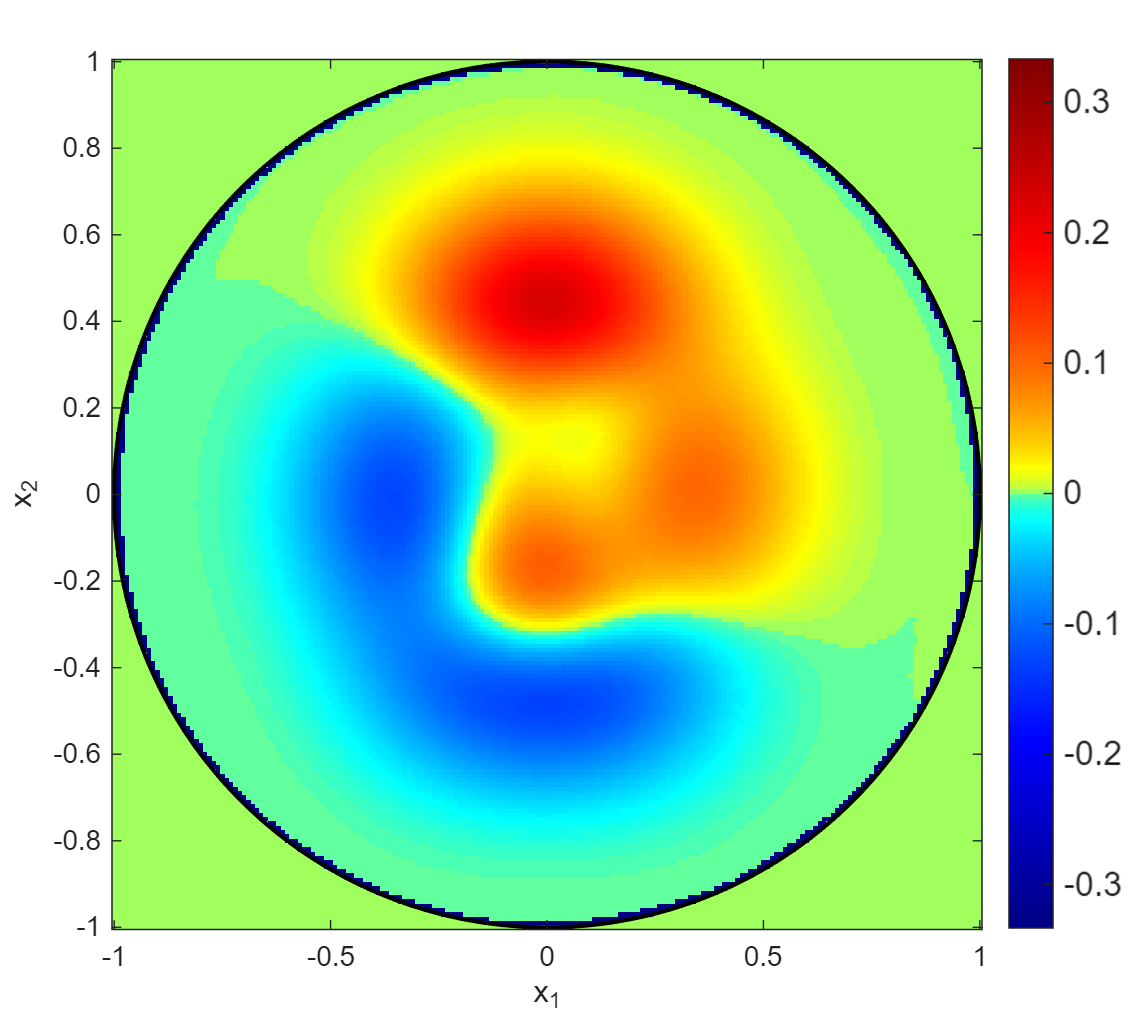}&
\includegraphics[width=0.25\linewidth]{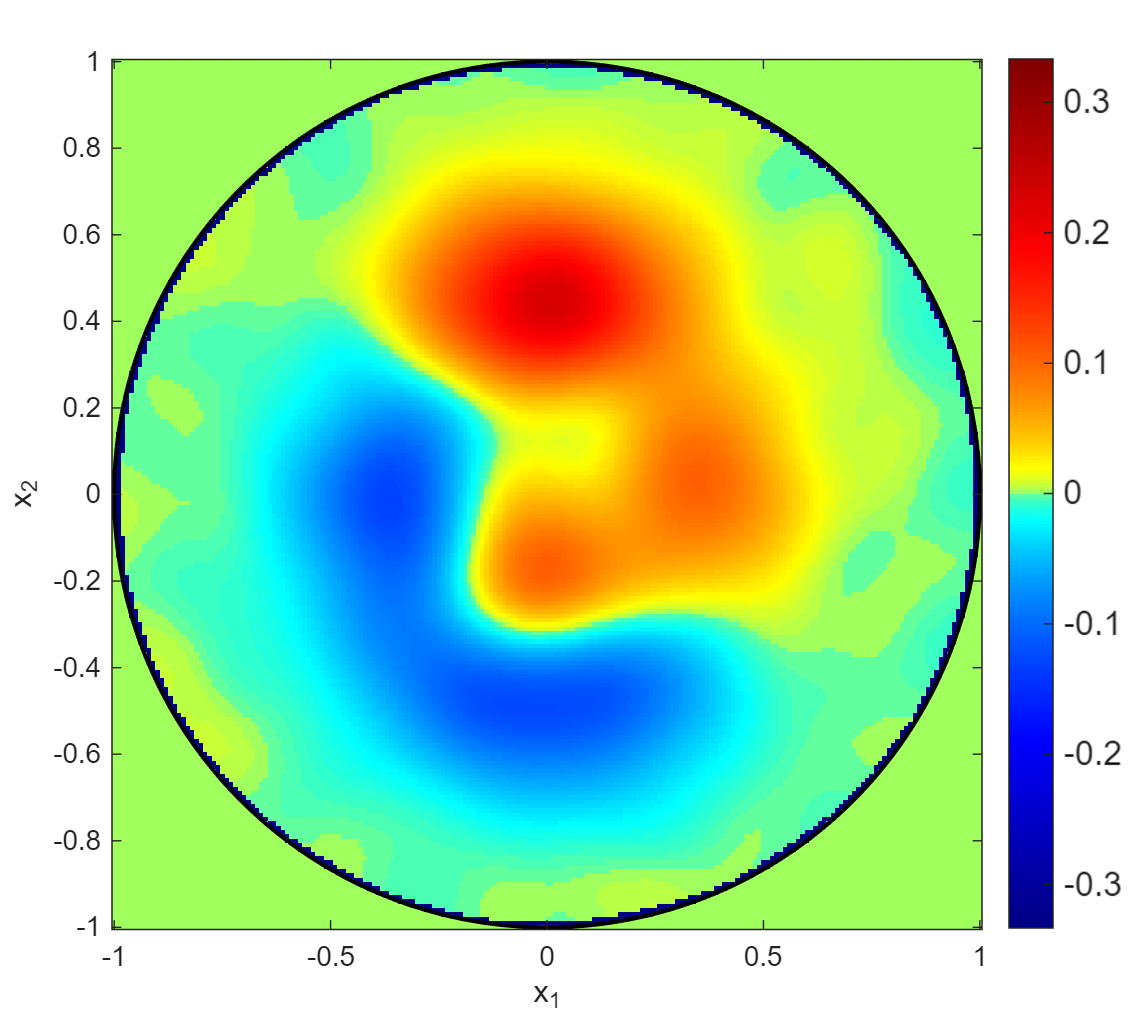} &
\hspace{-0.3cm}    \includegraphics[width=0.35\linewidth]{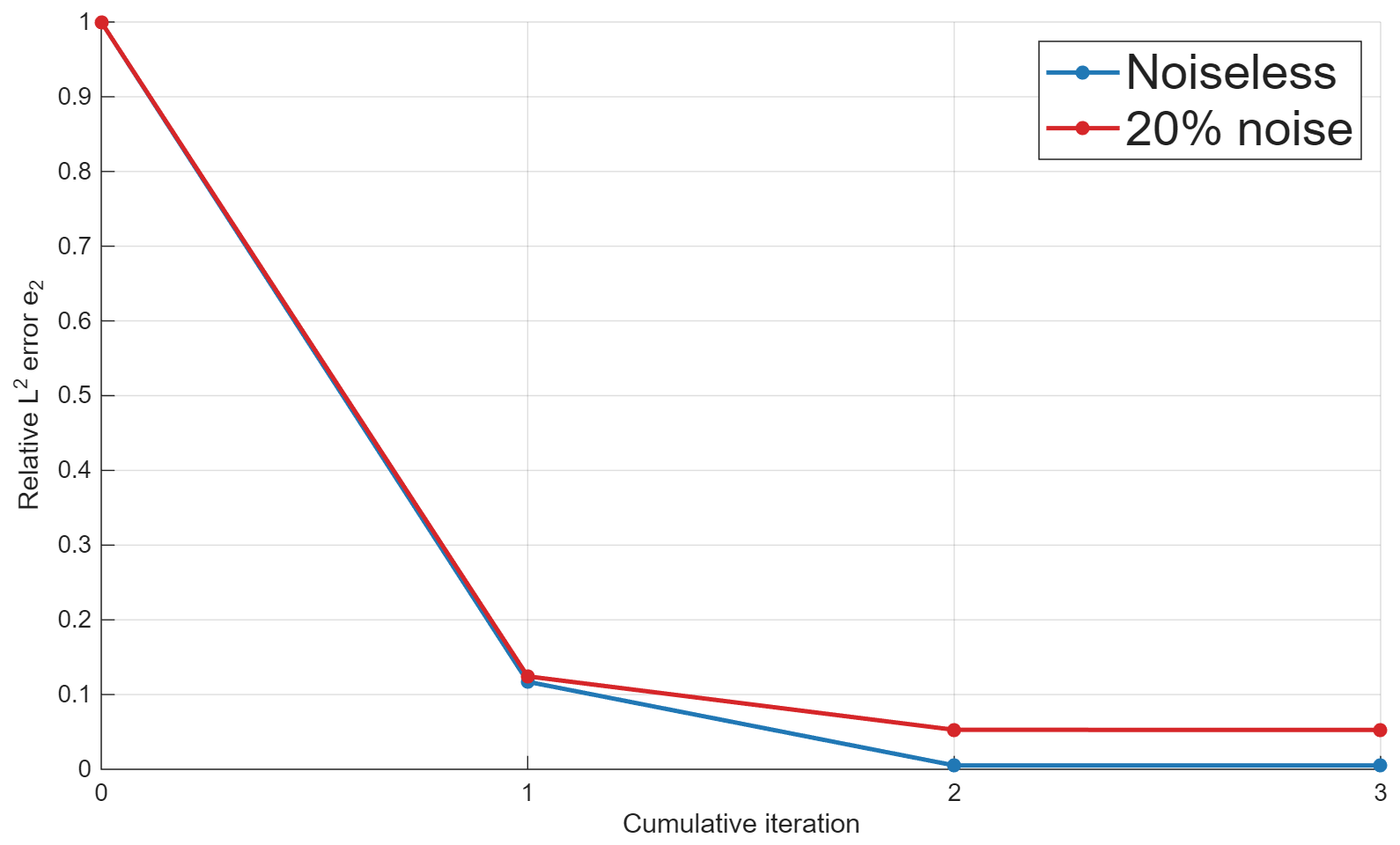}
\\
(a) true $q$ &(b) clean data &(c) data with $20\%$ noise & (d) convergence
\end{tabular}
\caption{Reconstruction results for Example ~\ref{ex: greengard}.}
\label{fig:Greengard}
\end{figure}

% \begin{figure}[H]
%     \centering
%     \includegraphics[width=0.5\linewidth]{noisy data/example 2/fig_convergence.png}
%     \caption{Convergence of $e_2$.}
%     \label{fig: Convergence of Greengard}
% \end{figure}

\begin{example}\label{ex: fine feature} We consider a compactly supported oscillatory contrast with relatively high spatial-frequency components, defined by
    \[
q(x_1,x_2)
=
\begin{cases}
0.7 \times \left[
0.7\left|\sin(3\pi x_1)\right|
+
0.3\left|\sin(6\pi x_2)\right|
\right]
\cos^2\left(\dfrac{\pi x_1}{1.4}\right)
\cos^2\left(\dfrac{\pi x_2}{1.4}\right),
&
|x_1|<0.7,\ |x_2|<0.7,
\\[8pt]
0,
&
\text{otherwise}.
\end{cases}
\]
\end{example}

% \paragraph{Example 3} Reconstruct a scatterer defined by
% \[
% q(x_1,x_2)
% =
% \begin{cases}
% 0.7 \times \left[
% 0.7\left|\sin(3\pi x_1)\right|
% +
% 0.3\left|\sin(6\pi x_2)\right|
% \right]
% \cos^2\left(\dfrac{\pi x_1}{1.4}\right)
% \cos^2\left(\dfrac{\pi x_2}{1.4}\right),
% &
% |x_1|<0.7,\ |x_2|<0.7,
% \\[8pt]
% 0,
% &
% \text{otherwise}.
% \end{cases}
% \]
In this example, the medium $q$ contains relatively high-frequency components, and we choose $k = 6\pi$. It takes 7 and 8 iterations to converge for the noiseless and noisy cases, respectively. An interesting phenomenon shown in Figure ~\ref{fig:Hierarchical}(d) is that the $L^2$ error does not decrease after the first Newton step, implying that $q$ lies outside the local contraction region linearized at the initial guess $q_0=0$. However, iTSVD still converges to a good result as an iterative method (see Figure ~\ref{fig:Hierarchical}). It illustrates the robustness of TSVD and the importance of using iterative linearization. Note that iTSVD is not an optimization method, so a monotone decrease of error with iteration may not be observed. 

On the other hand, Newton-type methods may need a good initial guess, especially when nonlinearity becomes stronger. In our case, when the amplitude 0.7 is increased to 1.0, the algorithm will diverge using $k = 6\pi$ starting from the initial guess $q_0=0$. However, this predicament can be resolved using frequency continuation as shown in Section ~\ref{sec: numerical on frequency continuation}.

\begin{figure}[htbp]
\centering
\begin{tabular}{c@{\vspace{2pt}}c@{\vspace{2pt}}cc}
\hspace{-0.3cm}\includegraphics[width=0.2\linewidth]{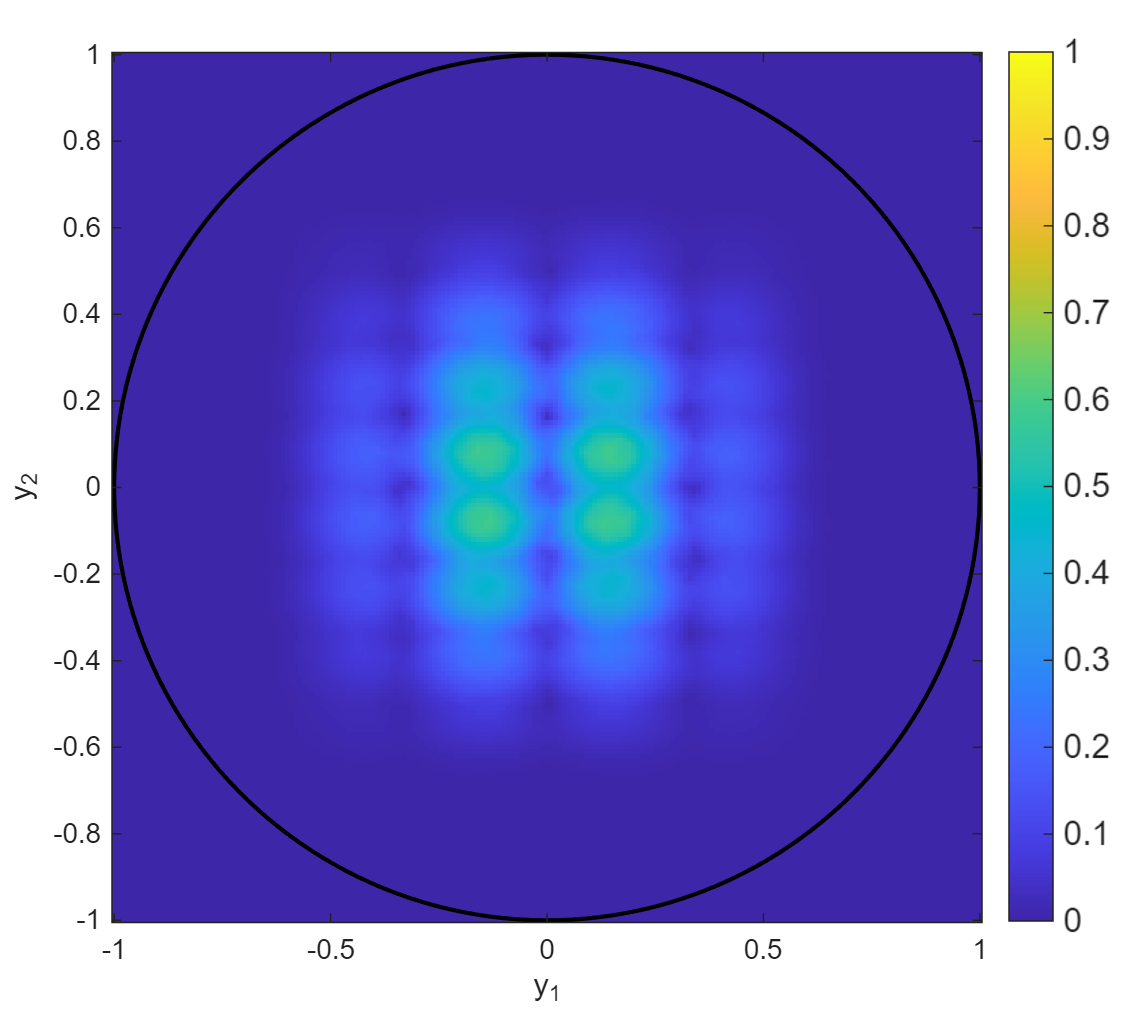}&
\includegraphics[width=0.2\linewidth]{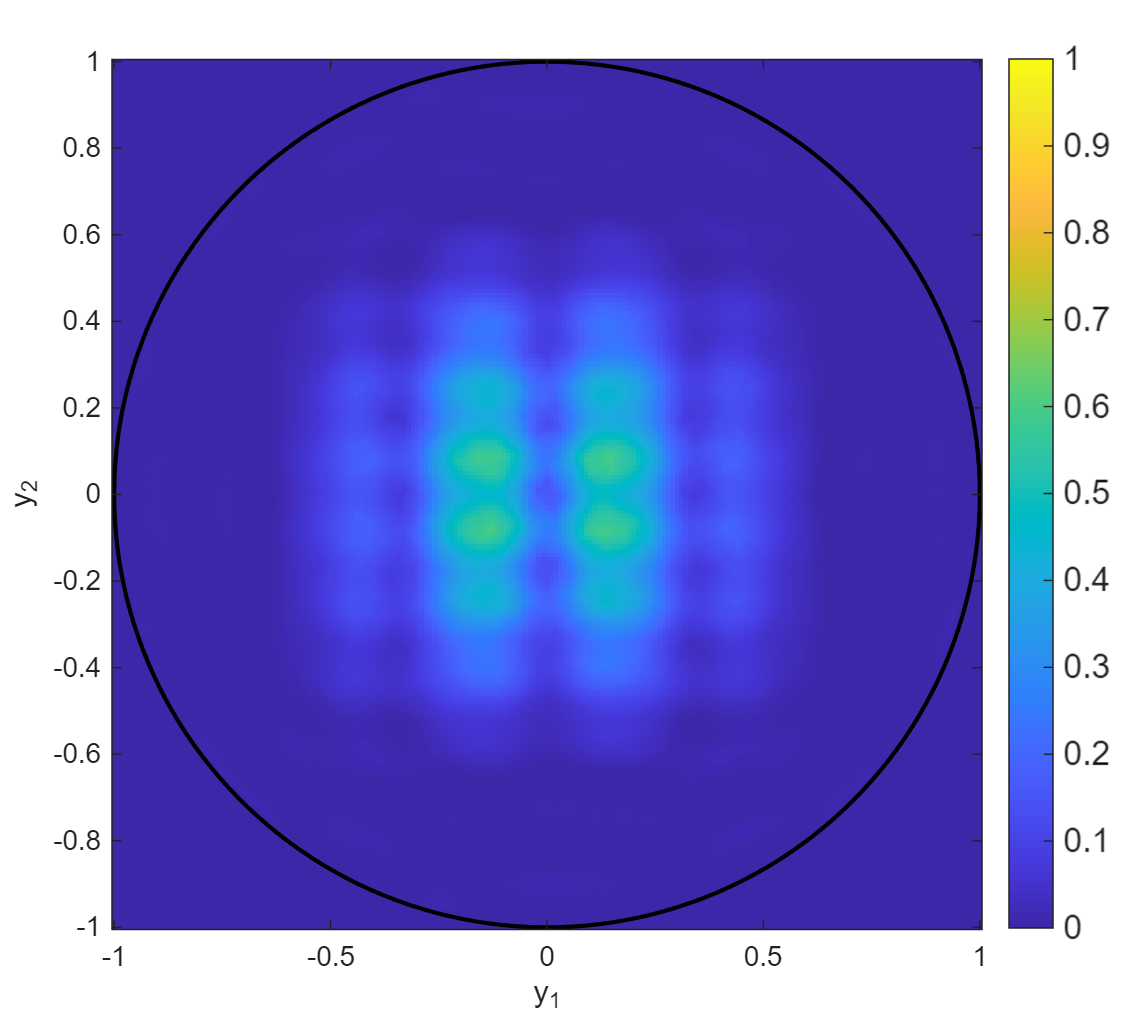}&
\includegraphics[width=0.2\linewidth]{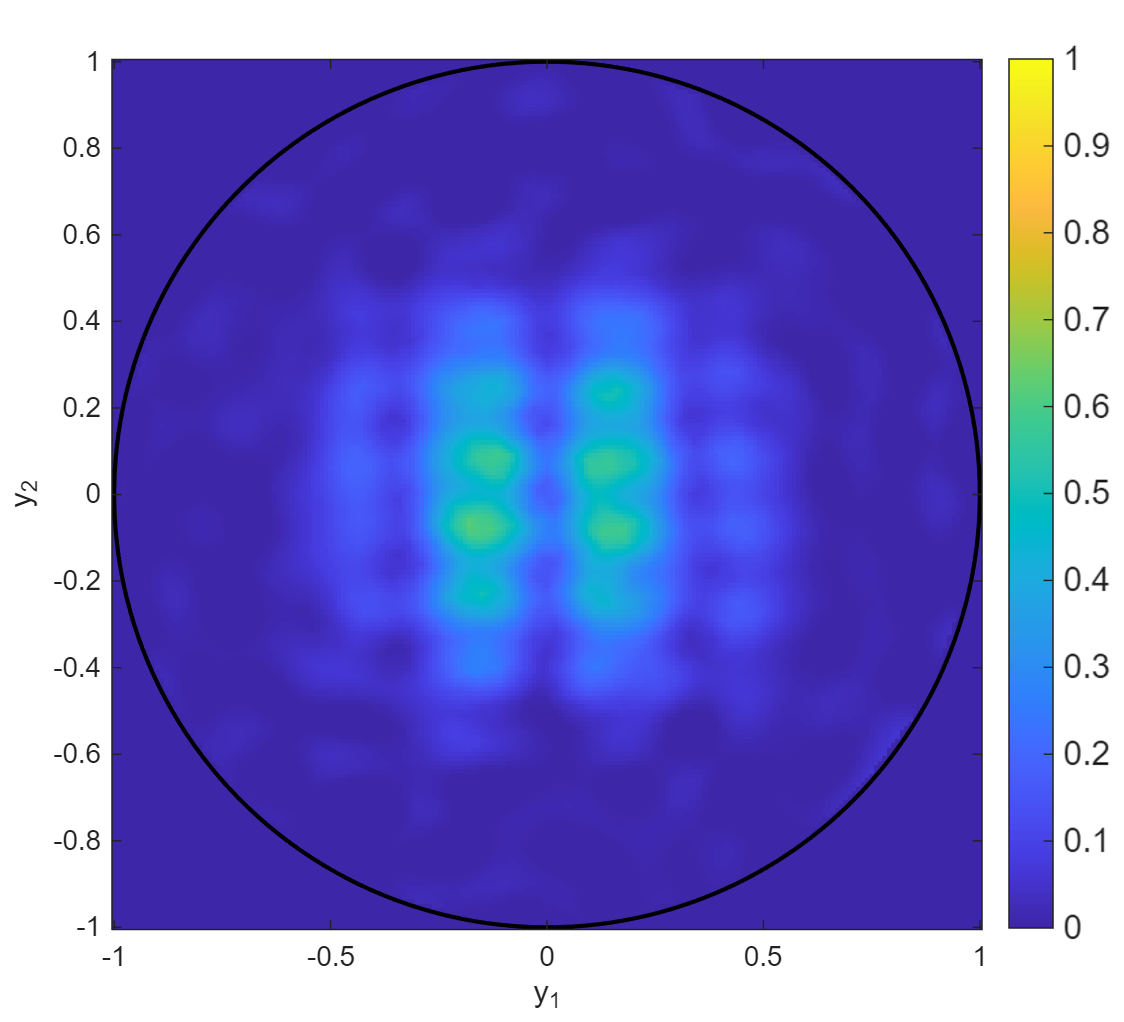} &
\hspace{-0.3cm}    \includegraphics[width=0.35\linewidth]{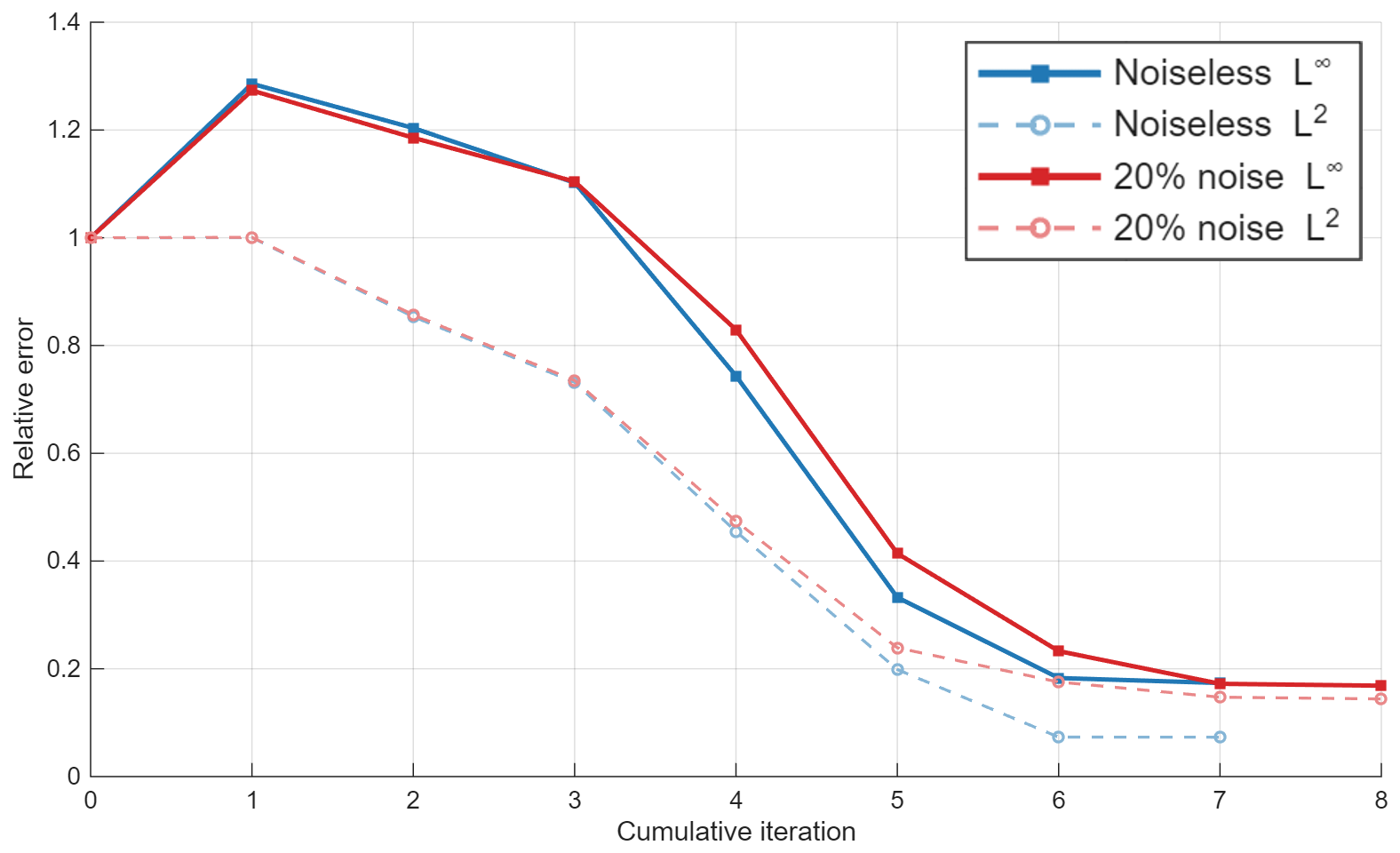}
\\
(a) true $q$ &(b) clean data &(c) data with $20\%$ noise & (d) convergence
\end{tabular}
\caption{Reconstruction results for Example ~\ref{ex: fine feature}.}
\label{fig:Hierarchical}
\end{figure}

%\subsection{3D Result}
\begin{example}
\label{ex:3d-five-gaussians}
We consider a contrast consisting of five Gaussian bumps truncated to the ball of radius $0.8$:
\[
 q(\mathbf{x})
 =\left[\sum_{j=1}^{5} A_j
       \exp\!\left(-\frac{|\mathbf{x}-c_j|^2}{\sigma_j^2}\right)\right]
       \mathbf 1_{\|\mathbf{x}\|_2\leq 0.8}(\mathbf{x}),
\]
where
\[
\begin{array}{c|c|c|c}
 j&A_j&c_j&\sigma_j\\
 \hline
 1&0.80&(-0.30,-0.20,-0.18)&0.15\\
 2&0.65&( 0.28,-0.22, 0.14)&0.14\\
 3&0.70&(-0.24, 0.26, 0.17)&0.16\\
 4&0.55&( 0.25, 0.24,-0.20)&0.13\\
 5&0.60&( 0,    0,    0.36)&0.14.
\end{array}
\]
\end{example}
This is a test for iTSVD on the 3D example. We take $k = 4\pi$ with $100$ incident directions and observation directions evenly distributed on $\mathbb{S}^2$. The reconstruction domain is $B^3(0,1)$. The Shannon number for this setup is $S = 1122$ according to ~\eqref{eq:Shannon number 3D}, which is used as the effective rank estimate.
% $[-1,1]^3$ with a $40\times40\times40$ grid. (? \textcolor{blue}{what is the rank estimate?}).
The $p$-transform Jacobian based on \eqref{eq:p transform in 3D} is used as the weight in discrete least squares. Figure ~\ref{fig:3D numerical} shows the reconstruction results and a convergence behavior similar to 2D, where both accuracy and efficiency are demonstrated.

% \begin{figure}[H]
% \centering
% \begin{tabular}{c@{\vspace{2pt}}c}
% \hspace{-0.3cm}\includegraphics[width=0.6\linewidth]{3D/comparison_1d.png}&
% % \includegraphics[width=0.25\linewidth]{3D/comparison_3d.png}&
% \includegraphics[width=0.3\linewidth]{3D/comparison_error_decay.png}
% \\
% (a) one-dimensional slice  &(b) convergence
% \end{tabular}
% \caption{Reconstruction results for Example ~\ref{ex:3d-five-gaussians}.}
% \label{fig:3D numerical}
% \end{figure}

\begin{figure}[htbp]
\centering
\begin{tabular}{c@{\vspace{2pt}}c@{\vspace{2pt}}c@{\vspace{2pt}}c}
\includegraphics[width=0.22\linewidth]{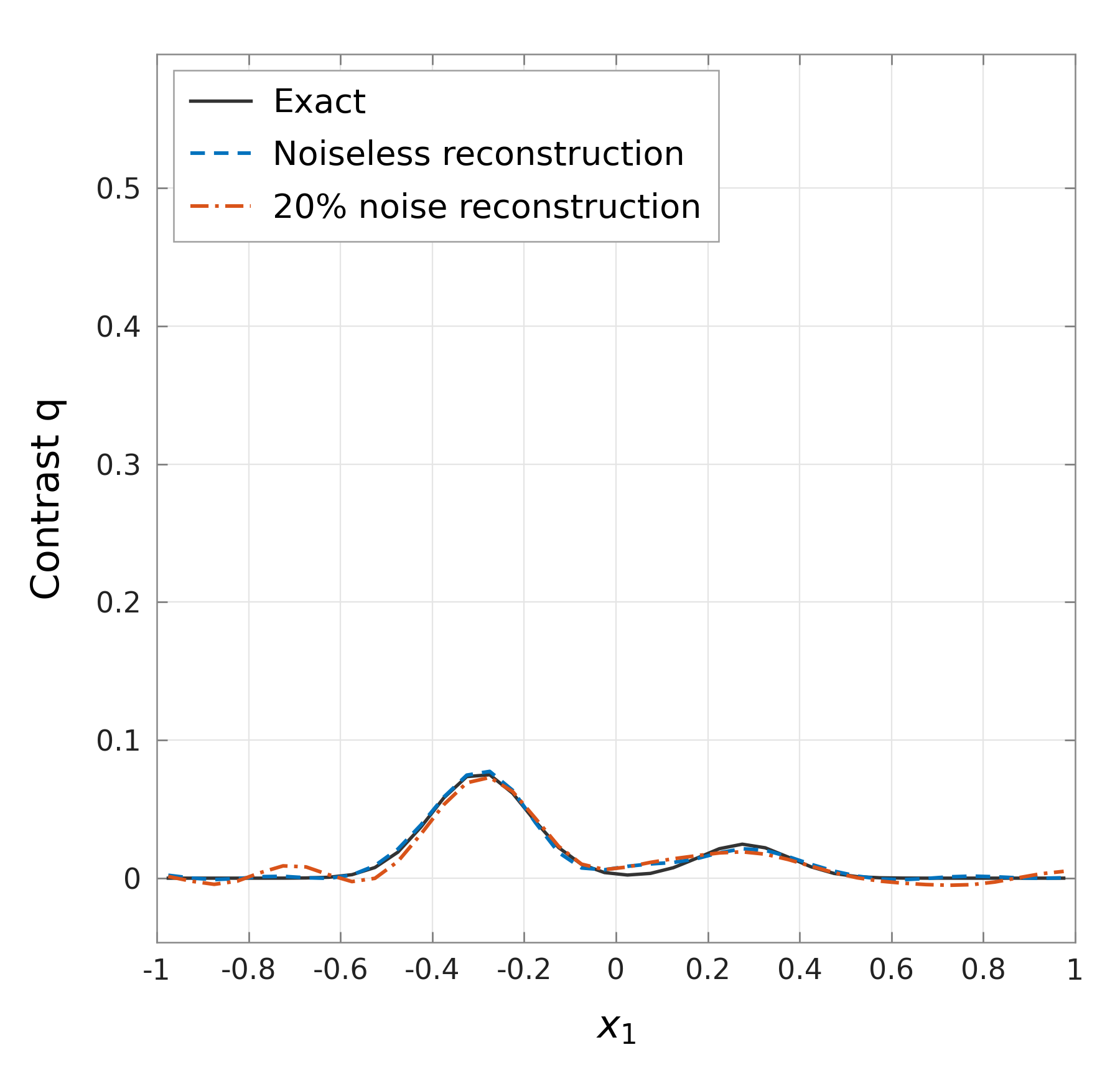}&
\includegraphics[width=0.22\linewidth]{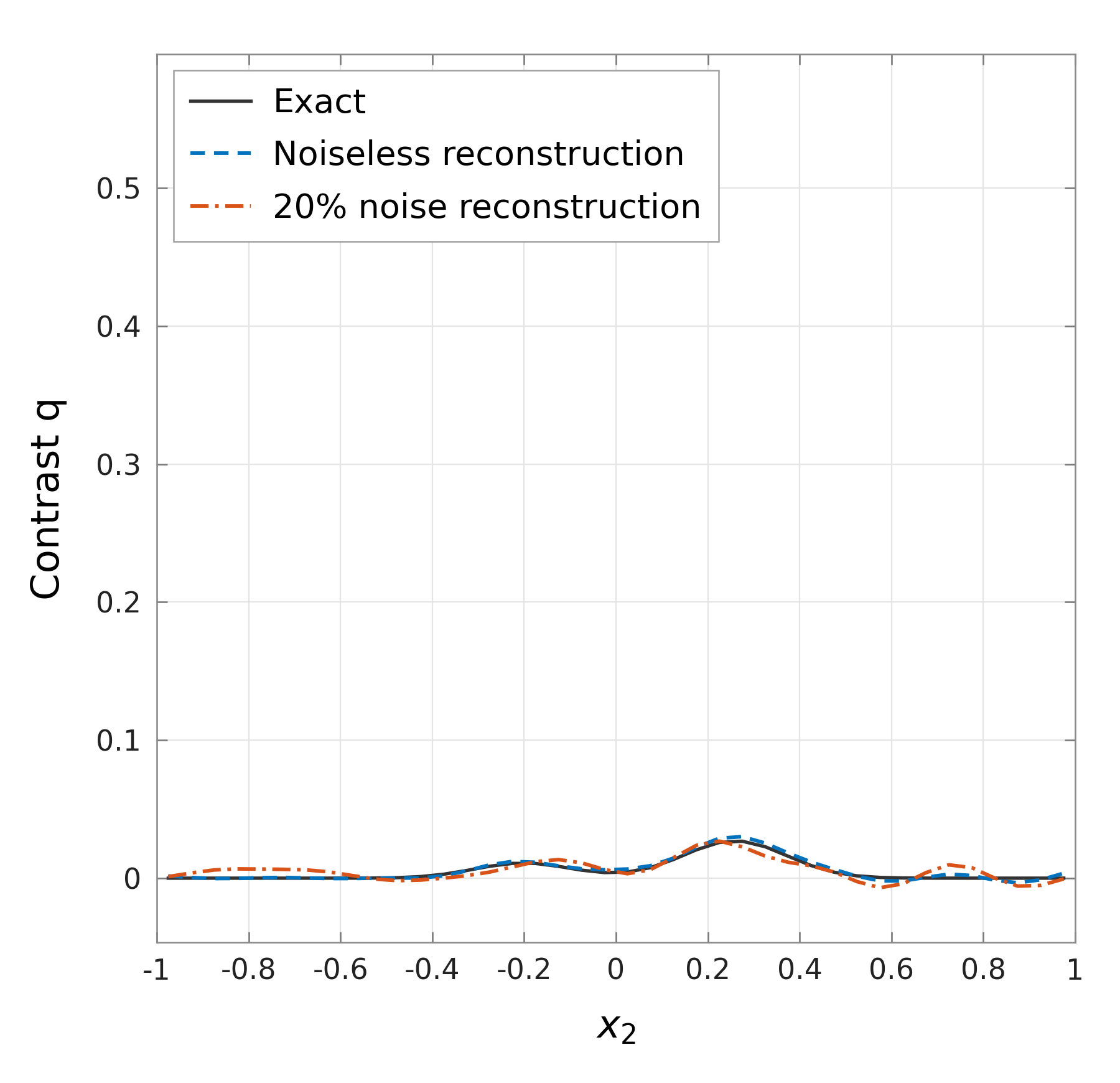}&
\includegraphics[width=0.22\linewidth]{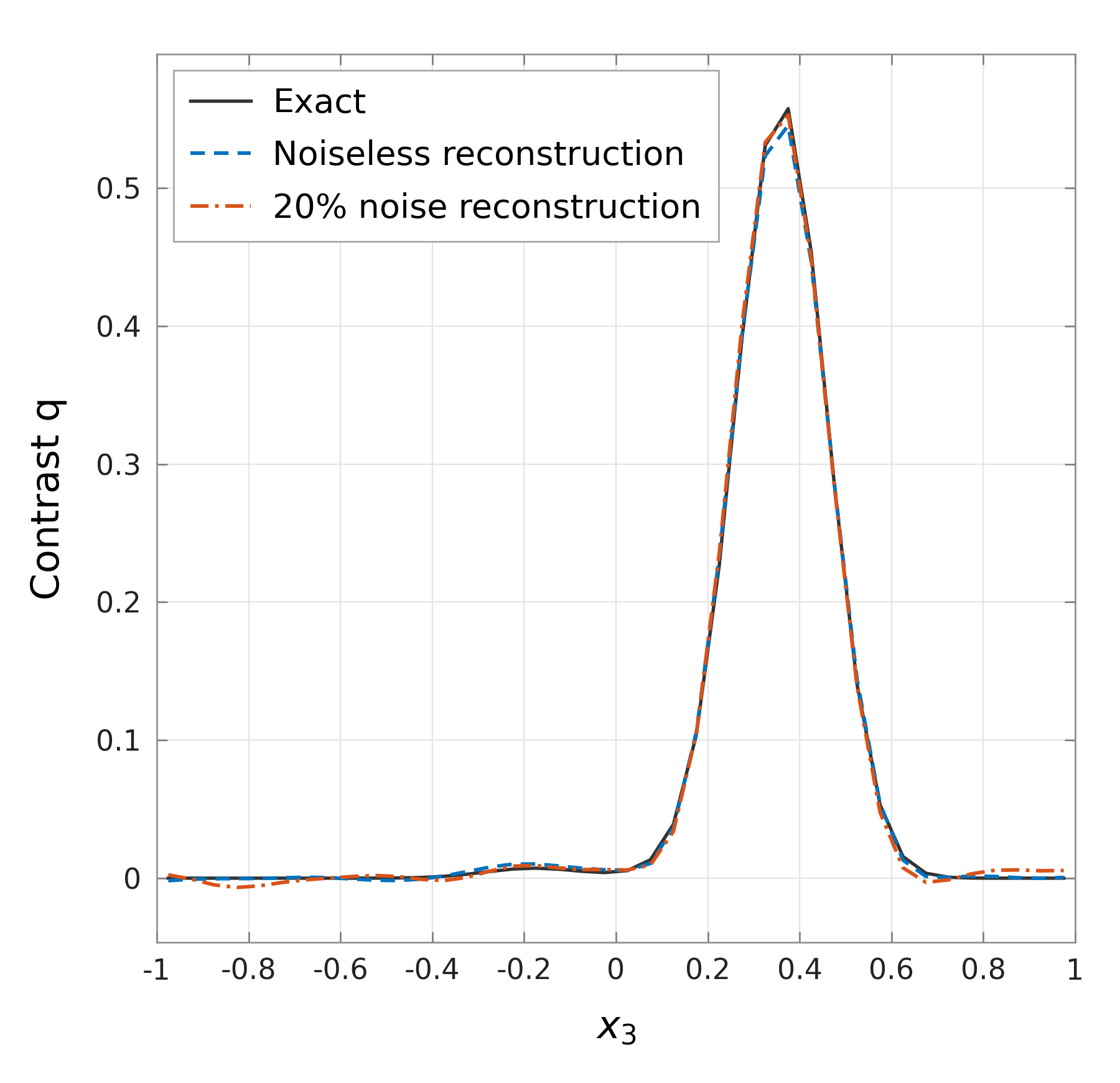}&
\includegraphics[width=0.3\linewidth]{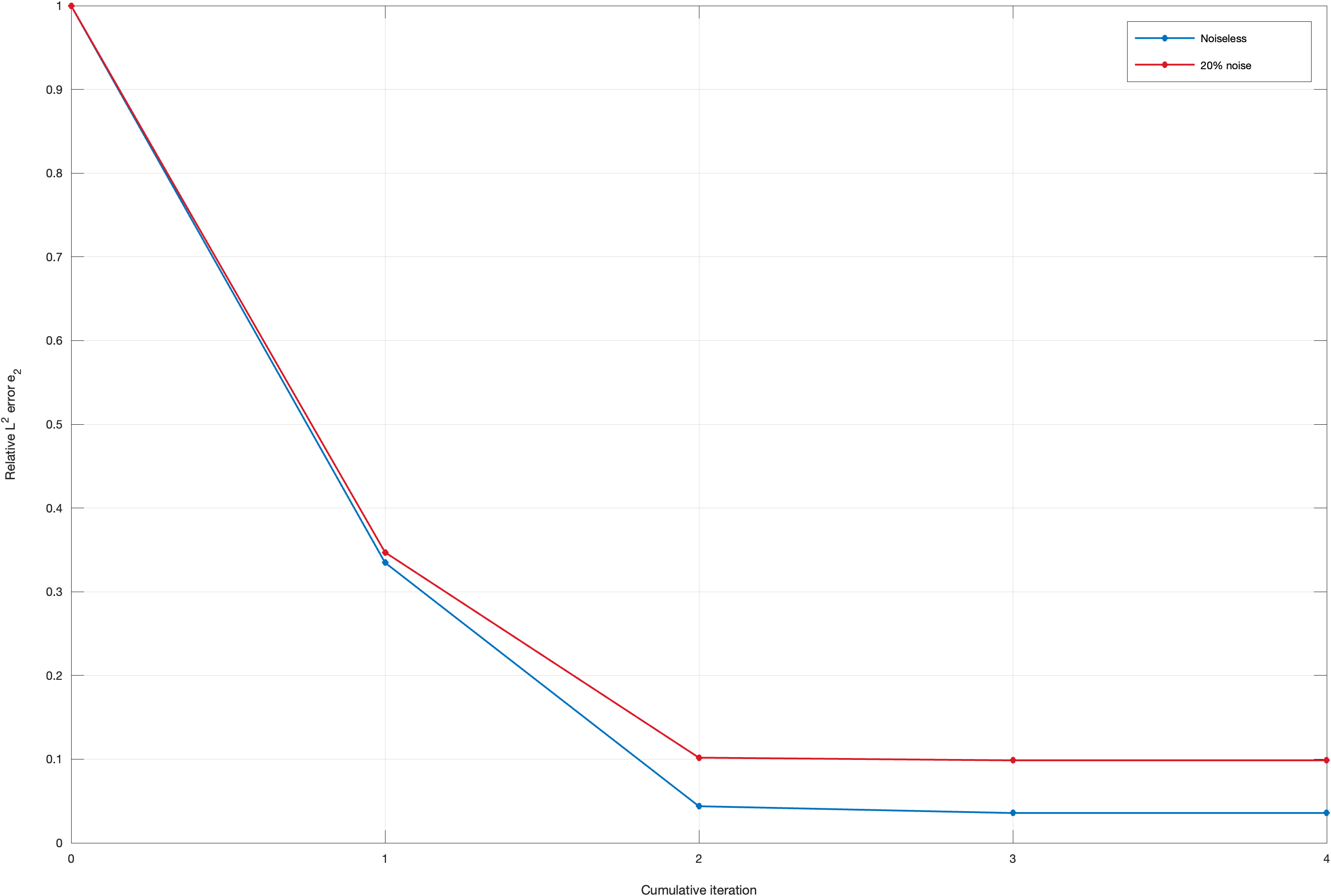}
\\
(a) $x_2 = x_3 = -0.025$ &(b) $x_1 = x_3 = -0.025$ &(c) $x_1 = x_2 = -0.025$ & (d) convergence
\end{tabular}
\caption{Reconstruction results for Example ~\ref{ex:3d-five-gaussians}.}
\label{fig:3D numerical}
\end{figure}

\subsection{Limited Aperture}
\label{sec: numerical on limited aperture}

In this section, we verify the discussions for limited aperture in Section ~\ref{sec: limited aperture analysis} and demonstrate how different aperture settings affect the reconstruction.
In general practice, a reconstruction result can be related to the extent to which the significant Fourier modes of the medium contrast $q(x)$ are contained in the frequency coverage after the $p$-transform ~\eqref{eq:p transform} corresponding to a specific aperture setup.   
%In particular, we verify the correspondence between the frequency domain coverage and reconstruction results for transmission/reflection data. To eliminate the effect of the aperture position, we choose the following symmetrical contrast.
\begin{example}[]\label{ex:single-gaussian}
To examine the effect of aperture size and relative orientation on reconstruction quality, we consider a symmetrical truncated Gaussian contrast defined by
\[
q(x_1,x_2)=
\begin{cases}
1.5\exp\!\left(-50(x_1^2+x_2^2)\right),
&
|x_1|\le \dfrac{1}{\sqrt{2}},\quad
|x_2|\le \dfrac{1}{\sqrt{2}},
\\[8pt]
0,
&
\text{otherwise}.
\end{cases}
\]
\end{example}
Set $\omega = |c_{\hat{\theta}} -c_{\hat{x}}|$ as described in Section ~\ref{sec: limited aperture analysis}. Figure ~\ref{fig:numerical on single Gaussian}(a) shows the reconstruction error for different aperture windows $L$ and different configurations quantified by the separation angle $\omega$ between the center of the incident aperture and the center of the observation aperture. First, with the same configuration, i.e., the same $\omega$, a larger aperture means a larger coverage in the frequency domain, resulting in better reconstruction. Moreover, the measurement configuration, i.e., the relative position between the incident and observation angles characterized by $\omega$, also leads to different coverage in the frequency domain and results in different reconstruction results. For example, for this particular $q(x)$ which is mainly composed of low Fourier modes, our tests show that for all aperture window sizes $L$, the transition from the reflection ($\omega = \pi$) data to the transmission ($\omega = 0$) data leads to better reconstruction in the $L^2$ norm due to the shift of the Fourier domain coverage from higher frequency to lower frequency as shown in Figure~\ref{fig:limited_angle}.  The extreme case is when $L=\frac{7}{4}\pi$, as shown in Figure~\ref{fig:numerical on single Gaussian}(a) and Figure~\ref{fig:limited_angle}(c), the transmission data cover enough low-frequency modes and result in a reconstruction as good as the full aperture data, while the reflection data contain not enough low-frequency information and result in a very poor reconstruction for this smooth $q(x)$.  

We also take $\omega = 0$ and $\omega = \pi$ to show the reconstruction plots for different apertures $L$ in Figure ~\ref{fig:numerical on aperture}. In the case of transmission data $\omega=0$, with respect to the center direction of the aperture,  we see that the resolution in the transverse direction is significantly better than that in the range direction, especially when the aperture is relatively small, e.g., $L=\frac{3\pi}{4}$. This corresponds to the geometry of the Fourier coverage region shown in Figure ~\ref{fig:limited_angle}, where modes along the transverse direction are covered to a greater extent than those along the range direction. 

%These results also verify the illumination of the different frequency components with different separations $\omega$ since $q$ is smooth and contains mostly low-frequency components. Inversion at large $p$ yields little update, while inversion at small $p$ yields a significant update. This is not a claim that the reconstruction is better in the case of transmission for any $q$.
% \begin{figure}[H]
%     \centering
%     \includegraphics[width=0.5\linewidth]{limited aperture/limited aperture position + window/reference_error_vs_separation_new.png}
%     \caption{$e_2$ for different separation $\omega$ with $L=\frac{3}{4}\pi, \frac{5}{4}\pi, \frac{7}{4}\pi, 2\pi$.}
%     \label{fig:e_2 with different L and separation delta}
% \end{figure}

\begin{figure}[htbp]
\centering
\begin{tabular}{c@{\vspace{2pt}}c@{\vspace{2pt}}c}
\includegraphics[width=0.3\linewidth]{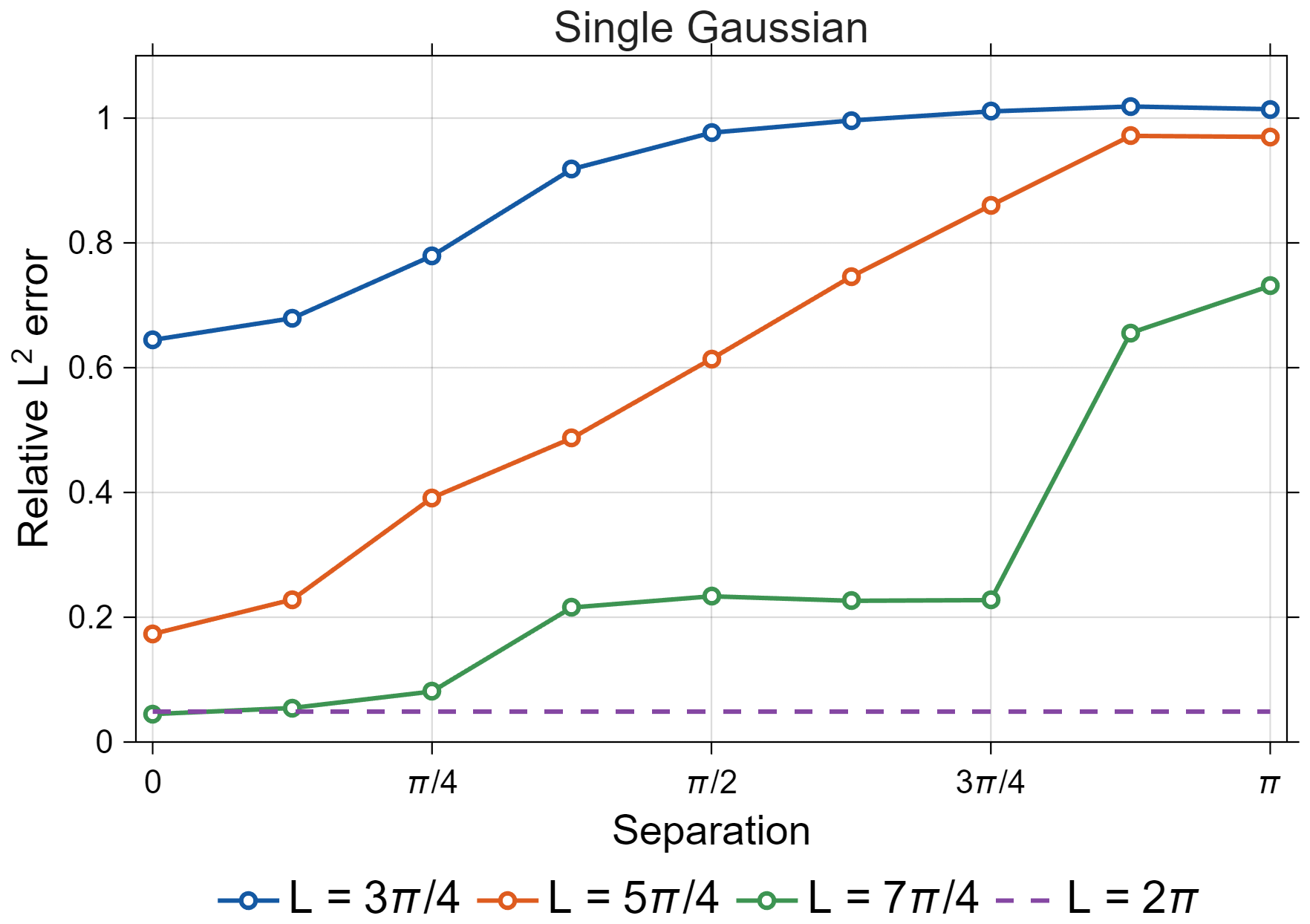}&
\includegraphics[width=0.2\linewidth]{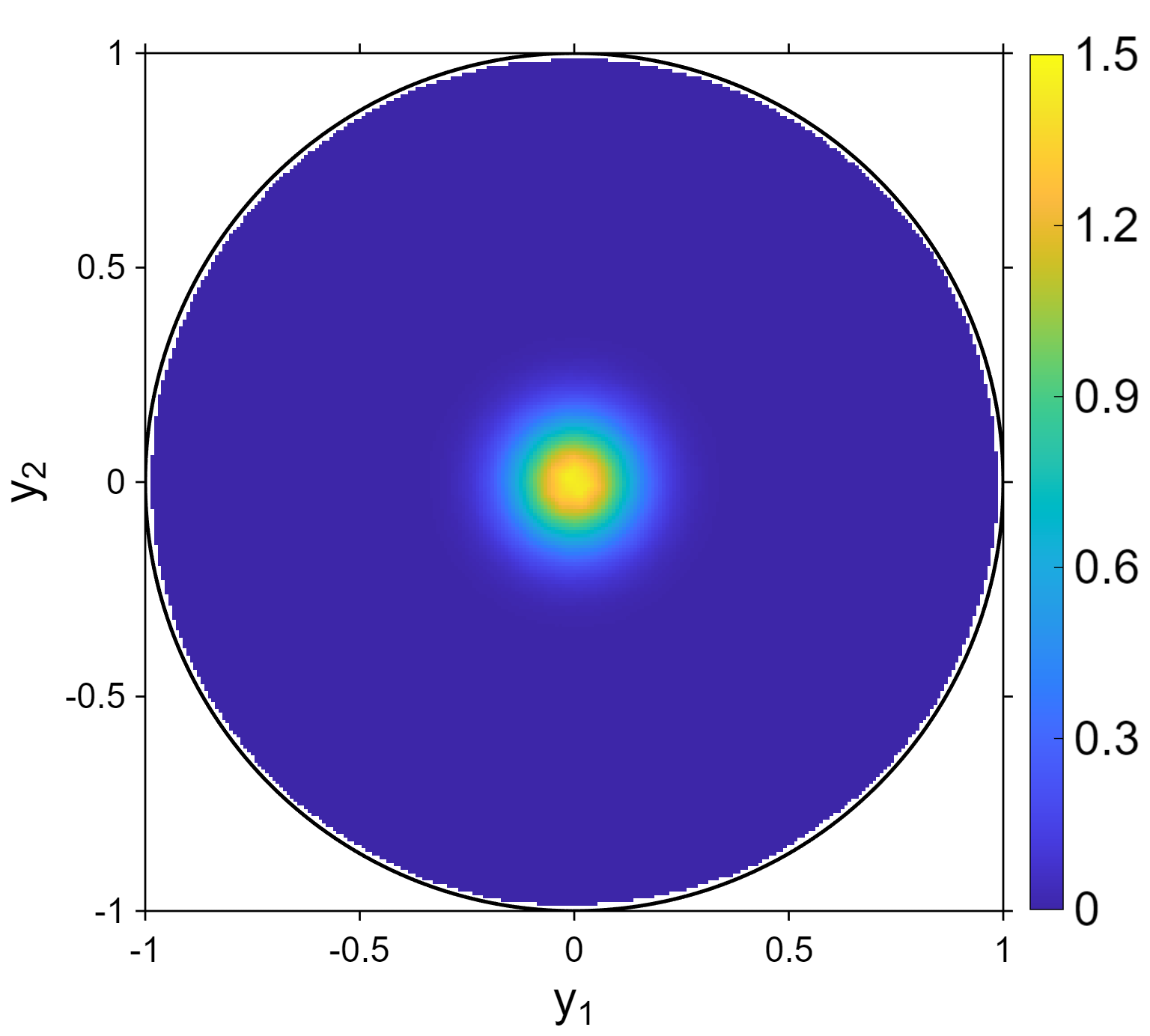}&
\includegraphics[width=0.2\linewidth]{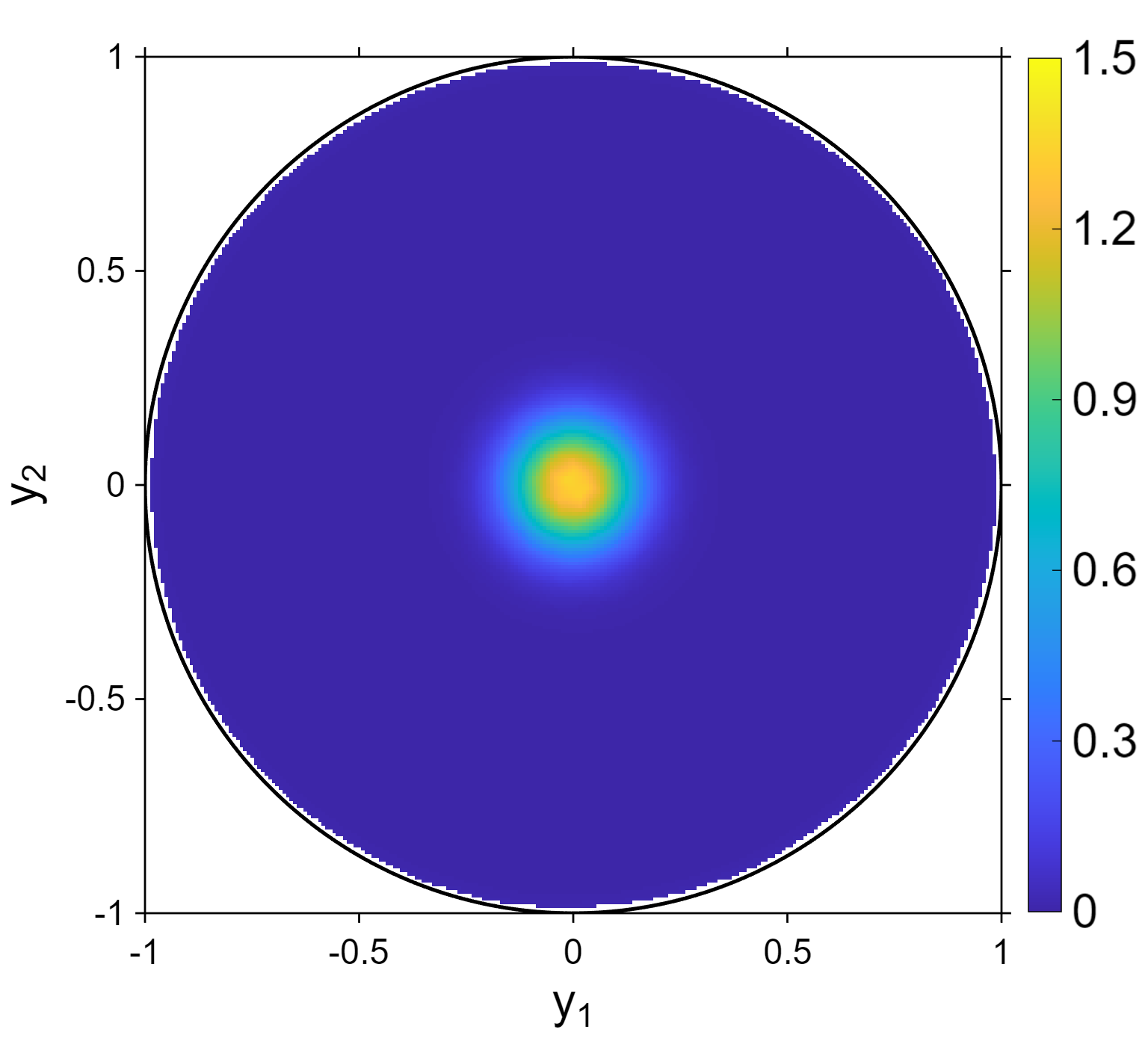}
\\
(a) $e_2$ v.s. $\omega$ for different $L$ &(b) true $q$  &(c) full aperture
\end{tabular}
\caption{Reference for Example ~\ref{ex:single-gaussian}.}
\label{fig:numerical on single Gaussian}
\end{figure}

% \begin{figure}[H]
%     \centering
%     \includegraphics[width=0.9\linewidth]{limited aperture/limited aperture position + window/reference_reconstructed_single_gaussian.png}
%     \caption{Reconstruction $\tilde{q}$ for different $L$ and separation $\omega$.}
%     \label{fig:reconstruction with different L and delta}
% \end{figure}

\begin{figure}[htbp]
\centering
\begin{tabular}{c@{\vspace{2pt}}c@{\vspace{2pt}}c}
\includegraphics[width=0.2\linewidth]{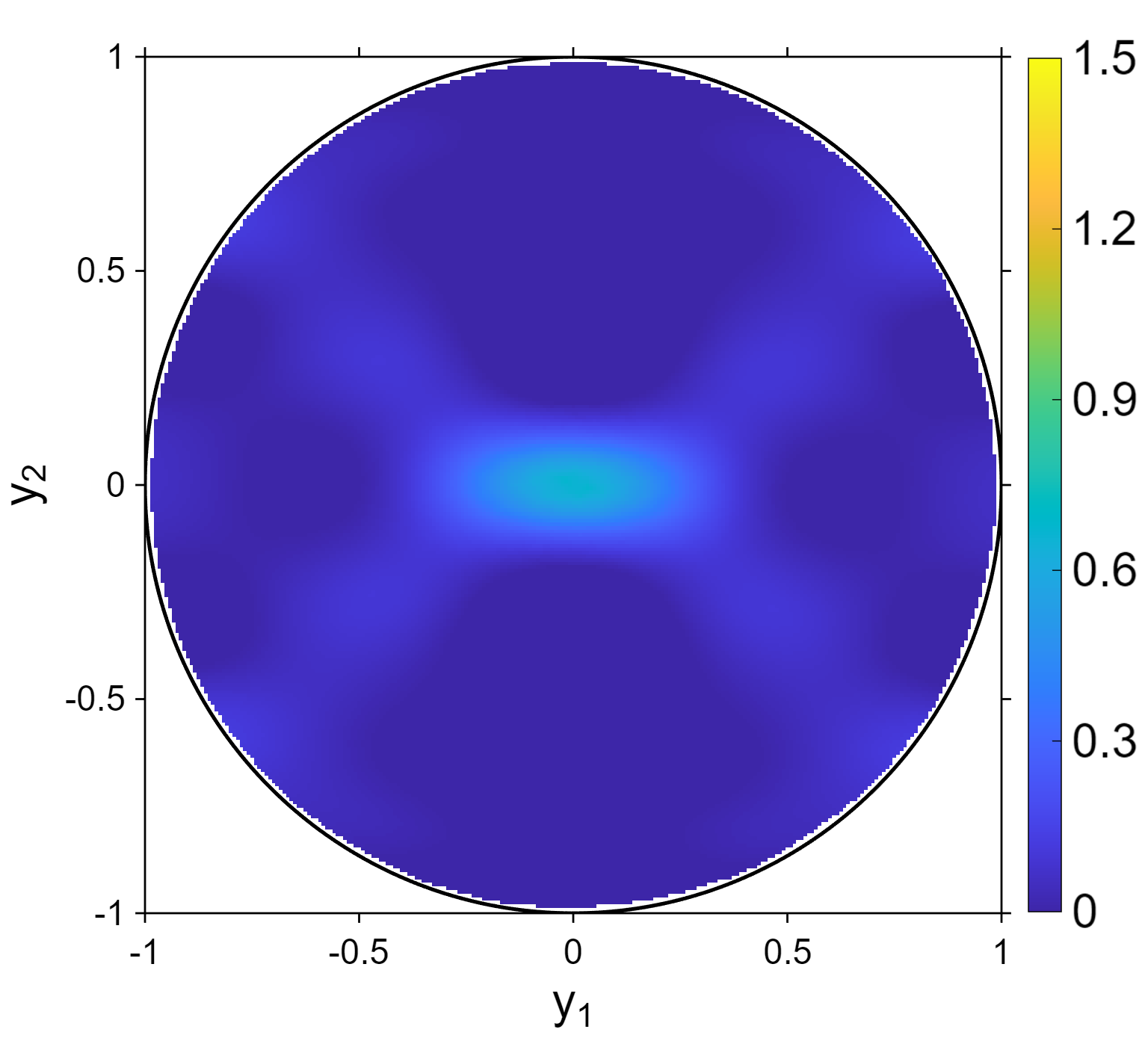}&
\includegraphics[width=0.2\linewidth]{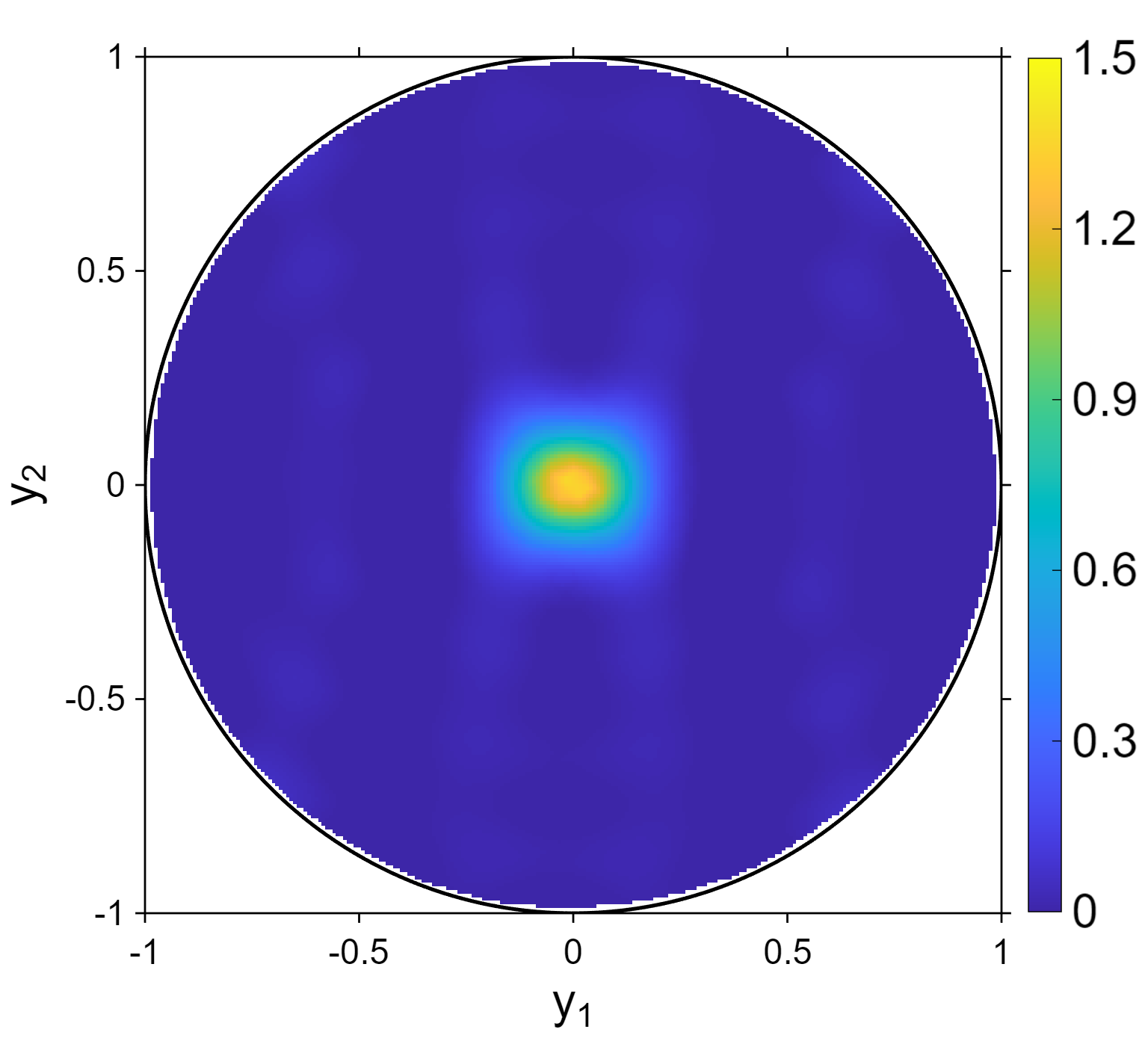}&
\includegraphics[width=0.2\linewidth]{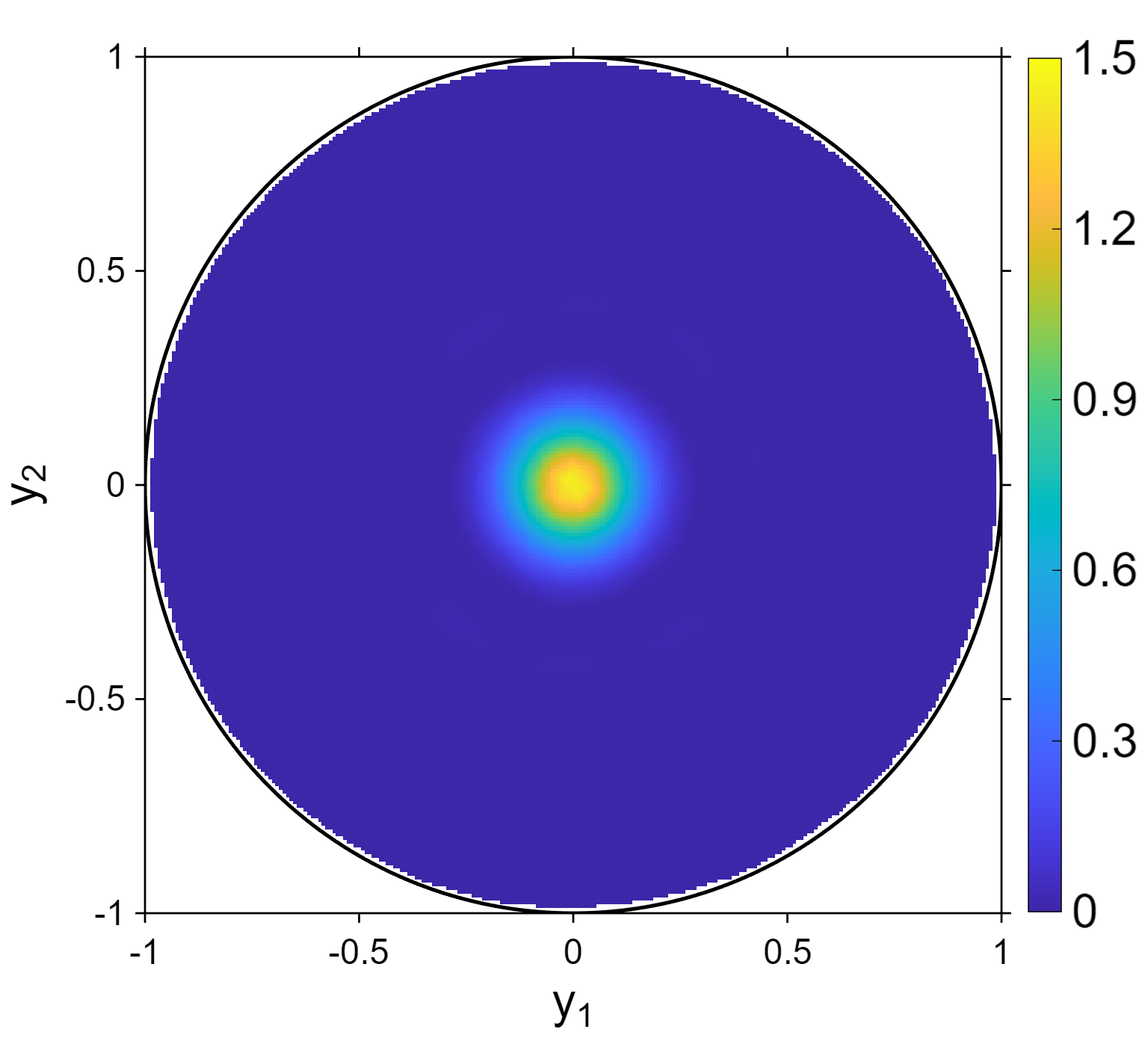}\\
\includegraphics[width=0.2\linewidth]{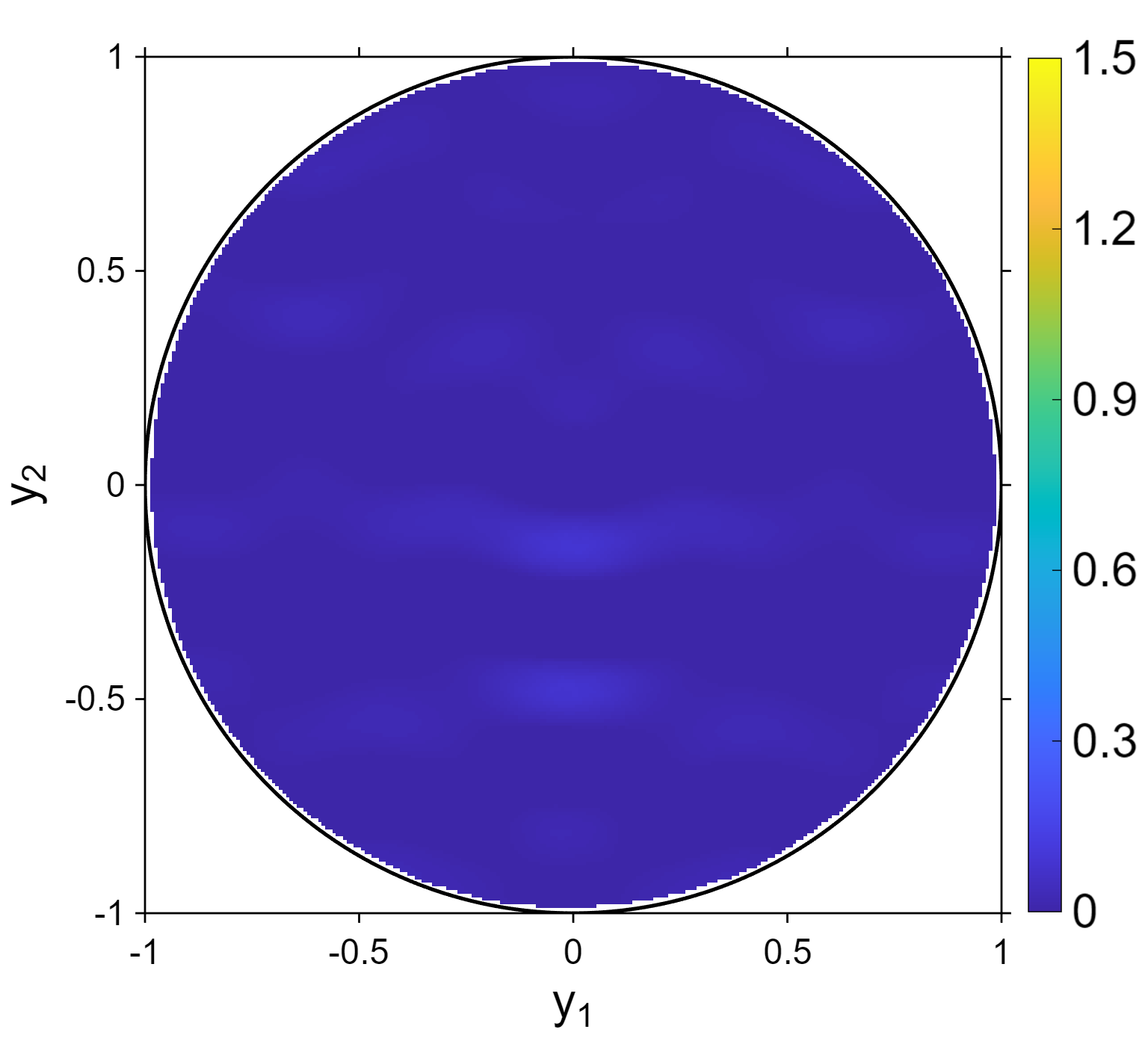}&
\includegraphics[width=0.2\linewidth]{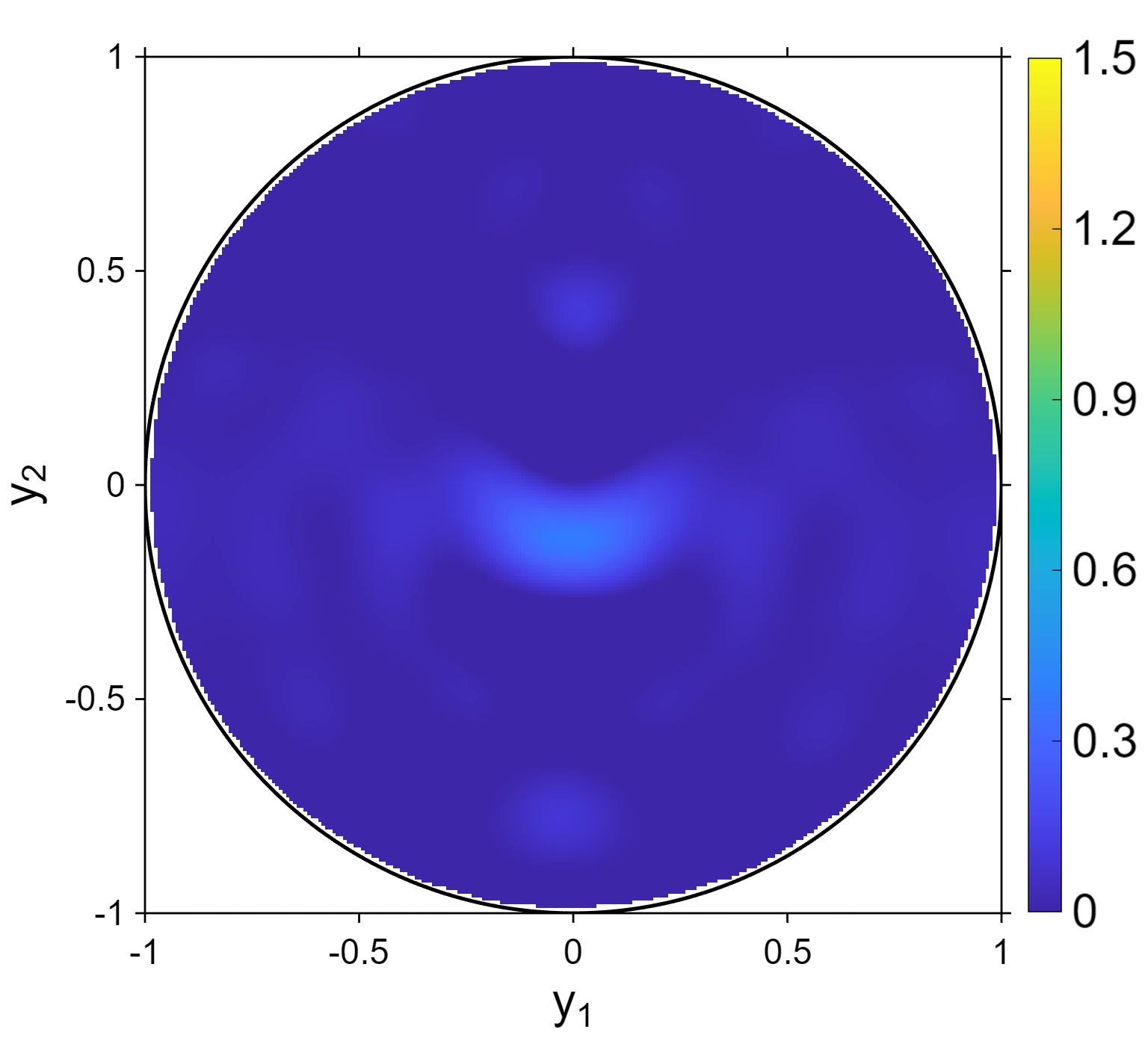}&
\includegraphics[width=0.2\linewidth]{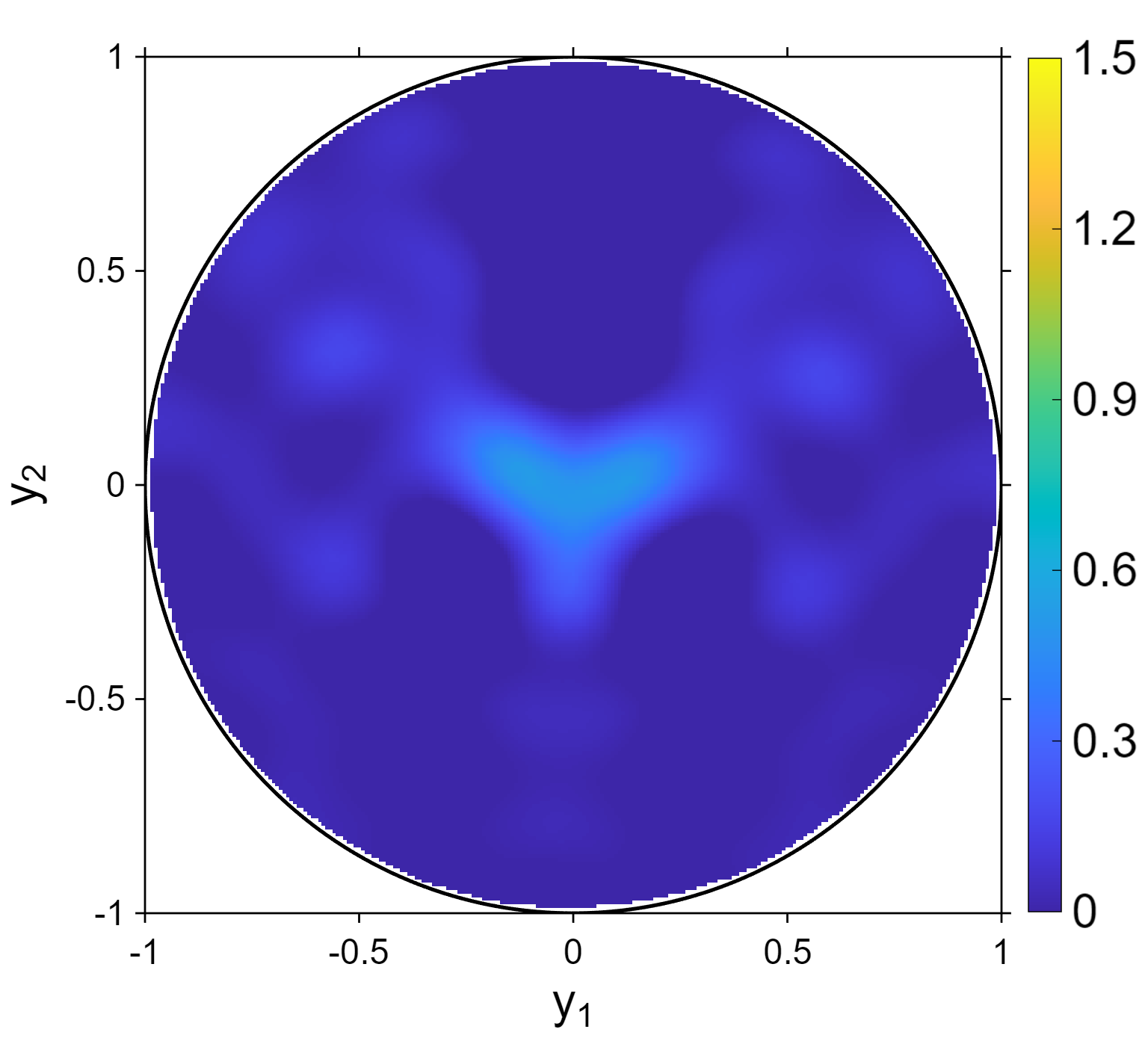}
\\
(a) $L=\frac{3}{4}\pi$ &(b) $L=\frac{5}{4}\pi$ &(c) $L=\frac{7}{4}\pi$
\end{tabular}
\caption{Reconstruction results for Example ~\ref{ex:single-gaussian}. 1st row: results from transmission. 2nd row: results from reflection.}
\label{fig:numerical on aperture}
\end{figure}

\begin{example}[]\label{ex:low and high}
We consider a contrast combining a localized low-frequency component with an oscillatory component with higher frequencies:
\[
q(x_1,x_2)=
\begin{cases}
\displaystyle
0.8\exp\!\left[
-\frac{
    \dfrac{(x_1+0.65)^2}{0.11^2}
    +\dfrac{x_2^2}{0.16^2}
}{
    2\left(1-\dfrac{(x_1+0.65)^2}{0.23^2}
             -\dfrac{x_2^2}{0.36^2}\right)
}
\right],
&
\displaystyle
\frac{(x_1+0.65)^2}{0.23^2}
+\frac{x_2^2}{0.36^2}<1,
\\[14pt]
\displaystyle
0.45\exp\!\left[
-\frac{
    \dfrac{(x_1-0.20)^2}{0.612^2}
    +\dfrac{x_2^2}{0.756^2}
}{
    2\left(1-\dfrac{(x_1-0.20)^2}{0.60^2}
             -\dfrac{x_2^2}{0.82^2}\right)
}
\right]\cos(6.8\pi x_2),
&
\displaystyle
\frac{(x_1-0.20)^2}{0.60^2}
+\frac{x_2^2}{0.82^2}<1,
\\[14pt]
0, & \text{otherwise}.
\end{cases}
\]
\end{example}

We compare two limited-aperture configurations as shown in Figure~\ref{fig:numerical on low-high}(a). The aperture size is $L=\frac{3}{2}\pi$. For the transmission setup, where the observation directions are the same as the incident directions, i.e., $c_{\hat{\theta}}=c_{\hat{x}}$, the corresponding Fourier domain coverage after the $p$-transform covers all low-frequency modes while leaving some high-frequency modes in the vertical directions in the Fourier disk uncovered. In contrast, for the reflection setup, where the observation aperture is opposite to the incident aperture, the corresponding Fourier domain coverage contains much fewer low-frequency modes, in particular those along the vertical direction, while covering high-frequency modes including those in the vertical direction near frequency $2k$. These observations correlate well with the reconstruction results shown in Figure~\ref{fig:numerical on low-high}(c). Of course, full aperture data cover all Fourier modes up to frequency $2k$ and hence produce the best result shown in Figure~\ref{fig:reference on low-high}(b).

% \begin{figure}[H]
%     \centering
%     \includegraphics[width=0.3\linewidth]{limited aperture/aperture difference/aperture_coverage.png}
%     \caption{Aperture setups of transmission and reflection with aperture window $L = \frac{3}{2}\pi$, and their corresponding coverage in $p$-space.}
%     \label{fig:aperture setup}
% \end{figure}

% \begin{figure}[H]
%     \centering
%     \includegraphics[width=1.0\linewidth]{limited aperture/aperture difference/fig_q_recon_comparison.png}
%     \caption{Left to right: true $q$, reconstruction $\tilde{q}$ with limited aperture in case of transmission, reflection, and with full aperture.}
%     \label{fig:aperture diff}
% \end{figure}

% \begin{figure}[H]
% \centering
% \begin{tabular}{c@{\vspace{2pt}}c@{\vspace{2pt}}cc}
% \hspace{-0.3cm}\includegraphics[width=0.25\linewidth]{limited aperture/aperture difference/fig_q_exact.png}&
% \includegraphics[width=0.25\linewidth]{limited aperture/aperture difference/fig_q_recon_transmission.png}&
% \includegraphics[width=0.25\linewidth]{limited aperture/aperture difference/fig_q_recon_reflection.png} &
% \hspace{-0.3cm}    \includegraphics[width=0.25\linewidth]{limited aperture/aperture difference/fig_q_recon_full_aperture.png}
% \\
% (a) true &(b) transmission &(c) reflection & (d) full aperture
% \end{tabular}
% \caption{Reconstruction results for Example 5.}
% \label{fig:aperture diff}
% \end{figure}

\begin{figure}[H]
\centering
\begin{tabular}{c@{\vspace{2pt}}c}
\includegraphics[width=0.25\linewidth]{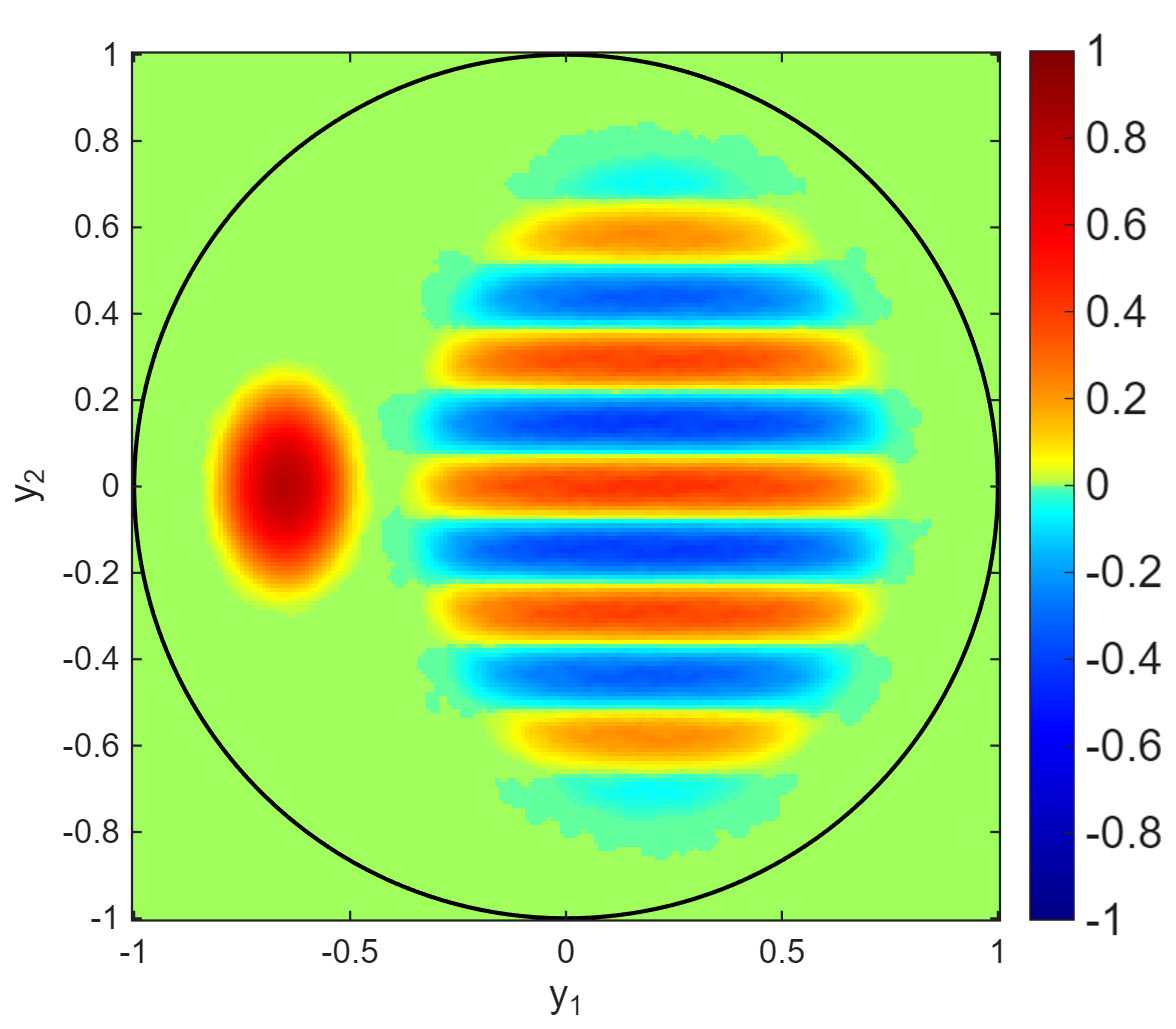}&
\includegraphics[width=0.25\linewidth]{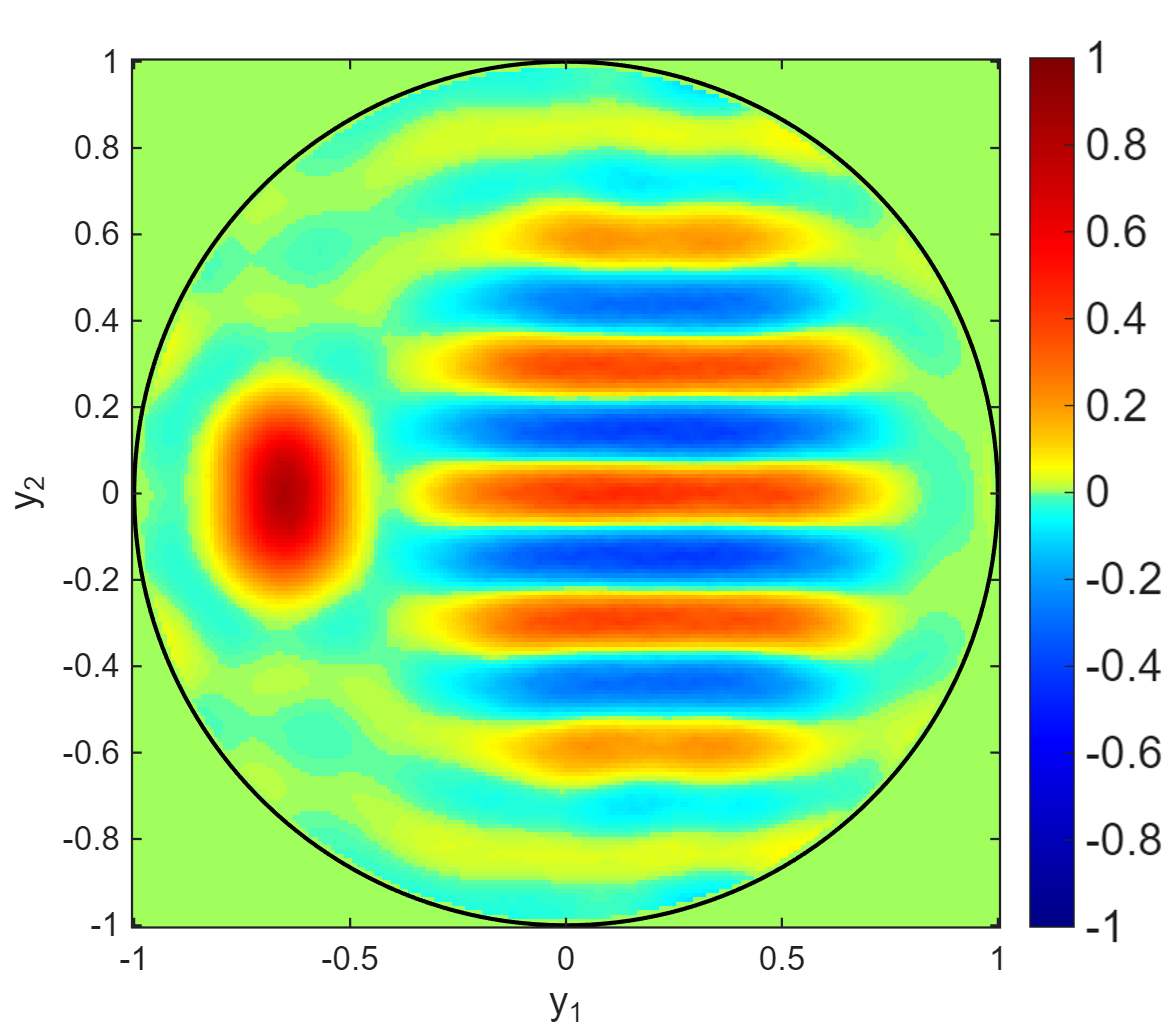}
\\
(a) true $q$ &(b) full aperture
\end{tabular}
\caption{Reference for Example ~\ref{ex:low and high}.}
\label{fig:reference on low-high}
\end{figure}

\begin{figure}[H]
\centering
\begin{tabular}{c@{\vspace{2pt}}c@{\vspace{2pt}}c}
\includegraphics[width=0.25\linewidth]{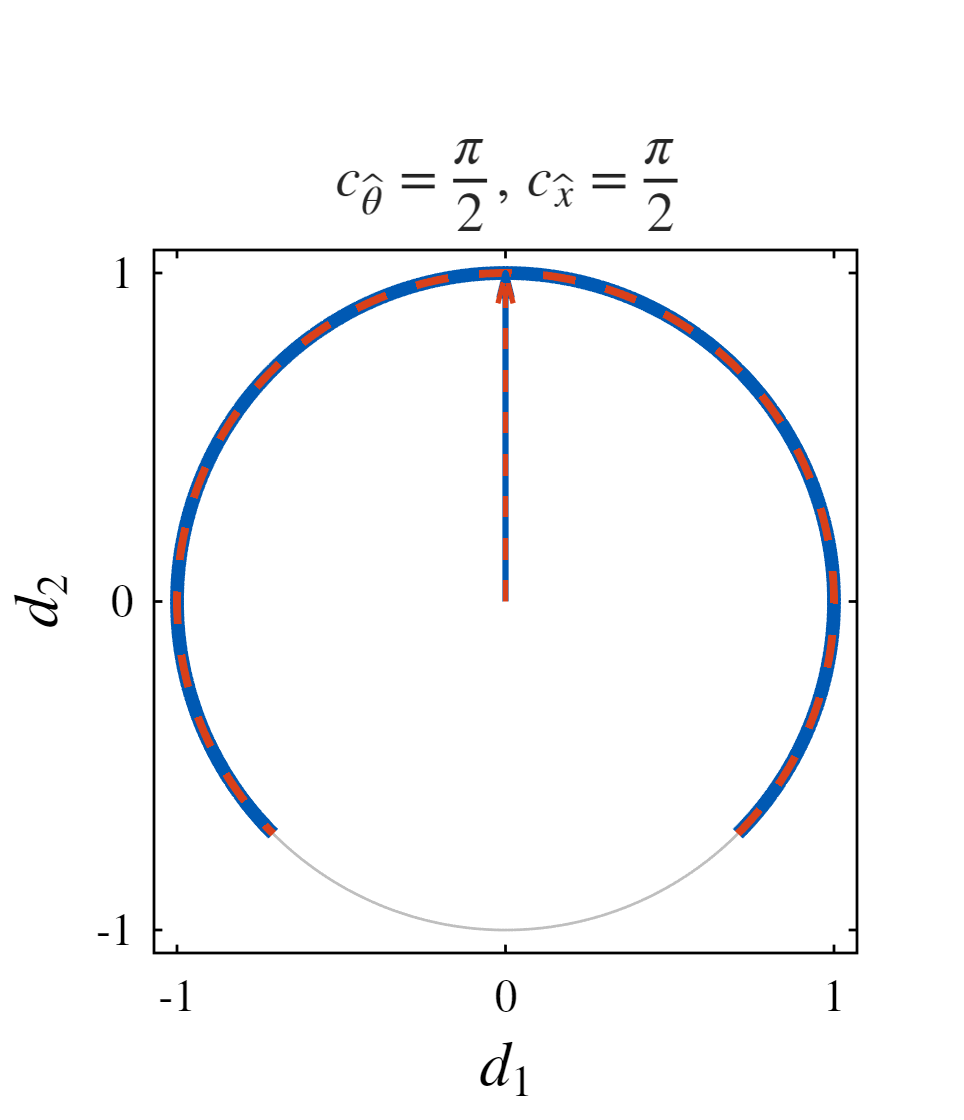}&
\includegraphics[width=0.25\linewidth]{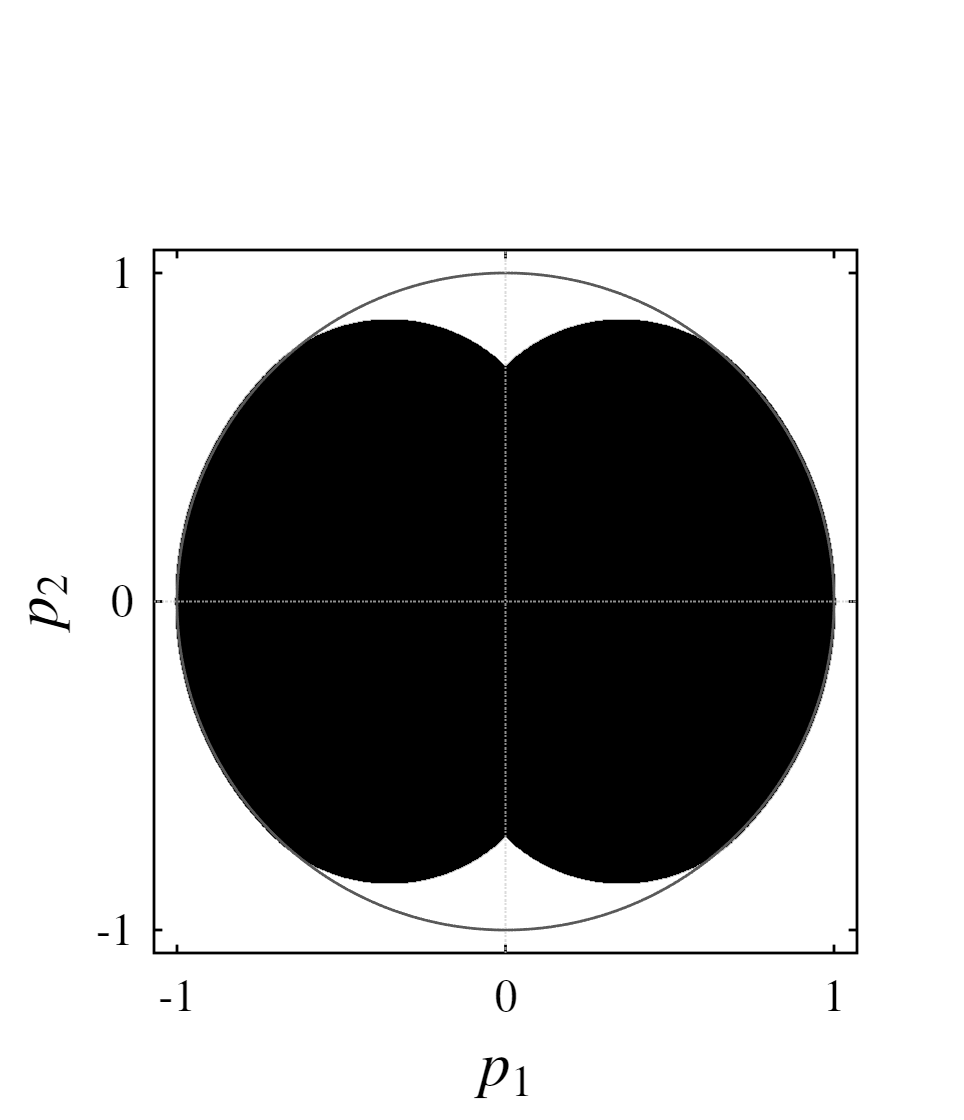}&
\includegraphics[width=0.25\linewidth]{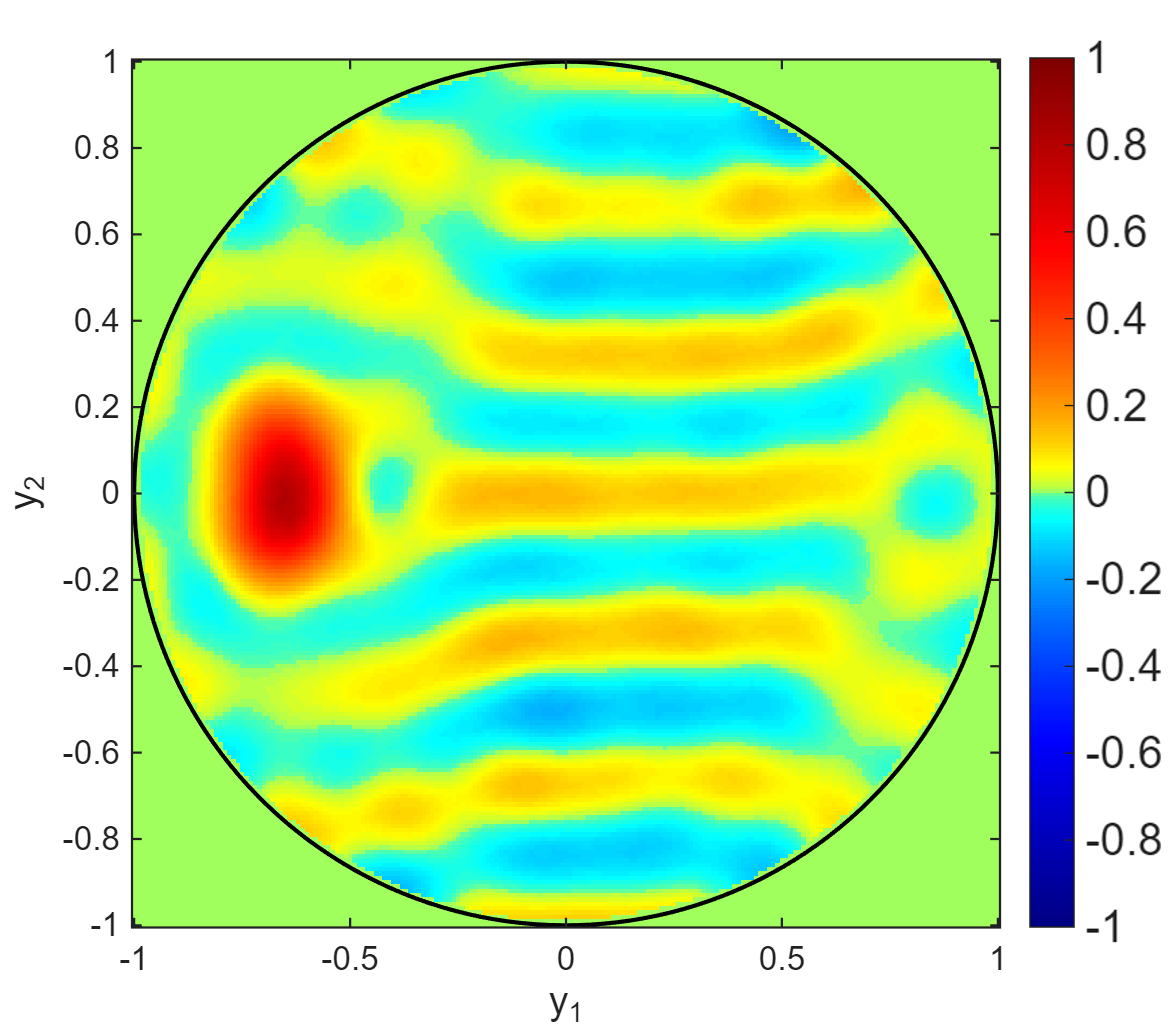}\\
\includegraphics[width=0.25\linewidth]{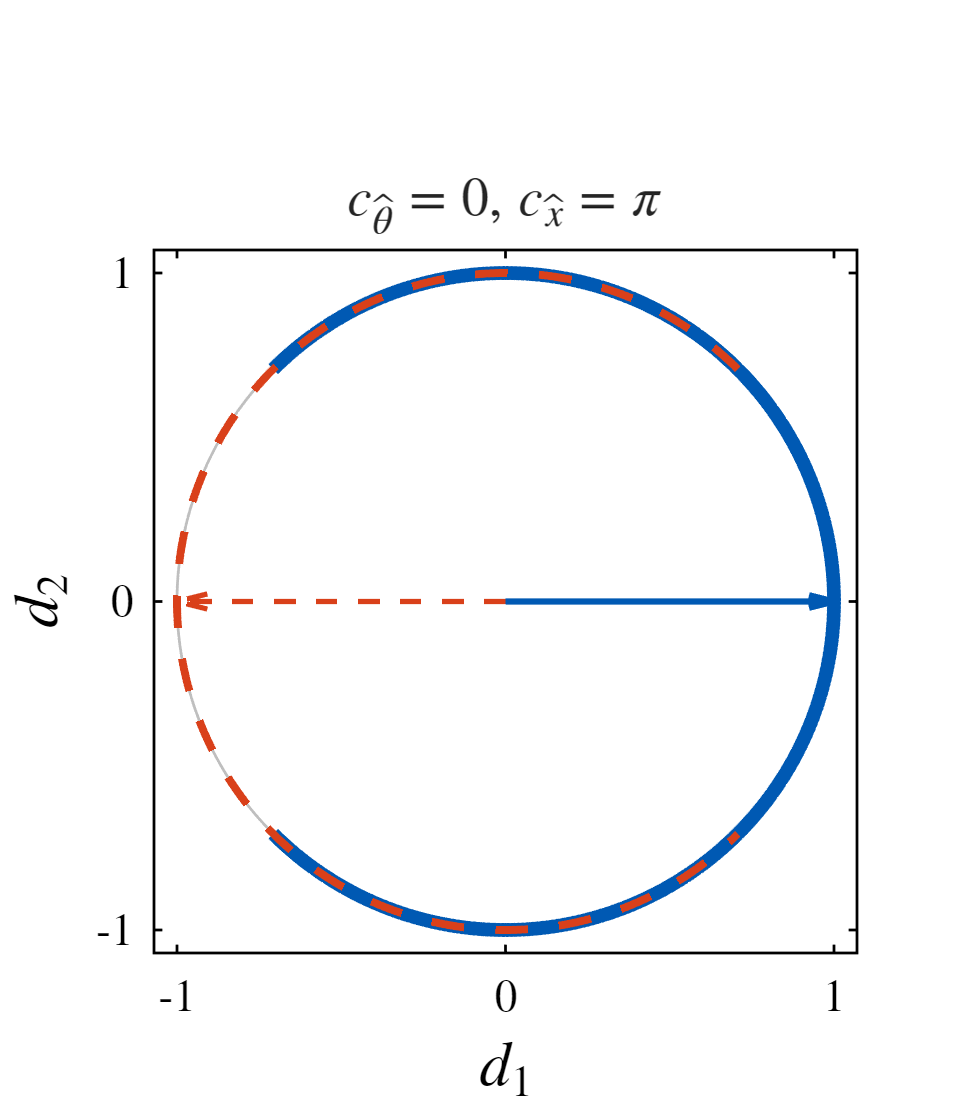}&
\includegraphics[width=0.25\linewidth]{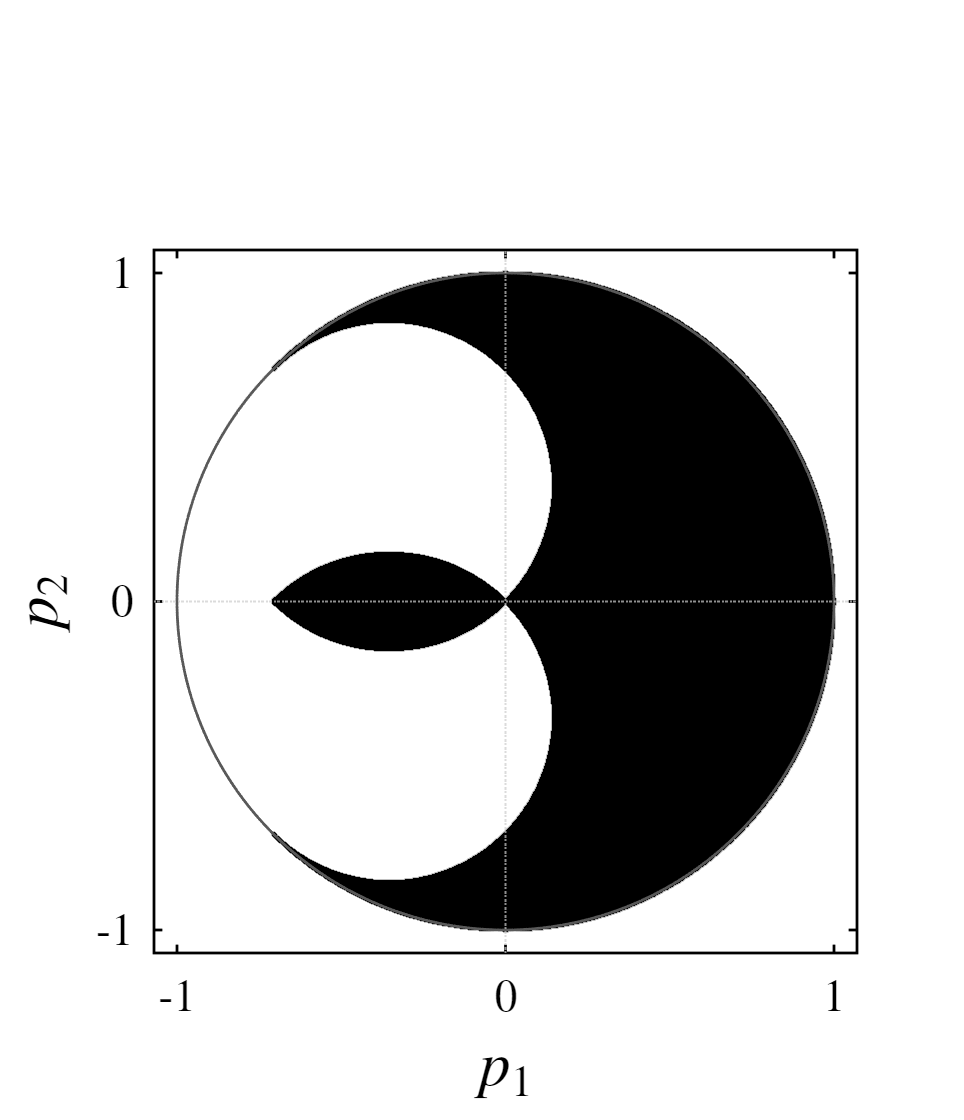}&
\includegraphics[width=0.25\linewidth]{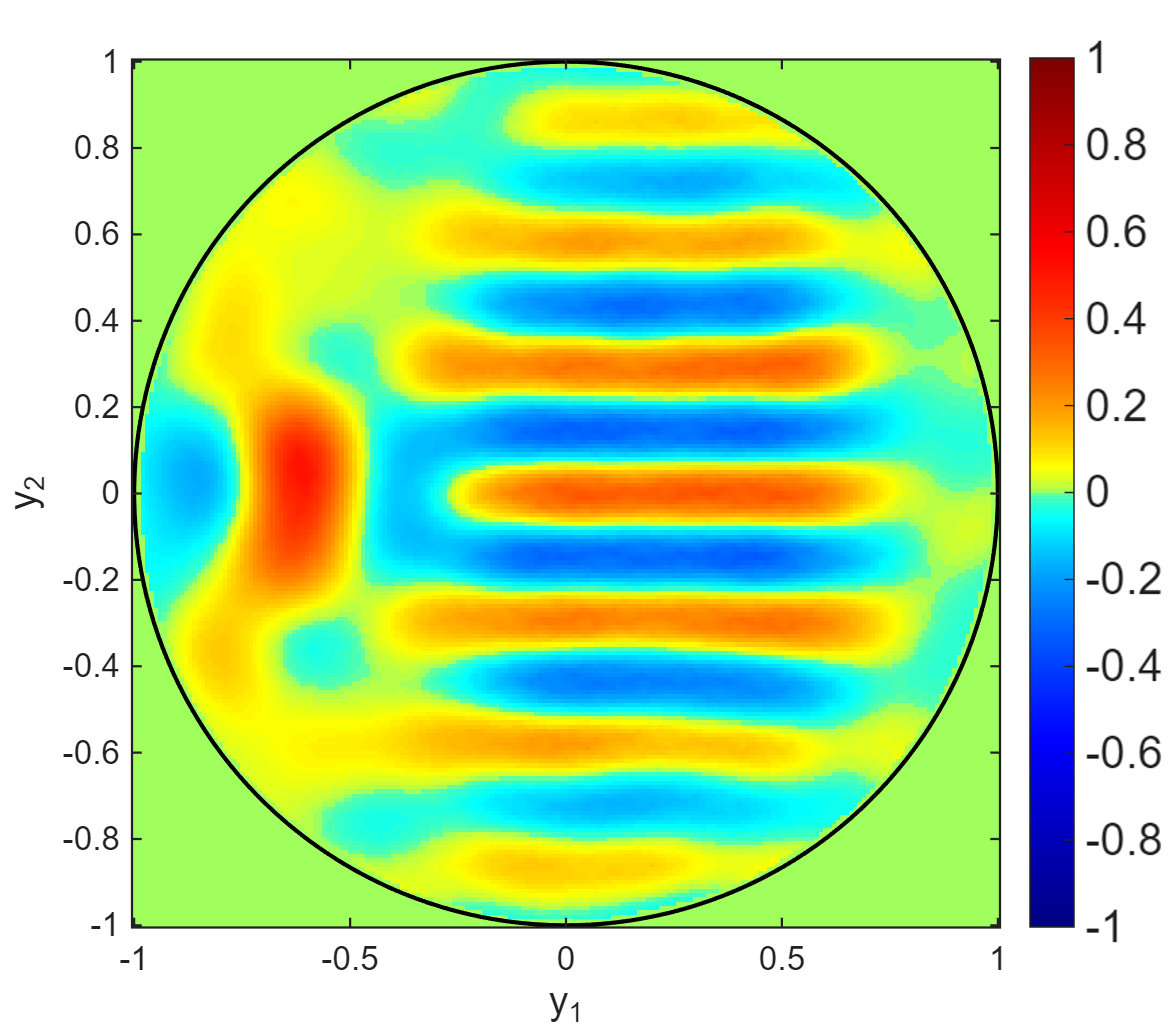}
\\
(a) aperture setup &(b) $p$-space coverage &(c) reconstruction
\end{tabular}
\caption{Reconstruction results for Example ~\ref{ex:low and high}. 1st row: transmission. 2nd row: reflection.}
\label{fig:numerical on low-high}
\end{figure}

\subsection{Discrete Sparse Data}
\label{sec: random data}
We live in a discrete world in practice. Each measurement, e.g., far-field data corresponding to a pairing of an incident direction $\hat{\theta}$ and observation direction $\hat{x}$, is mapped to a point $p=\frac{\hat{\theta}-\hat{x}}{2}$ in the frequency domain. If the measurement is dense enough to resolve the wavelength $\frac{2\pi}{k}$, then one can use the corresponding continuous limit to approximate the effective rank estimate. However, if the measurement is sparse or random, one needs to find a discrete approximation. Here, we propose a practical while empirical/ad hoc way to estimate the effective rank for an arbitrary discrete set of measurements for far-field data. The main idea is to extend the definition of the Shannon number~\eqref{eq:Shannon number} in the continuous setting to the discrete setting. Hence, for any discrete data set (after $p$-transform) $\{\hat{u}^\infty(p_j)\}_{j=1}^N$, one needs to define its numerical measure $\mu_r(\{p_j\}_{j=1}^N)$ in the frequency domain, although its mathematical measure is 0. A natural way is to associate each point $p_j$ with a ball centered at $p_j$ of an appropriate radius $r$. Here, we use the Born approximation for a homogeneous medium~\eqref{eq:J0} as the starting point. We follow the analysis in ~\cite{Shannon_number_estimate} to motivate our argument. The Shannon number is derived in the continuous case~\cite{Shannon_number_estimate} from evaluating the integral  
\begin{equation}\label{eq:trace}
\int_{\Omega} D(x,x)dx, 
\end{equation}
which is the trace of the correlation operator between two Fourier modes $x,x'\in \Omega$ in  $B(0,1)$: \begin{equation}\label{eq:correlation integral}
    D(x,x') = (2\pi)^{-2} \int_{B(0,1)} e^{i2k\xi\cdot (x-x')}d\xi
\end{equation}
Finding the appropriate radius can be formulated as an integral approximation problem: $\forall p\in B(0,1)$, find the largest admissible $r$ such that the correlation integral~\eqref{eq:correlation integral} on a small ball centered at $p_j$, $B(p_j,r)$, can be approximated by the constant phase at $p_j$. In other words, we choose the maximum radius $r = r_{max}$ such that
\begin{equation}\label{eq:3D integral}
\sup_{x,x' \in \Omega}|(2\pi)^{-2} \int_{B(p_j,r)} e^{i2k\xi\cdot (x-x')} d\xi - (2\pi)^{-2}  e^{i2kp_j \cdot (x-x')} \pi r^2|
\end{equation}
is small enough so that the trace integral~\eqref{eq:trace} can be approximated well.
%it is crucial to connect the continuous case and the discrete case in the need to evaluate an area of the frequency domain while $\mu(\{p_j\}_{j=1}^N) = 0$. Hence, a density argument connecting $\{p_j\}_{j=1}^N$ in the frequency domain to a positive measure needs to be made, i.e., $\{p_j\}_{j=1}^N$ is only dense in $\mathcal{K}_1$ for some $\mathcal{K}_1 \subset B(0,1)$.
% A sparse set of data implies a lack of information in $p$-space, where we need to evaluate the area of frequency domain for Shannon number estimate.
% We could translate this argument into a quantitative question: How big of an area in the frequency domain can a single data point represent? For simplicity, we consider the homogeneous case and refer to \cite{Shannon_number_estimate} for analysis. The Shannon number is derived from the continuous case by evaluating the integral
% \[
% \int_{\Omega} D(x,x)dx, 
% \]
% where $D(x,x') = (2\pi)^{-2} \int_{B(0,1)} e^{i2k\xi\cdot (x-x')}d\xi$. For the discrete case, we find $\mathcal{K}_1 \subseteq B(0,1)$ in which $\{p_j\}_{j=1}^N$ is dense. This leads to an integral approximation problem, i.e., for every $p_j$ and a tolerance $\epsilon_3$, we find the largest admissible radius $r$ such that
% \[
% \sup_{x,x' \in B(0,1)}|(2\pi)^{-2} \int_{B(p_j,r)} e^{i2k\xi\cdot (x-x')} d\xi - (2\pi)^{-2}  e^{i2kp_j \cdot (x-x')} \pi r^2| \leq \epsilon_3.
% \]
% In 2D, the integrand in ~\eqref{eq:3D integral} is replaced by $|1 - \frac{4kr}{2J_1(4kr)}|$ with a proper scaling. 
Once $r_{max}$ is chosen, we define $\mathcal{K}\subset B(0,1)$ by
\begin{equation}\label{eq:discrete}
\mathcal{K}
=
\left(\bigcup_{j=1}^{N} B(p_j,\, r_{\max})\right)\cap B(0,1)
\end{equation}
as the domain covered by the discrete measurement data and set $\mu_r(\{p_j\}_{j=1}^N)=\mu(\mathcal{K})$. We apply Green's theorem and refer to \cite{p-area} for computing the area of $\mathcal{K}$. Here we remark that $r_{max}$ imposes a local quadrature approximation rather than decorrelation across different modes. 
Let $\phi(p) = e^{i 2 k p\cdot (x-x')}$ and $\phi(p_0) = e^{i 2 k p_0\cdot (x-x')}$. From the approximation point of view,
\[
|\phi(p) - \phi(p_0)| \leq 4k|p-p_0|, \quad \forall x,x'\in B(0,1),
\]
which implies $r_{max} \sim O(1/k)$. 

For a random discrete measurement, it is difficult to carry out the spectral analysis as in ~\cite{Eigenvalue_distribution,hughes2026wavepacketseigenvalueestimates}, which is only applicable to $\mathcal{K}$ with a regular shape, e.g., an ellipse or a cube. Numerical experiments suggest that an empirical $r_{max} = \frac{1}{\sqrt{2}k}$ in~\eqref{eq:discrete} provides a good effective rank estimate based on the Shannon number formula~\eqref{eq:Shannon number}. First, we use discrete data to verify the effective rank estimate by choosing a disk in $p$-space with a different radius $R$, in which case the continuous Shannon number can be computed easily. We use a Fibonacci grid to sample on each disk (the selected points in Figure ~\ref{fig:effective rank} correspond to the Fibonacci grid in 2D) such that the minimum number of points is needed to cover the disk with $r_{max} = \frac{1}{\sqrt{2}k}$. The spectrum and the dashed line indicating the effective rank estimate using the discrete data in Figure ~\ref{fig:effective rank} show that the empirical $r$ is a good choice.
% It is important to note that when $\{p_j\}_{j=1}^N$ are randomly distributed in $B(0,1)$, our rank estimate above is conservative in that it allows retaining more modes without jeopardizing stability.
\begin{figure}[htbp]
\centering
\begin{tabular}{c@{\vspace{2pt}}c@{\vspace{2pt}}c@{\vspace{2pt}}c}
\includegraphics[width=0.2\linewidth]{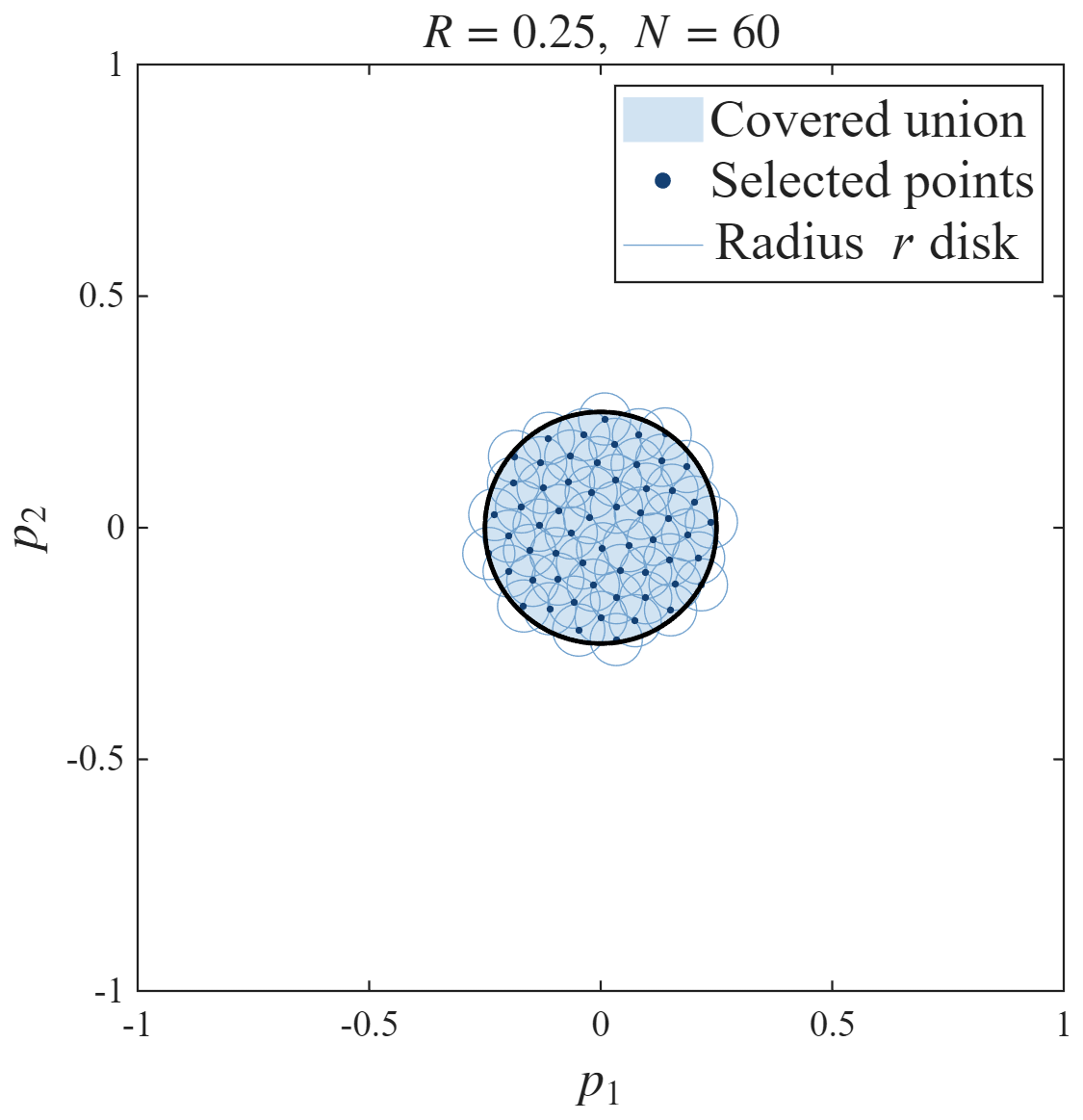}&
\includegraphics[width=0.2\linewidth]{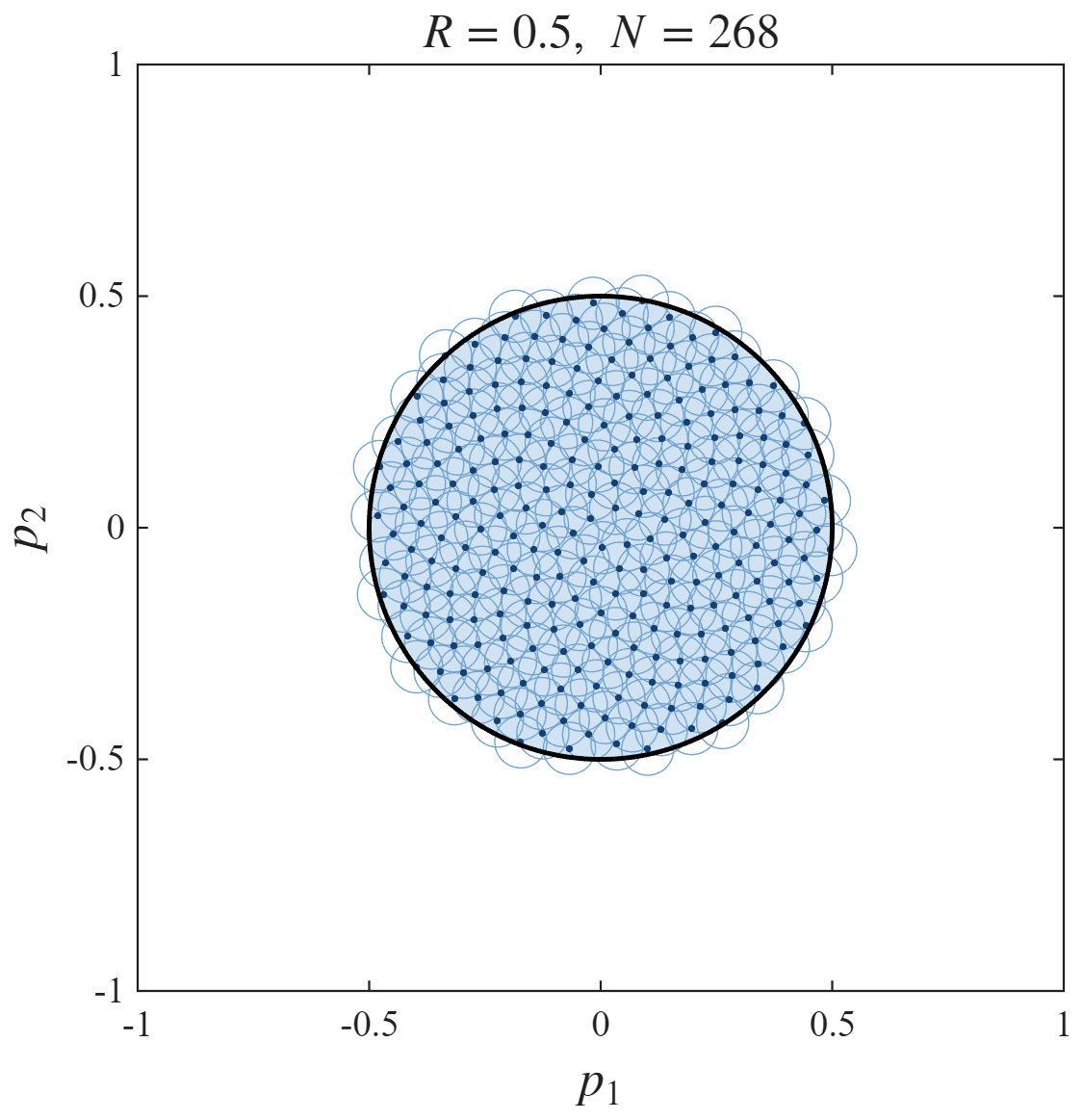}&
\includegraphics[width=0.2\linewidth]{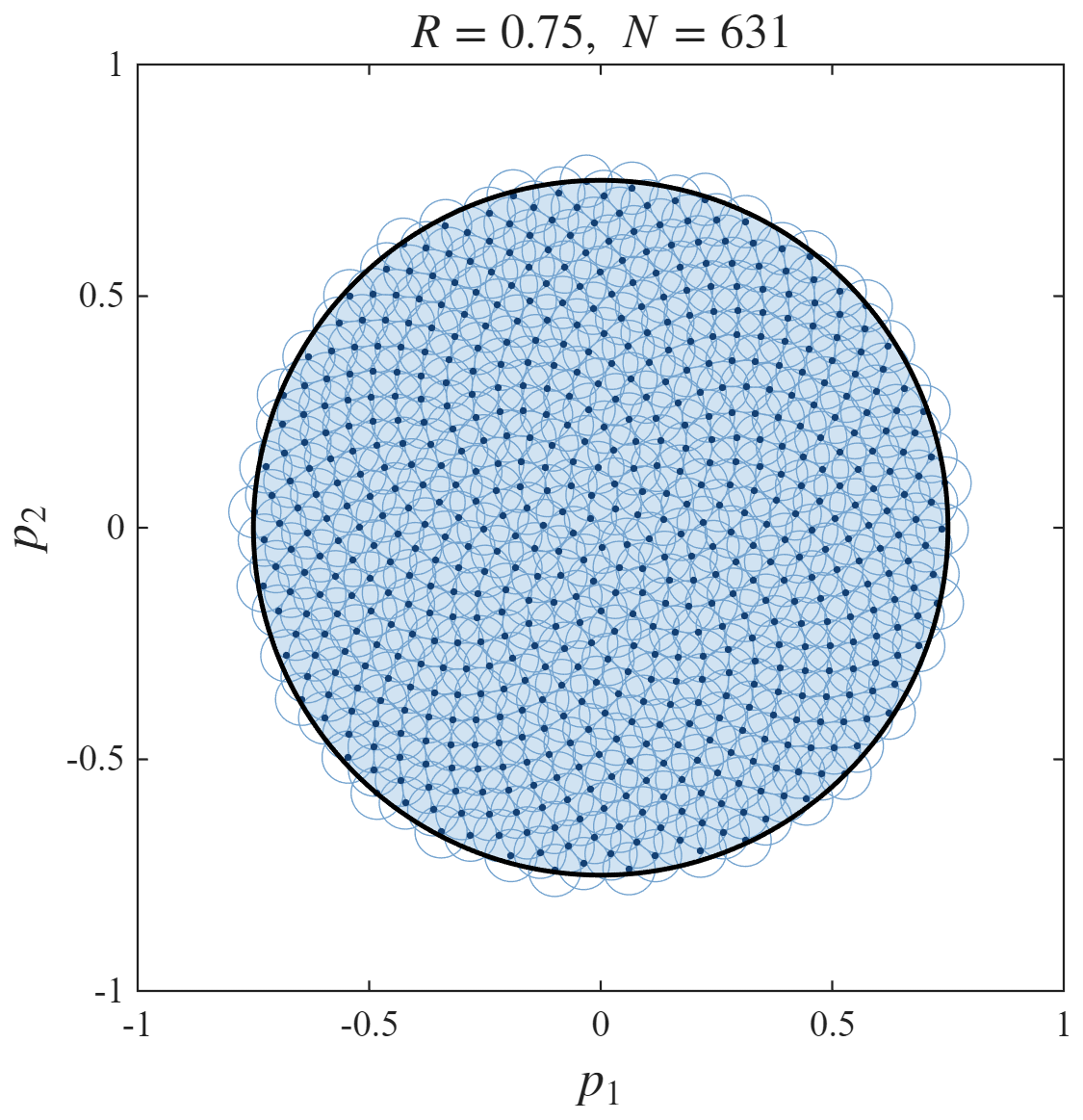}&
\includegraphics[width=0.2\linewidth]{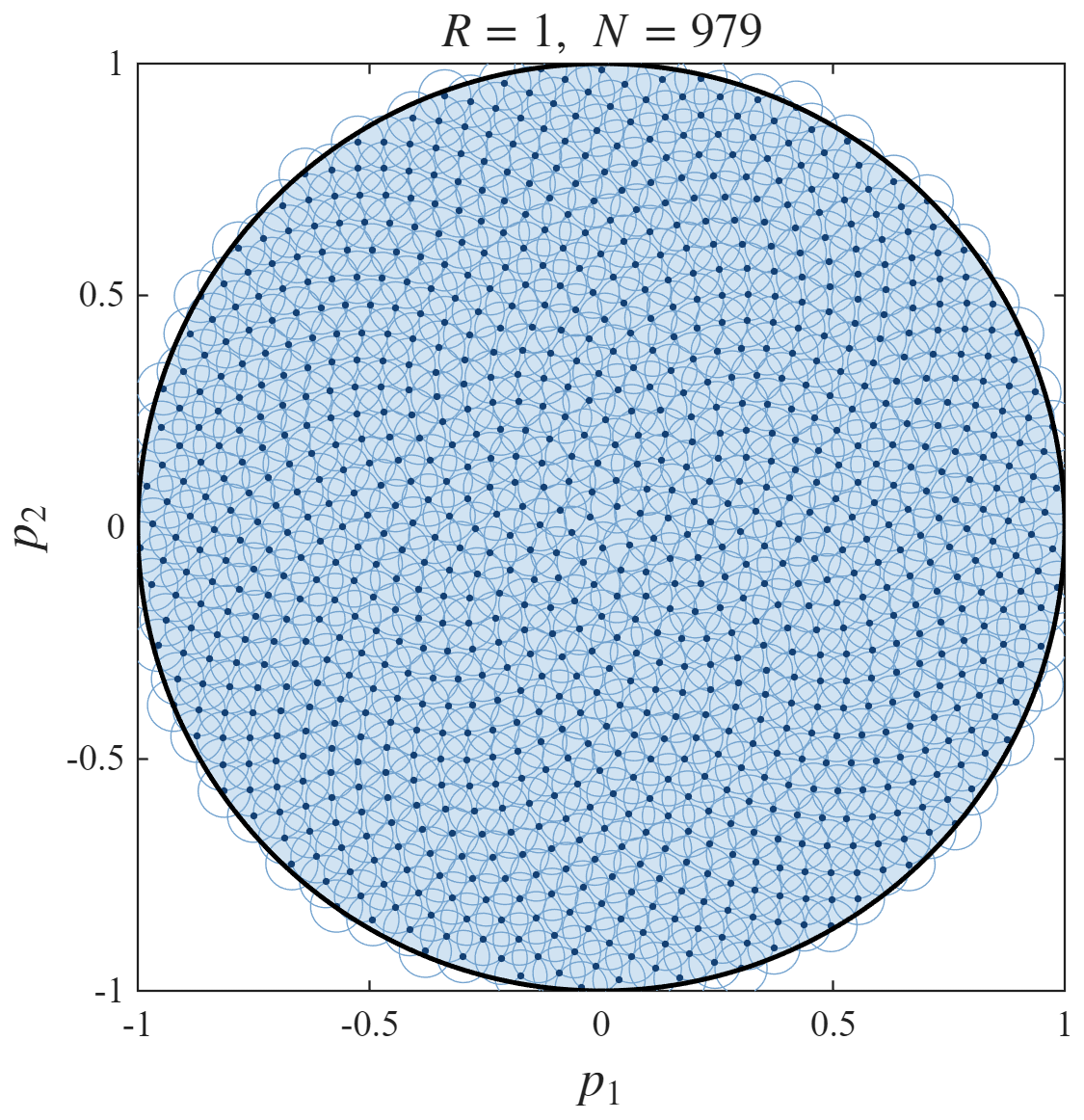}\\
\includegraphics[width=0.2\linewidth]{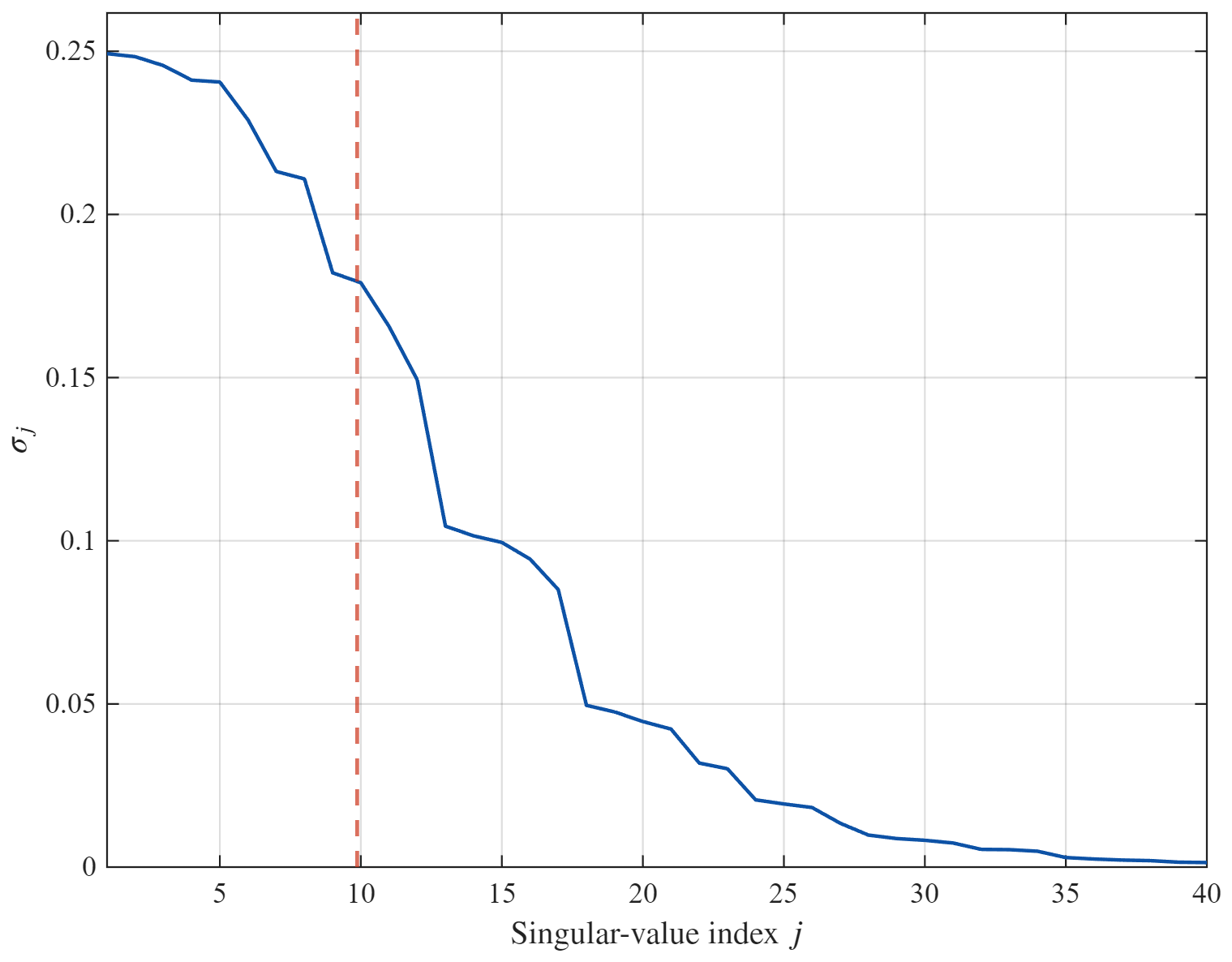}&
\includegraphics[width=0.2\linewidth]{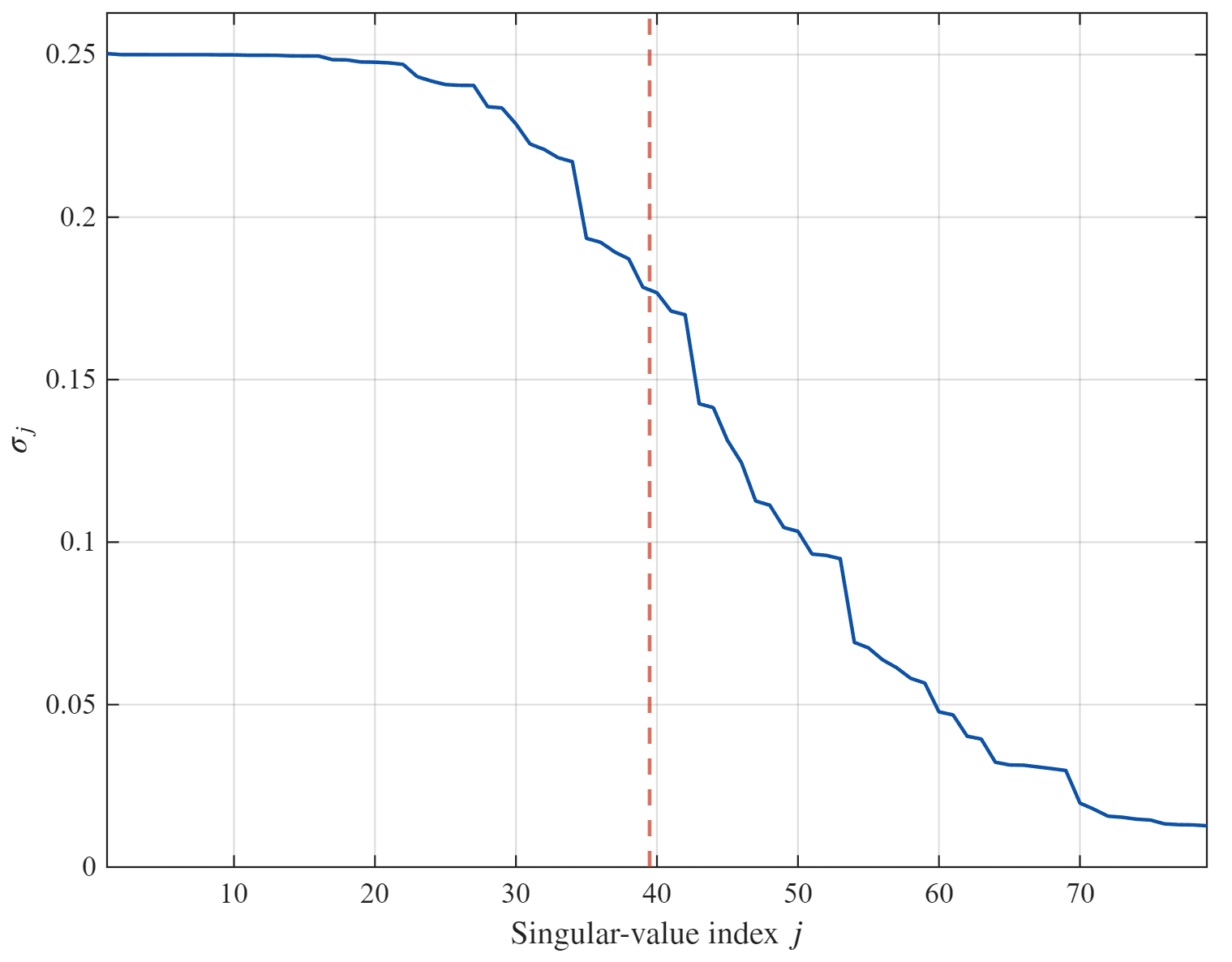}&
\includegraphics[width=0.2\linewidth]{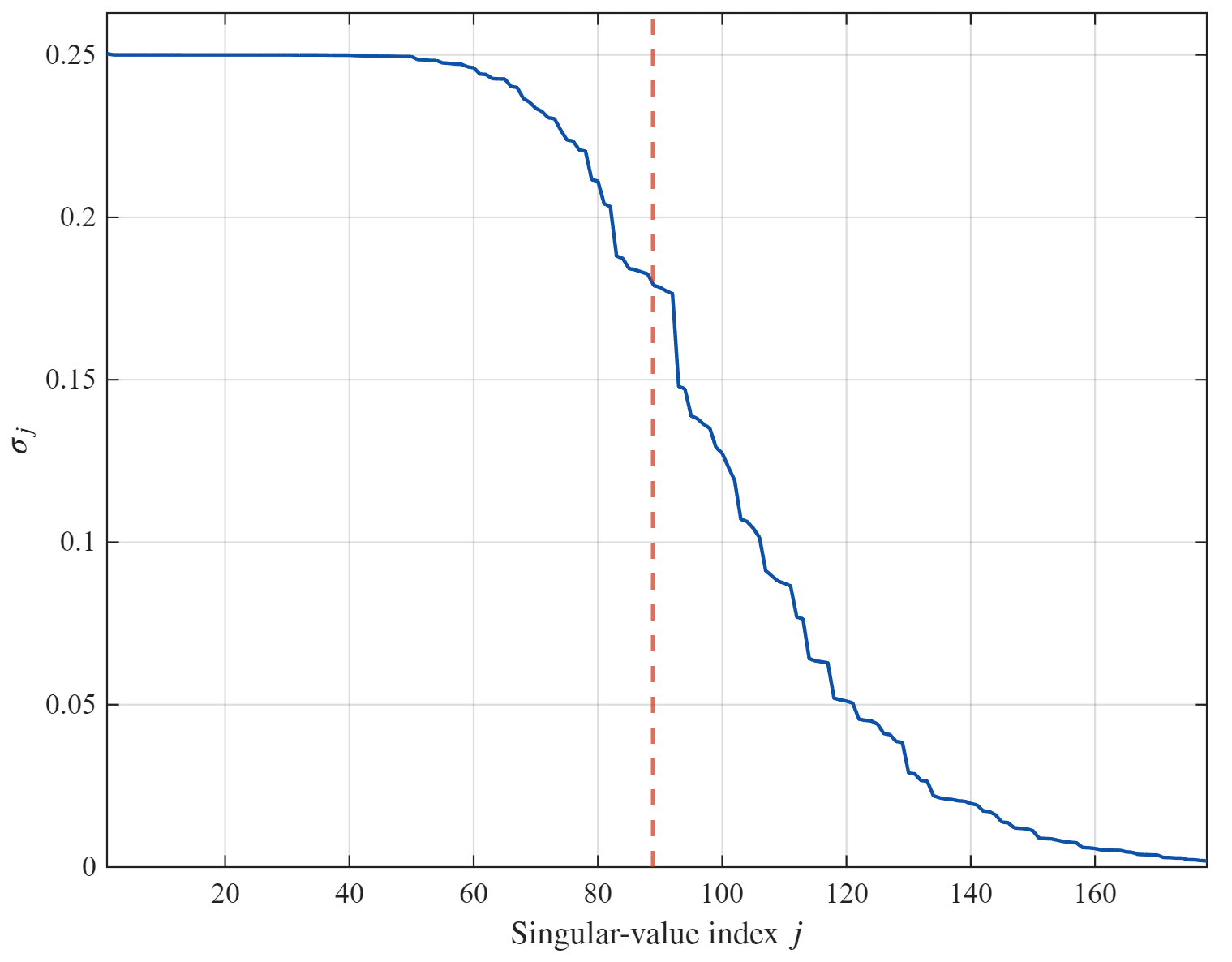}&
\includegraphics[width=0.2\linewidth]{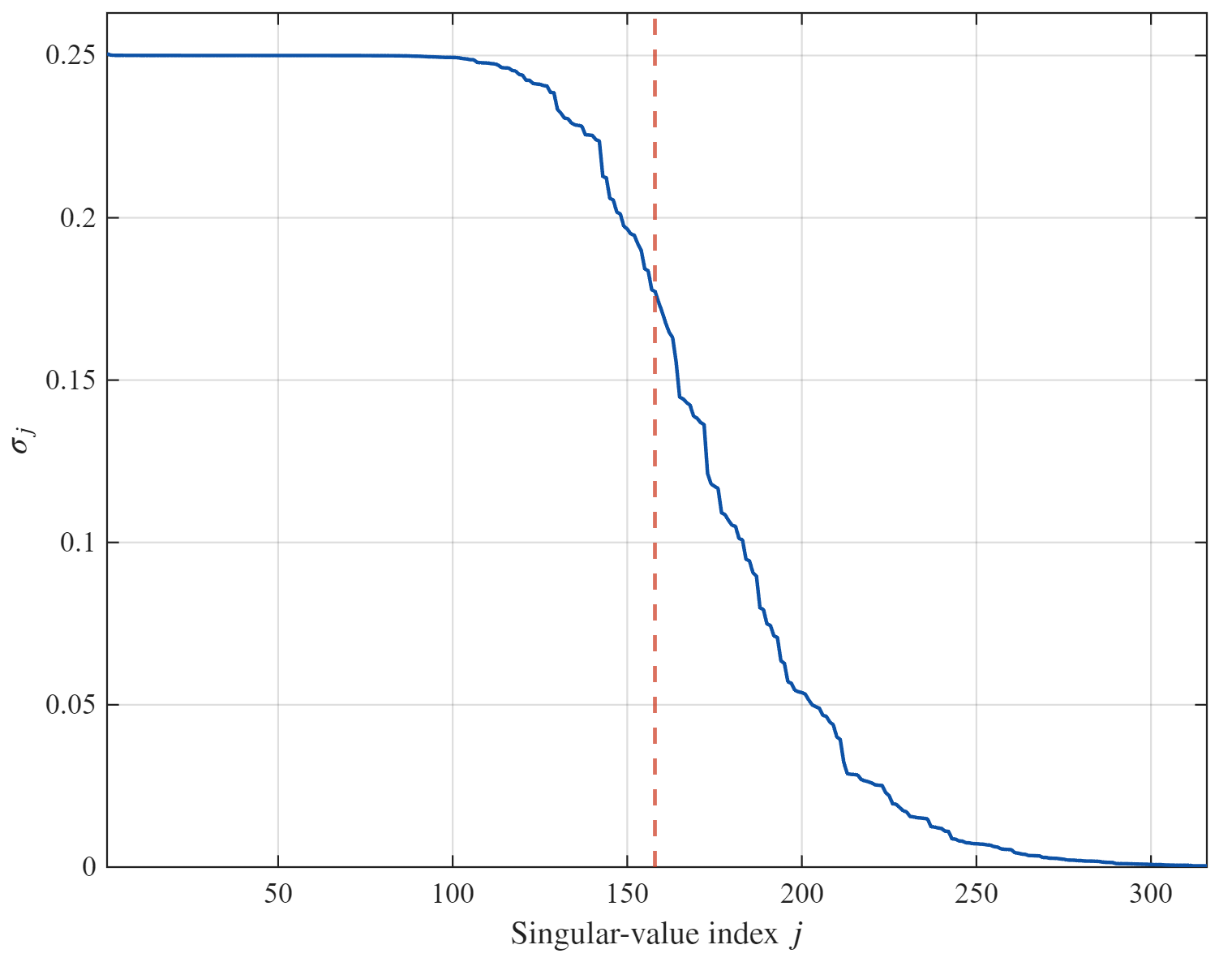}
\\
(a) $R=0.25$ &(b) $R=0.5$ &(c) $R=0.75$&(d) $R=1$
\end{tabular}
\caption{1st row: union coverage with empirical $r$. 2nd row: spectrum and rank estimate.}
\label{fig:effective rank}
\end{figure}

Next, we use the empirical $r_{max}$ in~\eqref{eq:discrete} and the Shannon number formula~\eqref{eq:Shannon number} to estimate the effective rank for random measurement data. We choose a different number of incident directions and observation directions $M$ randomly and independently from the full aperture set with 100 measurements to alter the sparsity of the data, which changes the spectral behavior of the linearized operator and does not correspond to any continuous setup.
The randomness of the data is controlled by a fixed random seed in MATLAB for reproducibility. We still consider Example ~\ref{ex:multi_gaussian} and compare $e_2$ with different $M$, see Figure \ref{fig:e_2 for different aperture M}(a). The stability of the algorithm is still observed. It also shows that with the same number of measurements, an even distribution of incident directions and observation directions gives better reconstruction than a random distribution, since an even distribution gives the largest subset $\mathcal{K}$ which covers more frequency components. As $\mathcal{K} \rightarrow B(0,1)$, data enriched in the frequency domain gives a better reconstruction.
\begin{remark} \label{remark: limited general}
    The limited aperture case can be considered as a special case of sparse random data.
\end{remark}

% \begin{figure}[H]
%     \centering   \includegraphics[width=0.5\linewidth]{random sparse data/fig_error_vs_M.png}
%     \caption{$e_2$ for different $M$.}
%     \label{fig:e_2 for different aperture M}
% \end{figure}

\begin{figure}[htbp]
\centering
\begin{tabular}{c@{\vspace{2pt}}c}
\includegraphics[width=0.4\linewidth]{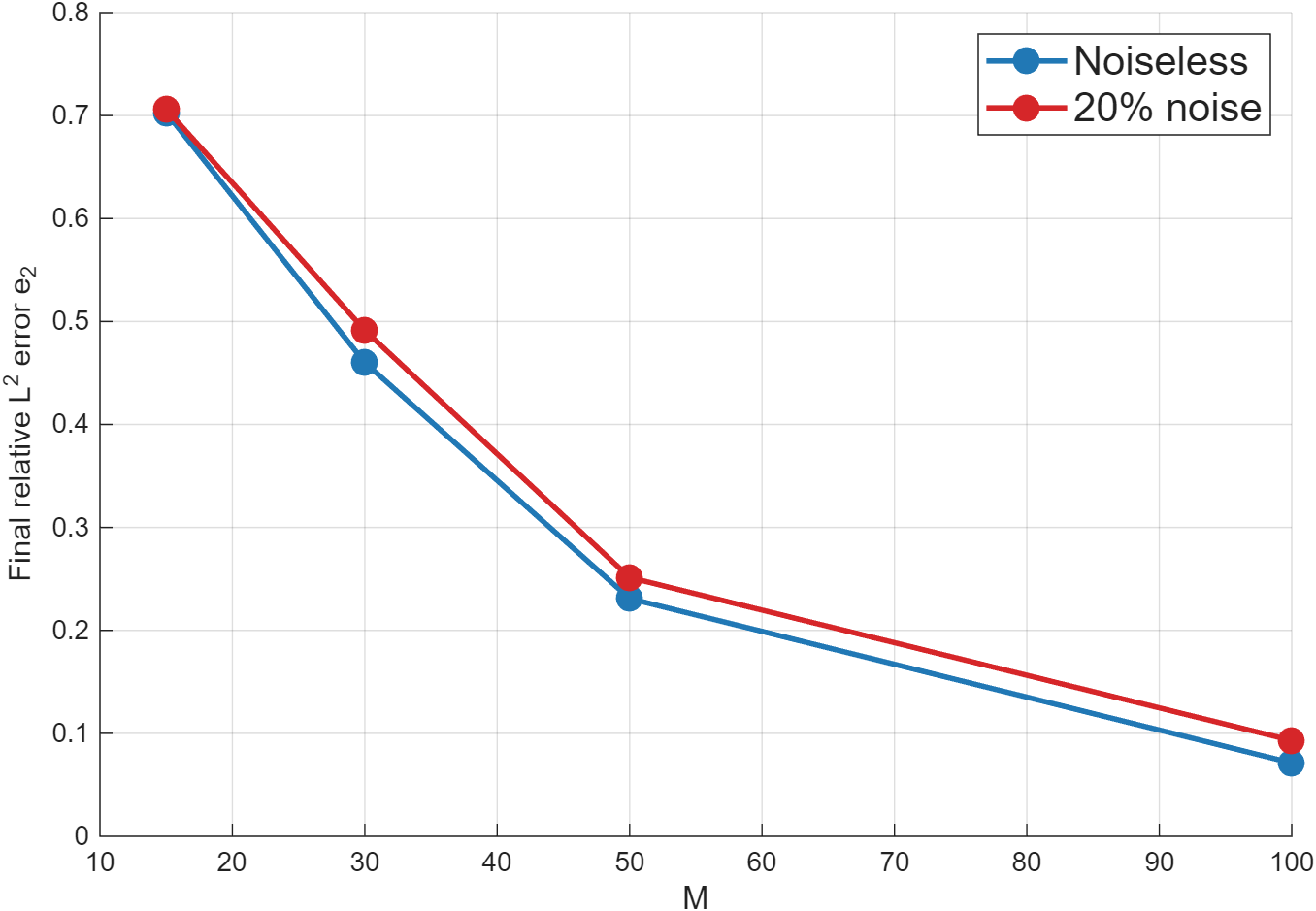}&
\includegraphics[width=0.4\linewidth]{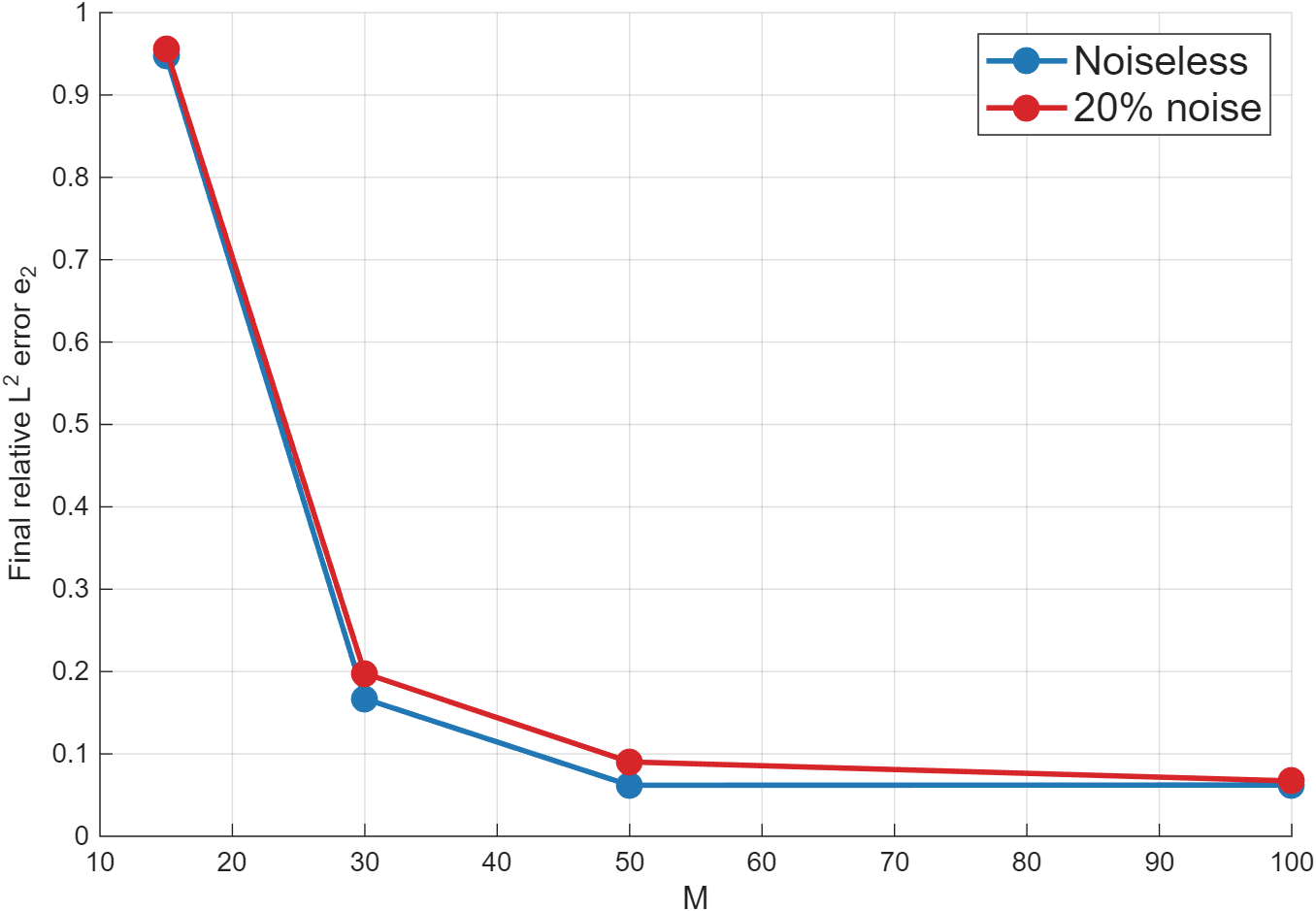}
\\
(a) random distribution &(b) even distribution
\end{tabular}
\caption{$e_2$ v.s. $M$ with different distributions}
\label{fig:e_2 for different aperture M}
\end{figure}

% \subsection{Near Field}
% \zhfn{Put this as future work since we only have experiments in this subsection?}
% We also applied the algorithm to the near-field case. A key difference is to find an optimal cutoff, in which case there is no reference for justification to separate the signal space and the noise space. In particular, we need to minimize the Rayleigh quotient with respect to the homogeneous medium case. This could be hard since the geometry of the boundary where the measurement is taken could be complex such that there is no simple parametrization. Therefore, the numerical experiment in this section mainly focuses on the capacity of the algorithm. For simplicity, we consider the following geometry setup. Let $\partial B = \partial B(0,2)$ be the boundary domain for measurement, then the operator mapping from contrast to data is given by
% \[\mathcal{F}_{near}f = \int_{B(0,1)} G_k(2\hat{x},y)e^{iky\cdot \hat{\theta}}f(y)dy,\]
% where $\hat{x},\hat{\theta} \in \mathbb{S}$. Here, we choose the spectral cutoff based on the estimate from far-field case and regularization $\varepsilon$ according to the noise, i.e.
% \[
% S_{near} = \min\{S,\#\{\sigma(F_{near}):\sigma(F_{near}) > \varepsilon \} \}.
% \]

\subsection{Frequency Continuation} \label{sec: numerical on frequency continuation}
In this section, we test iTSVD on frequency continuation as discussed in Section ~\ref{sec: analysis}. The idea is to start the reconstruction from a relatively lower-frequency reconstruction (hence weaker nonlinearity as well as faster and more stable) on a coarse scale. Then use the reconstruction as a better initial guess for the next higher frequency. This tests the issue in Example ~\ref{ex: fine feature} mentioned in Section ~\ref{sec: noisy data}. We modify the amplitude of Example ~\ref{ex: fine feature} from $0.7$ to $1.0$, and use multiple frequency data $k = 2\pi,\,4\pi,\,6\pi$ and $8\pi$ consecutively for the final reconstruction. We also use multiple frequency data $k = 4\pi,\,8\pi$ for comparison to show that Algorithm ~\ref{alg:iterative-low-rank} can still achieve good reconstruction with a larger frequency leap. Figure~\ref{fig:Hierarchical-convergence} plots the true $q$ and shows the convergence for the frequency continuation approach with 2 different frequency sequences. Only for this example, we adopt the Hierarchical Poincaré-Steklov (HPS) solver of \cite{HPS_solver} as the forward solver to obtain a higher solution accuracy since a higher frequency is used. Figure ~\ref{fig:multi_k_hier} gives the reconstruction at each frequency, which takes $11$ and $12$ iterations in total for the noiseless and noisy cases, respectively. The result also verifies that higher-frequency data are needed to resolve finer features.

% \begin{figure}[H]
% \centering
%     \begin{subfigure}{0.3\linewidth}
%     \captionsetup{labelformat=empty}  % Remove the alphabetic label
%         \includegraphics[width=\linewidth]{frequency continuation/hier/fig_q_exact.png}
%         % \caption{exact $q$}
%     \end{subfigure}
%     \begin{subfigure}{0.45\linewidth}
%     \captionsetup{labelformat=empty}  % Remove the alphabetic label
%         \includegraphics[width=\linewidth]{frequency continuation/hier/fig_convergence.png}
%         % \caption{noiseless}
%     \end{subfigure}
%     \caption{From left to right: true $q$, convergence of $e_2$; the dash line indicates the change of frequency.}
%     \label{fig:Hierarchical-convergence}
% \end{figure}

\begin{figure}[htbp]
\centering
\begin{tabular}{c@{\vspace{2pt}}c@{\vspace{2pt}}c}
\includegraphics[width=0.25\linewidth]{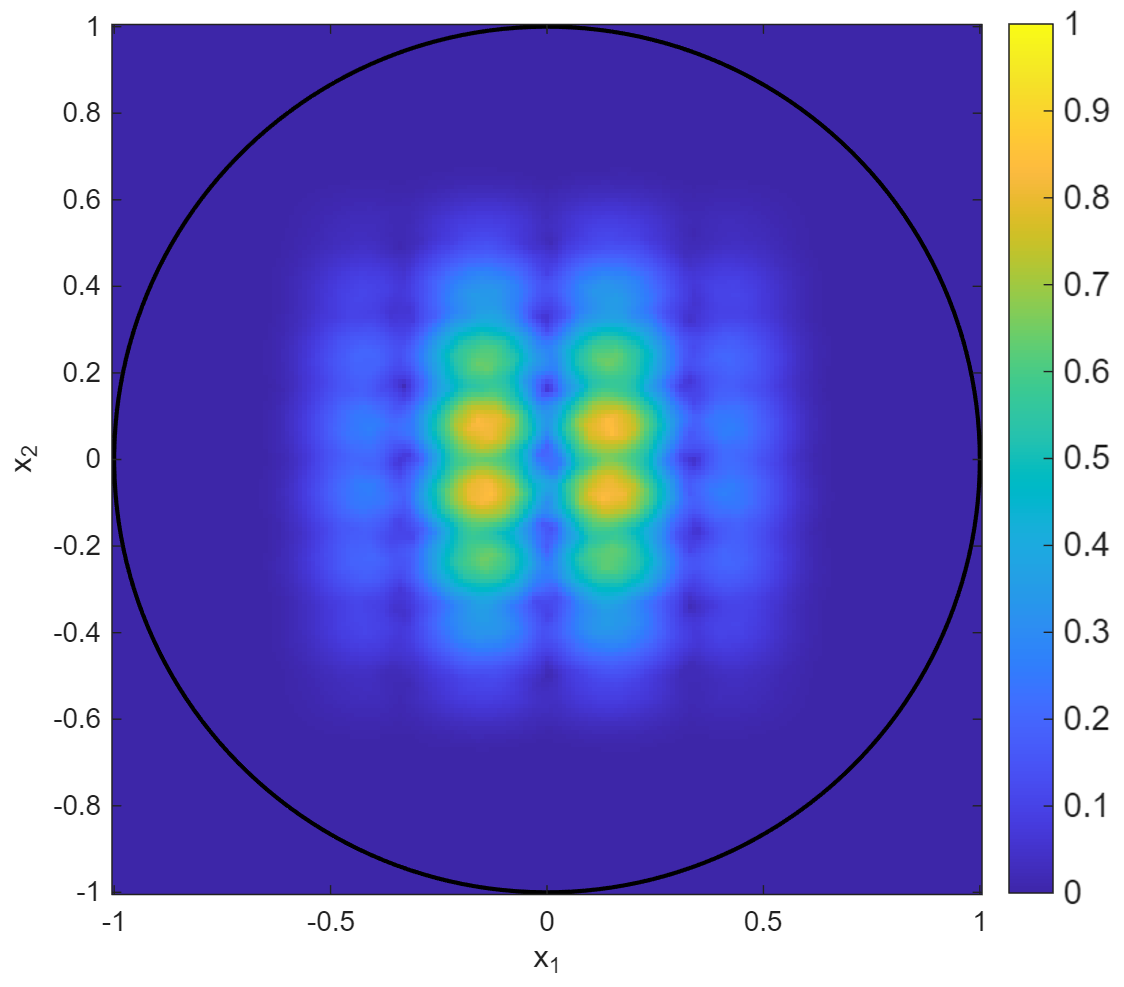}&
\includegraphics[width=0.35\linewidth]{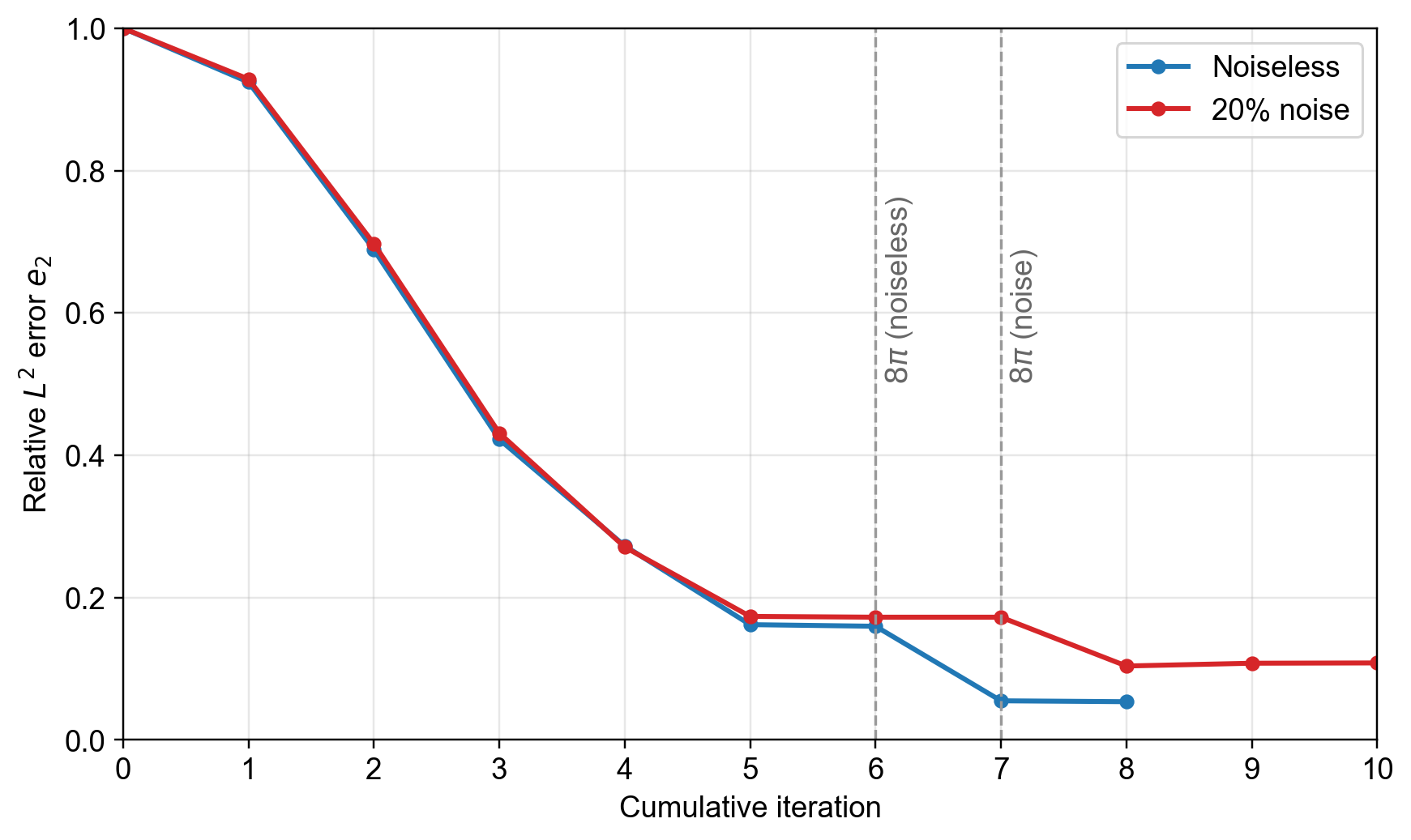}&
\includegraphics[width=0.35\linewidth]{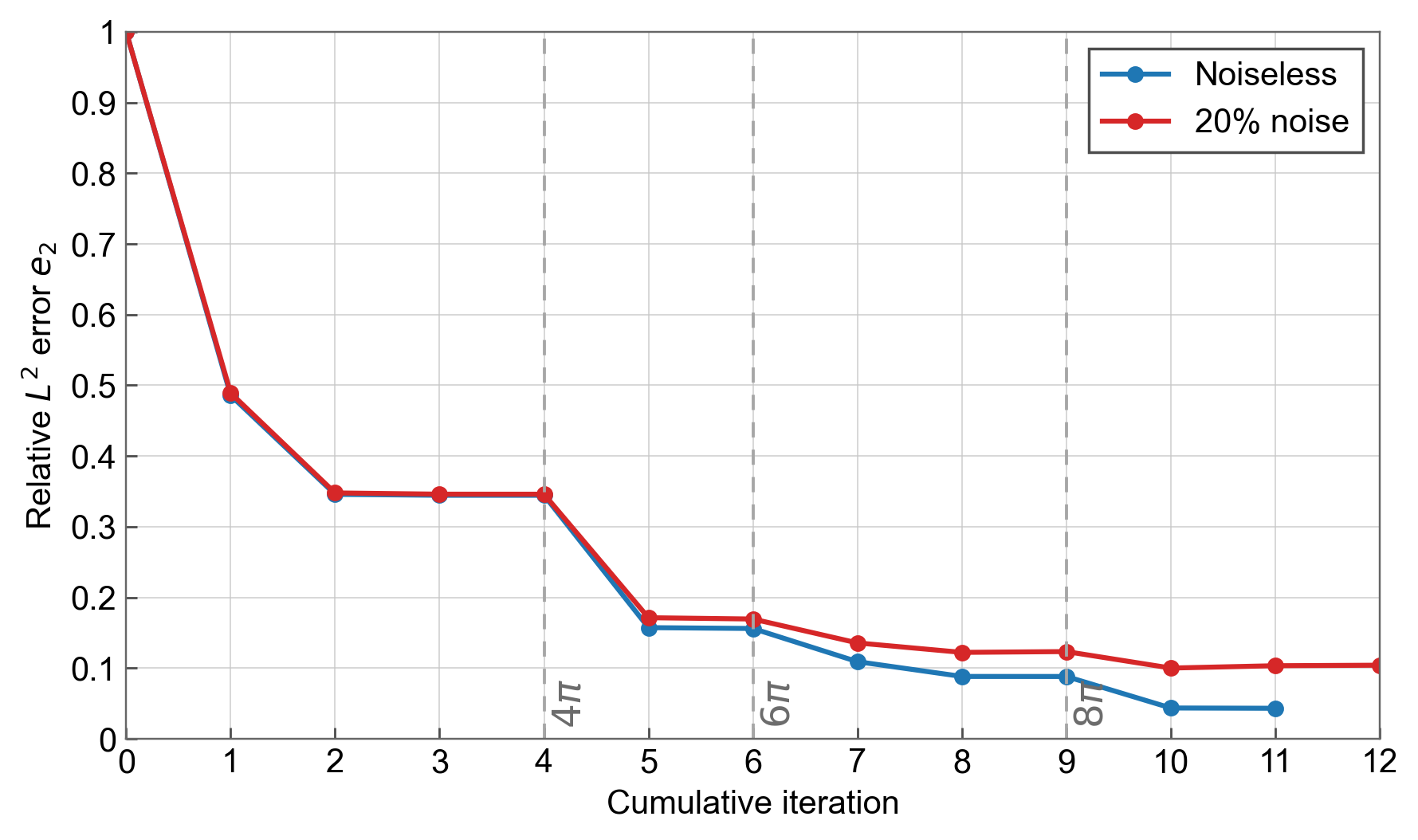}
\\
(a) true $q$ &(b) convergence with $k = 4\pi,\,8\pi$ &(c) convergence with $k = 2\pi,\,4\pi,\,6\pi,\,8\pi$
\end{tabular}
\caption{Reference and Convergence}
\label{fig:Hierarchical-convergence}
\end{figure}

\begin{figure}[htbp]
\centering
\begin{tabular}{c@{\vspace{2pt}}c@{\vspace{2pt}}c@{\vspace{2pt}}c}
\includegraphics[width=0.25\linewidth]{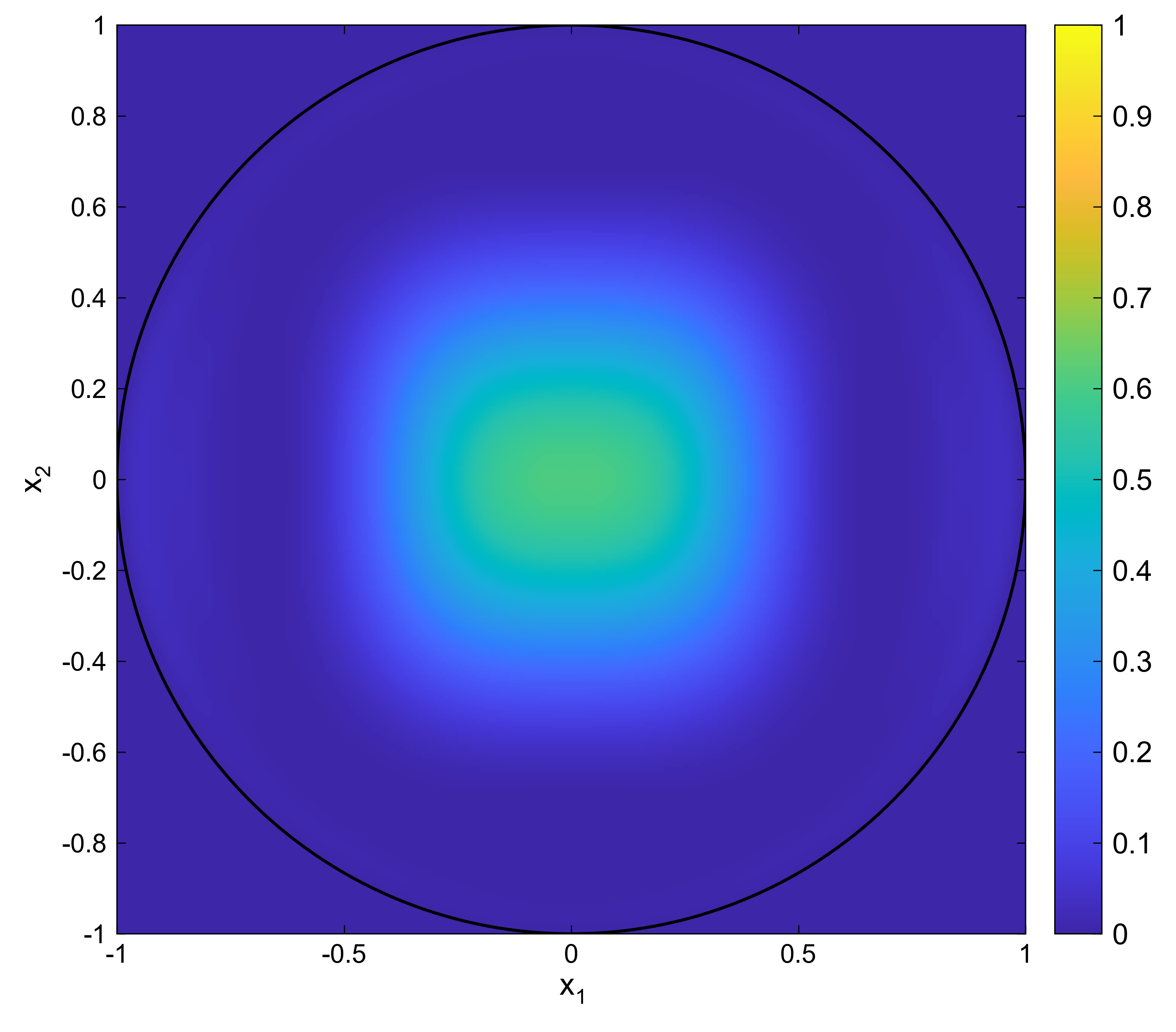}&
\includegraphics[width=0.25\linewidth]{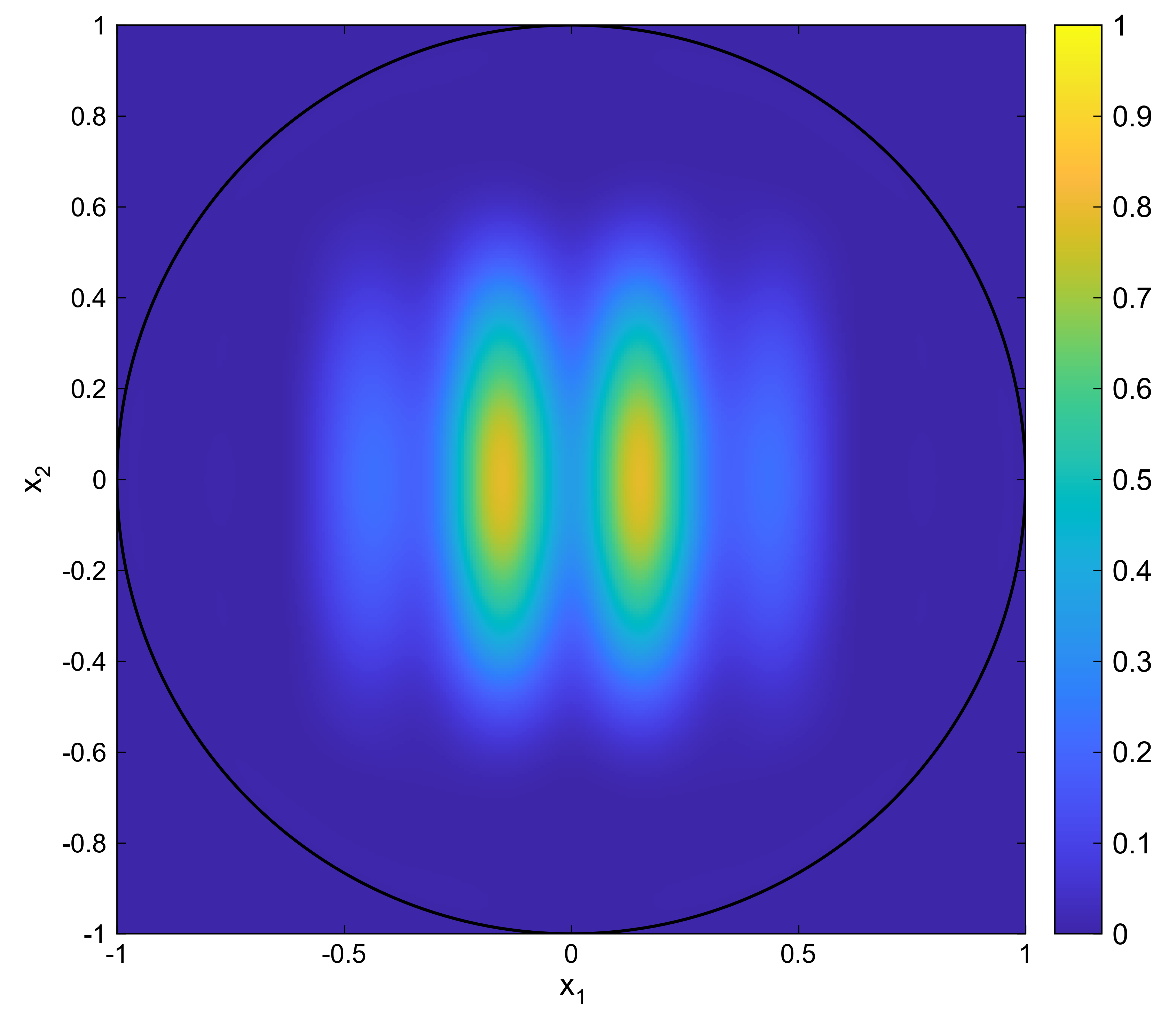}&
\includegraphics[width=0.25\linewidth]{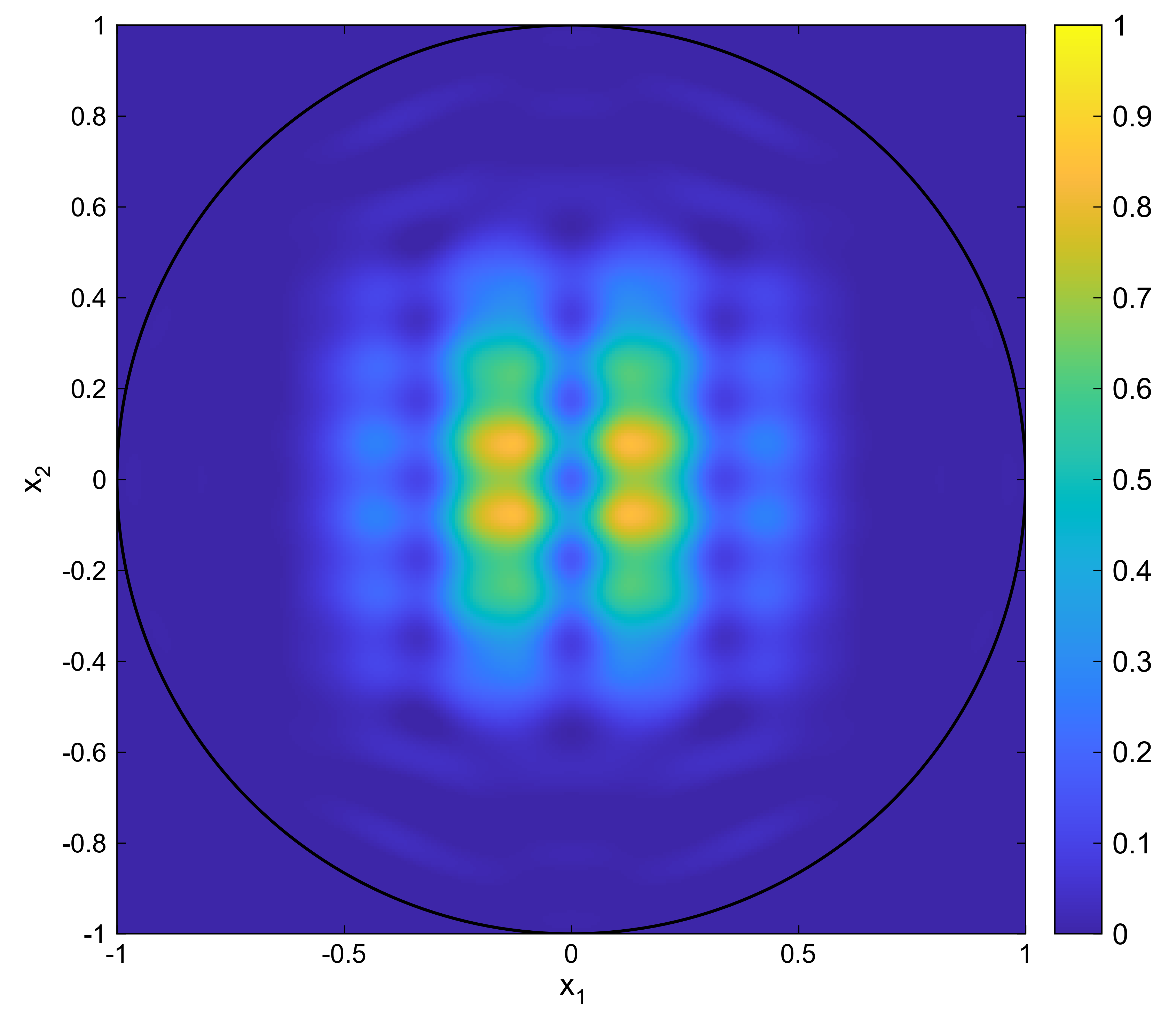}&
\includegraphics[width=0.25\linewidth]{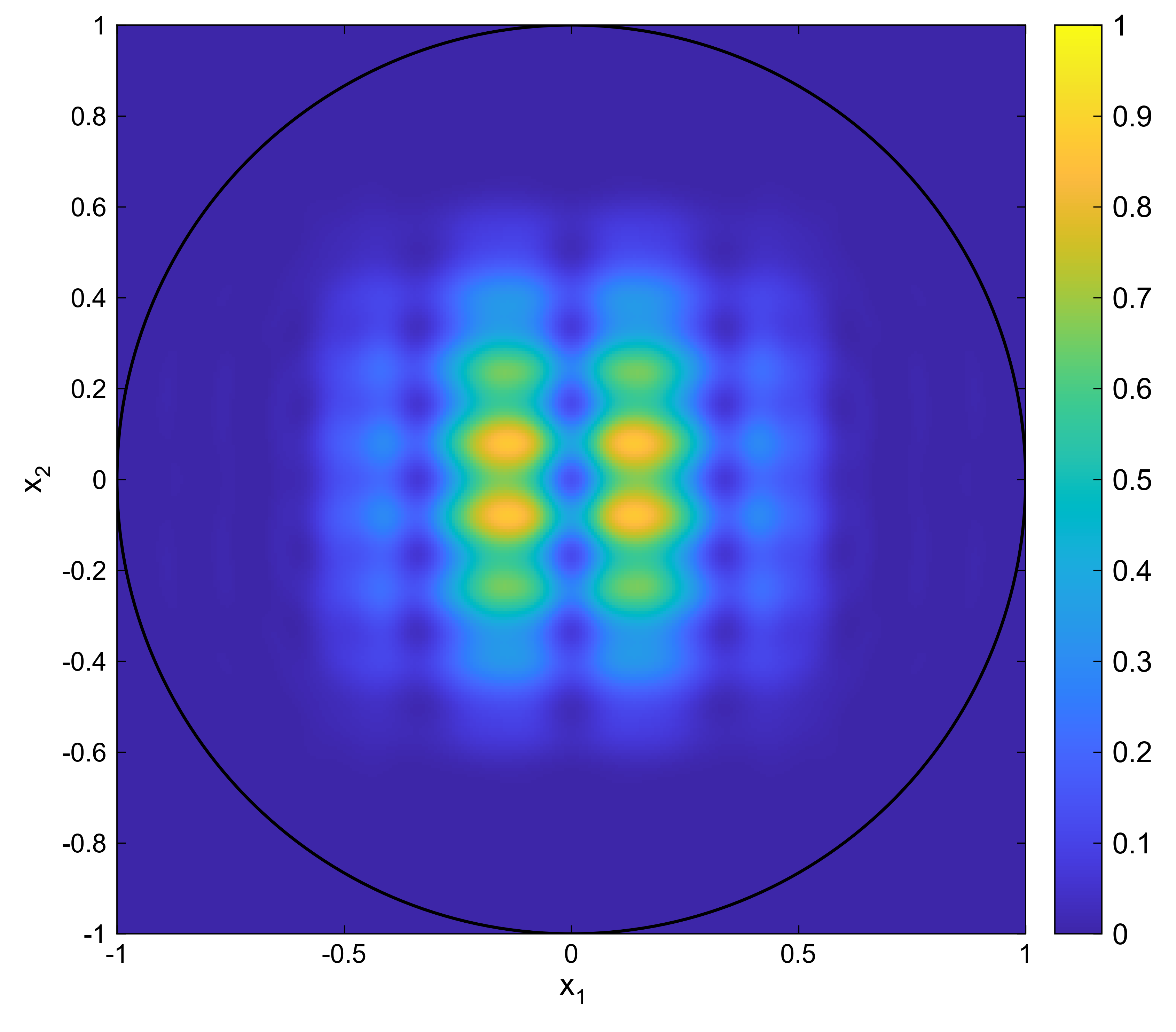}\\
\includegraphics[width=0.25\linewidth]{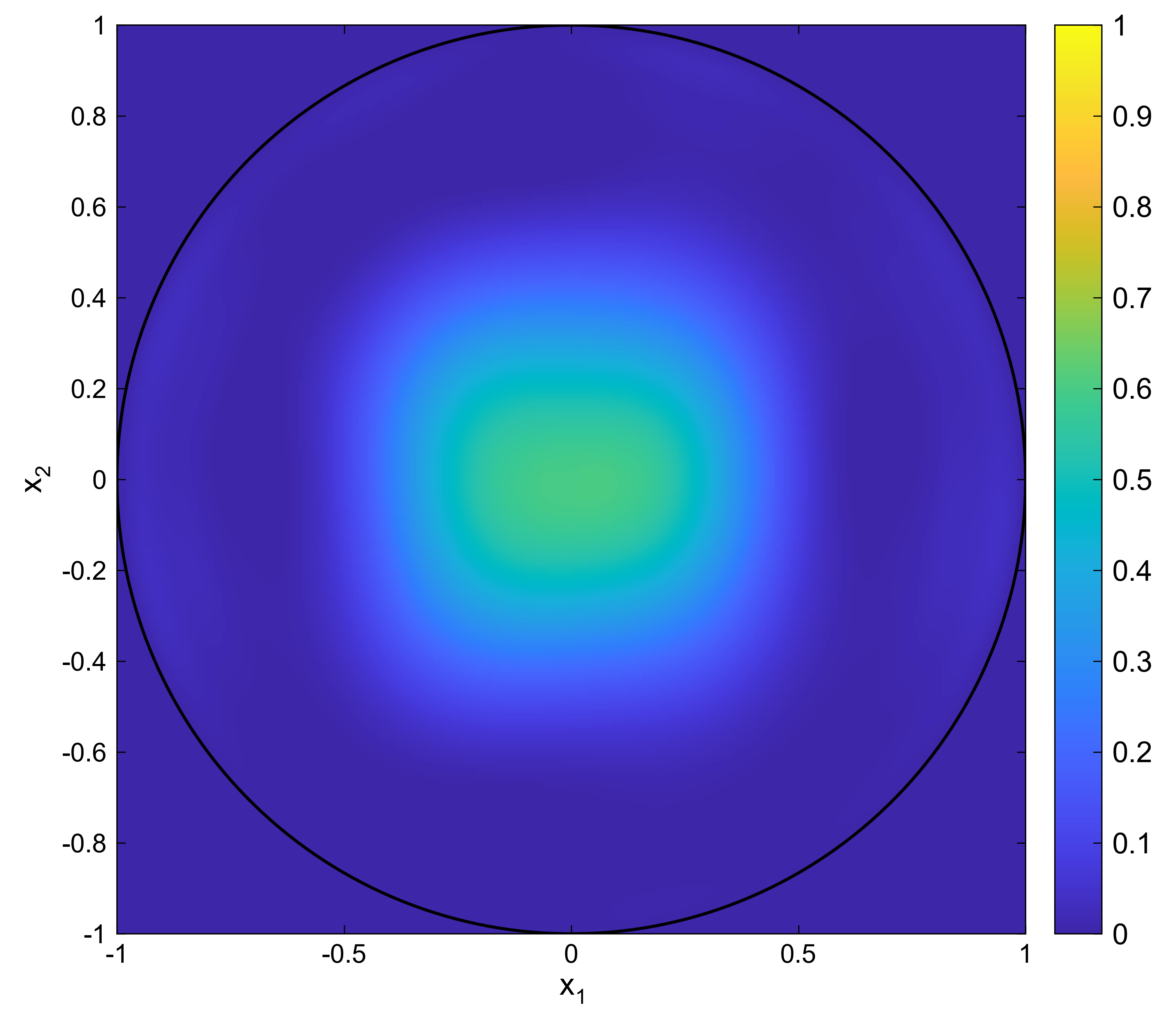}&
\includegraphics[width=0.25\linewidth]{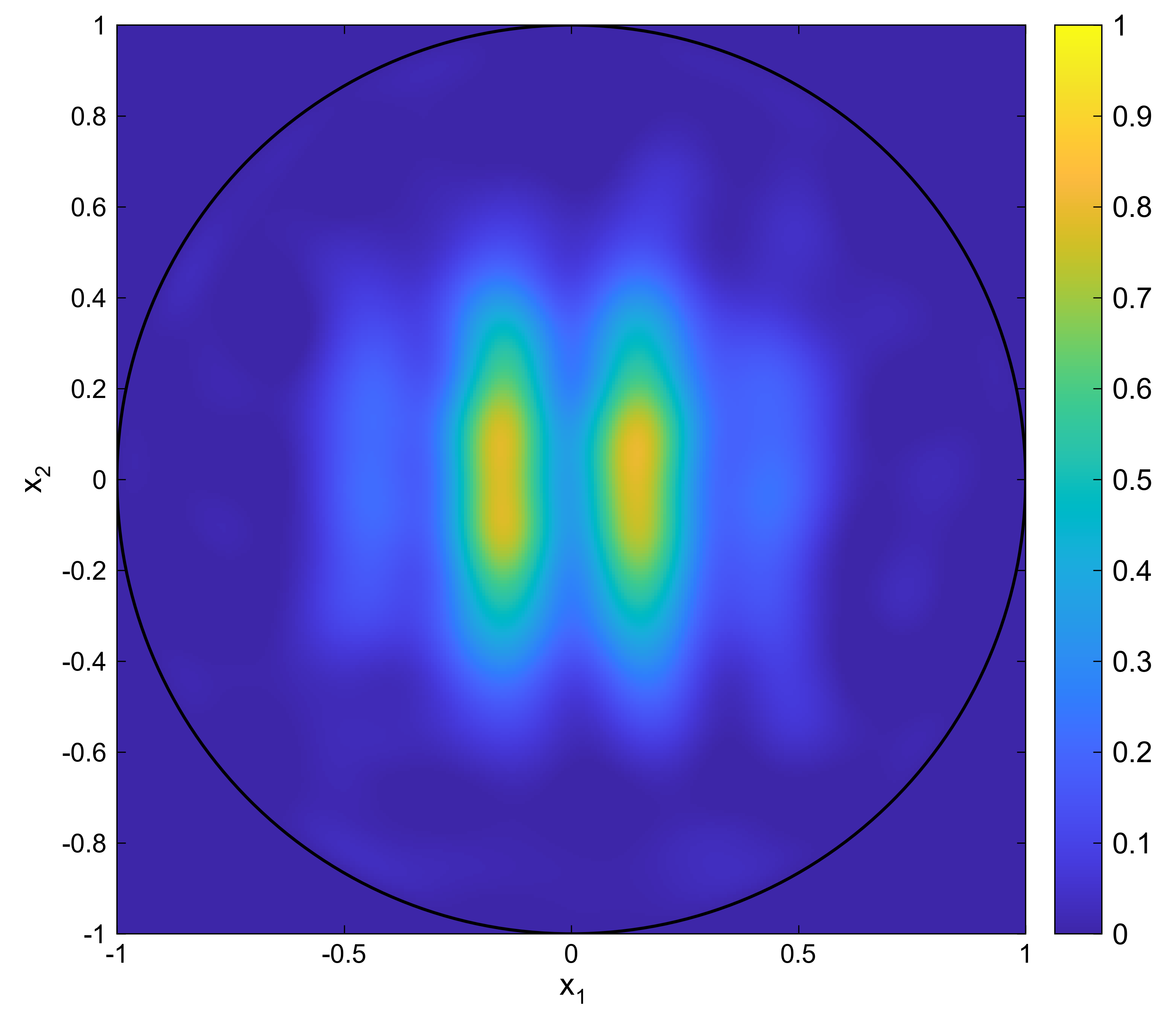}&
\includegraphics[width=0.25\linewidth]{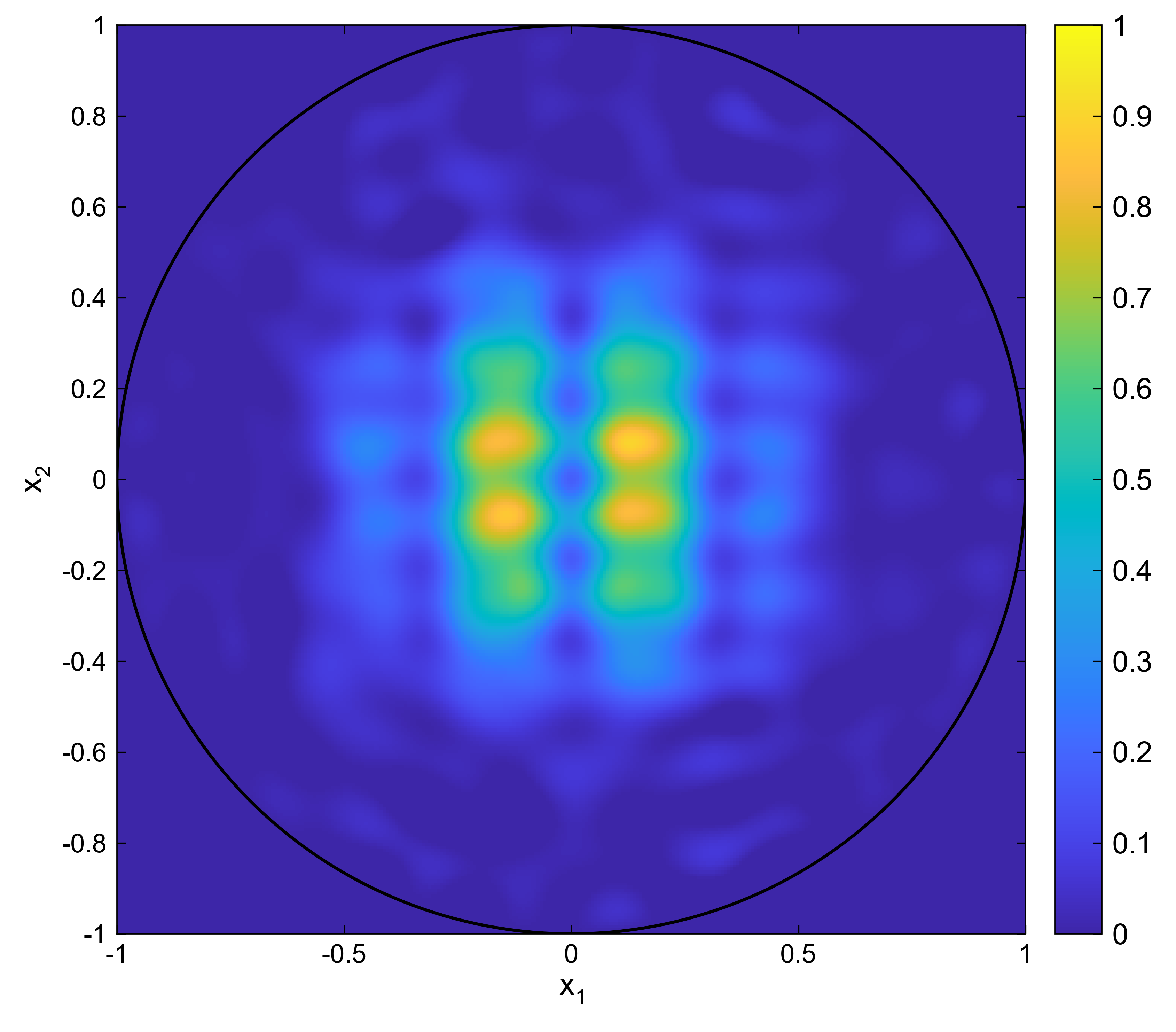}&
\includegraphics[width=0.25\linewidth]{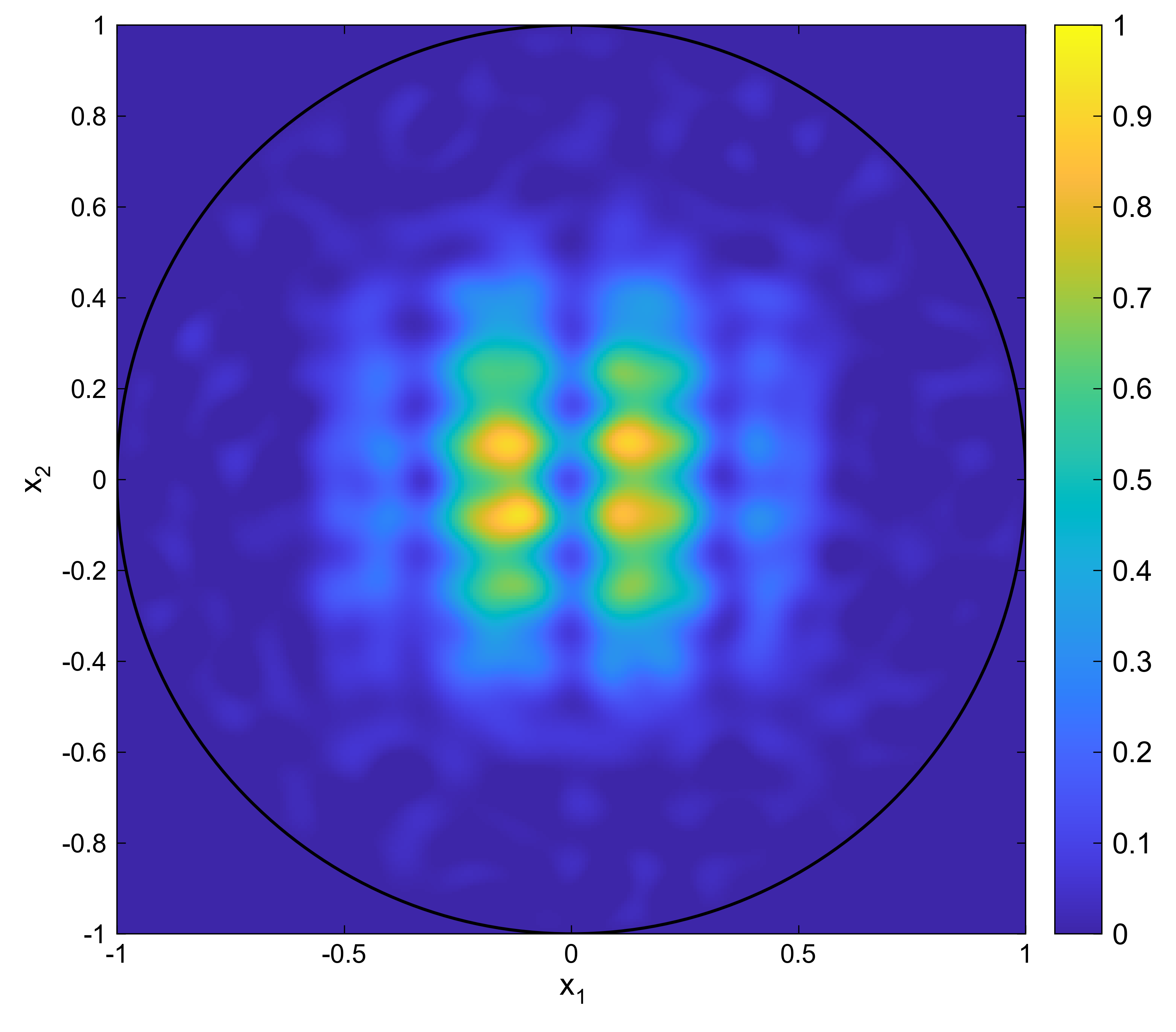}
\\
(a) $k=2\pi$ &(b) $k=2\pi, 4\pi$ &(c) $k=2\pi, 4\pi, 6\pi$ &(d)$k=2\pi, 4\pi, 6\pi, 8\pi$, 
\end{tabular}
\caption{Results by frequency continuation. 1st row: clean data. 2nd row:  data with $20\%$ noise.}
\label{fig:multi_k_hier}
\end{figure}

\subsection{Near-Field Data}
\label{sec: near field data}
\begin{example}[]\label{ex:BAO}
\[
q(x_1,x_2)=
\begin{cases}
0.3(1-3x_1)^2 e^{-9x_1^2-(3x_2+1)^2}\\
\qquad -\left(0.6x_1-27x_1^3-243x_2^5\right)
  e^{-9(x_1^2+x_2^2)}
-0.03e^{-(3x_1+1)^2-9x_2^2},
& x_1^2+x_2^2<1,\\[1ex]
0, & x_1^2+x_2^2\geq 1.
\end{cases}
\]
\end{example}
For this test, we use near-field data to show behavior and reconstruction results similar to the far-field case. We position $50$ point sources and $50$ measurements at evenly distributed locations on $\partial B(0,2)$. Since there is no $p$-transform for near-field data, we use equally weighted least squares for all data. This example is taken from Section 2.6 of \cite{Bao_example}. The efficiency and robustness of iTSVD are again demonstrated by 1) the use of a single frequency $k=4\pi$, which aligns with the highest frequency used in \cite{Bao_example}, and 2) the number of iterations, 8, for both noiseless and noisy data to achieve similar accuracy; see Figure ~\ref{fig:near-field}. In comparison, the reconstruction in \cite{Bao_example} is based on recursive linearization from $k = 0.5\pi$ to $k = 4\pi$ with frequency step size $\delta k = 0.5\pi$, with 20 Landweber iterations for each frequency, and the final relative $L^2$ error is $1.41\times 10^{-2}$ with noiseless data and $2.55\times 10^{-2}$ with $9\%$ noise, which is comparable to our results.

\begin{figure}[htbp]
\centering
\begin{tabular}{c@{\vspace{2pt}}c@{\vspace{2pt}}c@{\vspace{2pt}}c@{\vspace{2pt}}c}
\includegraphics[width=0.2\linewidth]{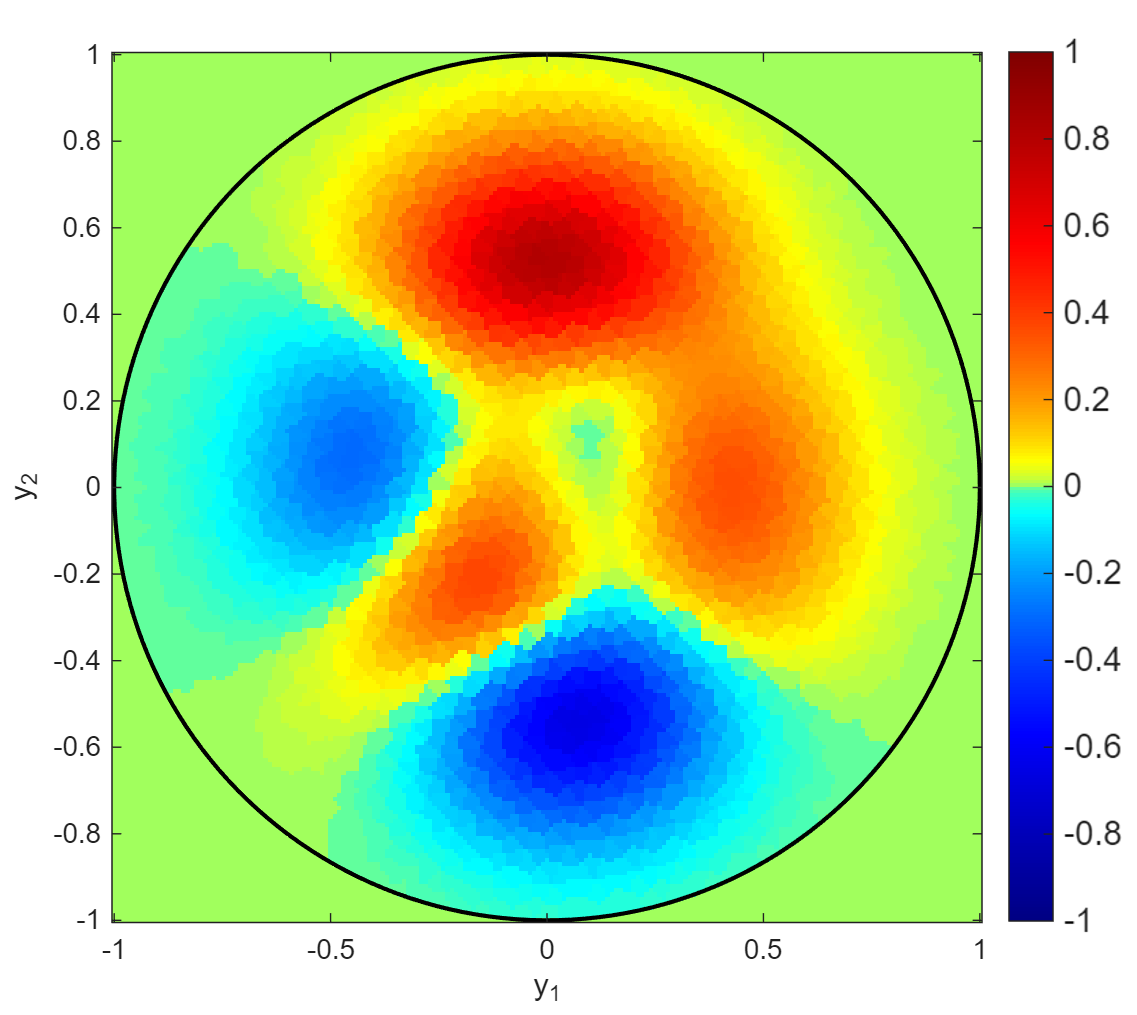}&
\includegraphics[width=0.2\linewidth]{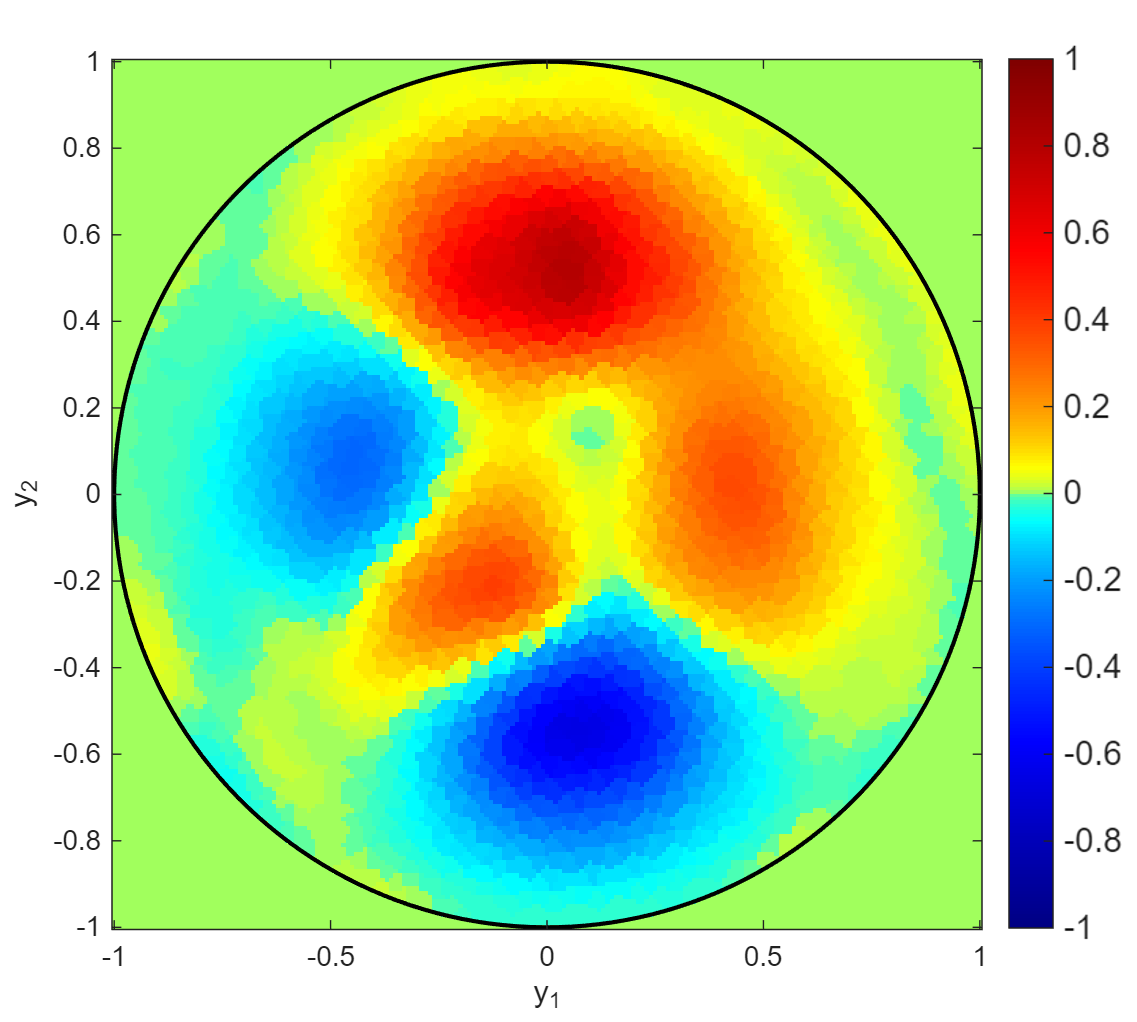}&
\includegraphics[width=0.2\linewidth]{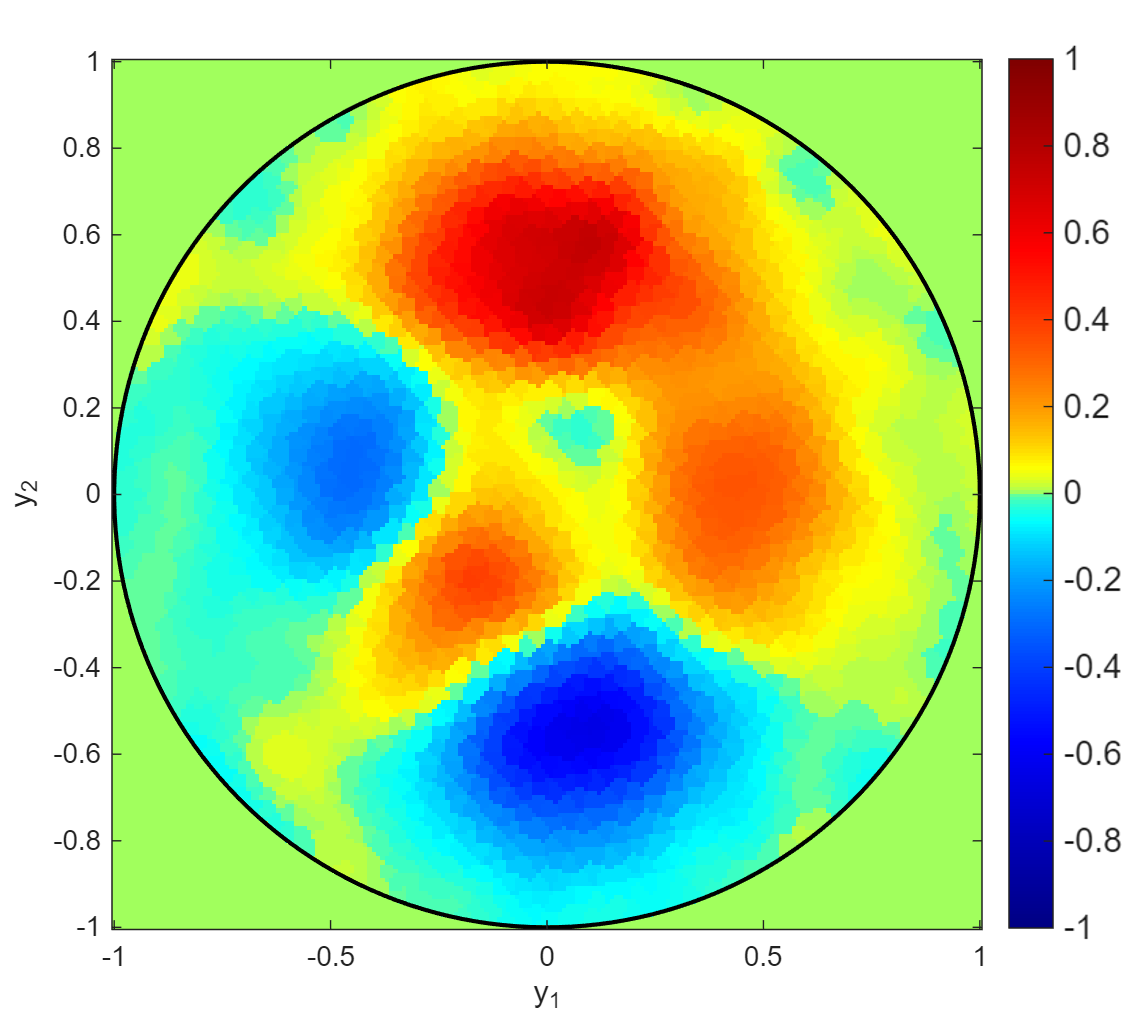}&
\includegraphics[width=0.3\linewidth]{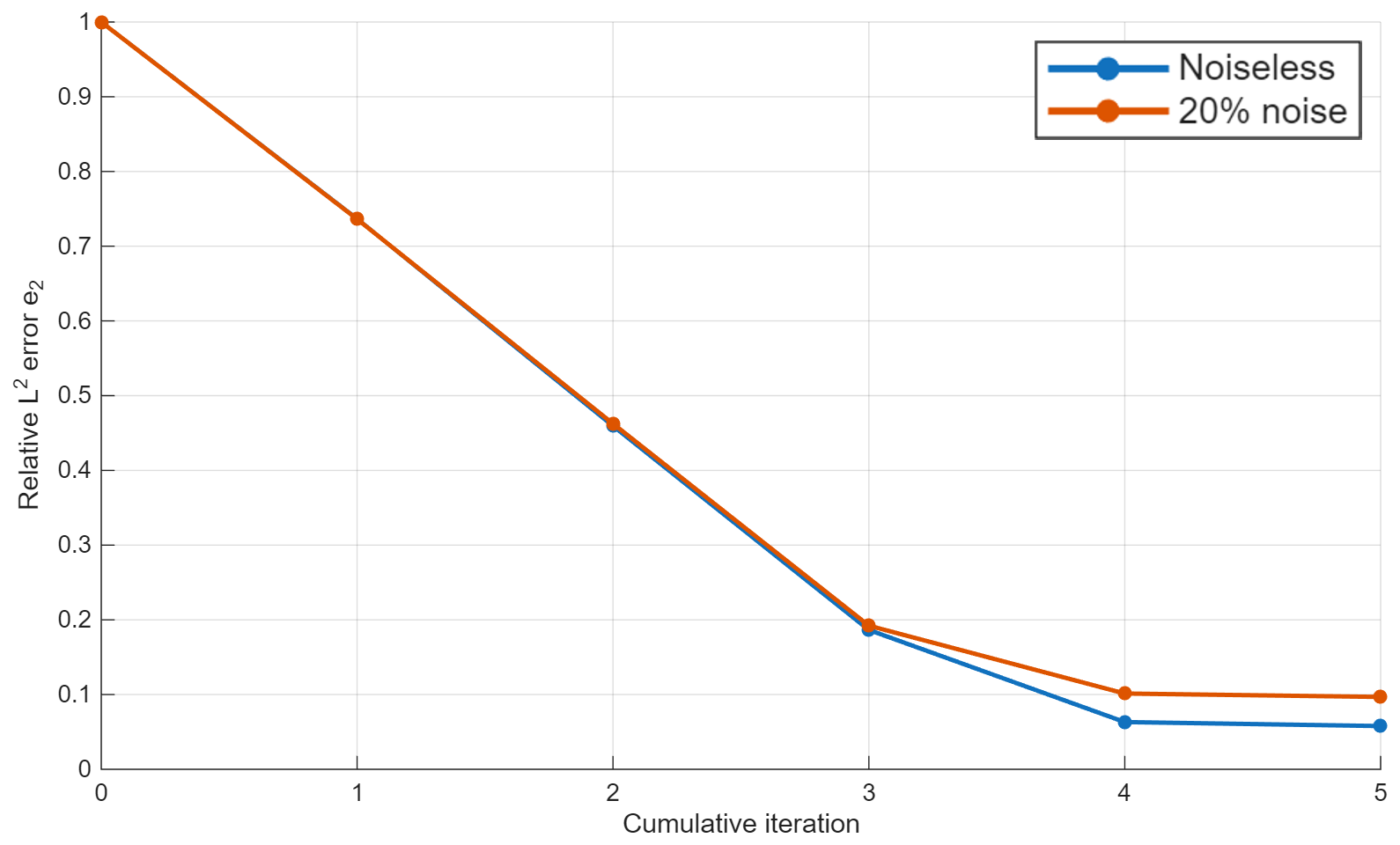}
\\
(a) true $q$ &(b) clean data &(c) data with $20\%$ noise & (d) convergence
\end{tabular}
\caption{Reconstruction results for Example ~\ref{ex:BAO}.}
\label{fig:near-field}
\end{figure}

\paragraph{Consistent threshold for effective rank estimate}
In general, the effective rank can be affected by many factors, e.g., the domain geometry, measurement setups, noise level, etc. Also, the $p$-transform doesn't apply to general measurements either. Here, we propose a different rank estimate strategy to test the near-field case: use RSVD for the linearized operator for the first iteration for a homogeneous medium to reveal the rank structure and use it for subsequent iterations. The two key rationales behind this are: 1) when everything else is fixed, the effective rank, or the dimension of signal space, is quite stable with respect to the medium perturbation for the wave operator; 2) for a homogeneous medium, RSVD is simple and efficient, as no PDE needs to be solved for the matrix-free formulation~\eqref{eq:operator-adj} since the total field is the corresponding incoming wave field. We use this strategy to test on different sets of near-field measurements with $M = 20,30,40,50$ evenly distributed sources and observations on $\partial B(0,2)$. In the first iteration with initial guess $q_0(x)=0$, we perform RSVD and pick a cutoff threshold $\underline{\sigma} = 0.1$ for TSVD solution. The number of singular values above this threshold is used as the effective rank for the rest of the iterations. Figure ~\ref{fig:near field SVD with fixed threshold} shows the singular value at the threshold rank in each iteration for different $M$, which shows consistency and stability in using the rank estimate of the first iteration. The reconstruction results in Figure ~\ref{fig:near field reconstruction with fixed threshold} also demonstrate the effectiveness of the strategy.

\begin{figure}[H]
\centering
\setlength{\tabcolsep}{2pt}

\begin{tabular}{c@{\vspace{2pt}}c@{\vspace{2pt}}c}
\includegraphics[width=0.3\linewidth]{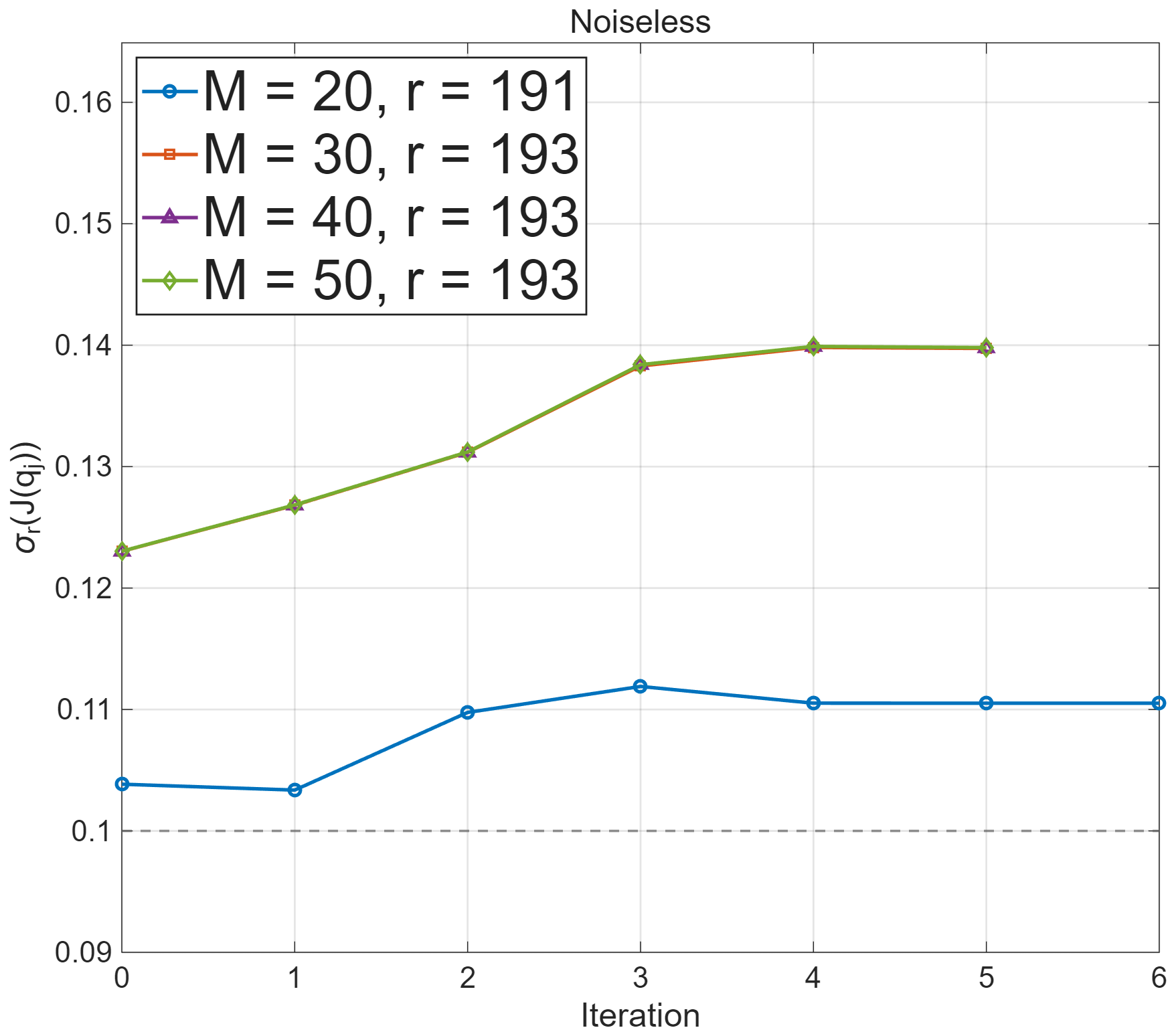}&
\includegraphics[width=0.3\linewidth]{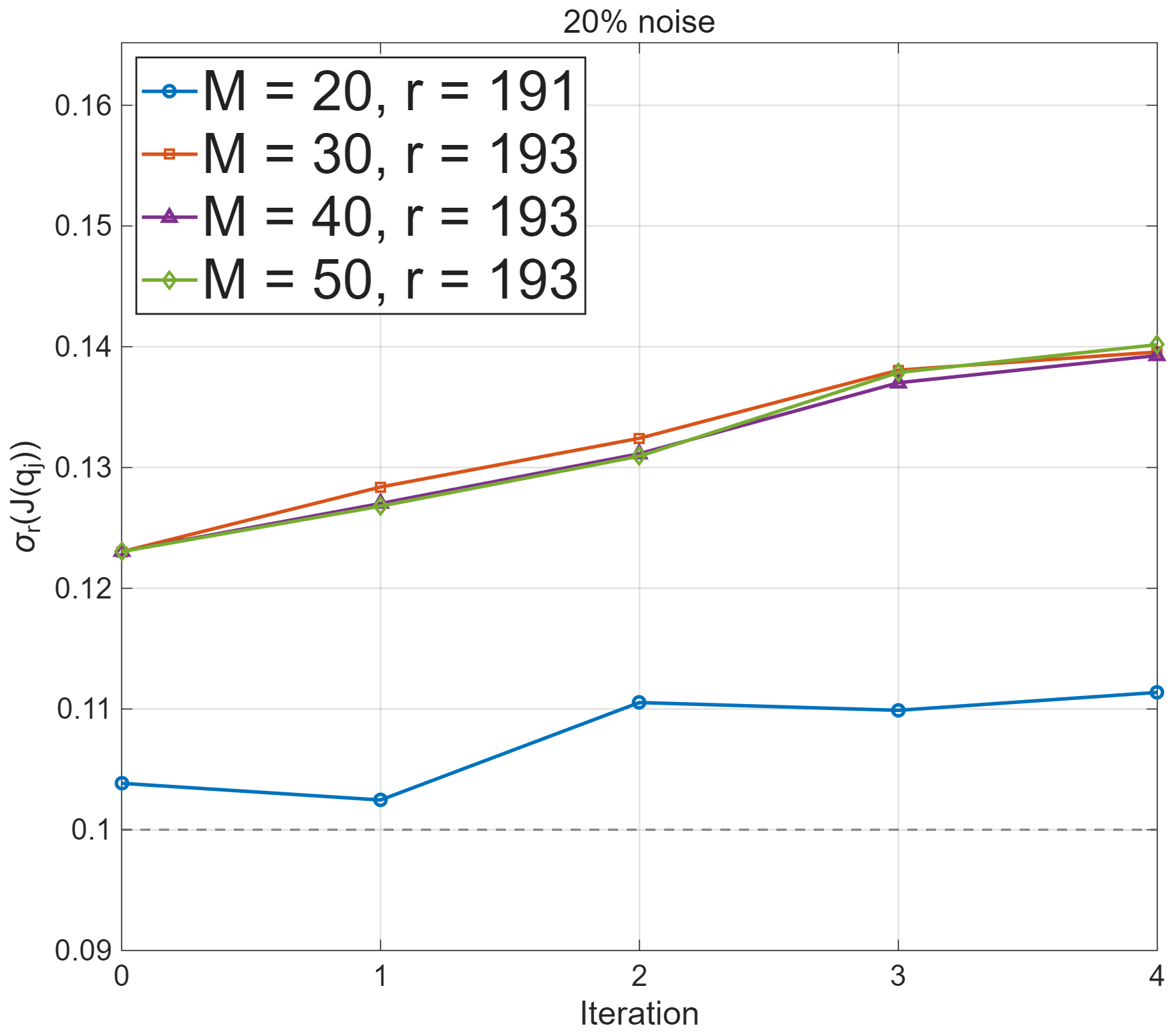}
\\
(a) clean data &(b) data with $20\%$ noise
\end{tabular}
\caption{Singular value at the threshold rank in each iteration with different measurements.}
\label{fig:near field SVD with fixed threshold}

\par\vspace{8pt}

\begin{tabular}{c@{\vspace{2pt}}c@{\vspace{2pt}}c@{\vspace{2pt}}c}
\includegraphics[width=0.2\linewidth]{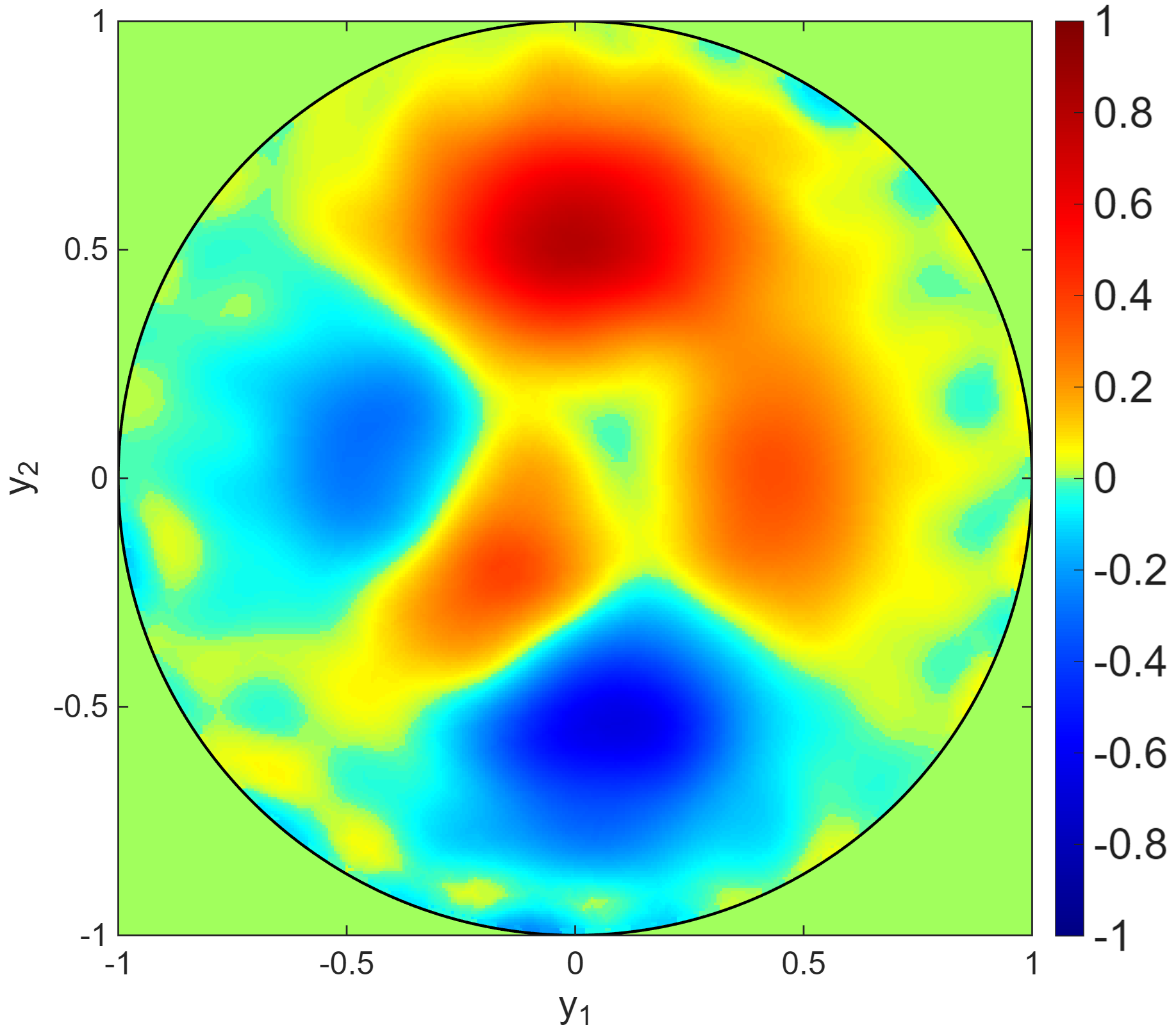}&
\includegraphics[width=0.2\linewidth]{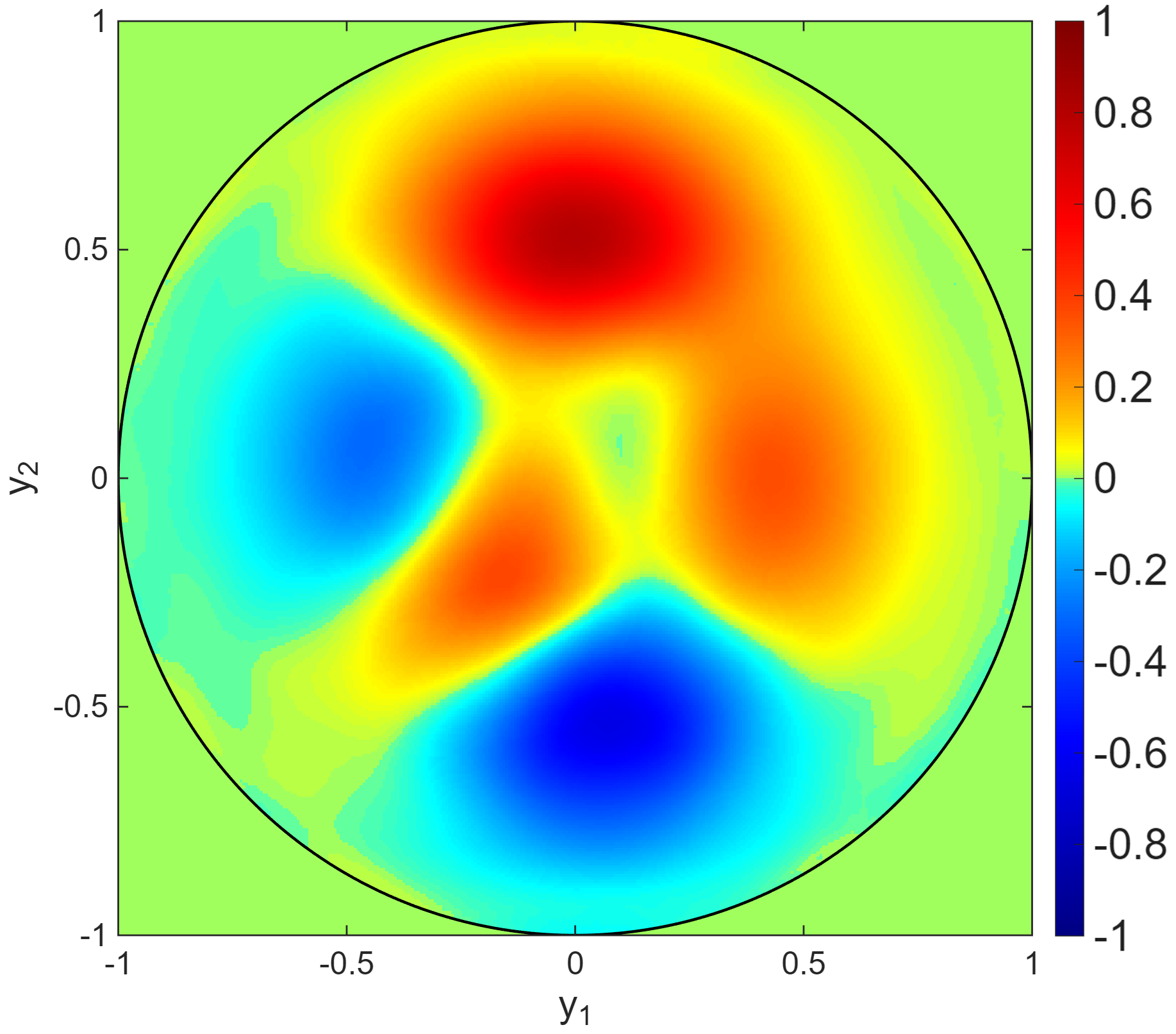}&
\includegraphics[width=0.2\linewidth]{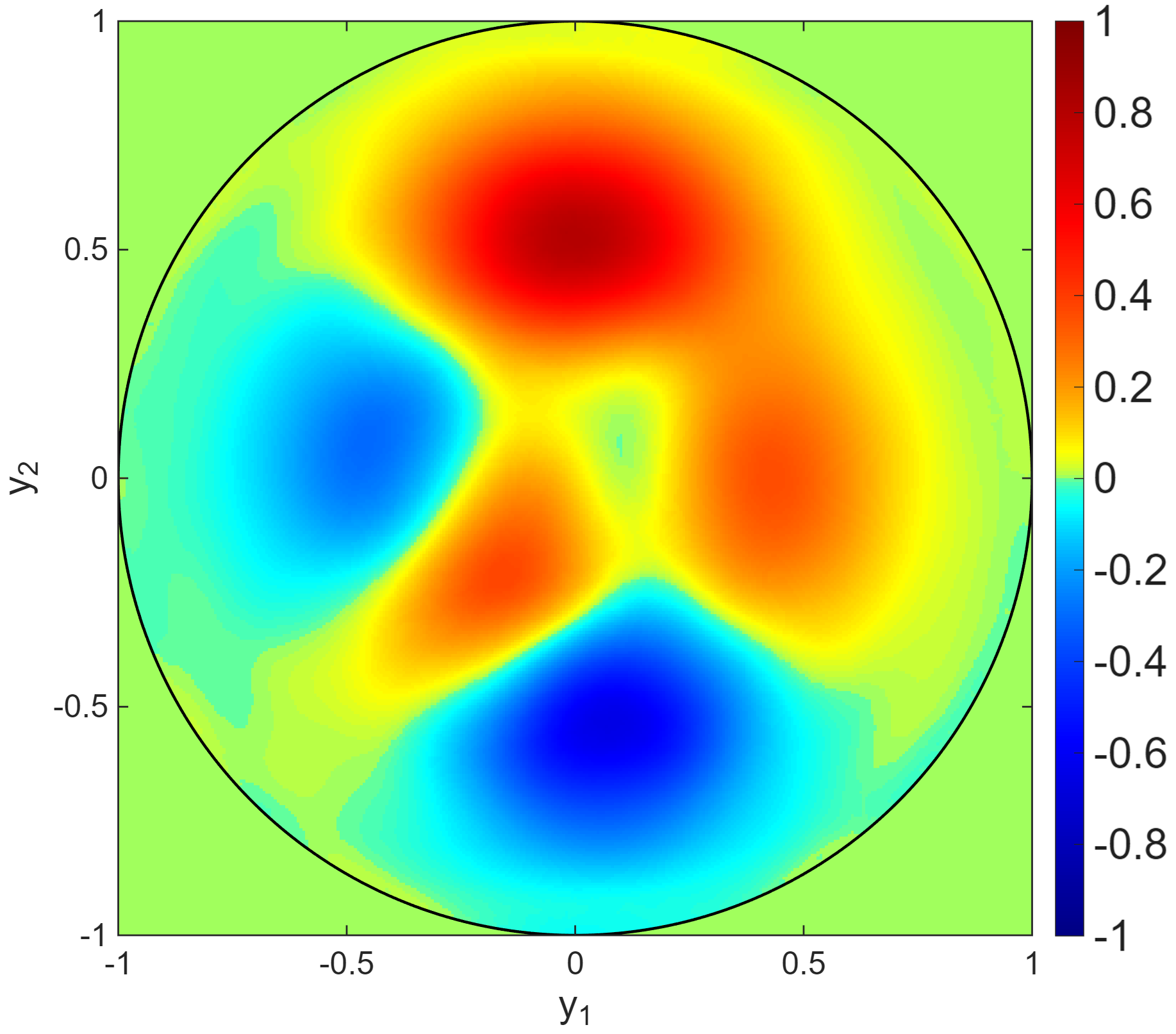}&
\includegraphics[width=0.2\linewidth]{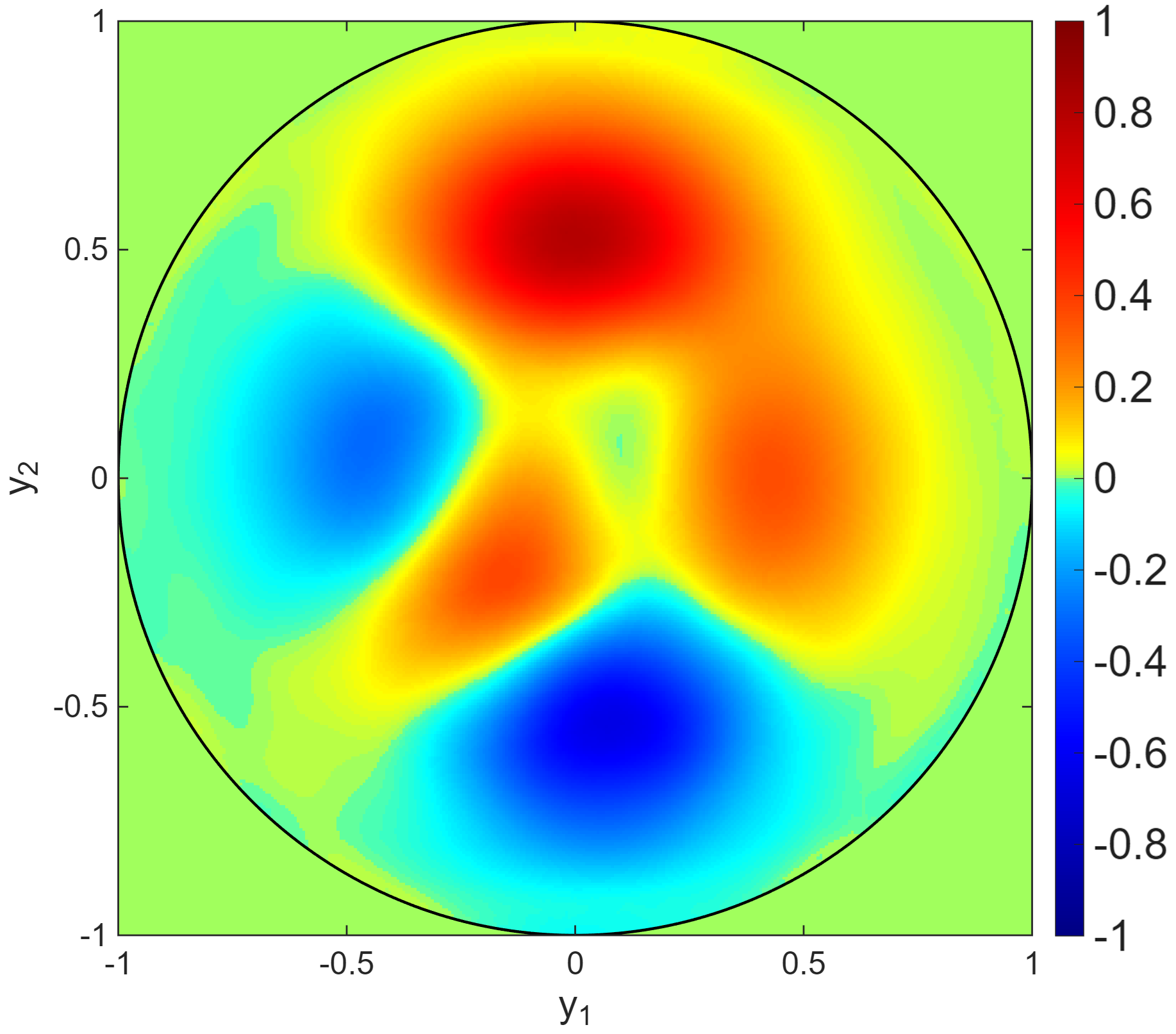}\\
\includegraphics[width=0.2\linewidth]{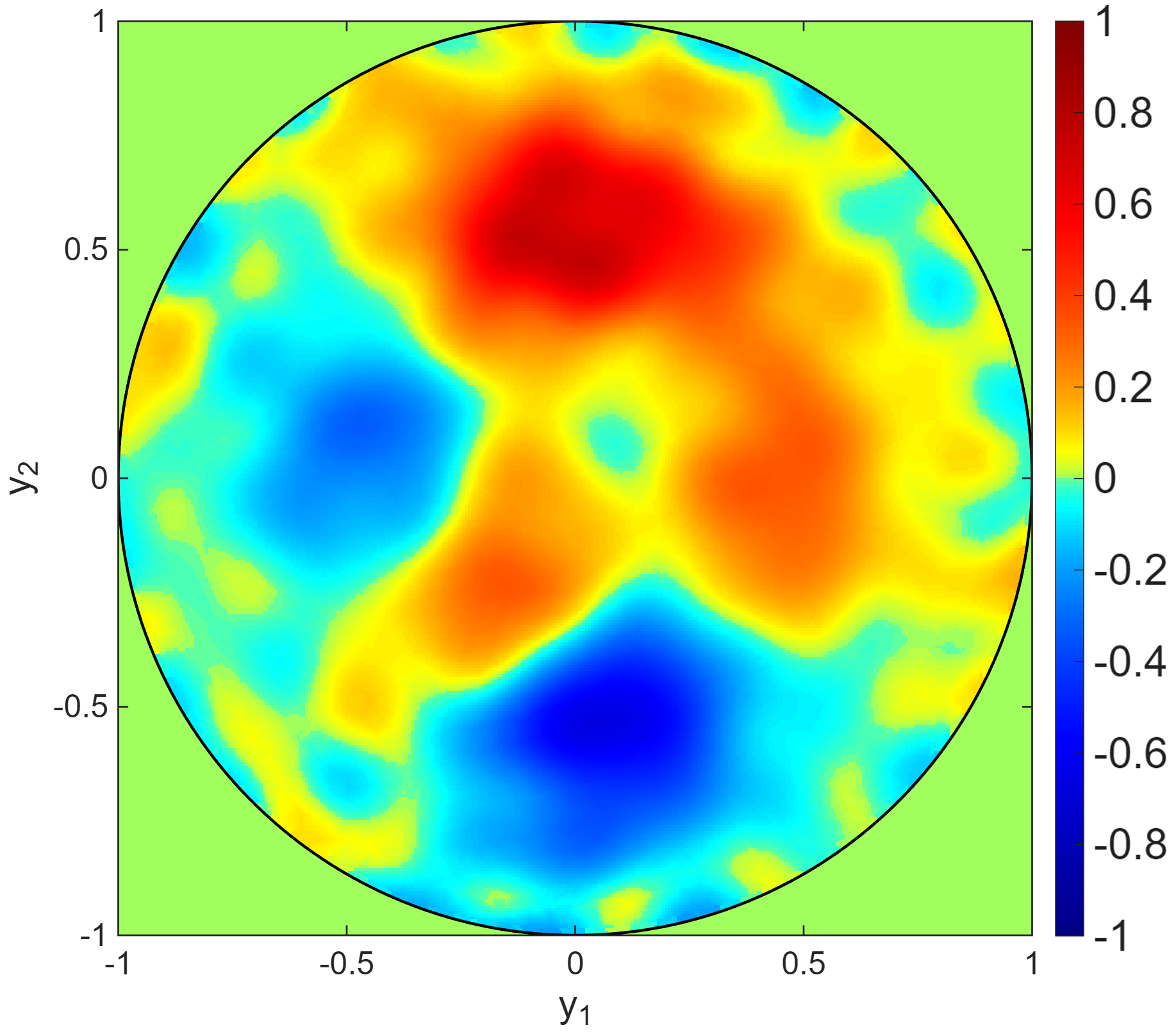}&
\includegraphics[width=0.2\linewidth]{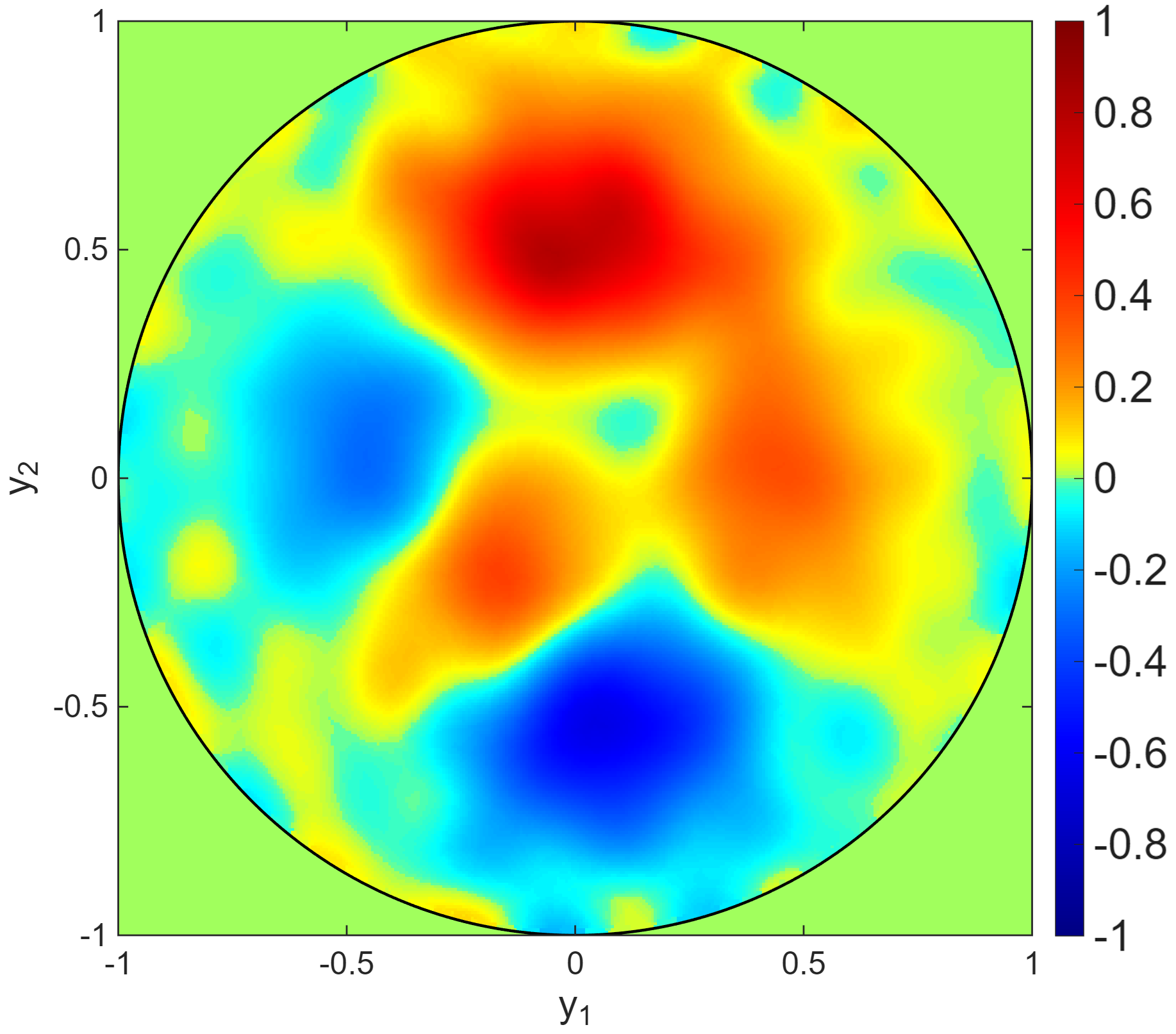}&
\includegraphics[width=0.2\linewidth]{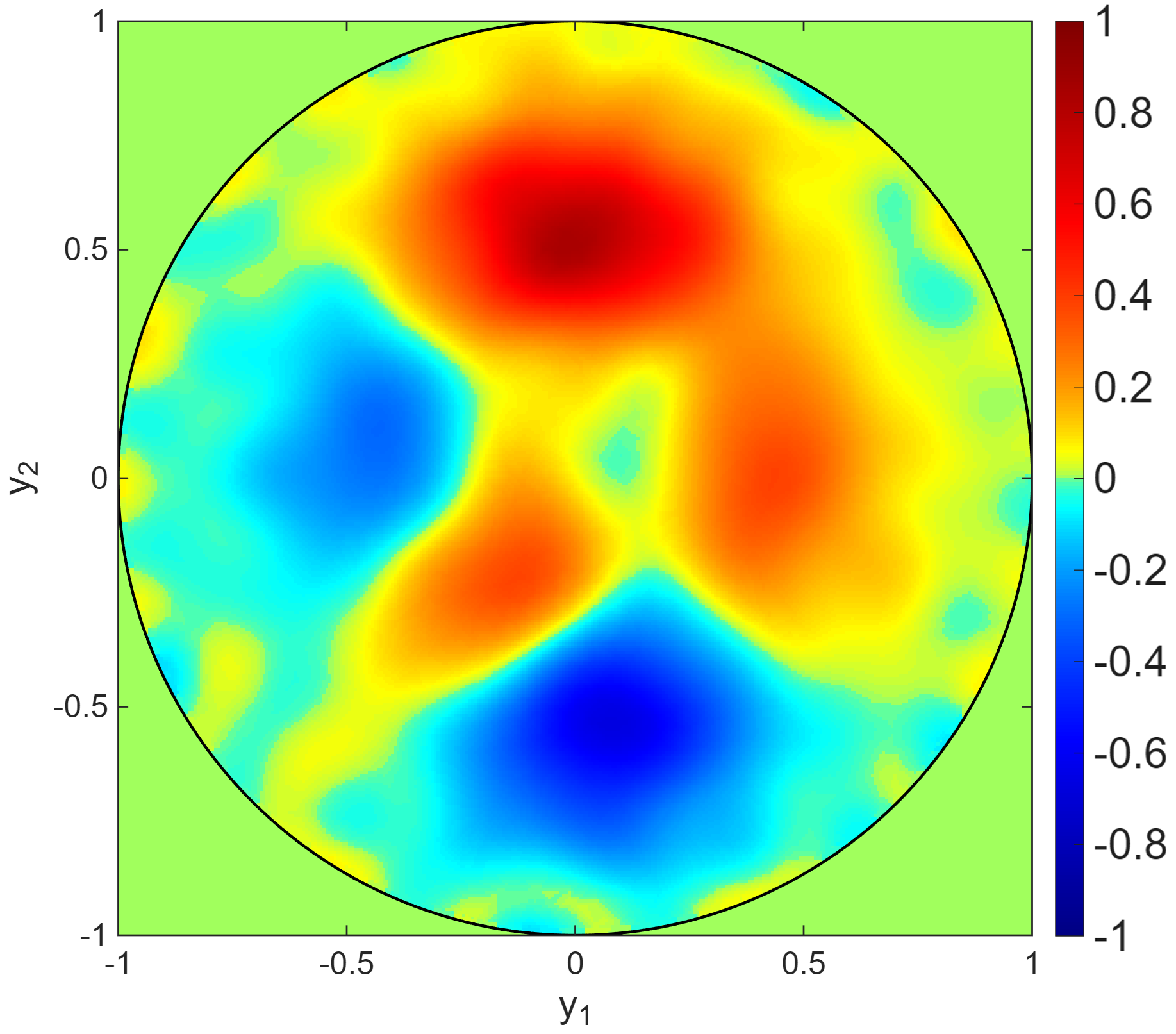}&
\includegraphics[width=0.2\linewidth]{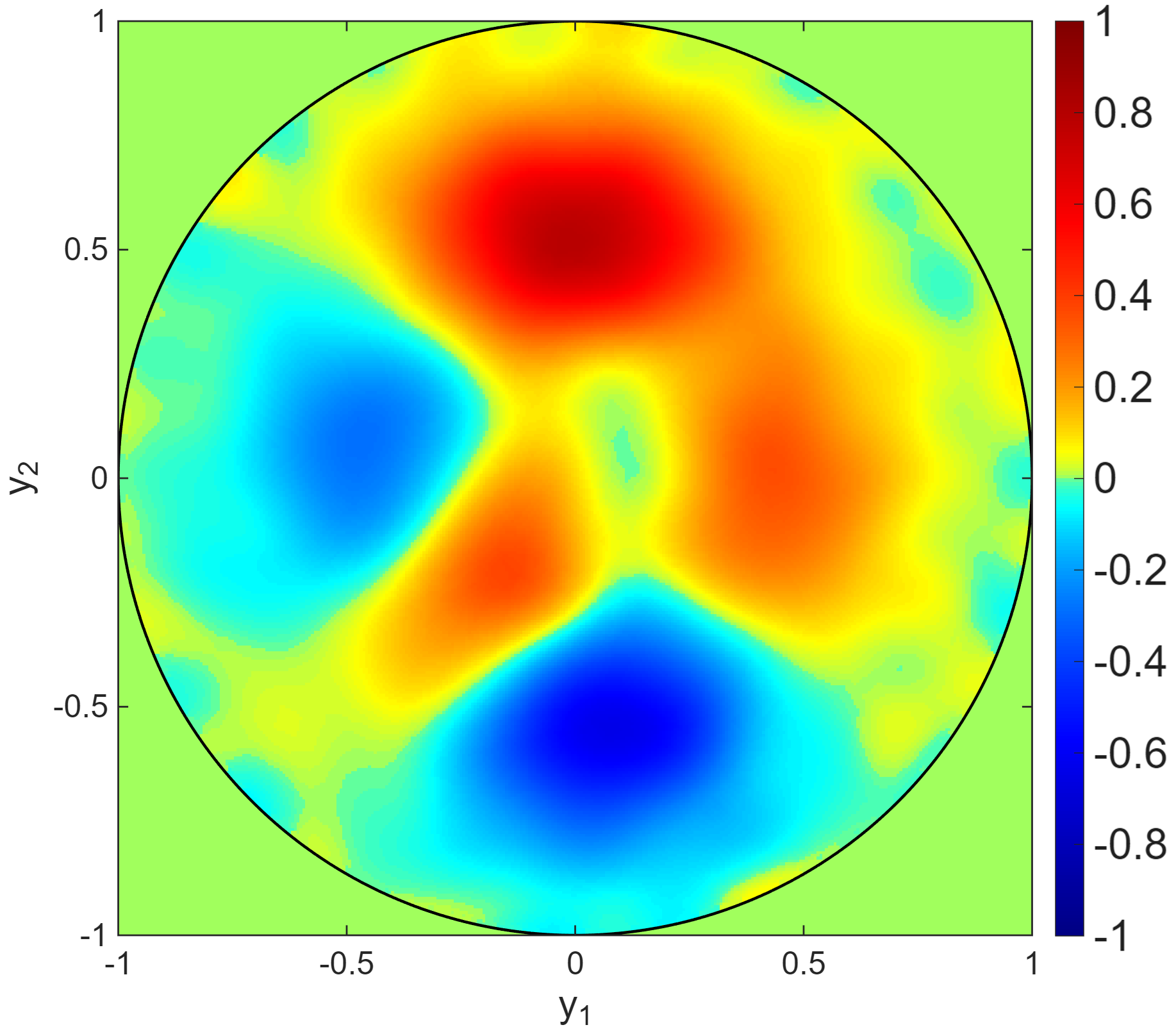}
\\
(a) $M=20$ &(b) $M=30$ &(c) $M=40$&(d) $M=50$
\end{tabular}
\caption{Reconstruction results for different sets of measurements. 1st row: clean data. 2nd row: data with $20\%$ noise.}
\label{fig:near field reconstruction with fixed threshold}
\end{figure}

\section{Conclusion and future work}
\label{sec: conclusion}
In this work, an iterative truncated SVD (iTSVD) algorithm based on Newton's method is proposed to solve inverse scattering problems. The efficiency and robustness of iTSVD come from two key ingredients: 1) an effective rank estimate for the linearized operator for a stable TSVD inversion, taking into account the forward model and measurement; 2) random SVD based on a matrix-free formulation, both feasible and efficient in practice. Extensive experiments are presented and compared with gradient-based methods to demonstrate the efficiency of the proposed method.

We plan to study and extend iTSVD to other inverse problems based on different forward models in the future.

\paragraph{Acknowledgment.} The authors thank Dr. Shixu Meng and Dr. Thomas G. Anderson for helpful discussion and for sharing their codes. H. Zhao is partially supported by NSF DMS 2309551.

\appendix\label{sec: appendix}

\section{Proof of Adjoint Equivalence}\label{sec: adjoint equivalence}
In this section, we prove \eqref{eq:adj-equivalence} in 2D. Extension to 3D is straightforward.

Choose \(R>0\) sufficiently large such that \(\Omega\subset B_R\subset \mathbb{R}^2\). Since
\(q_0\) and \(\delta q\) are supported in \(\Omega\), the source terms in
\eqref{eq:lin-PDE} and \eqref{eq:adj-PDE} are also supported
in \(\Omega\).
Multiplying \eqref{eq:lin-PDE} by \(\lambda_{q_0}^s\), multiplying
\eqref{eq:adj-PDE} by \(\delta u\), and subtracting, we obtain
\begin{equation}\label{eq:subtract_pde_2d}
\lambda_{q_0}^s\Delta\delta u
-
\delta u\,\Delta\lambda_{q_0}^s
=
-k^2\lambda_{q_0}^s\delta q\,u^t_{q_0}
+
k^2\delta u\,q_0 e^{-ik\hat x\cdot y}.
\end{equation}
Integrating \eqref{eq:subtract_pde_2d} over \(B_R\) and applying Green's
second identity gives
\begin{align}
\int_{\partial B_R}
\left(
\lambda_{q_0}^s\partial_\nu\delta u
-
\delta u\,\partial_\nu\lambda_{q_0}^s
\right)dS
&=
-k^2\int_\Omega
\lambda_{q_0}^s\delta q\,u^t_{q_0}\,dy
+
k^2\int_\Omega
\delta u\,q_0 e^{-ik\hat x\cdot y}\,dy .
\label{eq:green_identity_2d}
\end{align}
It remains to show that the boundary term vanishes as \(R\to\infty\). Outside
a sufficiently large ball, both \(\delta u\) and \(\lambda_{q_0}^s\) solve the
homogeneous Helmholtz equation
\[
(\Delta+k^2)v=0
\]
and satisfy the outgoing Sommerfeld radiation condition. Therefore, they admit the far-field asymptotics
\begin{equation}\label{eq:delta_u_far_field_2d}
\delta u(y)
=
\frac{e^{ikR}}{\sqrt R}
\delta u^\infty(\hat y)
+
o(R^{-\frac{1}{2}}),
\end{equation}
and
\begin{equation}\label{eq:lambda_far_field_2d}
\lambda_{q_0}^s(y)
=
\frac{e^{ikR}}{\sqrt R}
\lambda^\infty(\hat y)
+
o(R^{-\frac{1}{2}}),
\end{equation}
where \(R=|y|\) and \(\hat y=y/|y|\). Hence
\[
\delta u(y)=O(R^{-\frac{1}{2}}),
\qquad
\lambda_{q_0}^s(y)=O(R^{-\frac{1}{2}}).
\]
Since \(\hat y\) is fixed under the radial derivative,
\[
\partial_r\left(\frac{e^{ikR}}{\sqrt R}v^\infty(\hat y)\right)
=
\left(ik-\frac{1}{2R}\right)
\frac{e^{ikR}}{\sqrt R}v^\infty(\hat y).
\]
Hence
\[
\partial_r v
=
ikv-\frac{1}{2R}v+o(R^{-\frac{1}{2}}).
\]
% and therefore
% \[
% \partial_r v-ikv
% =
% -\frac{1}{2R}v+O(R^{-3/2})
% =
% O(R^{-3/2}),
% \]
% because \(v=O(R^{-1/2})\).
Let $v$ = $\delta u$ and $\lambda_{q_0}^s$, then
\[
\partial_r\delta u
=
ik\delta u
-
\frac{1}{2R}\delta u
+
o(R^{-\frac{1}{2}}),
\]
and
\[
\partial_r\lambda_{q_0}^s
=
ik\lambda_{q_0}^s
-
\frac{1}{2R}\lambda_{q_0}^s
+
o(R^{-\frac{1}{2}}).
\]
On \(\partial B_R\), one has \(\partial_\nu=\partial_r\). Therefore as $R \rightarrow \infty$, 
\begin{align*}
\lambda_{q_0}^s\partial_r\delta u
-
\delta u\,\partial_r\lambda_{q_0}^s
&=
\lambda_{q_0}^s
\left(
ik\delta u-\frac{1}{2R}\delta u+o(R^{-\frac{1}{2}})
\right)
-
\delta u
\left(
ik\lambda_{q_0}^s-\frac{1}{2R}\lambda_{q_0}^s+o(R^{-\frac{1}{2}})
\right) = o(R^{-1}).
\end{align*}
Since $|\partial B_R|=2\pi R=O(R),$
we obtain
\[
\int_{\partial B_R}
\left(
\lambda_{q_0}^s\partial_r\delta u
-
\delta u\,\partial_r\lambda_{q_0}^s
\right)dS
=
o(1)
\qquad \text{as }R\to\infty .
\]
Passing to the limit \(R\to\infty\) in \eqref{eq:green_identity_2d}, we obtain
\[
0
=
-k^2\int_\Omega
\lambda_{q_0}^s\delta q\,u^t_{q_0}\,dy
+
k^2\int_\Omega
\delta u\,q_0 e^{-ik\hat x\cdot y}\,dy .
\]
Dividing by \(k^2\) gives
\[
\int_\Omega 
\lambda_{q_0}^s(y;\hat{x})\,
u^t_{q_0}(y;\hat\theta)\,
\delta q(y)\,dy
=
\int_\Omega 
e^{-ik\hat x\cdot y}\,
q_0(y)\,
\delta u(y;\hat\theta)\,dy .
\]
This proves \eqref{eq:adj-equivalence}.

\section{More Implementation Details}
\label{sec: implementation}

% Since $\hat{u} \in L^2(B,\mathbb{C})$, $\hat{\Jacobian}_{q_0}: L^2(\Omega;\mathbb{R}) \to L^2(B,\mathbb{C})$, and $\delta q \in L^2(\Omega;\mathbb{R})$, we split the real and imaginary part of $\hat{\Jacobian}_{q_0}$ to obtain a viable basis for $\delta q$, i.e.,
% \begin{equation}\label{eq:real-aug-ls_continuous}
% \min_{\delta q}
% \left\|
% \begin{bmatrix}
% \Re(\hat{u}^{\infty}) \\
% \Im(\hat{u}^{\infty})
% \end{bmatrix}
% -
% \begin{bmatrix}
% \Re(\hat{\Jacobian}_{q_0})\\
% \Im(\hat{\Jacobian}_{q_0})
% \end{bmatrix}
% \,\delta q
% \right\|_{L^2(B,\mathbb{R}^2)}=: \|\hat{u}^\infty_R - \hat{\Jacobian}_{q_0,R}\,\delta q \|_{L^2(B,\mathbb{R}^2)}.
% \end{equation}
% Equivalence between the norms on $L^2(B,\mathbb{C})$ and $L^2(B,\mathbb{R}^2)$ is asserted in \cite{BLECHER2026111226} so that we can carry all the estimates by considering $\hat{\Jacobian}_{q_0}$. The forward solutions and adjoint solutions are obtained by solving the corresponding integral equations:
% \[
% (I-\mathcal K_{q_0})u_{q_0}^t(y;\hat{\theta}) = u^{in}(y;\hat{\theta}),
% \]
% \[
% (I-\mathcal K_{q_0}) \lambda_{q_0}^t(y;\hat{x}) = e^{-ik\hat{x}\cdot y}.
% \]
% The corresponding mesh is generated by a Julia package, while the algorithm is implemented in MATLAB. Specifically, the mesh size is set to $h = 0.1$, resulting in $2283$ elements in $\Omega$. We also choose to set $h = 0.05$, resulting in $8916$ elements for the sake of accuracy when a higher frequency is chosen.

{\bf Rank estimate} 
In our numerical experiments, when the Shannon number $S$ in~\eqref{eq:Shannon number} is used as the effective rank estimate, we treat the logarithmic singular-value spectrum as a planar curve \(\mathbf{r}(t)=(x(t),y(t))=(t,\log_{10}\sigma(t))\), and search within a small range around $S$ for the point with the maximum curvature in the spectrum to characterize the transition between the signal subspace and the noise subspace in each iteration. We denote $\hat{S}$ as the index of such a point which is close to $S$ in practice. We numerically approximate the curvature \(\kappa(t)=\dfrac{|y''(t)|}{\left(1+y'(t)^2\right)^{3/2}}\), where \(y(t)=\log_{10}\sigma(t)\) by using \(y_i=\log_{10}(\max(\sigma_i,\varepsilon_{\mathrm{mach}}))\), \(y_i'\approx (y_{i+1}-y_{i-1})/2\), \(y_i''\approx y_{i+1}-2y_i+y_{i-1}\), to compute $\kappa_i$.  The curvature is evaluated in a search window \([0.9S,1.1S]\). In addition, we select the beginning of the sharp bend in the singular-value decay by choosing the first index satisfying \(\kappa_i\geq0.8\max_j\kappa_j\) as $\hat{S}$ for stability reasons.

% \begin{figure}[H]
% \centering
% \begin{tabular}{c@{\vspace{2pt}}c@{\vspace{2pt}}cc}
% \hspace{-0.3cm}\includegraphics[width=0.25\linewidth]{3D/fig_profile_x1.png}&
% \includegraphics[width=0.25\linewidth]{3D/fig_profile_x2.png}&
% \includegraphics[width=0.25\linewidth]{3D/fig_profile_x3.png} &
% \hspace{-0.3cm}    \includegraphics[width=0.25\linewidth]{3D/fig_convergence.png}
% \\
% (a) $x_2 = x_3 = 0$ &(b) $x_1 = x_3 = 0$ &(c) $x_1 = x_2 = 0$ & (d) convergence
% \end{tabular}
% \caption{Reconstruction results for Example ~\ref{ex:3d-five-gaussians}.}
% \label{fig:3D numerical}
% \end{figure}

% \bibliographystyle{plain}  
\bibliography{reference}

\end{document}